\documentclass[12pt,a4paper]{amsart} 
\usepackage[utf8]{inputenc}
\usepackage[english]{babel}

\usepackage{amsmath}
\usepackage{amsfonts}
\usepackage{amssymb}
\usepackage{color}

\usepackage{amsthm}

\newtheorem{theorem}{Theorem} [section]
\newtheorem{proposition}[theorem]{Proposition}
\newtheorem{corollary}[theorem]{Corollary} 
\newtheorem{lemma}[theorem]{Lemma}
\newtheorem{conjecture}{Conjecture}
\newtheorem{open}{Open Question}

\newcounter{defno}

\usepackage{graphics}
\usepackage{epsfig}

\usepackage{url}
\usepackage{hyperref}

\usepackage[a4paper,top=2.5cm,bottom=2.8cm,
left=3.2cm,right=3.2cm]{geometry}
\markleft{\hfill \textsc{Relationships between Central Quadrilateral and Reference Quadrilateral} \hfill}
\newcommand{\degrees}{^\circ}
\newcommand{\relbox}[1]{\medskip\fboxrule 1.5pt\fcolorbox{red}{yellow}{#1}\smallskip}

\long\def\void#1{}
\begin{document}
% ===================
International Journal of  Computer Discovered Mathematics (IJCDM) \\
ISSN 2367-7775 \copyright IJCDM \\
Volume 7, 2022, pp. 214--287  \\
web: \url{http://www.journal-1.eu/} \\
Received 11 May 2022. Published on-line 27 July 2022 \\ 

\copyright The Author(s) This article is published 
with open access.\footnote{This article is distributed under the terms of the Creative Commons Attribution License which permits any use, distribution, and reproduction in any medium, provided the original author(s) and the source are credited.} \\
% ===========================   
\bigskip
\bigskip

\begin{center}
	{\Large \textbf{Relationships between a Central Quadrilateral and its Reference Quadrilateral}} \\
	\medskip
	\bigskip
        \bigskip

	\textsc{Stanley Rabinowitz$^a$ and Ercole Suppa$^b$} \\

	$^a$ 545 Elm St Unit 1,  Milford, New Hampshire 03055, USA \\
	e-mail: \href{mailto:stan.rabinowitz@comcast.net}{stan.rabinowitz@comcast.net}\footnote{Corresponding author} \\
	web: \url{http://www.StanleyRabinowitz.com/} \\
	
	$^b$ Via B. Croce 54, 64100 Teramo, Italia \\
	e-mail: \href{mailto:ercolesuppa@gmail.com}{ercolesuppa@gmail.com} \\
	web: \url{https://www.esuppa.it} \\

\bigskip
%draft revised: May 9, 2022

\end{center}
\bigskip
\bigskip

% ==============================
\textbf{Abstract.}
Let $P$ be a point inside a convex quadrilateral $ABCD$.
The lines from $P$ to the vertices of the quadrilateral divide the quadrilateral
into four triangles.
If we locate a triangle center in each of these triangles, the four triangle
centers form another quadrilateral called a central quadrilateral.
For each of various shaped quadrilaterals, and each of 1000 different triangle
centers, we compare the reference quadrilateral to the central quadrilateral.
Using a computer, we determine how the two quadrilaterals are related.
For example, we test to see if the two quadrilaterals are congruent, similar, have the same area,
or have the same perimeter.
We also look for such relationships when $P$ is a special point associated with
the reference quadrilateral, such as being the diagonal point, Steiner point, or Poncelet point.



\newcommand{\ru}{\rule[-7pt]{0pt}{20pt}}

%\font\bigf=cmb10 at 16pt

\bigskip
\bigskip
% ================================
% 1 Introduction 
% ================================
%
\section{Introduction}
\label{section:introduction}

In this study, $ABCD$ always represents a convex quadrilateral known as the \emph{reference quadrilateral}.
A point $E$ in the plane of the quadrilateral is chosen and will be called the \emph{radiator}.
The radiator can be an arbitrary point or it can be a notable point associated with the quadrilateral.
Lines are drawn from the radiator to the vertices of the reference quadrilateral forming four triangles
with the sides of the quadrilateral as shown in Figure \ref{fig:centerPointTriangles}.
These triangles will be called the \emph{radial triangles}.

\begin{figure}[h!t]
\centering
\includegraphics[width=0.4\linewidth]{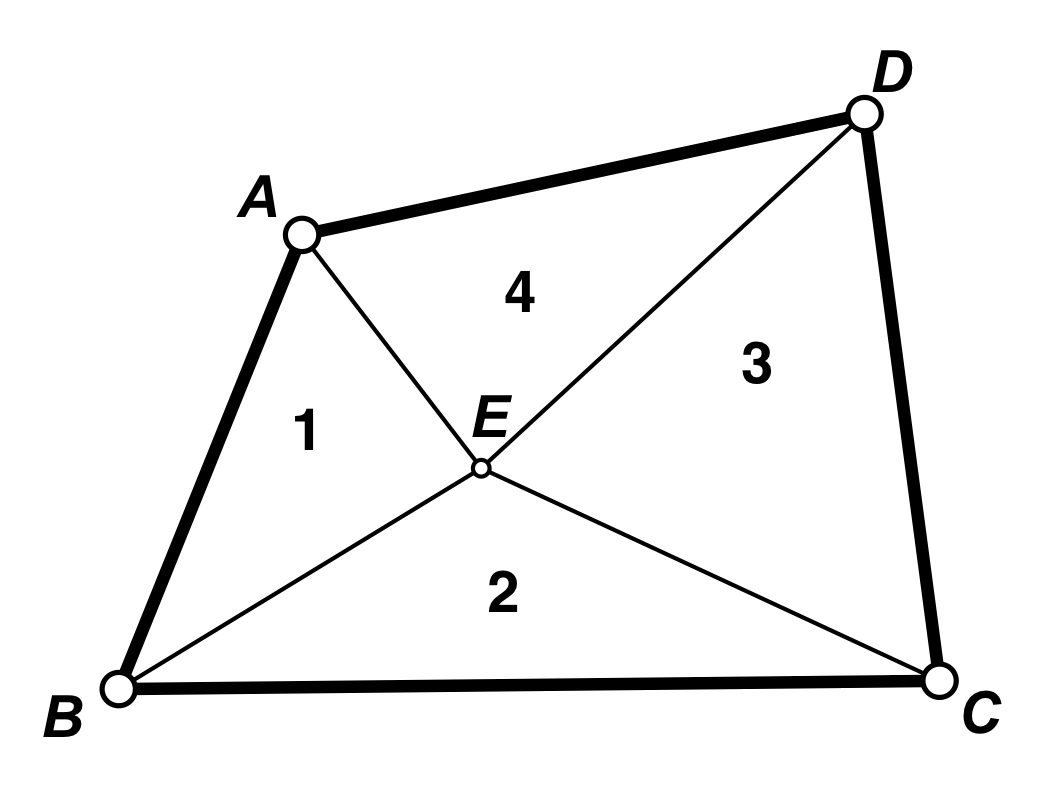}
\caption{Radial Triangles}
\label{fig:centerPointTriangles}
\end{figure}

In the figure, the radial triangles have been numbered in a counterclockwise order starting with side $AB$:
$\triangle ABE$, $\triangle BCE$, $\triangle CDE$, $\triangle DAE$.
Triangle centers (such as the incenter, centroid, or circumcenter) are selected in each triangle.
The same type of triangle center is used with each radial triangle.
In order, the names of these points are $F$, $G$, $H$, and $I$ as shown in Figure \ref{fig:centralQuadrilateral}.
These four centers form a quadrilateral $FGHI$ that will be called the \emph{central quadrilateral} (of quadrilateral $ABCD$
with respect to $E$). Quadrilateral $FGHI$ need not be convex.

\begin{figure}[h!t]
\centering
\includegraphics[width=0.4\linewidth]{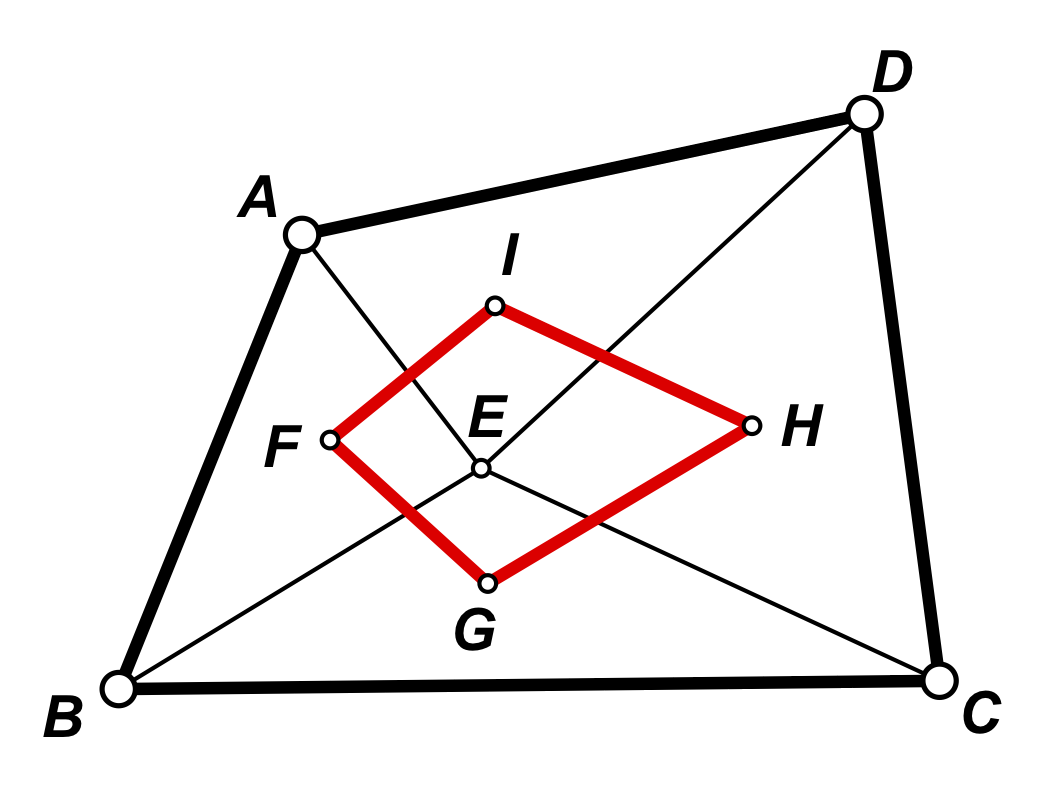}
\caption{Central Quadrilateral}
\label{fig:centralQuadrilateral}
\end{figure}

The purpose of this paper is to determine interesting relationships between a reference quadrilateral
and its central quadrilateral.

\newpage

\section{Types of Quadrilaterals Studied}
\label{section:quadrilaterals}

We are only interested in reference quadrilaterals that have a certain amount of symmetry.
For example, we excluded bilateral quadrilaterals (those with two equal sides),
bisect-diagonal quadrilaterals (where one diagonal bisects another), right kites,
right trapezoids, and golden rectangles.
The types of quadrilaterals we studied are shown in Table \ref{table:quadrilaterals}.
The sides of the quadrilateral, in order, have lengths $a$, $b$, $c$, and $d$.
The diagonals have lengths $p$ and $q$.
The measures of the angles of the quadrilateral, in order, are $A$, $B$, $C$, and $D$.

\begin{table}[ht!]
\caption{}
\label{table:quadrilaterals}
\begin{center}
\footnotesize
\begin{tabular}{|l|l|l|}\hline
\multicolumn{3}{|c|}{\textbf{\color{blue}\Large \strut Types of Quadrilaterals Considered}}\\ \hline
\textbf{Quadrilateral Type}&\textbf{Geometric Definition}&\textbf{Algebraic Condition}\\ \hline
general&convex&none\\ \hline
cyclic&has a circumcircle&$A+C=B+D$\\ \hline
tangential&has an incircle&$a+c=b+d$\\ \hline
extangential&has an excircle&$a+b=c+d$\\ \hline
parallelogram&opposite sides parallel&$a=c$, $b=d$\\ \hline
equalProdOpp&product of opposite sides equal&$ac=bd$\\ \hline
equalProdAdj&product of adjacent sides equal&$ab=cd$\\ \hline
orthodiagonal&diagonals are perpendicular&$a^2+c^2=b^2+d^2$\\ \hline
equidiagonal&diagonals have the same length&$p=q$\\ \hline
Pythagorean&equal sum of squares, adjacent sides&$a^2+b^2=c^2+d^2$\\ \hline
%harmonic&cyclic, equal-product&$ac=bd$ plus cyclic\\ \hline
kite&two pair adjacent equal sides&$a=b$, $c=d$\\ \hline
trapezoid&one pair of opposite sides parallel&$A+B=C+D$\\ \hline
rhombus&equilateral&$a=b=c=d$\\ \hline
rectangle&equiangular&$A=B=C=D$\\ \hline
Hjelmslev&two opposite right angles&$A=C=90^\circ$\\ \hline
%right trapezoid&trapezoid with a right angle&$A=B=90^\circ$\\ \hline
isosceles trapezoid&trapezoid with two equal sides&$A=B$, $C=D$\\ \hline
%trilateral&three equal sides&$a=b=c$\\ \hline
%triangular&three equal angles&$A=B=C$\\ \hline
%equiRecipSum&sum of reciprocals of opp sides equal&$1/a+1/c=1/b+1/d$\\ \hline
%cyclicEquiRecipSum&cyclic and equiRecipSum&above plus cyclic\\ \hline
APquad&sides in arithmetic progression&$d-c=c-b=b-a$\\ \hline
%bicentric&cyclic and tangential&$a+c=b+d$ plus cyclic\\ \hline
%exbicentric&cyclic and extangential&$a+b=c+d$ plus cyclic\\ \hline
\end{tabular}
\end{center}
\end{table}

The following combinations of entries in the above list were also considered:
bicentric quadrilaterals (cyclic and tangential), exbicentric quadrilaterals (cyclic and extangential),
bicentric trapezoids, cyclic orthodiagonal quadrilaterals, equidiagonal kites,
equidiagonal orthodiagonal quadrilaterals, equidiagonal orthodiagonal trapezoids,
harmonic quadrilaterals (cyclic and equalProdOpp), orthodiagonal trapezoids, tangential trapezoids,
and squares (equiangular rhombi).

So, in addition to the general convex quadrilateral, a total of 27 other types of quadrilaterals
were considered in this study.

A graph of the types of quadrilaterals considered is shown in Figure \ref{fig:quadShapes}.
An arrow from A to B means that any quadrilateral of type B is also of type A.
For example: all squares are rectangles and all kites are orthodiagonal.
If a directed path leads from a quadrilateral of type A to a quadrilateral of type B, then we will
say that A is an \emph{ancestor} of B. For example, an equidiagonal quadrilateral is an ancestor of a rectangle.
In other words, all rectangles are equidiagonal.

\begin{figure}[h!t]
\centering
\scalebox{1}[1.5]{\includegraphics[width=1\linewidth]{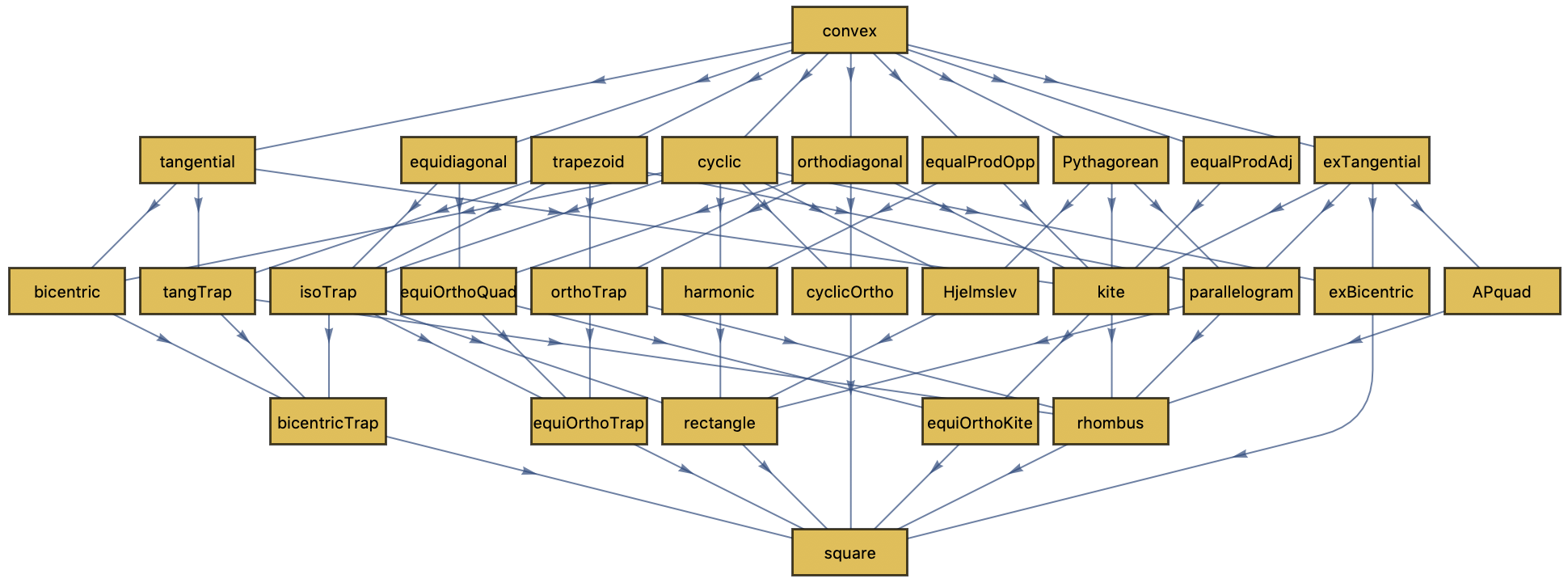}}
\caption{Quadrilateral Shapes}
\label{fig:quadShapes}
\end{figure}

Unless otherwise specified, when we give a theorem or table of properties of a quadrilateral, we will omit an entry
for a particular shape quadrilateral if the property is known to be true for an ancestor of that quadrilateral.

\newpage

\section{Centers}
\label{section:centers}

In this study, we will place triangle centers in the four radial
triangles. We use Clark Kimberling's definition of a triangle center \cite{KimberlingA}.

A \emph{center function} is a nonzero function $f(a,b,c)$
homogeneous in $a$, $b$, and $c$ and symmetric in $b$ and $c$.
\emph{Homogeneous} in $a$, $b$, and $c$ means that
$$f(ta,tb,tc)=t^nf(a,b,c)$$
for some nonnegative integer $n$,
all $t>0$, and all positive real numbers $(a,b,c)$ satisfying $a<b+c$, $b<c+a$, and $c<a+b$.
\emph{Symmetric} in $b$ and $c$ means that
$$f(a,c,b)=f(a,b,c)$$
for all $a$, $b$, and $c$.

A \emph{triangle center} is an equivalence class $x:y:z$ of ordered triples $(x,y,z)$
given by
$$x=f(a,b,c),\quad y=f(b,c,a),\quad z=f(c,a,b).$$

Tens of thousands of interesting triangle centers have been cataloged in the Encyclopedia of Triangle Centers \cite{ETC}. We use $X_n$ to denote the nth named center in this encyclopedia.

Note that if the center function of a certain center is $f(a,b,c)$,
then the trilinear coordinates of that point with respect to a triangle with sides $a$, $b$, and $c$ are
$$\Bigl(f(a,b,c):f(b,c,a):f(c,a,b)\Bigr).$$
The barycentric coordinates for that point would then be
$$\Bigl(af(a,b,c):bf(b,c,a):cf(c,a,b)\Bigr).$$

\newpage

\section{Methodology}
\label{section:methodology}

We used a computer program called GeometricExplorer to compare quadrilaterals
with their central quadrilateral. Starting with each type of quadrilateral listed in
Figure~\ref{fig:quadShapes} for the reference quadrilateral, we picked various choices
for point $E$, the radiator. The types of radiators studied are shown in Table \ref{table:radiators}.

\begin{table}[ht!]
\caption{}
\label{table:radiators}
\begin{center}
\begin{tabular}{|l|l|}
\hline
\multicolumn{2}{|c|}{\Large \strut \textbf{\color{blue}Points Used as Radiators}}\\
\hline
\textbf{name}&\textbf{description}\\ \hline
arbitrary point&any point in the plane of $ABCD$\\ \hline
diagonal point&intersection of the diagonals (QG--P1)\\ \hline
Poncelet point&(QA--P2)\\ \hline
Steiner point&(QA--P3)\\ \hline
circumcenter&center of circumscribed circle\\ \hline
incenter&center of inscribed circle\\ \hline
anticenter&(QA--P2 in a cyclic quadrilateral)\\ \hline
vertex centroid&(QA-P1)\\ \hline
midpoint of 3rd diagonal&(QG-P2)\\ \hline
\void{
quasi centroid$\dagger$&(QG-P4)\\ \hline
quasi circumcenter$\dagger$&(QG-P5)\\ \hline
quasi orthocenter$\dagger$&(QG-P6)\\ \hline
quasi nine-point center$\dagger$&(QG-P7)\\ \hline
quasi incenter$\dagger$&\\ \hline
Kirikami center$\dagger$&(QG-P15)\\ \hline
Miquel point$\dagger$&(QL-P1)\\ \hline
\multicolumn{2}{|l|}{$\dagger$ \small This point might be omitted from the final paper.}\\ \hline
}
\end{tabular}
\end{center}
\end{table}

Some notable points only exist for certain shape quadrilaterals.
For example, the circumcenter only applies to cyclic quadrilaterals.
A code in parentheses represents the name for the point as listed in
the Encyclopedia of Quadri-Figures~\cite{EQF}.
These will be defined in the section that reports relationships using these points.

For each $n$ from 1 to 1000, we placed center $X_n$
in each of the radial triangles of the reference quadrilateral.
The program then analyzes the central quadrilateral formed by these four centers
and reports if the central quadrilateral is related to the reference quadrilateral.
Points at infinity were omitted.
The types of relationships checked for are shown in Table \ref{table:relationships}.

\begin{table}[ht!]
\caption{}
\label{table:relationships}
\begin{center}
\begin{tabular}{|l|l|}
\hline
\multicolumn{2}{|c|}{\Large \strut \textbf{\color{blue}Relationships Checked For}}\\
\hline
\textbf{notation}&\textbf{description}\\ \hline
\ru $[ABCD]=[FGHI]$&the quadrilaterals have the same area\\ \hline
\ru $[ABCD]=k[FGHI]$&the area of $ABCD$ is $k$ times the area of $FGHI$ $\dagger$\\ \hline
\ru $ABCD\cong FGHI$&the quadrilaterals are congruent\\ \hline
\ru $ABCD\sim FGHI$&the quadrilaterals are similar\\ \hline
\ru $\partial ABCD=\partial FGHI$&the quadrilaterals have the same perimeter\\ \hline
\ru $\odot ABCD\cong\odot FGHI$&the quadrilaterals have congruent circumcircles\\ \hline
\ru $\odot ABCD\equiv \odot FGHI$&the quadrilaterals have the same circumcircle\\ \hline
\multicolumn{2}{|l|}{$\dagger$ \small Only rational values of $k$ were checked for with denominators less than 6.}\\
\hline
\end{tabular}
\end{center}
\end{table}

\newpage

\section{Barycentric Coordinates and Quadrilaterals}

The program we used to find results about central quadrilaterals (GeometricExplorer) is a useful
tool for discovering results, but it does not prove that these results are true.
GeometricExplorer uses numerical coordinates (to 15 digits of precision) for locating
all the points. Thus, a relationship found by this program does not constitute a proof that the result is correct,
but gives us compelling evidence for the validity of the result.

If a theorem in this paper is accompanied by a figure, this means that the
figure was drawn using either Geometer's Sketchpad or GeoGebra.
In either case, we used the drawing program to dynamically vary the points
in the figure. Noticing that the result remains true as the points vary offers
further evidence that the theorem is true. But again, this does not constitute a proof.

To prove the results that we have discovered, we use geometric methods, when possible.
If we could not find a purely geometrical proof, we turned to analytic methods using
barycentric coordinates and performing exact symbolic computation using Mathematica.

We assume the reader is familiar with barycentric coordinates. We give below some
useful results that will be used when providing proofs of some of the theorems that we found.

The following result about the centroid  of a triangle is well known \cite[p.~65]{Altshiller-Court}.

\begin{lemma}
\label{lemma:centroid}
The centroid of a triangle divides a median in the ratio $2:1$.
\end{lemma}

The following result about the area of a quadrilateral is well known \cite[p.~124]{Altshiller-Court}.

\begin{lemma}[Varignon Parallelogram]
\label{lemma:Varignon}
The midpoints of the sides of a convex quadrilateral form a parallelogram.
The area of the parallelogram is half the area of the quadrilateral.
The sides of the parallelogram are parallel to the diagonals of the quadrilateral.
\end{lemma}

This parallelogram is called the \emph{Varignon paralellogram} of the given quadrilateral.

If $(u:v:w)$ are barycentric coordinates with the property that $u+v+w=1$,
then we call the coordinates \emph{normalized}.

The formula for the distance between two points in terms of their normalized barycentric coordinates
is well known \cite[Section~2]{Grozdev}.

\begin{lemma}[Distance Formula]
\label{lemma:distanceFormula}
Let $ABC$ be a triangle with sides of lengths $a$, $b$, and $c$.
Relative to $\triangle ABC$, let the normalized barycentric coordinates for points $P$ and $Q$ be $(u_1:v_1:w_1)$ and $(u_2:v_2:w_2)$,
respectively. Let $x=u_1-u_2$, $y=v_1-v_2$, and $z=w_1-w_2$.
Then the distance from $P$ to $Q$ is
$$PQ=\sqrt{-a^2yz-b^2zx-c^2xy}.$$
\end{lemma}

\newpage
The formula for the area of a triangle is also well known \cite[Equation~2]{Grozdev}.

\begin{lemma}[Area Formula]
\label{lemma:areaFormula}
Let $ABC$ be a triangle with area $K$.
Relative to $\triangle ABC$, let the normalized barycentric coordinates for points $P$, $Q$, and $R$
be $(u_1:v_1:w_1)$, $(u_2:v_2:w_2)$, and $(u_3:v_3:w_3)$ respectively.
Then the area of $\triangle PQR$ is
$$[PQR]=\left|
\begin{array}{ccc}
u_1&v_1&w_1\\
u_2&v_2&w_2\\
u_3&v_3&w_3\\
\end{array}
\right|K.$$
\end{lemma}

Note that the area is signed. It is positive if the triangle has the same orientation as $\triangle ABC$
and negative if it has the opposite orientation.

The signed area of a quadrilateral $PQRS$ is defined as
$$[PQRS]=[PQR]+[RSP].$$
This definition is used even when the quadrilateral is self-intersecting.

Given the barycentric coordinates for a point $P$ with respect to $\triangle ABC$,
we sometimes want to find the corresponding point in some other triangle, $\triangle DEF$.
This is accomplished using the well-known Change of Coordinates Formula \cite[Section~3]{Grozdev}.

\begin{lemma}[Change of Coordinates Formula]
\label{lemma:changeOfCoordinates}
Relative to $\triangle ABC$, let the normalized barycentric coordinates for points $D$, $E$, and $F$
be $(u_1:v_1:w_1)$, $(u_2:v_2:w_2)$, and $(u_3:v_3:w_3)$ respectively.
Let the normalized barycentric coordinates for point $P$ with respect to $\triangle DEF$ be $(p:q:r)$.
Then the barycentric coordinates for $P$ with respect to $\triangle ABC$ are $(u:v:w)$
where
$$
\begin{aligned}
u&=u_1p+u_2q+u_3r\\
v&=v_1p+v_2q+v_3r\\
w&=w_1p+w_2q+w_3r.\\
\end{aligned}
$$
\end{lemma}

\begin{lemma}
\label{lemma:hypotenuseMidpoint}
Let $ABC$ be a right triangle with right angle at $A$.
Let $f$ be a center function with the properties
$$f(a,b,c)=0\qquad\mathrm{and}\qquad f(b,c,a)=f(c,a,b)$$
when $c^2=a^2+b^2$. Then the center of $\triangle ABC$ corresponding to this center function
coincides with the midpoint of the hypotenuse.
\end{lemma}

\begin{proof}
Since $\triangle ABC$ is a right triangle with hypotenuse $BC$, we must have $c^2=a^2+b^2$.
The trilinear coordinates for the center then are
$$
\begin{aligned}
\Bigl(f(a,b,c):f(b,c,a):f(c,a,b)\Bigr)&=\Bigl(0:f(b,c,a):f(b,c,a)\Bigr)\\
&=(0:1:1)
\end{aligned}
$$
which corresponds to the trilinear coordinates for the midpoint of the hypotenuse.
\end{proof}

\newpage

\begin{lemma}
\label{lemma:midpointMedian}
The nine-point center of a right triangle coincides with the midpoint of
the median to the hypotenuse.
\end{lemma}

\begin{figure}[h!t]
\centering
\includegraphics[width=0.4\linewidth]{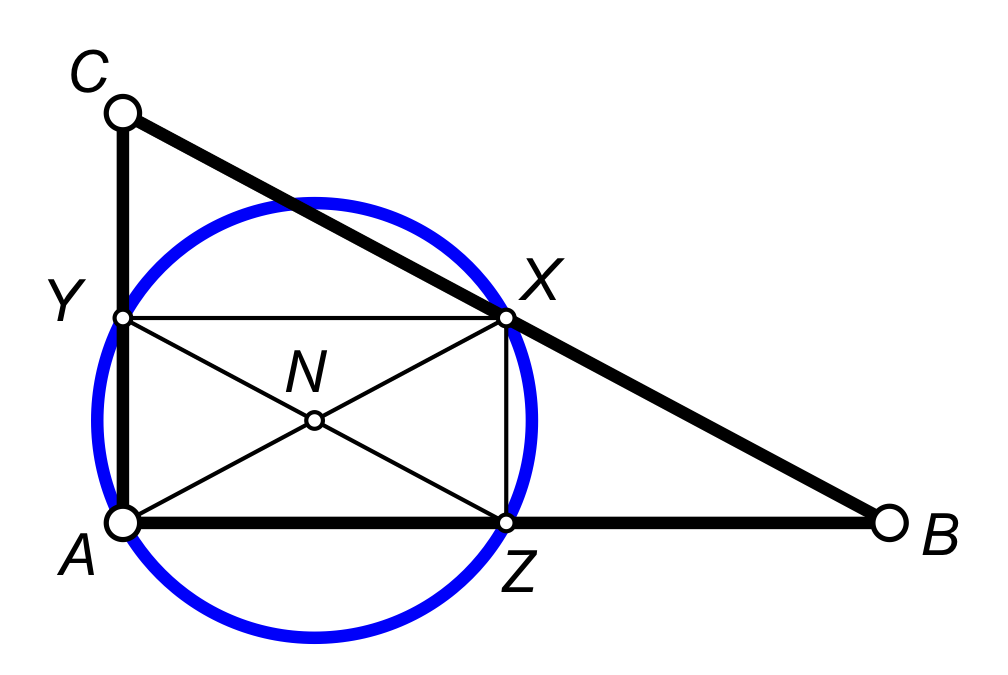}
\caption{Nine point center of a right triangle}
\label{fig:dp9pt}
\end{figure}

\begin{proof}
Let $X$, $Y$, and $Z$ be the midpoints of the sides of right triangle $ABC$ as shown
in Figure \ref{fig:dp9pt}.
Since $X$, $Y$, and $Z$ are midpoints, $XZ\parallel CA$ and $XY\parallel BA$.
Since $\angle BAC$ is a right angle, $AYXZ$ must be a rectangle.
The nine-point circle of $\triangle ABC$ passes through $X$, $Y$, and $Z$ and is therefore
the circumcircle of this rectangle. The nine-point center is the center of this rectangle
and is therefore the midpoint of $AX$.
\end{proof}

\begin{lemma}
\label{lemma:rightTriangle}
Let $M$ be midpoint of the hypotenuse $BC$ of right triangle $ABC$.
Let $X$ be any point on $AM$.
Let $Y$ and $Z$ be the feet of the perpendiculars dropped from $X$ to $AC$ and $AB$,
respectively (Figure \ref{fig:dpRightTriangle}).
Then $$\frac{[ABC]}{[AZXY]}=2\left(\frac{AM}{AX}\right)^2.$$
\end{lemma}

\begin{figure}[h!t]
\centering
\includegraphics[width=0.4\linewidth]{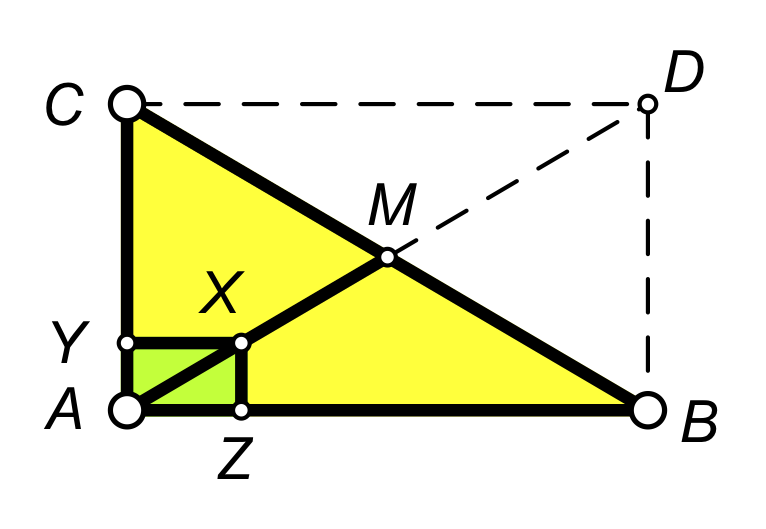}
\caption{right triangle, $\frac{[ABC]}{[AZXY]}=2\left(\frac{AM}{AX}\right)^2.$}
\label{fig:dpRightTriangle}
\end{figure}

\begin{proof}
Reflect $A$ about $M$ to get point $D$, making $ACDB$ a rectangle.
Rectangles $AYXZ$ and $ACDB$ are similar. The ratio of the area of two similar figures
is the square of the ratio of similarity. So
$$\frac{[ABDC]}{[AZXY]}=\left(\frac{AD}{AX}\right)^2.$$
Thus,
$$\frac{[ABC]}{[AZXY]}=2\left(\frac{AM}{AX}\right)^2$$
since $AD=2AM$ and $[ABCD]=2[ABC]$.
\end{proof}

The following result comes from \cite{shapes}.

\begin{lemma}
\label{lemma:angleBisector}
The condition for a triangle center with center function $f(x,y,z)$ to lie
on the angle bisector at vertex $A$ in right triangle $ABC$ (with right angle at $A$) is
$$f(x,y,z)=f(y,x,z)$$
for all $x$, $y$, and $z$ satisfying $x^2+y^2=z^2$.
\end{lemma}

\begin{lemma}
\label{lemma:dpIsoscelesTriangle}
Let $ABC$ be an isosceles triangle with $AB=AC=b$ and $BC=a$.
Let $M$ be the midpoint of $BC$.
Let $X$ be any triangle center of $\triangle ABC$.
Suppose the barycentric coordinates for $X$ are $(u:v:w)$ with respect to $\triangle ABC$
(Figure~\ref{fig:dpIsoscelesTriangle}).
Then $X$ lies on $AM$, the median to side $BC$, $v=w$, and
$$\frac{XM}{AM}=\frac{u}{u+2v}.$$
\end{lemma}

\begin{figure}[h!t]
\centering
\includegraphics[width=0.4\linewidth]{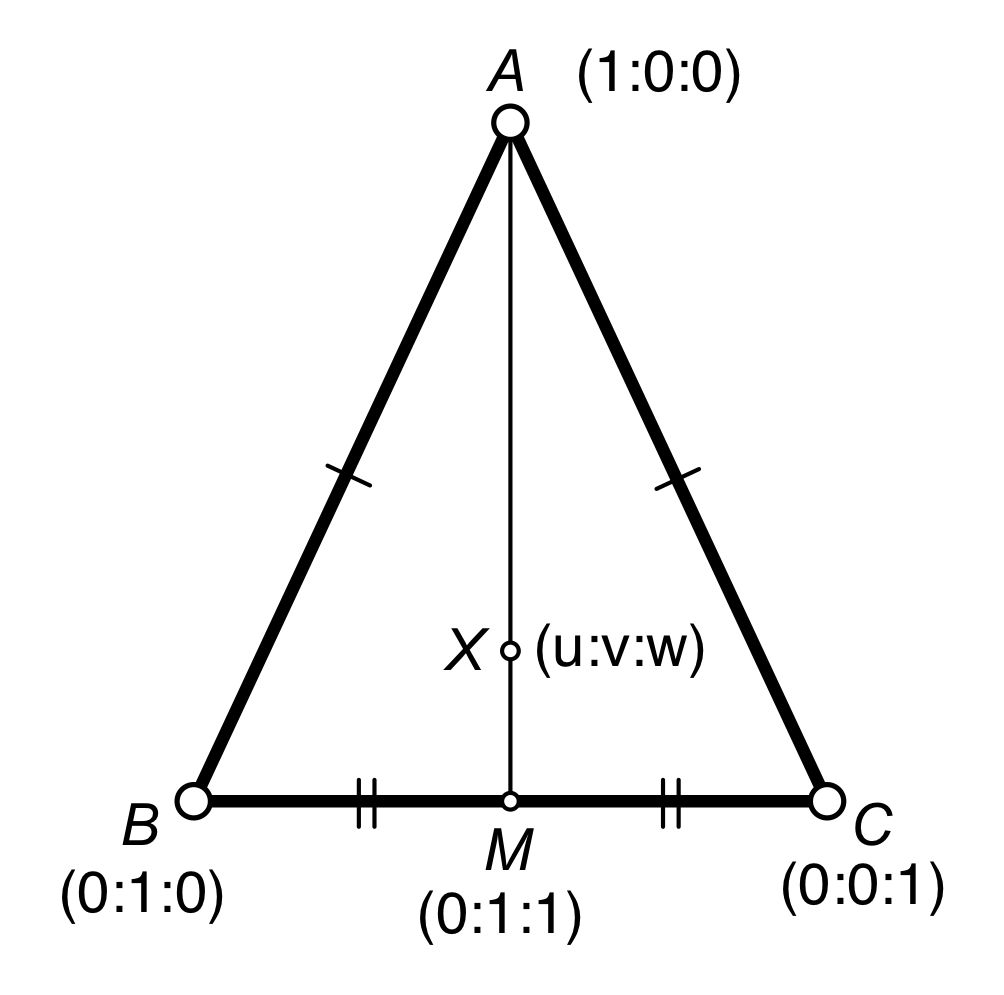}
\caption{Center $X$ of an isosceles triangle}
\label{fig:dpIsoscelesTriangle}
\end{figure}

\begin{proof}
Since the barycentric coordinates of a point are proportional to the areas of the triangles
formed by that point and the sides of a triangle, we must have
$$\frac{[AXB]}{[AXC]}=\frac{v}{w}.$$
But triangles $AXB$ and $AXC$ are congruent. Therefore $v=w$.

The equation of line $BC$ is $x=0$ and the equation of line $AM$ is $y=z$, so the barycentric coordinates for $M$ are $(0:1:1)$.

Again, by the area property,
$$\frac{[BXC]}{[ABC]}=\frac{u}{u+v+w}.$$
But the area of a triangle is half the base times the height, so
$$\frac{[BXC]}{[ABC]}=\frac{XM}{AM}.$$
Thus, 
$$\frac{XM}{AM}=\frac{u}{u+v+w}=\frac{u}{u+2v}.$$
\void{
and
$$\frac{AX}{AM}=\frac{AM-AX}{AM}=1-\frac{u}{u+2v}=\frac{2v}{u+2v}$$
}
\end{proof}

\newpage

\begin{lemma}
\label{lemma:dpIsoscelesTriangleA}
Let $ABC$ be an isosceles triangle with $AB=AC$.
Then the center $X_n$ coincides with $A$ for the following values of $n$:
\end{lemma}
59, 99, 100, 101, 107, 108, 109, 110, 112, 162, 163, 190, 249, 250, \
476, 643, 644, 645, 646, 648, 651, 653, 655, 658, 660, 662, 664, 666, \
668, 670, 677, 681, 685, 687, 689, 691, 692, 765, 769, 771, 773, 777, \
779, 781, 783, 785, 787, 789, 791, 793, 795, 797, 799, 803, 805, 807, \
809, 811, 813, 815, 817, 819, 823, 825, 827, 831, 833, 835, 839, 874, \
877, 880, 883, 886, 889, 892, 898, 901, 906, 907, 919, 925, 927, 929, \
930, 931, 932, 933, 934, 935.

\begin{lemma}
\label{lemma:dpIsoscelesTriangleMidpoint}
Let $ABC$ be an isosceles triangle with $AB=AC$.
Let $M$ be the midpoint of $BC$.
Then the center $X_n$ coincides with $M$ for the following values of $n$:
\end{lemma}
11, 115, 116, 122, 123, 124, 125, 127, 130, 134, 135, 136, 137, 139, \
244, 245, 246, 247, 338, 339, 865, 866, 867, 868.

\begin{lemma}
\label{lemma:dpIsoscelesTriangleInfinity}
The center $X_n$
lies on the line at infinity for all isosceles triangles, but not for all triangles,
for the following values of $n$:
\end{lemma}
351, 647, 649, 650, 652, 654, 656, 657, 659, 661, 663, 665, 667, 669, \
676, 684, 686, 693, 764, 770, 798, 810, 822, 850, 875, 876, 878, 879, \
881, 882, 884, 885, 887, 890, 905.

For reference,
$X_n$ lies on the line at infinity for all triangles
for the following values of $n$:
30, 511, 512, 513, 514, 515, 516, 517, 518, 519, 520, 521, 522, 523, \
524, 525, 526, 527, 528, 529, 530, 531, 532, 533, 534, 535, 536, 537, \
538, 539, 540, 541, 542, 543, 544, 545, 674, 680, 688, 690, 696, 698, \
700, 702, 704, 706, 708, 710, 712, 714, 716, 718, 720, 722, 724, 726, \
730, 732, 734, 736, 740, 742, 744, 746, 752, 754, 758, 760, 766, 768, \
772, 776, 778, 780, 782, 784, 786, 788, 790, 792, 794, 796, 802, 804, \
806, 808, 812, 814, 816, 818, 824, 826, 830, 832, 834, 838, 888, 891, \
900, 912, 916, 918, 924, 926, 928, 952, 971.

%**************************************
%    Arbitrary Point
%**************************************

\newpage

\section{Results Using an Arbitrary Point}
\label{section:arbitraryPoint}

In this configuration, the radiator, $E$, is any point in the plane of
the reference quadrilateral $ABCD$.

Our computer analysis found only two relationships that hold for all quadrilaterals
when $E$ is an arbitrary point in the plane. We examined all the types of
quadrilaterals listed in Table 1 and all triangle centers from $X_1$ to $X_{1000}$.
The two relationships only occur when the chosen center is $X_2$, the centroid.
The relationships are shown in Table \ref{table:arbitrary}.
%\bigskip

\begin{table}[ht!]
\caption{}
\label{table:arbitrary}
\begin{center}
\begin{tabular}{|l|l|p{2.2in}|}
\hline
\multicolumn{3}{|c|}{\textbf{\color{blue}\large \strut Central Quadrilaterals formed by an Arbitrary Point}}\\ \hline
\textbf{Quadrilateral Type}&\textbf{Relationship}&\textbf{centers}\\ \hline
\ru general&$[ABCD]=\frac92[FGHI]$&2\\
\hline
\ru square&$ABCD\sim FGHI$&2\\
\hline
\end{tabular}
\end{center}
\end{table}

%\bigskip

We were able to find geometric proofs for these relationships.

\relbox{Relationship $[ABCD]=\frac92[FGHI]$}

\begin{theorem}
\label{theorem:cpCentroids}
Let $E$ be any point in the plane of convex quadrilateral $ABCD$ not on a sideline of the quadrilateral.
Let $F$, $G$, $H$, and $I$ be the centroids of $\triangle EAB$, $\triangle EBC$, 
$\triangle ECD$, and $\triangle EDA$, respectively (Figure~\ref{fig:cpCentroids}).
Then $FGHI$ is a parallelogram
and
$$[ABCD]=\frac92[FGHI].$$
\end{theorem}

\begin{figure}[h!t]
\centering
\includegraphics[width=0.3\linewidth]{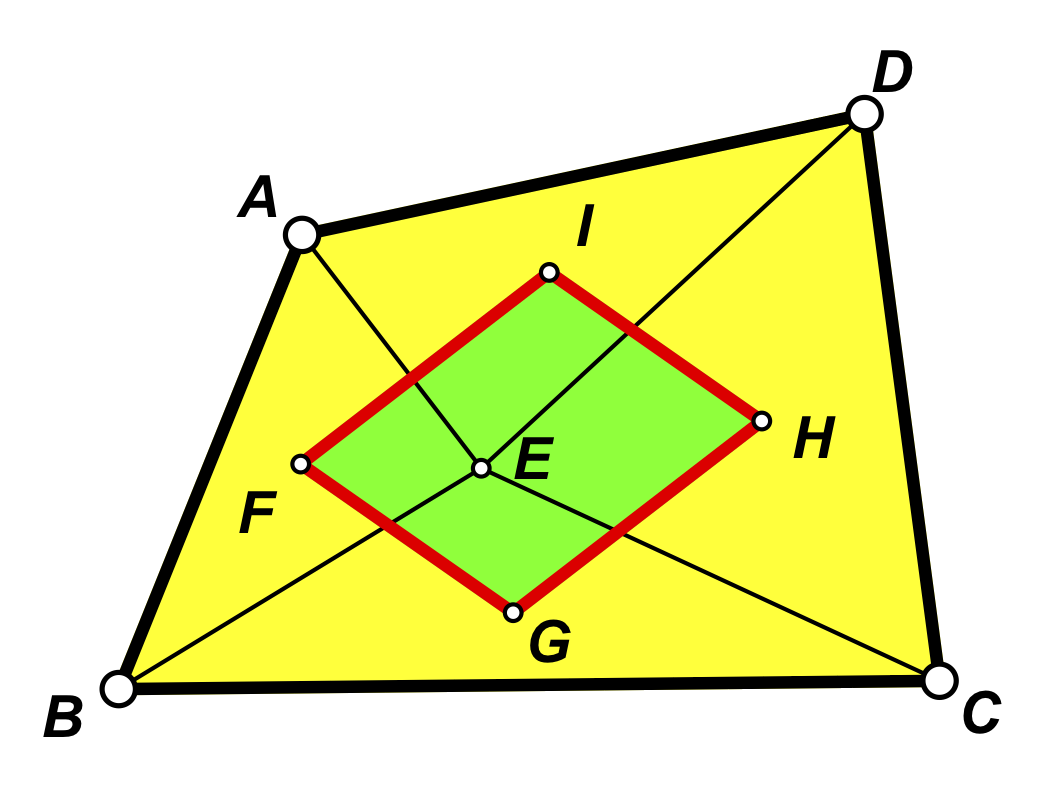}
\caption{$E$ arbitrary, centroids $\implies$ $\displaystyle\frac{[ABCD]}{[FGHI]}=\frac92$}
\label{fig:cpCentroids}
\end{figure}

\begin{proof}

Let $P$ be the midpoint of $BC$ and let $Q$ be the midpoint of $CD$ (Figure~\ref{fig:cpCentroidsProof}).

\begin{figure}[h!t]
\centering
\includegraphics[width=0.27\linewidth]{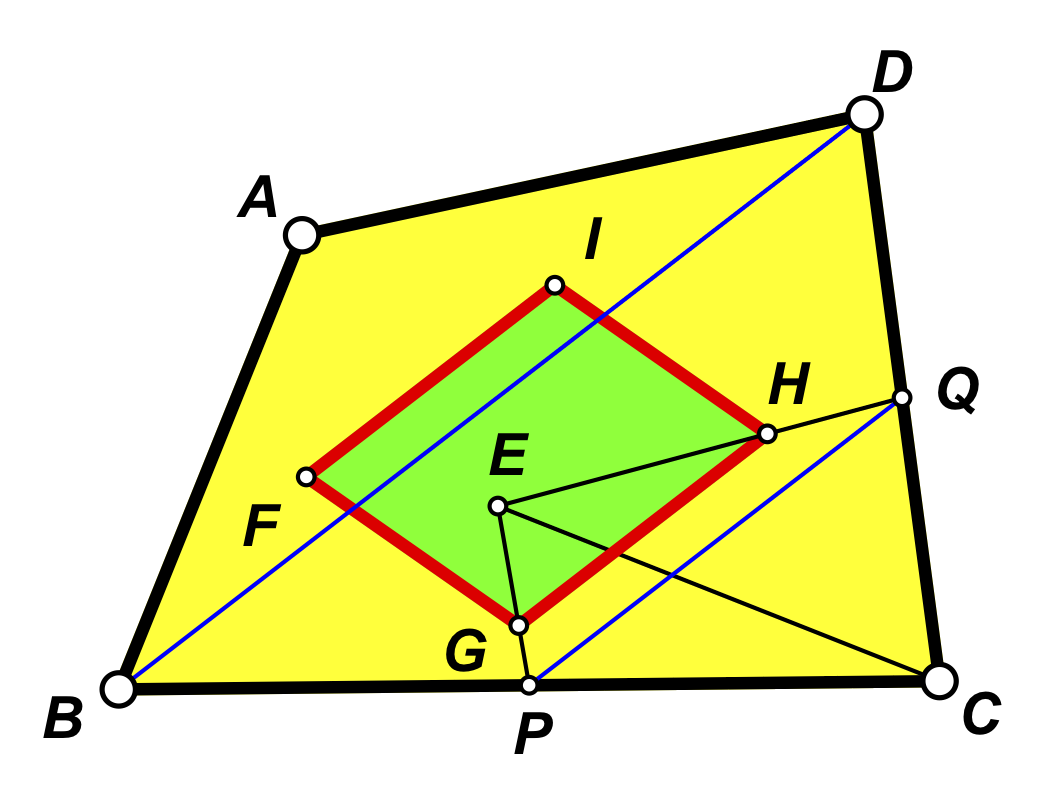}
\caption{}
\label{fig:cpCentroidsProof}
\end{figure}

Since $G$ is the centroid of $\triangle BEC$, $EP$ is a median of $\triangle BEC$
and $EG/GP=2$ by Lemma \ref{lemma:centroid}. Similarly, $EH/HQ=2$. Thus, $GH\parallel PQ$.
But $PQ\parallel BD$, so $GH\parallel BD$.
In the same manner, $FI\parallel BD$. Hence, $GH\parallel FI$.
Likewise, $FG\parallel IH$. Thus, $FGHI$ is a parallelogram.

Now, let $R$ be the midpoint of $DA$ and let $S$ be the midpoint of $AB$ (Figure~\ref{fig:cpCentroidsProof2}).

\begin{figure}[h!t]
\centering
\includegraphics[width=0.35\linewidth]{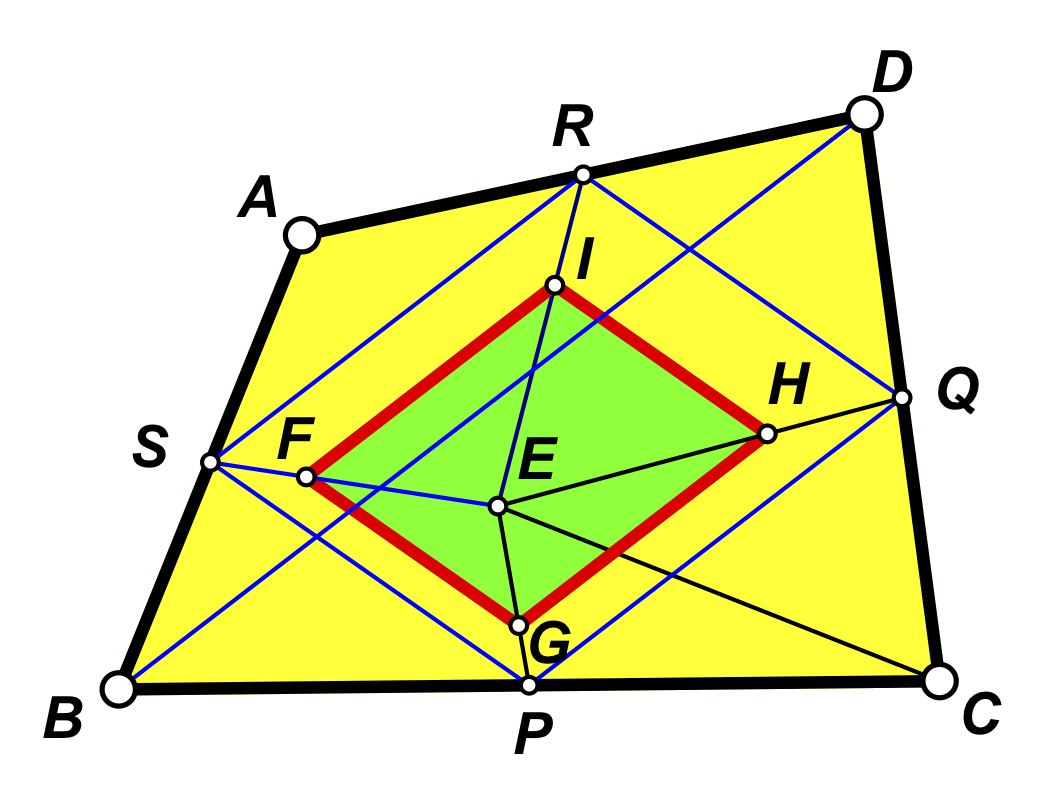}
\caption{}
\label{fig:cpCentroidsProof2}
\end{figure}

Then $PQRS$ is a parallelogram similar to parallelogram $FGHI$ with ratio of
similarity $3:2$ since $EQ/EH=3/2$. Thus,
\begin{equation}
\label{eq:cp1}
\frac{[PQRS]}{[FGHI]}=\frac94.
\end{equation}

Now
\begin{equation}
\label{eq:cp2}
\frac{[ABCD]}{[PQRS]}=2
\end{equation}
by Lemma \ref{lemma:Varignon}.
Combining equations (\ref{eq:cp1}) and (\ref{eq:cp2}) gives
$$\frac{[ABCD]}{[FGHI]}=\frac92$$
and we are done.
\end{proof}

\newpage

\relbox{Relationship $ABCD\sim FGHI$}

\begin{theorem}
\label{theorem:cpSquareCentroids}
Let $E$ be any point in the plane of square $ABCD$ not on a sideline of the square.
Let $F$, $G$, $H$, and $I$ be the centroids of $\triangle EAB$, $\triangle EBC$, 
$\triangle ECD$,  and $\triangle EDA$, respectively (Figure~\ref{fig:cpSquareCentroids}).
Then $FGHI$ is a square.
\end{theorem}

\begin{figure}[h!t]
\centering
\includegraphics[width=0.3\linewidth]{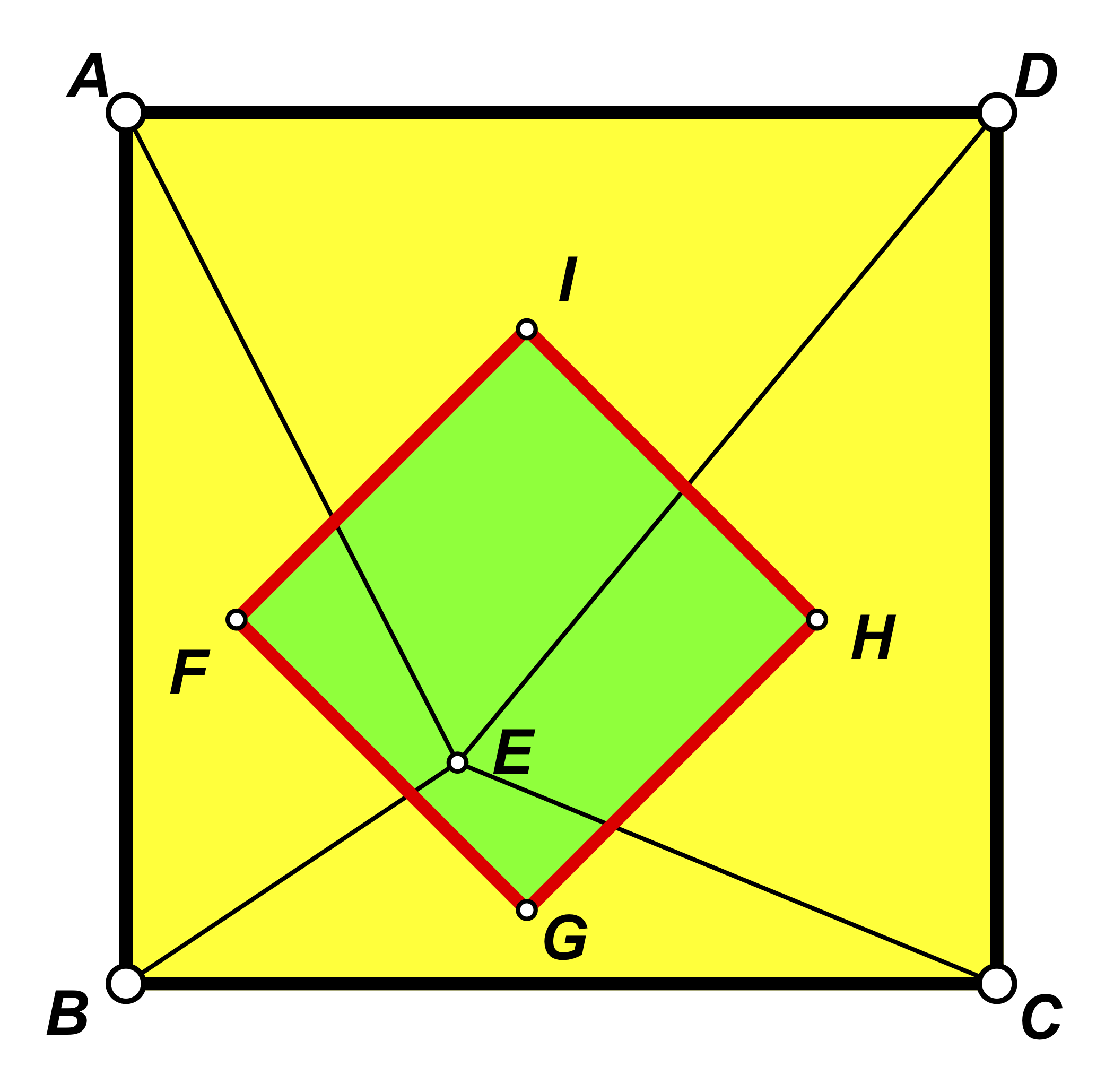}
\caption{$E$ arbitrary, square, centroids $\implies$ square}
\label{fig:cpSquareCentroids}
\end{figure}

\begin{proof}
By Theorem \ref{theorem:cpCentroids}, $FGHI$ is a parallelogram. From the proof of Theorem \ref{theorem:cpCentroids},
we see that each side of this parallelogram has length equal to half the length of one of the diagonals of
the square. Since the diagonals of a square are equal, the parallelogram must also be a square.
\end{proof}

\begin{conjecture}
Let $E$ be any point inside square $ABCD$.
Let $F$, $G$, $H$, and $I$ be centers of $\triangle EAB$, $\triangle EBC$, 
$\triangle ECD$, and $\triangle EDA$, respectively, with the same center function.
If $FGHI$ is a square independent of point $E$, then the four centers must be centroids.
\end{conjecture}

\newpage

%**************************************
%    Diagonal Point / Quarter Triangles
%**************************************

\section{Results Using the Diagonal Point}
\label{section:quarterTriangles}

In this configuration, the radiator, $E$, is the diagonal point of
the reference quadrilateral $ABCD$ (the point of intersection of the diagonals).
In this case, the radial triangles are also called \emph{quarter triangles}.

\begin{figure}[h!t]
\centering
\includegraphics[width=0.4\linewidth]{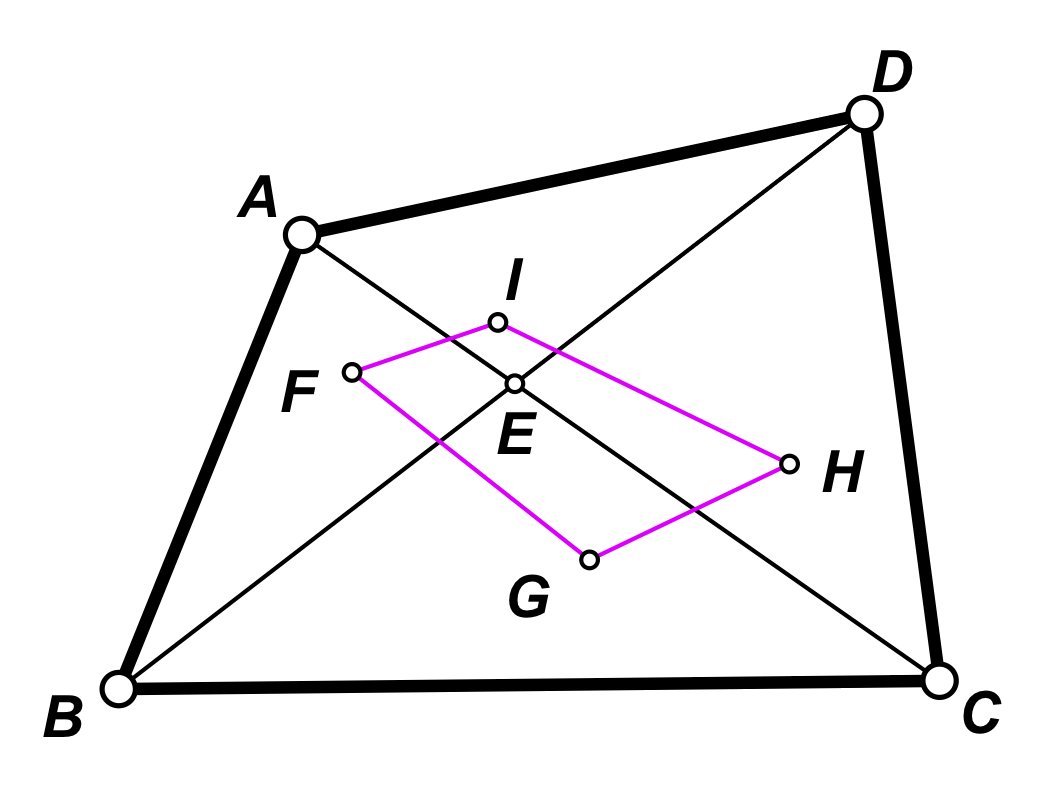}
\caption{Central quadrilateral formed using the diagonal point}
\label{fig:diagonalPointCenters}
\end{figure}

Our computer analysis found a number of relationships
between the reference quadrilateral $ABCD$ and the central quadrilateral $FGHI$.
Table~\ref{table:dp0} shows the relationship found for an arbitrary quadrilateral.
Table~\ref{table:dp1} on page \pageref{table:dp1} shows the relationships found for specific shaped quadrilaterals,
other than a square.
%Table~\ref{table:dp2} on page \pageref{table:dp2} shows the relationships found when the reference quadrilateral is a square.

\begin{table}[ht!]
\caption{}
\label{table:dp0}
\begin{center}
\begin{tabular}{|l|l|p{2.2in}|}
\hline
\multicolumn{3}{|c|}{\textbf{\color{blue}\large \strut Central Quadrilaterals formed by the Diagonal Point}}\\ \hline
\textbf{Quadrilateral Type}&\textbf{Relationship}&\textbf{centers}\\ \hline
\ru general&$[ABCD]=\frac12[FGHI]$&20\\
\hline
\end{tabular}
\end{center}
\end{table}

\subsection{Proofs for General Quadrilaterals}\ 

We now give a proof for the result listed in Table~\ref{table:dp0} for a general quadrilateral.

\relbox{Relationship $[ABCD]=\frac12[FGHI]$}

\begin{theorem}
\label{theorem:dpX20}
Let $E$ be the diagonal point of convex quadrilateral $ABCD$.
Let $F$, $G$, $H$, and $I$ be the $X_{20}$ points of $\triangle EAB$, $\triangle EBC$, 
$\triangle ECD$, and $\triangle EDA$, respectively (Figure~\ref{fig:dpX20}).
Then $FGHI$ is a parallelogram
and
$$[ABCD]=\frac12[FGHI].$$
\end{theorem}

\begin{figure}[h!t]
\centering
\includegraphics[width=0.4\linewidth]{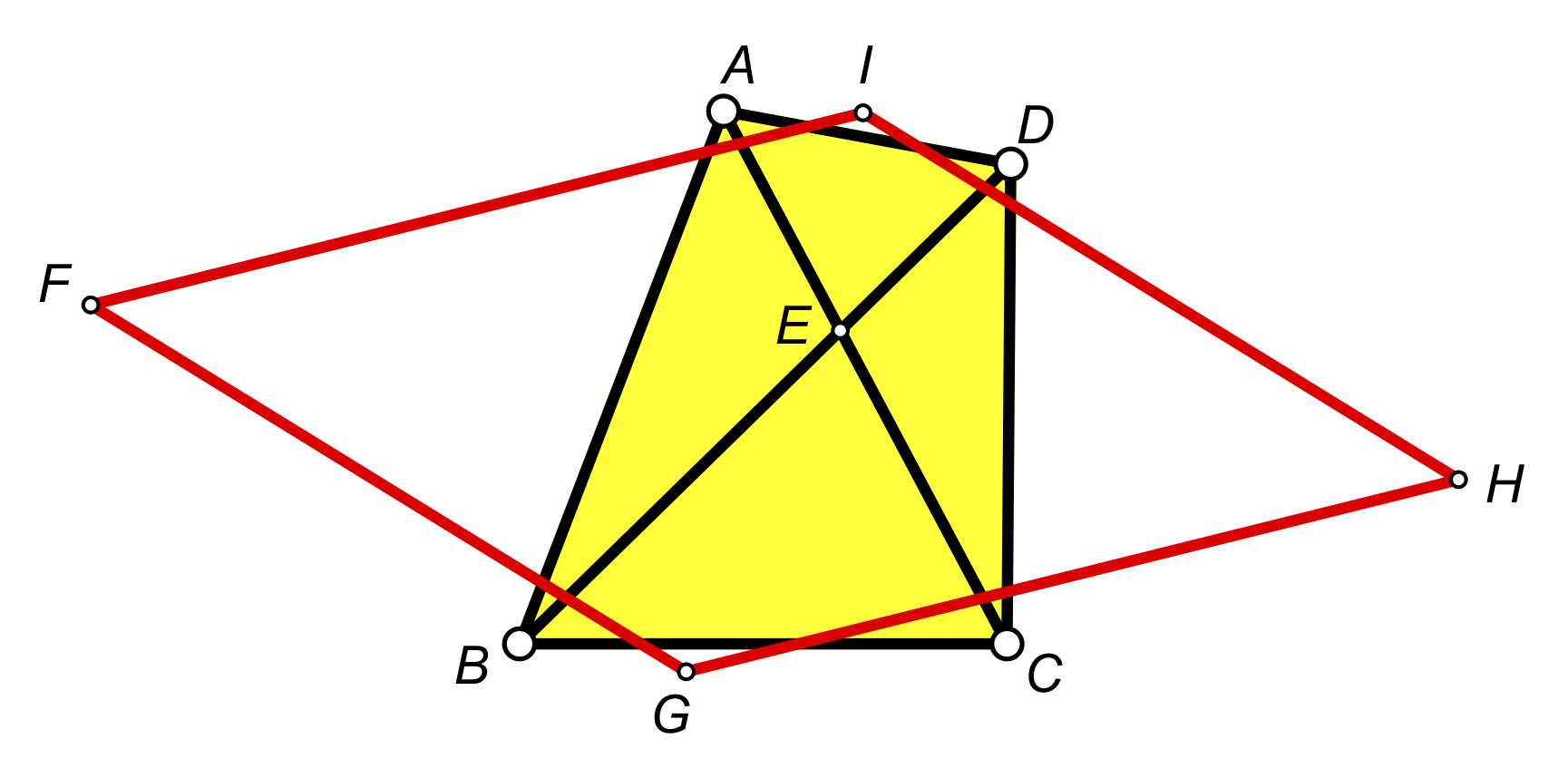}
\caption{$X_{20}\implies [ABCD]=\frac12[FGHI]$}
\label{fig:dpX20}
\end{figure}

\begin{proof}
We set up a barycentric coordinate system using $\triangle ABC$ as the reference triangle,
so that
$$
\begin{aligned}
A&=(1:0:0)\\
B&=(0:1:0)\\
C&=(0:0:1).
\end{aligned}
$$
We let the barycentric coordinates of $D$ be $(p:q:r)$ with $p+q+r=1$
and without loss of generality, assume $p>0$, $q<0$, and $r>0$
(Figure~\ref{fig:dpQA-CT}).
%This set-up is sometimes known as the QA-CT coordinate system for quadrilaterals.

\begin{figure}[h!t]
\centering
\includegraphics[width=0.5\linewidth]{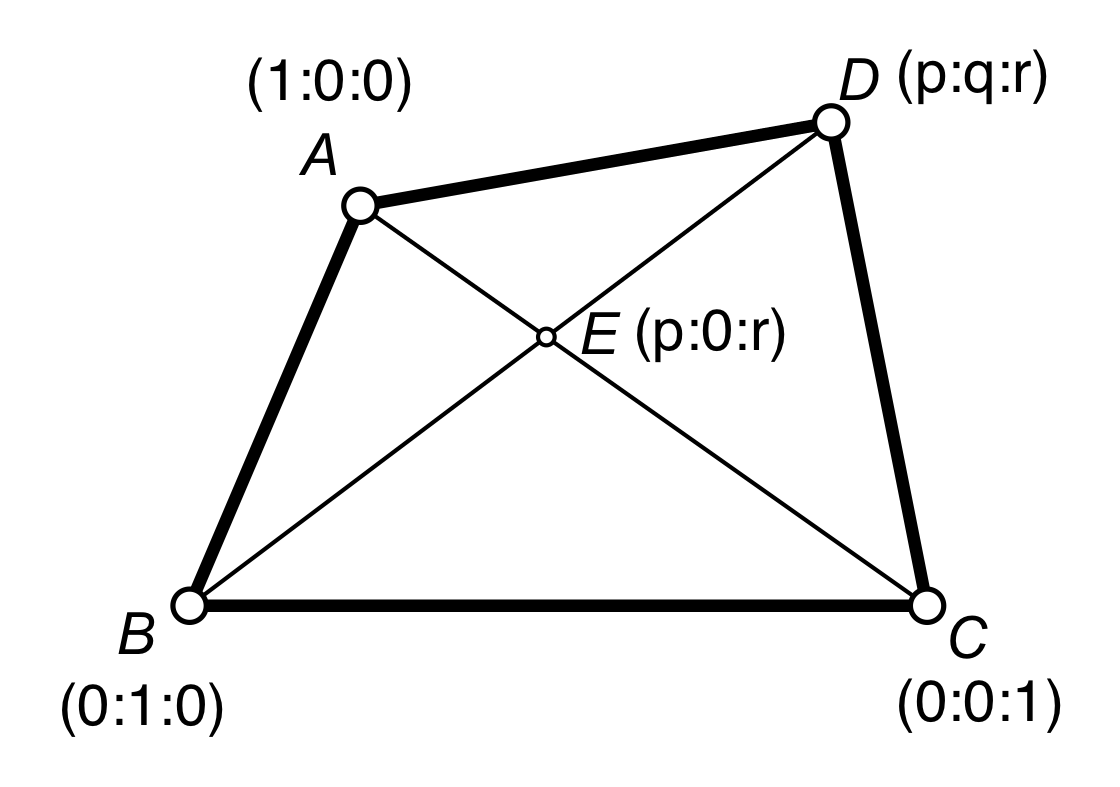}
\caption{Set-up for Quadrilateral coordinate system}
\label{fig:dpQA-CT}
\end{figure}

The equation for line $AC$ is $y=0$.
The equation for line $BD$ is $rx=pz$.
Therefore, the barycentric coordinates for the diagonal point $E$ are $(p:0:r)$.
Note that the barycentric coordinates for $A$, $B$, $C$, and $D$ are already normalized.
The normalized barycentric coordinates for $E$ are
$$E=\left(\frac{p}{p+r}:0:\frac{r}{p+r}\right).$$

From \cite{ETC20}, we find that barycentric coordinates for the $X_{20}$ point of
a triangle $ABC$ with sides $a$, $b$, and $c$ are
$$
\begin{aligned}
X_{20}=&\bigl(3 a^4-2 a^2 b^2-2 a^2 c^2-b^4+2 b^2 c^2-c^4:\\
&-a^4-2 a^2 b^2+2 a^2 c^2+3 b^4-2 b^2c^2-c^4:\\
&-a^4+2 a^2 b^2-2 a^2 c^2-b^4-2 b^2 c^2+3 c^4\bigr).
\end{aligned}
$$
To convert these to normalized barycentric coordinates, divide each coordinate by their sum,
$$a^4+b^4+c^4-2 a^2 b^2-2 b^2 c^2-2 c^2 a^2.$$

Now we want to find the barycentric coordinates (with respect to $\triangle ABC$) for the $X_{20}$ point of $\triangle ABE$. First, we compute the lengths of the sides of $\triangle ABE$ using the Distance Formula (Lemma~\ref{lemma:distanceFormula}). We find that
$$
\begin{aligned}
AB&=c\\
BE&=\frac{\sqrt{a^2 r (p+r)-b^2 p r+c^2 p (p+r)}}{p+r}\\
CE&=\frac{b r}{p+r}
\end{aligned}
$$
If we call these lengths $a_1$, $b_1$, and $c_1$, respectively, then we can use the Change of Coordinates
Formula (Lemma~\ref{lemma:changeOfCoordinates}) to find the coordinates for point $F$, the $X_{20}$ point of $\triangle ABE$, by substituting $a_1$, $b_1$, and $c_1$ for $a$, $b$, and $c$ in the expression for the
normalized barycentric coordinates for $X_{20}$. We find that the barycentric coordinates for $F$
before normalization (with respect to $\triangle ABC$) are
$$
F=\Bigl(a^4 (2 p+3 r)-2 a^2 \left(b^2 (2 p+r)+c^2 r\right)+\left(b^2-c^2\right)
   \left(b^2 (2 p-r)+c^2 (2 p+r)\right):$$
$$-a^4 (p+r)+2 a^2 \left(b^2
   (p-r)+c^2 (p+r)\right)-b^4 (p-3 r)-2 b^2 c^2 (3 p+r)-c^4 (p+r):$$
$$a^4 (-r)-2 a^2 \left(c^2 (2
   p+r)-b^2 r\right)-b^4 r+2 b^2 c^2 (2 p-r)+c^4 (4 p+3 r)\Bigr).
$$
Using the same procedure, similar expressions are found for the coordinates of points $G$, $H$, and $I$,
the $X_{20}$ points of triangles $BCE$, $CDE$, and $DAE$, respectively.

Next, we want to compare the area of quadrilaterals $ABCD$ and $FGHI$. Letting $K$ be the area of $\triangle ABC$,
we can use the Area Formula (Lemma~\ref{lemma:areaFormula}) to find the area of $\triangle CDA$. We find that
$$[CDA]=-qK.$$
Note that this area is positive since we assumed $q<0$.
Thus,
$$[ABCD]=[ABC]+[CDA]=K-qK=K(1-q).$$

To compute the area of quadrilateral $FGHI$, we take a shortcut. By Theorem~5.1 of \cite{shapes}, quadrilateral
$FGHI$ is a parallelogram. Thus,
$$[FGHI]=2[FGH].$$
Computing the area of $\triangle FGH$ using the Area Formula, we find (after simplifying and using the fact that $p+q+r=1$):
$$[FGH]=K(1-q).$$
Consequently, $[ABCD]=[FGH]=\frac12[FGHI]$.
\end{proof}

\begin{open}
Is there a simpler or purely geometric proof for Theorem~\ref{theorem:dpX20}?
\end{open}

\newpage

\subsection{Proofs for Orthodiagonal Quadrilaterals}\ 

We now give proofs for the results listed in Table~\ref{table:dp1} for orthodiagonal quadrilaterals.

\begin{table}[ht!]
\caption{}
\label{table:dp1}
\begin{center}
\begin{tabular}{|l|l|p{2.2in}|}
\hline
\multicolumn{3}{|c|}{\textbf{\color{blue}\large \strut Central Quadrilaterals formed by the Diagonal Point}}\\ \hline
\textbf{Quadrilateral Type}&\textbf{Relationship}&\textbf{centers}\\ \hline
\ru orthodiagonal&$[ABCD]=32[FGHI]$&546\\
\cline{2-3}
\ru &$[ABCD]=18[FGHI]$&381\\
\cline{2-3}
\ru &$[ABCD]=8[FGHI]$&5, 402\\
\cline{2-3}
\ru &$[ABCD]=2[FGHI]$&3, 97, 122, 123, 127, 131, 216, 268,
339, 382, 408, 417, 418, 426, 440, 441, 454, 464--466, 577, 828, 852, 856\\
\cline{2-3}
\ru &$[ABCD]=\frac12[FGHI]$&22, 23, 151, 175, 253, 280, 347, 401, 858, 925\\
\hline
\ru equiorthodiagonal&$[ABCD]=2[FGHI]$&124\\
\cline{2-3}
\ru &$[ABCD]=\frac12[FGHI]$&102\\
\hline
\ru rhombus&$[ABCD]=4[FGHI]$&10\\
\cline{2-3}
\ru &$[ABCD]=[FGHI]$&40, 84\\
\cline{2-3}
\ru &$\partial ABCD=\partial FGHI$&40, 84\\
\hline
\ru rectangle&$[ABCD]=8[FGHI]$&402, 620\\
\cline{2-3}
\ru &$[ABCD]=6[FGHI]$&395, 396\\
\cline{2-3}
\ru &$[ABCD]=2[FGHI]$&11, 115, 116, 122--125, 127, 130, 134--137, 139, 244--247, 338, 339, 865--868\\
\cline{2-3}
\ru &$[ABCD]=\frac32[FGHI]$&616, 617\\
\cline{2-3}
\ru &$[ABCD]=\frac12[FGHI]$&146--153\\
\hline
\end{tabular}
\end{center}
\end{table}

\relbox{Relationship $[ABCD]=32[FGHI]$}

\begin{proposition}[$X_{546}$ Property of a Right Triangle]
\label{proposition:rightTriangleX546}
Let $P$ be the $X_{546}$ point of right triangle $ABC$, with $A$ being the vertex of the right angle. Then $BC=8AP$.
\end{proposition}

\begin{figure}[h!t]
\centering
\includegraphics[width=0.4\linewidth]{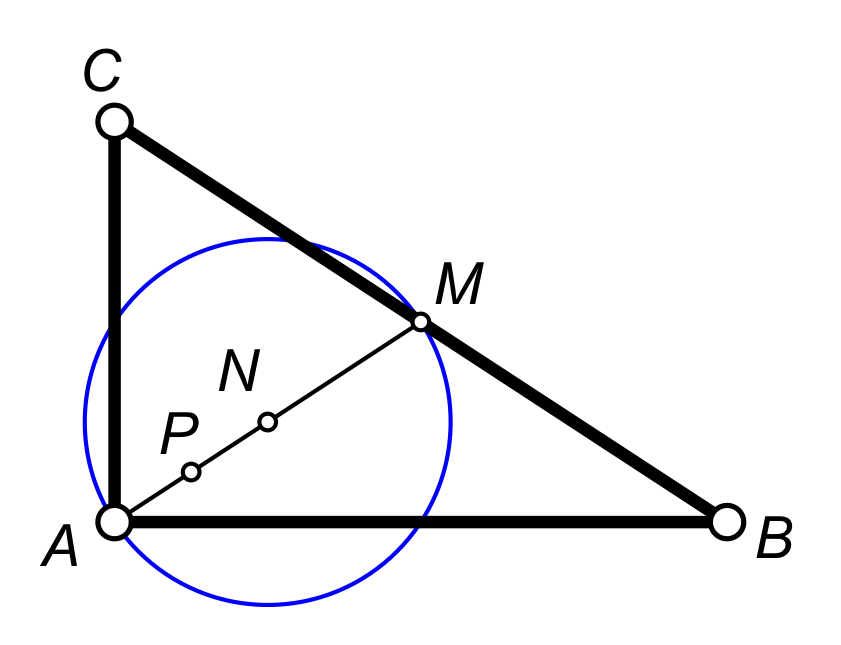}
\caption{right triangle, $P=X_{546}\implies BC=8AP$}
\label{fig:rightTriangleX546}
\end{figure}

\begin{proof}
Let $M$ be the midpoint of the hypotenuse of $\triangle ABC$.
Let $N$ be the nine-point center of $\triangle ABC$.
Let $P$ be the $X_{546}$ point of $\triangle ABC$ (Figure~\ref{fig:rightTriangleX546}).
By Lemma \ref{lemma:midpointMedian}, $N$ is the midpoint of $AM$.
Since $AM=\frac12BC$, this means $AN=\frac14BC$.
According to \cite{ETC546}, $P$ is the midpoint of $HN$ where $H$ is the
orthocenter of $\triangle ABC$. But the orthocenter of a right triangle
is the vertex of the right angle, so $P$ is the midpoint of $AN$
and $AP=\frac12AN$. Thus $AP=\frac18BC$.
\end{proof}

\begin{theorem}
\label{thm:dpX546}
Let $E$ be the diagonal point of orthodiagonal quadrilateral $ABCD$.
Let $F$, $G$, $H$, and $I$ be the $X_{546}$ point of $\triangle EAB$, $\triangle EBC$, 
$\triangle ECD$, and $\triangle EDA$, respectively.
Then
%$FGHI$ is a parallelogram and
$$[ABCD]=32[FGHI].$$
\end{theorem}

\begin{figure}[h!t]
\centering
\includegraphics[width=0.4\linewidth]{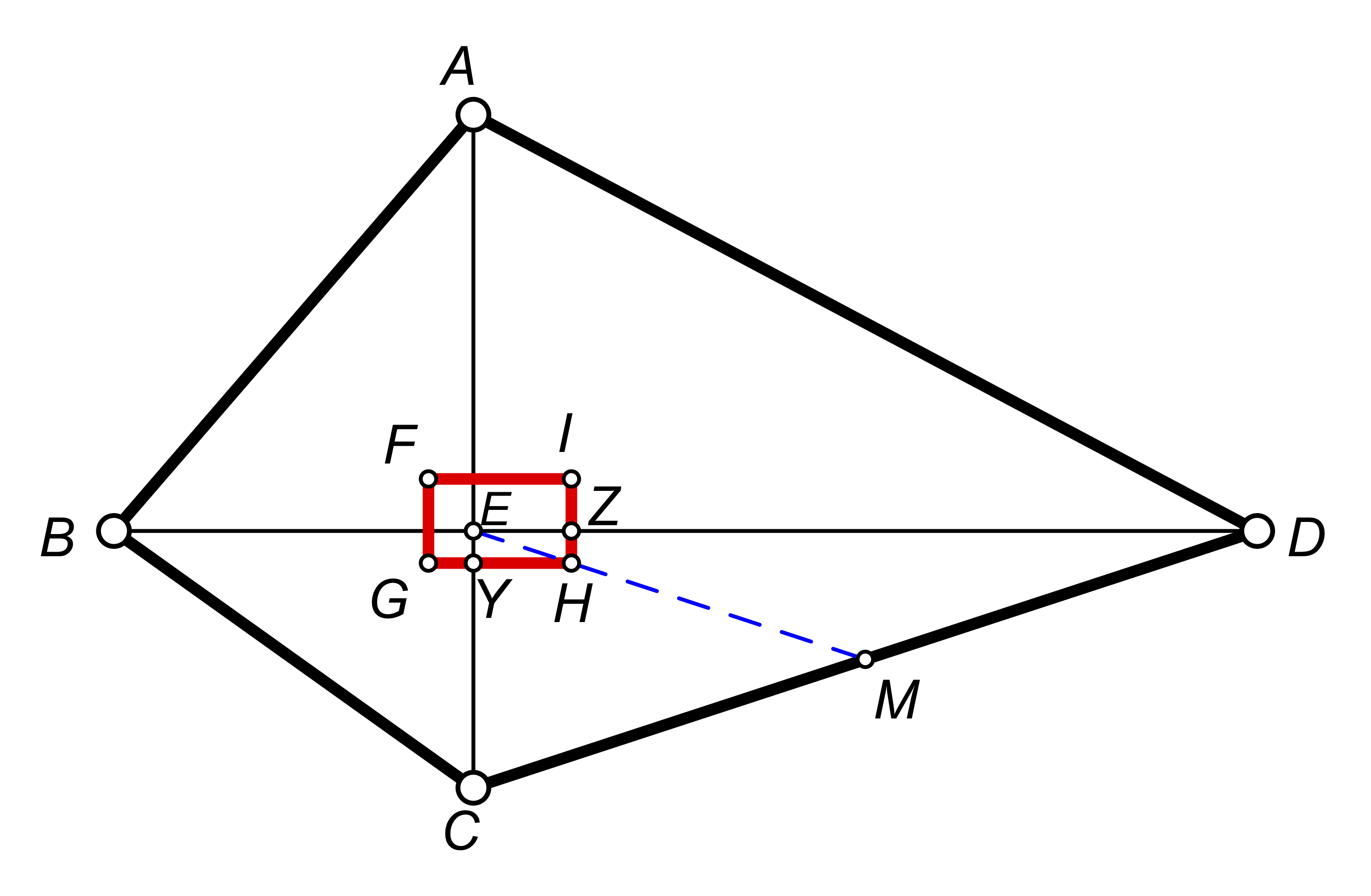}
\caption{orthodiagonal, $X_{546}\implies [ABCD]=32[FGHI]$}
\label{fig:dpOrthoX546}
\end{figure}

\begin{proof}
Since $ABCD$ is orthodiagonal, $\triangle ECD$ is a right triangle.
By Proposition~\ref{proposition:rightTriangleX546}, $EH=\frac18CD$.
Let $M$ be the midpoint of $CD$ and let $Y$ and $Z$ be the feet of the perpendiculars
from $H$ to $EC$ and $ED$, respectively.
By Lemma~\ref{lemma:rightTriangle},
$$\frac{[ECD]}{[EYHZ]}=2\left(\frac{EM}{EH}\right)^2=2\left(\frac{CD/2}{CD/8}\right)^2=32.$$
By symmetry, the same is true for the other three radial triangles.
Thus, $$[ABCD]=32[FGHI].$$
\end{proof}

\relbox{Relationship $[ABCD]=18[FGHI]$}

\begin{theorem}
\label{thm:dpOrthoX381}
Let $E$ be the diagonal point of orthodiagonal quadrilateral $ABCD$.
Let $F$, $G$, $H$, and $I$ be the $X_{381}$ point of $\triangle EAB$, $\triangle EBC$, 
$\triangle ECD$, and $\triangle EDA$, respectively (Figure~\ref{fig:dpOrthoX381}).
Then
$$[ABCD]=18[FGHI].$$
\end{theorem}

\begin{figure}[h!t]
\centering
\includegraphics[width=0.4\linewidth]{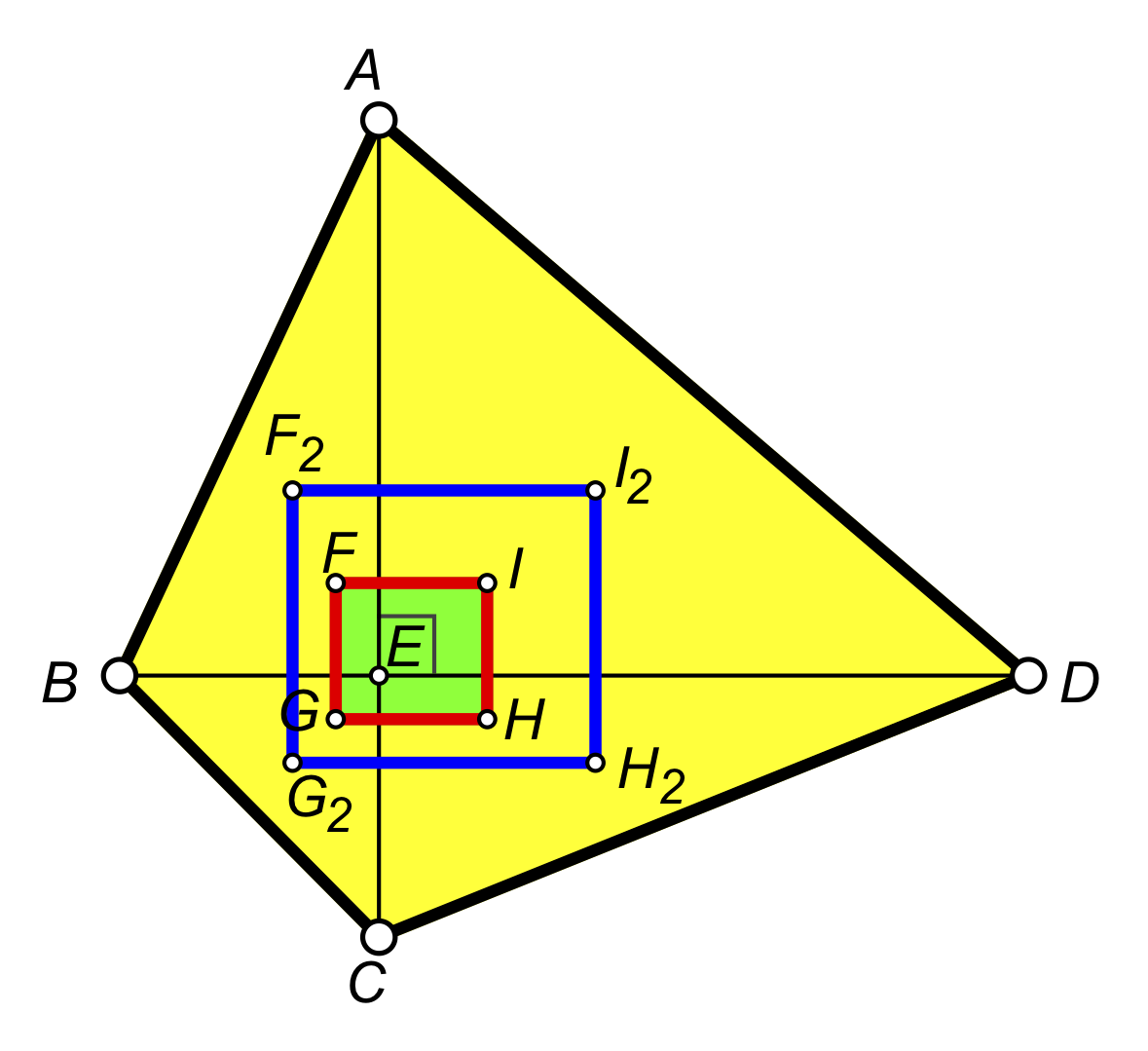}
\caption{orthodiagonal, $X_{381}\implies [ABCD]=18[FGHI]$}
\label{fig:dpOrthoX381}
\end{figure}

\begin{proof}
Let $F_2$, $G_2$, $H_2$, and $I_2$ be the $X_{2}$ point (centroid) of $\triangle EAB$, $\triangle EBC$, 
$\triangle ECD$, and $\triangle EDA$, respectively (Figure~\ref{fig:dpOrthoX381}).
By Theorem~\ref{theorem:cpCentroids},
$$[FGHI]=\frac29[ABCD].$$
From \cite{ETC381} we learn that the $X_{381}$ point of a right triangle is the midpoint of the
centroid and orthocenter of that triangle. The orthocenter of a right triangle coincides with the vertex
of the right angle. Therefore, $F$ is the midpoint of $EF_2$. The same reasoning applies to $G$, $H$, and $I$.
Therefore, quadrilateral $FGHI$ is similar to quadrilateral $F_2G_2H_2I_2$, with similarity ratio 2.
Thus,
$$[FGHI]=\frac14[F_2G_2H_2I_2]=\frac14\left(\frac29[ABCD]\right)=\frac{1}{18}[ABCD]$$
or $[ABCD]=18[FGHI]$.
\end{proof}

\relbox{Relationship $[ABCD]=8[FGHI]$}

\begin{lemma}
\label{lemma:dpX402}
In a right triangle, the $X_5$ point coincides with the $X_{402}$ point.
\end{lemma}

\begin{proof}
From \cite{ETC402}, we find that the barycentric coordinates for the $X_{402}$ point
are $(p:q:r)$ where

$p=(2 a^4-a^2 b^2-a^2 c^2-b^4+2 b^2 c^2-c^4) (a^8-a^6 b^2-a^6 c^2-2 a^4
   b^4+5 a^4 b^2 c^2-2 a^4 c^4+3 a^2 b^6-3 a^2 b^4 c^2-3 a^2 b^2 c^4+3 a^2 c^6-b^8-b^6
   c^2+4 b^4 c^4-b^2 c^6-c^8)$,

$q=(-a^4-a^2 b^2+2 a^2 c^2+2 b^4-b^2 c^2-c^4)(-a^8+3 a^6 b^2-a^6 c^2-2 a^4
   b^4-3 a^4 b^2 c^2+4 a^4 c^4-a^2 b^6+5 a^2 b^4 c^2-3 a^2 b^2 c^4-a^2 c^6+b^8-b^6 c^2-2
   b^4 c^4+3 b^2 c^6-c^8)$,

and\\
$r=(-a^4+2 a^2 b^2-a^2 c^2-b^4-b^2 c^2+2 c^4)(-a^8-a^6 b^2+3 a^6 c^2+4 a^4
   b^4-3 a^4 b^2 c^2-2 a^4 c^4-a^2 b^6-3 a^2 b^4 c^2+5 a^2 b^2 c^4-a^2 c^6-b^8+3 b^6
   c^2-2 b^4 c^4-b^2 c^6+c^8)$.

When $a^2=b^2+c^2$, these coordinates simplify to
$$X_{402}=\Bigl(16 b^6 c^6:8 b^6 c^6: 8b^6 c^6\Bigr)=(2:1:1).$$

The barycentric coordinates for the $X_5$ point are
$$\Bigl(a^2 \left(b^2+c^2\right)-\left(b^2-c^2\right)^2:b^2
   \left(a^2+c^2\right)-\left(c^2-a^2\right)^2:c^2
   \left(a^2+b^2\right)-\left(a^2-b^2\right)^2\Bigr).$$
When $a^2=b^2+c^2$, these coordinates simplify to
$$X_{5}=\Bigl(4 b^2 c^2:2 b^2 c^2:2 b^2 c^2\Bigr)=(2:1:1).$$
Thus, for right triangles, $X_{402}$ and $X_{5}$ coincide.
The common point is the midpoint of the hypotenuse by Lemma \ref{lemma:midpointMedian}.
\end{proof}

\begin{theorem}
\label{thm:dpX5}
Let $E$ be the diagonal point of orthodiagonal quadrilateral $ABCD$.
Let $F$, $G$, $H$, and $I$ be the $X_5$ or $X_{402}$ point of $\triangle EAB$, $\triangle EBC$, 
$\triangle ECD$, and $\triangle EDA$, respectively (Figure~\ref{fig:dpOrthoX5}).
Then $FGHI$ is a rectangle and
$$[ABCD]=8[FGHI].$$
\end{theorem}

\void{
\begin{figure}[h!t]
\centering
\includegraphics[width=0.4\linewidth]{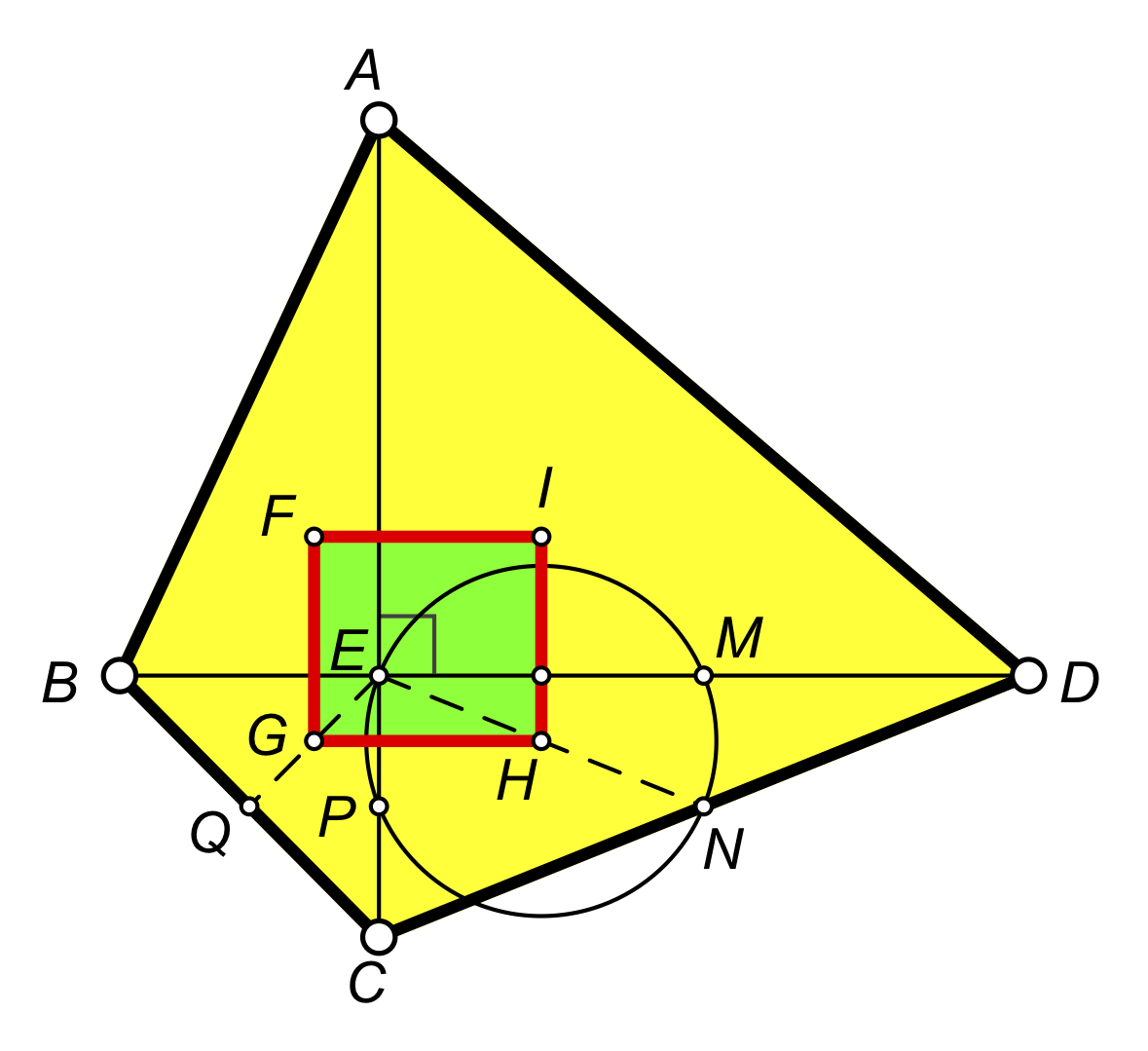}
\caption{OBSOLETE - orthodiagonal, $X_5\implies [ABCD]=8[FGHI]$}
\label{fig:dpOrthoX8}
\end{figure}
}

\begin{figure}[h!t]
\centering
\includegraphics[width=0.4\linewidth]{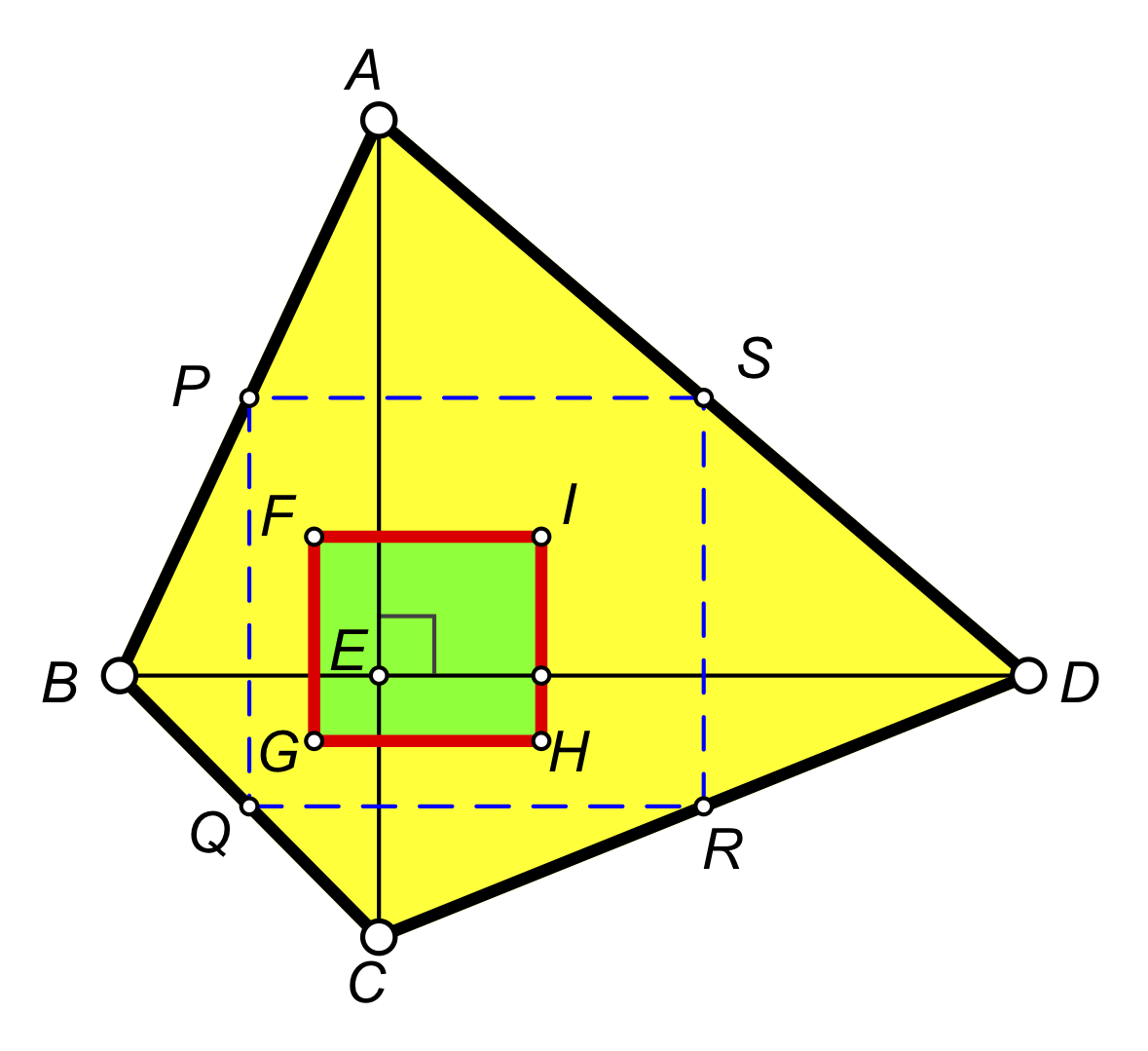}
\caption{orthodiagonal, $X_5\implies [ABCD]=8[FGHI]$}
\label{fig:dpOrthoX5}
\end{figure}

\begin{proof}
By Lemma~\ref{lemma:dpX402}, we need only prove this theorem when the chosen center is the $X_5$ point (the nine-point
center).
Since the quadrilateral is orthodiagonal, each of the radial triangles is a right triangle.
By Lemma~\ref{lemma:midpointMedian}, the $X_5$ point (nine-point center) of each of these triangles
is the midpoint of the median to the hypotenuse. Thus, $FGHI$ is similar to the Varignon parallelogram $PQRS$
of the quadrilateral with ratio of similarity $\frac12$ (Figure~\ref{fig:dpOrthoX5}). So the ratio of their areas is $1:4$.
By Lemma~\ref{lemma:Varignon}, the Varignon parallelogram has half the area of the quadrilateral.
So
$$[FGHI]=\frac14[PQRS]=\frac14\left(\frac12[ABCD]\right)=\frac18[ABCD]$$
or $[ABCD]=8[FGHI]$.
\end{proof}

\relbox{Relationship $[ABCD]=2[FGHI]$}

\begin{theorem}
\label{thm:hypotenuseMidpoint}
Let $E$ be the diagonal point of orthodiagonal quadrilateral $ABCD$.
Let $F$, $G$, $H$, and $I$ be centers of $\triangle EAB$, $\triangle EBC$, 
$\triangle ECD$, and $\triangle EDA$, respectively (Figure~\ref{fig:dpOrthoX3}).
If the center function for the four centers have
the properties $f(a,b,c)=0$ and $f(b,c,a)=f(c,a,b)$ when $c^2=a^2+b^2$,
then $FGHI$ is a rectangle and $[ABCD]=2[FGHI]$.
\end{theorem}

\begin{figure}[h!t]
\centering
\includegraphics[width=0.35\linewidth]{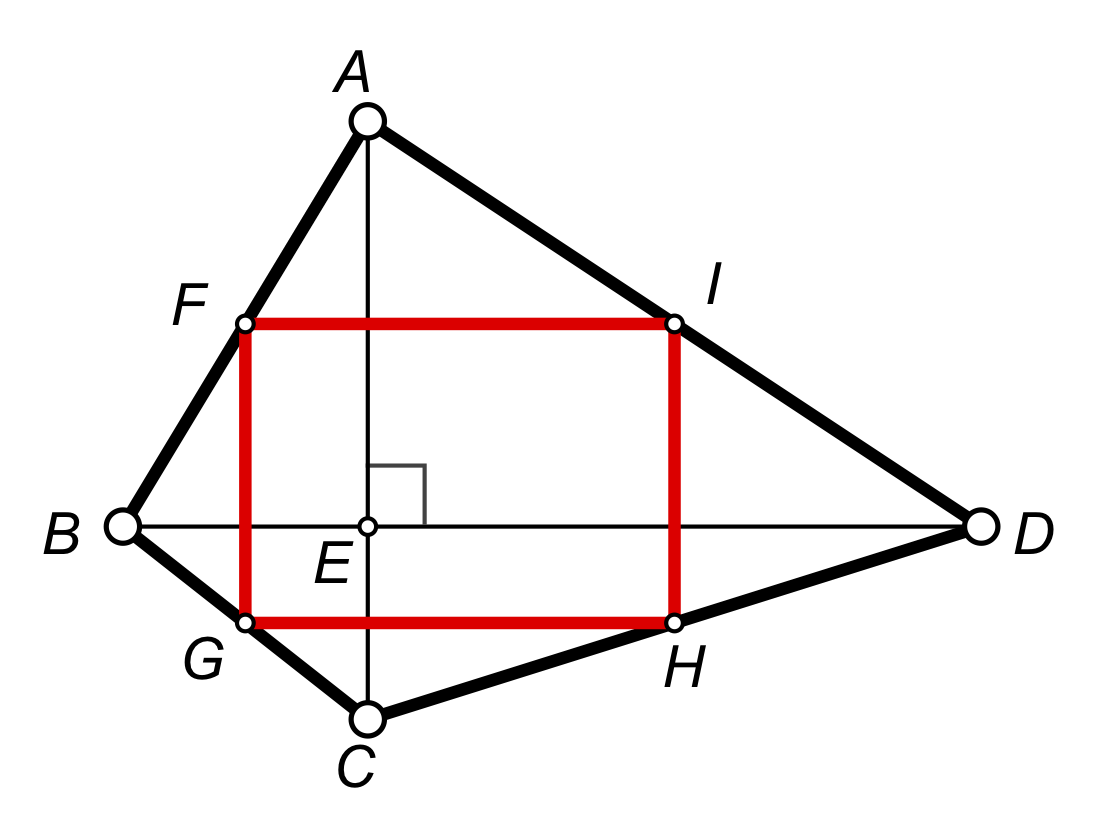}
\caption{$[ABCD]=2[FGHI]$}
\label{fig:dpOrthoX3}
\end{figure}

\begin{proof}
By Lemma \ref{lemma:hypotenuseMidpoint}, $F$, $G$, $H$, and $I$ are the midpoints of
the sides of the quadrilateral (Figure \ref{fig:dpOrthoX3}). By Lemma \ref{lemma:Varignon},
these midpoints form a parallelogram whose sides are parallel to the diagonals of the quadrilateral.
Since the diagonals of the quadrilateral are perpendicular, this parallelogram must be a rectangle.
Also by Lemma \ref{lemma:Varignon}, the area of this rectangle is half the area of the quadrilateral.
\end{proof}

\textbf{Examples.} Some examples of centers described by Theorem \ref{thm:hypotenuseMidpoint}
are $X_{3}$, $X_{97}$, $X_{122}$, $X_{123}$, $X_{127}$, $X_{131}$, $X_{216}$, $X_{268}$, $X_{339}$, $X_{382}$, $X_{408}$, $X_{417}$, $X_{418}$, $X_{426}$, $X_{440}$, $X_{441}$, $X_{454}$, $X_{464}$, $X_{465}$, $X_{466}$, $X_{577}$, $X_{828}$, $X_{852}$, and $X_{856}$. These agree with the centers found for orthodiagonal quadrilaterals
listed in Table \ref{table:dp1} with relationship $[ABCD]=2[FGHI]$.
For an orthodiagonal quadrilateral, these centers all coincide with the midpoints of
the sides of the quadrilateral.

\begin{theorem}
\label{thm:atE}
Let $ABC$ be a right triangle with right angle at $A$. Let $M$ be the midpoint of
the hypotenuse. Then the center $X_n$ coincides with $M$ for the following values of $n$:
\end{theorem}
3, 97, 122, 123, 127, 131, 216, 268, 339, 408, 417, 418, 426, 440, \
441, 454, 464, 465, 466, 577, 828, 852, 856.

\begin{theorem}
\label{thm:onMedian}
Let $ABC$ be a right triangle with right angle at $A$. Let $M$ be the midpoint of
the hypotenuse. Then the center $X_n$ lies on the line $AM$ (but does
not coincide with $A$ or $M$) for the following values of $n$:
\end{theorem}
2, 5, 20, 21, 22, 23, 95, 140, 199, 233, 237, 253, 280, 347, 376, \
377, 379, 381, 382, 383, 401, 402, 404, 405, 409, 411, 413, 416, 439, \
442, 443, 446, 448, 449, 452, 453, 474, 546, 547, 548, 549, 550, 631, \
632, 851, 853, 854, 855, 857, 858, 859, 861, 863, 864, 865, 866, 867, \
868, 925, 964.

We have excluded values of $n$ for which $X_n$ lies on the line at infinity.

\begin{proof}
%By the Midpoint Theorem, the Cartesian coordinates for $M$ are $\left(\frac{c}{2},\frac12\right)$.
The normalized barycentric coordinates for $M$ are $\left(0:\frac12:\frac12\right)$, using
Formula (12) from \cite{Grozdev}.
The barycentric equation for line $AM$ is $y=z$, using Equation (3) from \cite{Grozdev}.
If $P$ lies on line $AM$ and has barycentric coordinates $(u:v:w)$, we must have $v=w$.
Examining the first 1,000 centers in \cite{ETC}, those for which $v=w$ when $a^2=b^2+c^2$
are the ones given by Theorems \ref{thm:atE} and \ref{thm:onMedian}.
The ones for which $u=0$ are the ones given by Theorem \ref{thm:atE}.
\end{proof}

%\newpage

\begin{theorem}
\label{thm:ratio}
Let $ABC$ be a right triangle with right angle at $A$. Let $M$ be the midpoint of
the hypotenuse. Then the center $X_n$ lies on the line $AM$ and the ratio
of $AM$ to $AX_n$ is a constant for the values shown in
Table~\ref{table:diagonalPointRatios}:
\end{theorem}

\begin{table}[ht!]
\caption{Values of $n$ for which the ratio $AM:AX_n$ is a constant}
\label{table:diagonalPointRatios}
\begin{center}
\begin{tabular}{|l|c|}
\hline
$n$&\textbf{ratio}\\ \hline
\ru 2&$\frac32$\\ \hline
\ru 3&$1$\\ \hline
\ru 5&$2$\\ \hline
\ru 20&$\frac12$\\ \hline
\ru 22&$\frac12$\\ \hline
\ru 23&$\frac12$\\ \hline
\ru 95&$\frac54$\\ \hline
\ru 97&$1$\\ \hline
\ru 122&$1$\\ \hline
\ru 123&$1$\\ \hline
\ru 127&$1$\\ \hline
\ru 131&$1$\\ \hline
\end{tabular}
~
\begin{tabular}{|l|c|}
\hline
$n$&\textbf{ratio}\\ \hline
\ru 140&$\frac43$\\ \hline
\ru 216&$1$\\ \hline
\ru 233&$\frac53$\\ \hline
\ru 253&$\frac12$\\ \hline
\ru 268&$1$\\ \hline
\ru 280&$\frac12$\\ \hline
\ru 339&$1$\\ \hline
\ru 347&$\frac12$\\ \hline
\ru 376&$\frac34$\\ \hline
\ru 381&$3$\\ \hline
\ru 382&$1$\\ \hline
\ru 401&$\frac12$\\ \hline
\end{tabular}
~
\begin{tabular}{|l|c|}
\hline
$n$&\textbf{ratio}\\ \hline
\ru 402&$2$\\ \hline
\ru 408&$1$\\ \hline
\ru 417&$1$\\ \hline
\ru 418&$1$\\ \hline
\ru 426&$1$\\ \hline
\ru 440&$1$\\ \hline
\ru 441&$1$\\ \hline
\ru 454&$1$\\ \hline
\ru 464&$1$\\ \hline
\ru 465&$1$\\ \hline
\ru 466&$1$\\ \hline
\ru 546&$4$\\ \hline
\end{tabular}
~
\begin{tabular}{|l|c|}
\hline
$n$&\textbf{ratio}\\ \hline
\ru 547&$\frac{12}{7}$\\ \hline
\ru 548&$\frac45$\\ \hline
\ru 549&$\frac65$\\ \hline
\ru 550&$\frac23$\\ \hline
\ru 577&$1$\\ \hline
\ru 631&$\frac54$\\ \hline
\ru 632&$\frac{10}{7}$\\ \hline
\ru 828&$1$\\ \hline
\ru 852&$1$\\ \hline
\ru 856&$1$\\ \hline
\ru 858&$\frac12$\\ \hline
\ru 925&$\frac12$\\ \hline
\end{tabular}

\end{center}
\end{table}

\begin{proof}
The lengths of $AM$ and $AX_n$ are easily found (by computer) using the
Distance Formula (Lemma~\ref{lemma:distanceFormula}).
The resulting ratio is simplified using the constraint that $a^2=b^2+c^2$.
Values of this ratio that are not constant are discarded.
\end{proof}

Combining the data in Theorem \ref{thm:ratio} with Lemma \ref{lemma:rightTriangle}
gives cases where $\frac{[ABC]}{[AZXY]}$ is rational, which in turn gives cases
where $[ABCD]/[FGHI]$ is rational where $FGHI$ is the central quadrilateral
formed by the diagonal point using center $X_n$ in an orthodiagonal quadrilateral $ABCD$.
This therefore proves the entries in Table~\ref{table:dp1}
for the orthodiagonal quadrilateral entries.

It is interesting to note that the entries in Table~\ref{table:dp1}
for the orthodiagonal quadrilateral with $[ABCD]=2[FGHI]$
includes the point $X_{382}$ which does not appear in Theorem \ref{thm:atE}.
This is because the $X_{382}$ of a right triangle coincides with the reflection
of the hypotenuse midpoint $M$ about the vertex of the right angle, $A$.
This agrees with \cite{ETC382} which states that $X_{382}$ is the reflection
of the circumcenter about the orthocenter.
(In a right triangle, the circumcenter
is $M$ and the orthocenter is $A$.)

Table~\ref{table:diagonalPointRatios} implies additional results about central quadrilaterals
associated with orthodiagonal quadrilaterals using the diagonal point as the radiator
that do not appear in Table~\ref{table:dp1}. This is because the computer analysis that
produced Table~\ref{table:dp1} only checked area ratios where the denominator was less than 6.
These supplementary results are shown in Table~\ref{table:dp1a}. They are obtained
by applying Lemma \ref{lemma:rightTriangle} to the entries in Table~\ref{table:diagonalPointRatios}.

\begin{table}[ht!]
\caption{Supplementary Results}
\label{table:dp1a}
\begin{center}
\begin{tabular}{|l|l|p{2.2in}|}
\hline
\multicolumn{3}{|c|}{\textbf{\color{blue}\large \strut Central Quadrilaterals formed by the Diagonal Point}}\\ \hline
\textbf{Quadrilateral Type}&\textbf{Relationship}&\textbf{centers}\\ \hline
\ru orthodiagonal&$[ABCD]=\frac{288}{49}[FGHI]$&547\\
\cline{2-3}
\ru &$[ABCD]=\frac{50}{9}[FGHI]$&233\\
\cline{2-3}
\ru &$[ABCD]=\frac{200}{49}[FGHI]$&632\\
\cline{2-3}
\ru &$[ABCD]=\frac{32}{9}[FGHI]$&140\\
\cline{2-3}
\ru &$[ABCD]=\frac{25}{8}[FGHI]$&95, 631\\
\cline{2-3}
\ru &$[ABCD]=\frac{72}{25}[FGHI]$&549\\
\cline{2-3}
\ru &$[ABCD]=\frac{32}{25}[FGHI]$&548\\
\cline{2-3}
\ru &$[ABCD]=\frac98[FGHI]$&376\\
\cline{2-3}
\ru &$[ABCD]=\frac89[FGHI]$&550\\
\hline
\end{tabular}
\end{center}
\end{table}

\newpage

\subsection{Proofs for Equidiagonal Orthodiagonal Quadrilaterals}\ 

We now give proofs for the results listed in Table~\ref{table:dp1} for equidiagonal orthodiagonal quadrilaterals.

\begin{lemma}
\label{lemma:orthodiag}
The area of an orthodiagonal quadrilateral with diagonals of length $x$ and $y$
is $\frac12xy$.
\end{lemma}

\begin{figure}[h!t]
\centering
\includegraphics[width=0.35\linewidth]{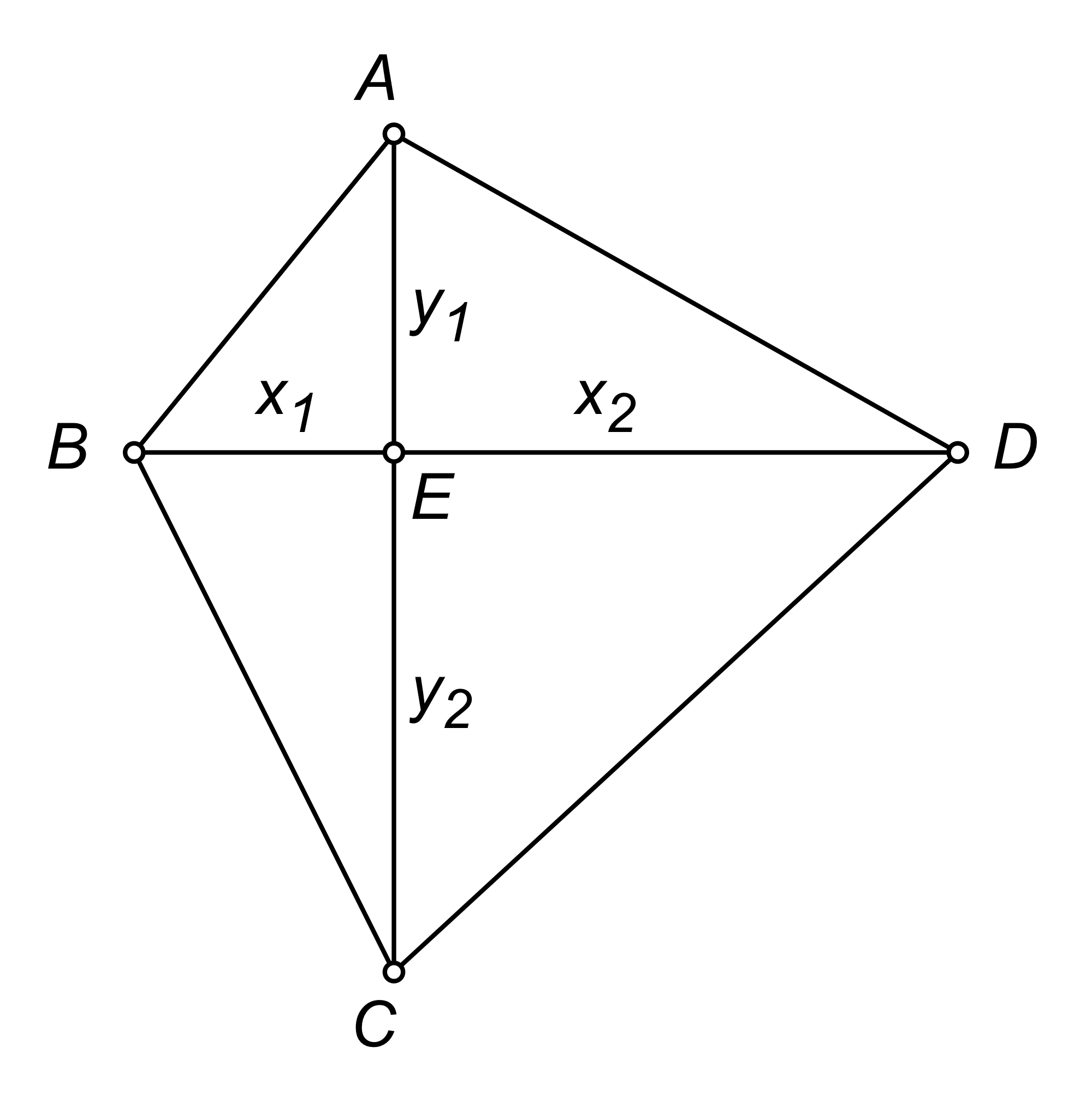}
\caption{An orthodiagonal quadrilateral}
\label{fig:dpOrthoArea}
\end{figure}

\begin{proof}
Let the segments of the diagonal of length $x$ be $x_1$ and $x_2$.
Let the segments of the diagonal of length $y$ be $y_1$ and $y_2$ (Figure~\ref{fig:dpOrthoArea}).
Then
$$
\begin{aligned}\ 
[ABCD]&=[ABE]+[BCE]+[CDE]+[DAE]\\
&=\frac12x_1y_1+\frac12x_1y_2+\frac12x_2y_2+\frac12x_2y_1\\
&=\frac12\left(x_1+x_2\right)\left(y_1+y_2\right)\\
&=\frac12xy.
\end{aligned}
$$
\end{proof}

\begin{lemma}
\label{lemma:equiorthodiag}
The area of an equidiagonal orthodiagonal quadrilateral with diagonal of length $d$
is $\frac12d^2$.
\end{lemma}

\begin{proof}
Let $x=y=d$ in Lemma~\ref{lemma:orthodiag}.
\end{proof}

\relbox{Relationship $[ABCD]=\frac12[FGHI]$}

\begin{lemma}
\label{lemma:rightTriangleX102}
Let $ABC$ be a right triangle with right angle at $A$ (Figure \ref{fig:dpRightTriangleX102}).
Then the $X_{102}$ point of $\triangle ABC$ lies on the angle bisector of $\angle BAC$
and also lies on the perpendicular bisector of $BC$.
\end{lemma}

\begin{figure}[h!t]
\centering
\includegraphics[width=0.35\linewidth]{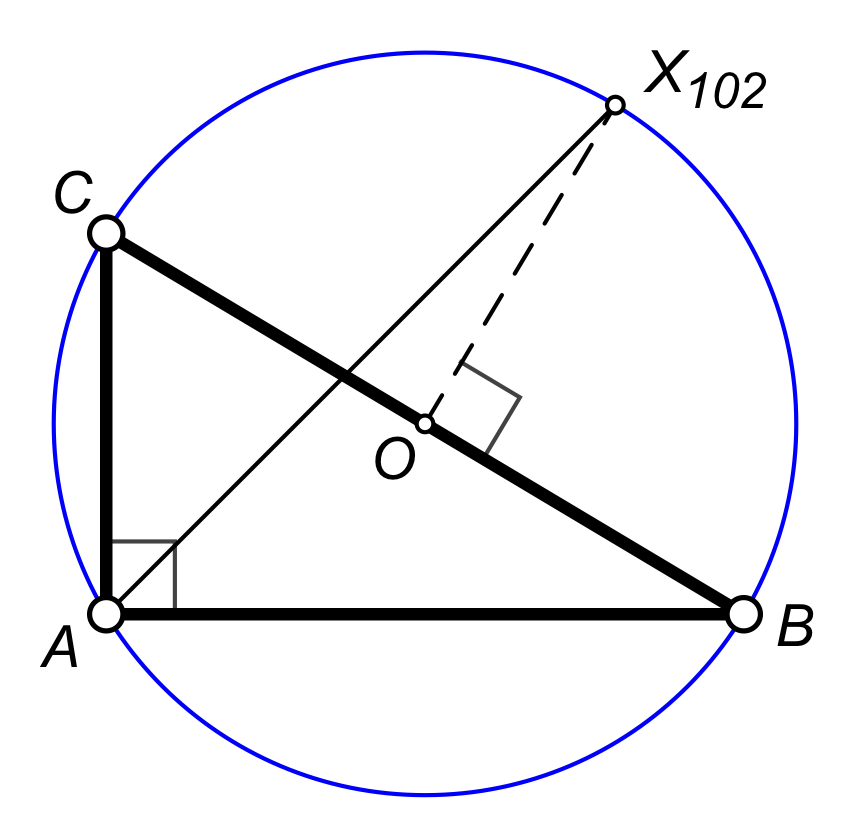}
\caption{The $X_{102}$ point of a right triangle}
\label{fig:dpRightTriangleX102}
\end{figure}

\begin{proof}
From \cite{ETC102}, we find that the center function corresponding to $X_{102}$ is
$$f(a,b,c)=a\times U\times V$$
where
$$U=a^4-a^3 c-2 a^2 b^2+a^2 b c+a^2 c^2+a b^2 c-2 a b c^2+a c^3+b^4-b^3 c+b^2 c^2+b
   c^3-2 c^4$$
and
$$V=a^4-a^3 b+a^2 b^2+a^2 b c-2 a^2 c^2+a b^3-2 a b^2 c+a b c^2-2
   b^4+b^3 c+b^2 c^2-b c^3+c^4.$$
(Recall that the center function is the first component of the \emph{trilinear} coordinates for $X_{102}$.)
A little computation shows that $f(a,b,c)-f(b,a,c)$ factors as
$$(a-b) (a+b-c) (a+b+c) \left(a^2+b^2-c^2\right)W$$
where
$$W=a^4-a^3 c+a^2 \left(-2 b^2+b c+c^2\right)+a c (b-c)^2+b^4-b^3 c+b^2 c^2+b c^3-2 c^4.$$
Thus, $f(a,b,c)=f(b,a,c)$ when $a^2+b^2=c^2$. Therefore, by Lemma \ref{lemma:angleBisector},
$X_{102}$ lies on the angle bisector of vertex $A$.

Also from \cite{ETC102}, we learn that the $X_{102}$ point of a triangle lies on its circumcircle.
Since the angle bisector of $\angle BAC$ bisects the arc from $B$ to $C$, $X_{102}$ must
lie on the perpendicular bisector of side $BC$.
\end{proof}

\begin{theorem}
\label{thm-equiorthoX102}
Let $E$ be the diagonal point of equidiagonal orthodiagonal quadrilateral $ABCD$.
Let $F$, $G$, $H$, and $I$ be the $X_{102}$ points of $\triangle EAB$, $\triangle EBC$, 
$\triangle ECD$,  and $\triangle EDA$, respectively (Figure~\ref{fig:dpEquiOrthoX102}).
Then $FGHI$ is an equidiagonal orthodiagonal quadrilateral and
$$\frac{[ABCD]}{[FGHI]}=\frac12.$$
Furthermore, quadrilaterals $ABCD$ and $FGHI$ have the same diagonal point
and centroid and the diagonals of $FGHI$ bisect the right angles
formed by the diagonals of $ABCD$. 
\end{theorem}

\begin{figure}[h!t]
\centering
\includegraphics[width=0.4\linewidth]{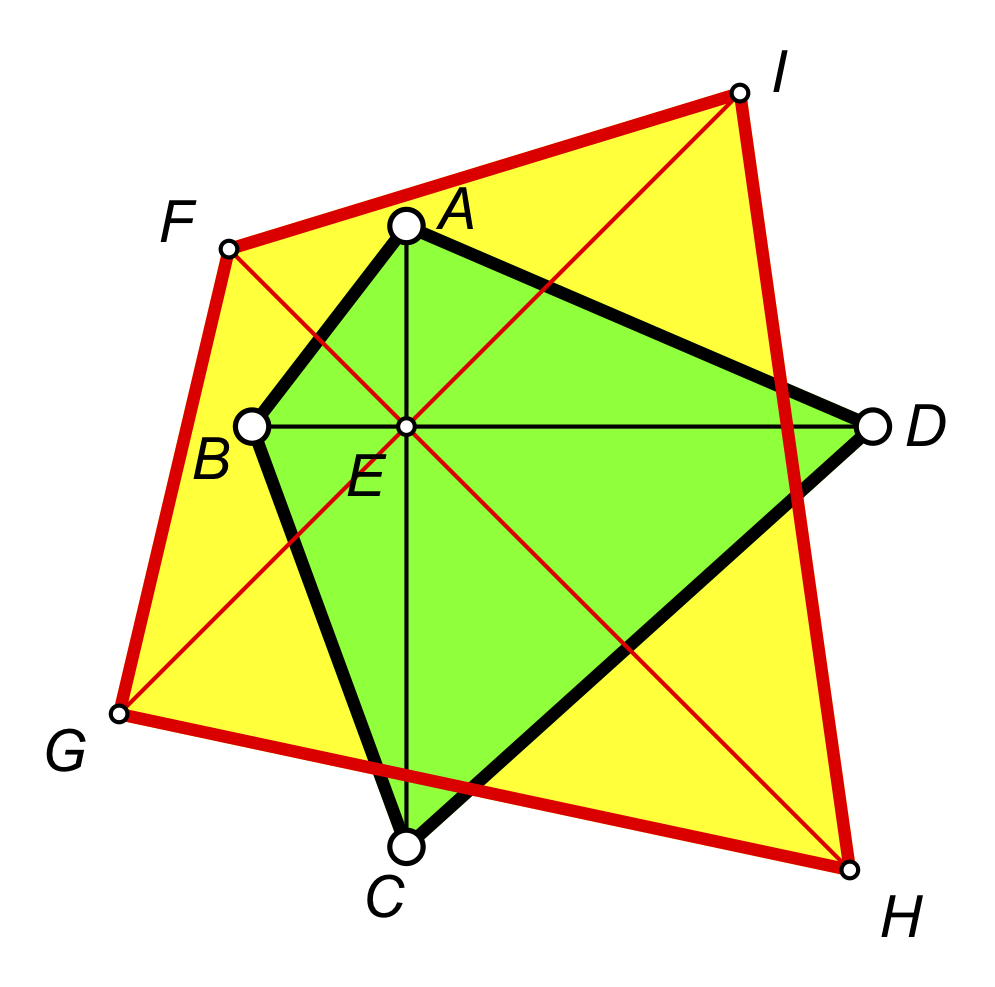}
\caption{equiortho, $X_{102}\implies \frac{[ABCD]}{[FGHI]}=\frac12$}
\label{fig:dpEquiOrthoX102}
\end{figure}

%\newpage

The following proof is due to Ahmet \c{C}etin \cite{Cetin}.

\begin{proof}

\begin{figure}[h!t]
\centering
\includegraphics[width=0.4\linewidth]{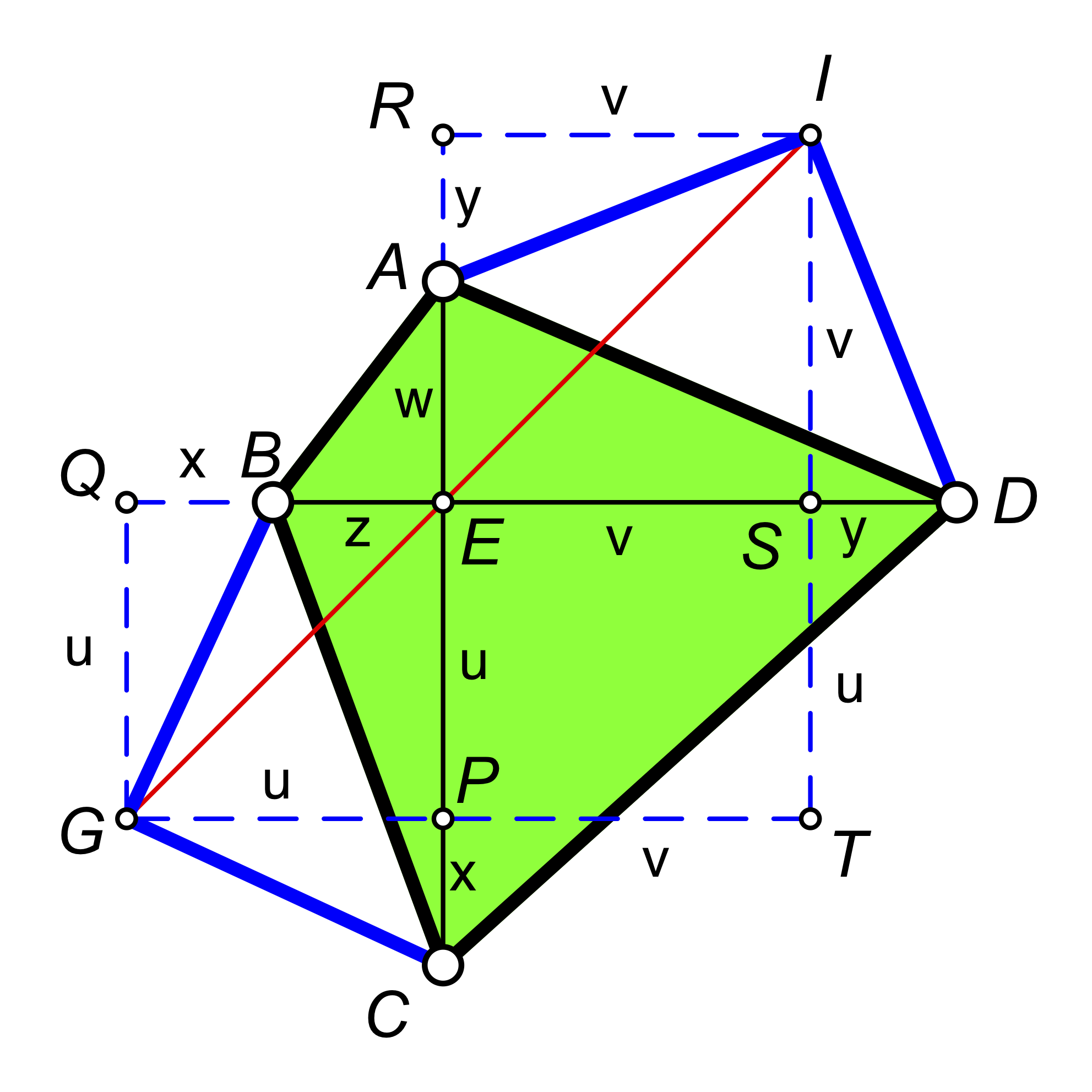}
\caption{}
\label{fig:dpEquiOrthoX102proof}
\end{figure}

Draw lines through $G$ and $I$ parallel to $AC$.
Draw lines through $G$ and $I$ parallel to $BD$.
These lines form intersection points, $P$, $Q$, $R$, $S$, and $T$ as shown in
Figure~\ref{fig:dpEquiOrthoX102proof}.
Let $\angle ADE=\theta$. By Lemma~\ref{lemma:rightTriangleX102},
$I$ lies on the perpendicular bisector of $AD$, so $\angle IDA=45\degrees$.
Hence $\angle IDS=\theta+45\degrees$.
Now
$$\angle IAR=180\degrees-\angle DAI-\angle EAD=180\degrees-45\degrees-(90\degrees-\theta)
=\theta+45\degrees.$$
Thus, $\angle IDS=\angle IAR$. Since angles $\angle DSI$ and $\angle ARI$ are right angles, and $ID=IA$,
we can conclude that $\triangle IDS\cong \triangle IAR$.
Hence, $DS=AR=y$.
Similarly, $CP=BQ=x$.

Since $ERIS$ is a parallelogram and diagonal $EI$ bisects $\angle SER$,
$ERIS$ mist be a square.
Similarly, $EPGQ$ is a square.
Thus $v=w+y$ where $ES=v$.
Similarly, $EPGQ$ is a square and $PG=PE=u=x+z$ where $BQ=x$ and $EB=z$.
Quadrilateral $PEST$ is a rectangle, so $PE=TS=u$.
Hence $GT=u+v=TI$. Thus $\triangle GTI$ is an isosceles right triangle and
therefore $GI=GT\sqrt2$.

Since $ABCD$ is equidiagonal, $$w+u+x=z+v+y.$$
Since $u=x+z$,
$$w+u+x=w+(x+z)+x.$$
Since $v=w+y$,
$$z+v+y=z+(w+y)+y.$$
Thus
$$w+(x+z)+x=z+(w+y)+y$$
which implies that $x=y$.

But $GT=u+v=u+w+y=u+w+x=AC$.
Since $GI=GT\sqrt2$, we have $GI^2=2AC^2$.
By Lemma~\ref{lemma:equiorthodiag}, $[FGHI]=2[ABCD]$.
\end{proof}

%\newpage

\relbox{Relationship $[ABCD]=2[FGHI]$}

\begin{theorem}
\label{thm-equiorthoX124}
Let $E$ be the diagonal point of equidiagonal orthodiagonal quadrilateral $ABCD$.
Let $F$, $G$, $H$, and $I$ be the $X_{124}$ points of $\triangle EAB$, $\triangle EBC$, 
$\triangle ECD$,  and $\triangle EDA$, respectively (Figure~\ref{fig:dpEquiOrthoX124}).
Then $FGHI$ is an equidiagonal orthodiagonal quadrilateral and
$$[ABCD]=2[FGHI].$$
\end{theorem}

\begin{figure}[h!t]
\centering
\includegraphics[width=0.4\linewidth]{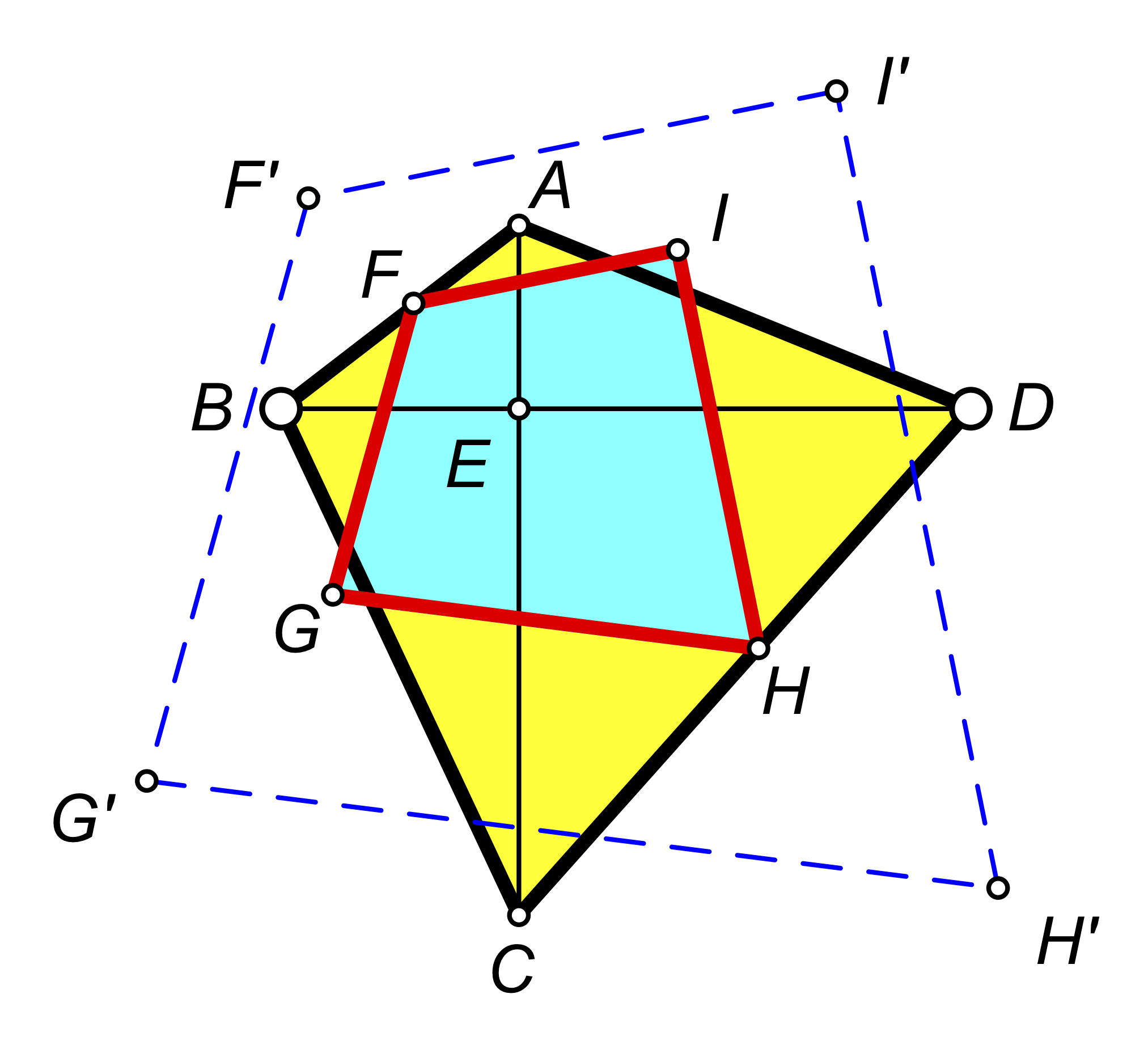}
\caption{equiortho, $X_{124}\implies [ABCD]=2[FGHI]$}
\label{fig:dpEquiOrthoX124}
\end{figure}

\begin{proof}
Let $F'$, $G'$, $H'$, and $I'$ be the $X_{102}$ points of $\triangle EAB$, $\triangle EBC$, 
$\triangle ECD$, and $\triangle EDA$, respectively (Figure~\ref{fig:dpEquiOrthoX124}).
According to \cite{ETC124}, the $X_{124}$ point of a triangle is the midpoint of its $X_4$
point and its $X_{102}$ point. But the $X_4$ point of a right triangle is the vertex
of the right angle,
so the $X_4$ points of all four radial triangles is point $E$.
Thus, $F$ is the midpoint of $EF'$ with similar results for $G$, $H$, and $I$.
Hence quadrilateral $FGHI$ is homothetic to quadrilateral $F'G'H'I'$ with ratio
of similitude 2.
But $[F'G'H'I']=2[ABCD]$ by Theorem~\ref{thm-equiorthoX102}.
Therefore, $[FGHI]=\frac12[ABCD]$. Also $F'G'H'I'$ is equidiagonal and orthodiagonal.
Since the quadrilaterals are homothetic, $FGHI$ must also be an equidiagonal orthodiagonal quadrilateral.
\end{proof}

\newpage

\subsection{Proofs for Rhombi}\ 

We now give proofs for the results listed in Table~\ref{table:dp1} for rhombi.

\relbox{Relationship $[ABCD]=4[FGHI]$}

\begin{lemma}
\label{lemma:rightTriangleX10}
Let $ABC$ be a right triangle with right angle at $A$.
Let $S$ be the Spieker center ($X_{10}$ point) of $\triangle ABC$.
Let $J$ and $K$ be the feet of perpendiculars from $S$ to $AC$ and $AB$, respectively,
so that $AKSJ$ is a rectangle (Figure \ref{fig:dpRightTriangleX10}).
Then $[ABC]=4[AKSJ]$.
\end{lemma}

\begin{figure}[h!t]
\centering
\includegraphics[width=0.4\linewidth]{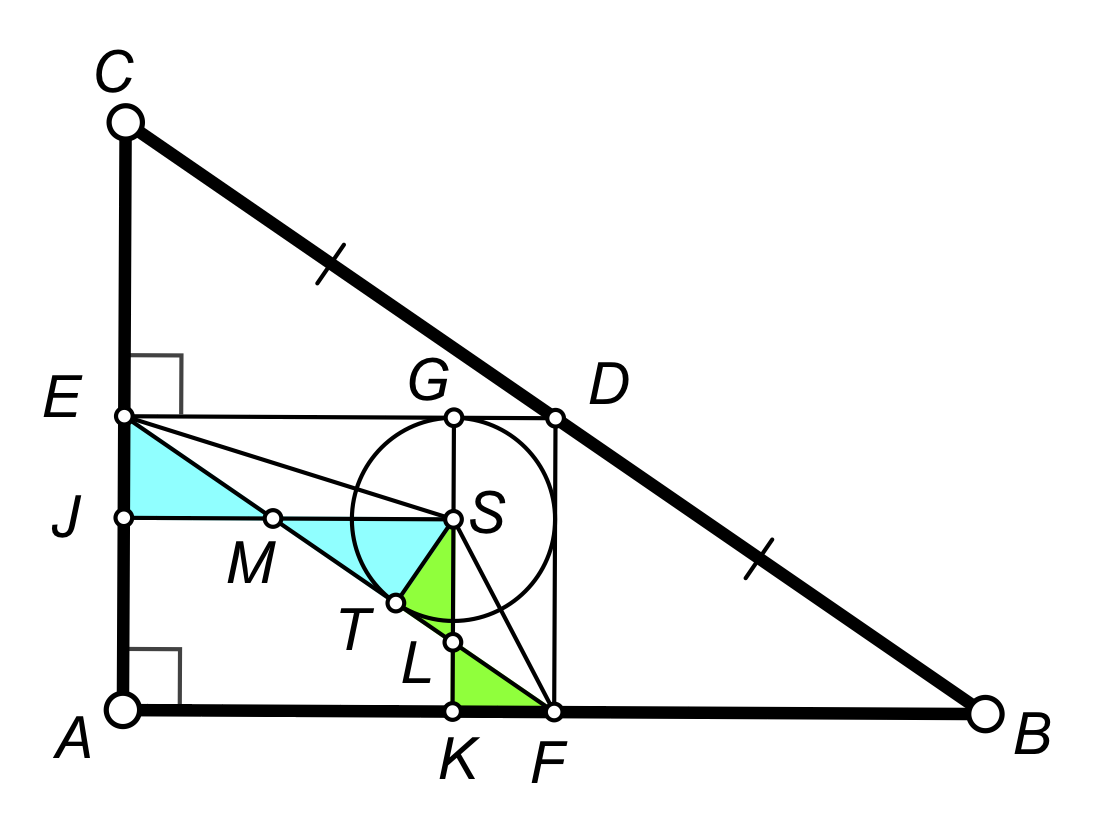}
\caption{$S$ is the Spieker center of $\triangle ABC$}
\label{fig:dpRightTriangleX10}
\end{figure}

\begin{proof}
Let $D$, $E$, and $F$ be the midpoints of the sides of $\triangle ABC$
as shown in Figure \ref{fig:dpRightTriangleX10}. From \cite{ETC10},
it is known that $S$ is the incenter of the medial triangle $DEF$.
Let $G$ and $T$ be the feet of the perpendiculars from $S$ to $DE$
and $EF$, respectively. Segments $SG$ and $ST$ are radii of the incircle
of $\triangle DEF$, so $ST=SG$. Since $SG\perp DE$, $SGEJ$ is a rectangle
and $SG=EJ$. Therefore, $EJ=ST$ and $\triangle EJM\cong \triangle STM$.
Similarly, $\triangle FKL\cong \triangle STL$.

So the green triangles
have the same area and the blue triangles have the same area.
Therefore, $[AKSJ]=[AFE]=\frac14[ABC]$.
\end{proof}

\begin{theorem}
Let $E$ be the diagonal point of rhombus $ABCD$.
Let $F$, $G$, $H$, and $I$ be the $X_{10}$ points of $\triangle EAB$, $\triangle EBC$, 
$\triangle ECD$,  and $\triangle EDA$, respectively (Figure~\ref{fig:dpRhombusX10}).
Then $FGHI$ is a rectangle and
$$[ABCD]=4[FGHI].$$
\end{theorem}

\begin{figure}[h!t]
\centering
\includegraphics[width=0.35\linewidth]{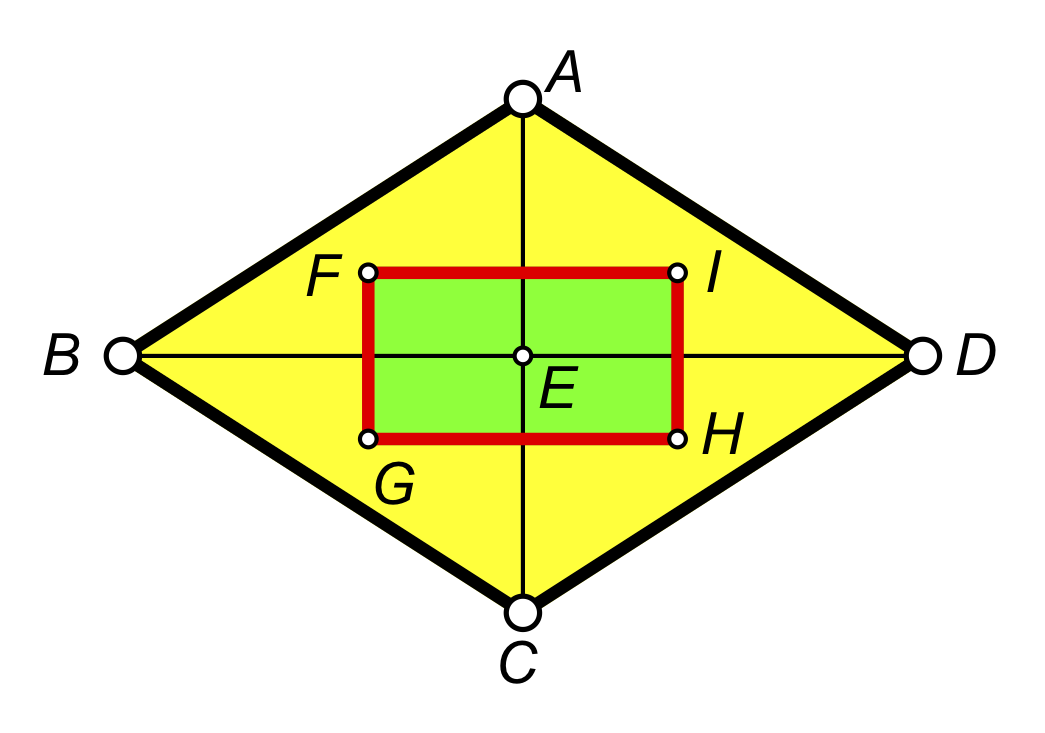}
\caption{rhombus, $X_{10}\implies [ABCD]=4[FGHI]$}
\label{fig:dpRhombusX10}
\end{figure}

\begin{proof}
Since $ABCD$ is a rhombus, $AE\perp ED$.
By Lemma \ref{lemma:rightTriangleX10}, the rectangle with diagonal $EI$
has one-fourth the area of $\triangle AED$.
By symmetry, the same is true for the rectangles with diagonals, $EF$, $EG$,
and $EH$. Thus $[ABCD]=4[FGHI]$.
\end{proof}

%\newpage

\relbox{Relationship $[ABCD]=[FGHI]$}

The following result is well known \cite{MathWorld-RightTriangle}.
\begin{lemma}
\label{lemma:inradius}
If $r$ is the inradius of a right triangle with hypotenuse $c$ and legs $a$ and $b$, then
$$r=\frac{a+b-c}{2}.$$
\end{lemma}

\begin{proposition}[$X_{40}$ Property of a Right Triangle]
\label{proposition:dpRightTriangleX40}
Let $\triangle ABC$ be a right triangle with right angle at C.
Let $F$ be its $X_{40}$ point.
Let $BC=a$, $AC=b$, and $AB=c$.
Let the distance from the $F$ to $BC$ be $p$
and let the distance from $F$ to $AC$ be $q$ as shown in Figure~\ref{fig:dpRightTriangleX40}.
Then
$$p+q=c\qquad\hbox{and}\qquad 2pq=ab.$$
\end{proposition}

\begin{figure}[h!t]
\centering
\includegraphics[width=0.45\linewidth]{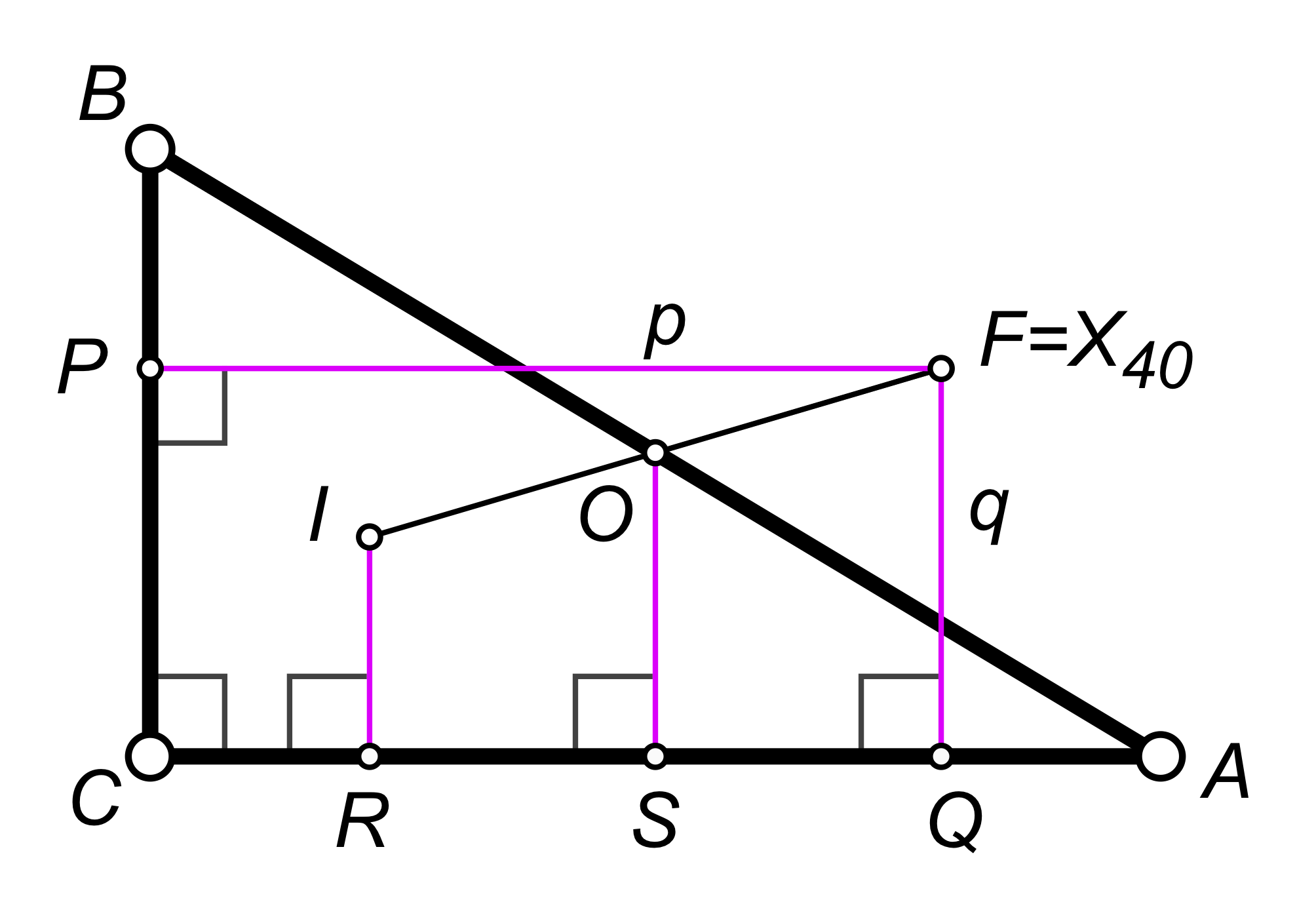}
\caption{$X_{40}$ point of a right triangle}
\label{fig:dpRightTriangleX40}
\end{figure}

\begin{proof}
Let $F$ be the Bevan point of $\triangle ABC$.
According to \cite{ETC40}, the Bevan point, $F=X_{40}$ is the reflection
of $I=X_1$ about $O=X_3$.
Let $R$ and $S$ be the projections of $I$ and $O$, respectively, on $AC$.
Since $\triangle ABC$ is a right triangle, $O$ is the midpoint of $AB$,
and $OS=a/2$.
Since $I$ is the center of the incircle of $\triangle ABC$, $IR=r$, where
$r$ is the inradius.
Since $O$ is the midpoint of $IF$, we have
$$OS=\frac{IR+FQ}{2}$$
or $a=q+r$. Thus, $q=a-r$.
Similarly, $p=b-r$.
Thus, using Lemma~\ref{lemma:inradius}, we have
$$p+q=a+b-2r=c$$
and
$$
\begin{aligned}
pq&=(a-r)(b-r)\\
&=\left(a-\frac{a+b-c}{2}\right)\left(b-\frac{a+b-c}{2}\right)\\
&=\left(\frac{a-b+c}{2}\right)\left(\frac{b-a+c}{2}\right)\\
&=\frac{c^2-(a-b)^2}{4}\\
&=\frac{a^2+b^2-(a-b)^2}{4}\\
&=\frac{ab}{2}.
\end{aligned}
$$
\end{proof}

\begin{proposition}[$X_{84}$ Property of a Right Triangle]
\label{proposition:dpRightTriangleX84}
Let $\triangle ABC$ be a right triangle with right angle at C.
Let $F$ be its $X_{84}$ point.
Let $BC=a$, $AC=b$, and $AB=c$.
Let the distance from the $F$ to $BC$ be $p$
and let the distance from $F$ to $AC$ be $q$ as shown in Figure~\ref{fig:dpRightTriangleX84}.
Then
$$p+q=c\qquad\hbox{and}\qquad 2pq=ab.$$
\end{proposition}

\begin{figure}[h!t]
\centering
\includegraphics[width=0.35\linewidth]{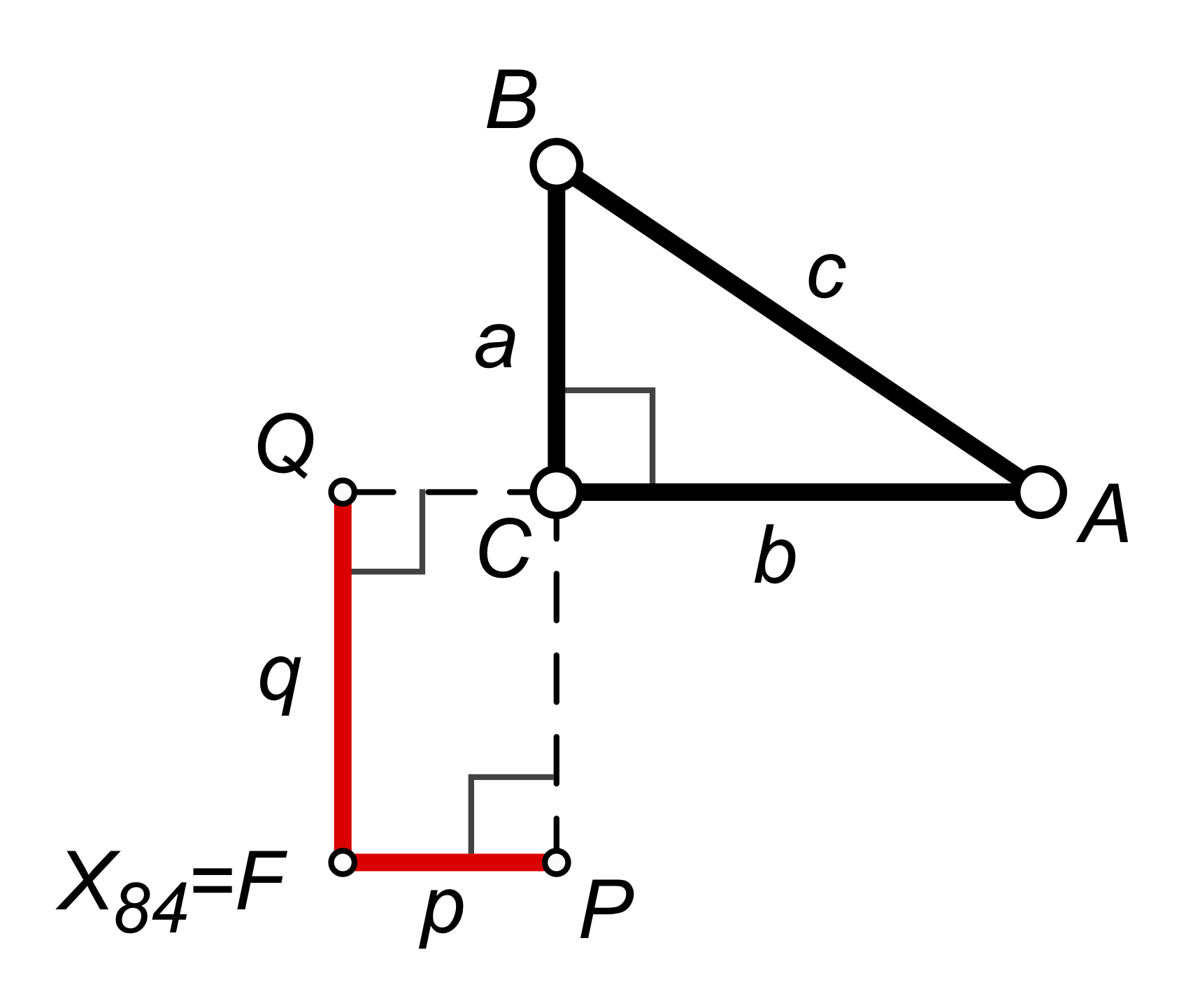}
\caption{$X_{84}$ point of a right triangle}
\label{fig:dpRightTriangleX84}
\end{figure}

\begin{proof}
According to \cite{ETC84}, the barycentric coordinates for the $X_{84}$ point of a triangle are
$$\left(a^3 \left(a^2-(b-c)^2\right)^2-a(b-c)^2 \left(a^2-(b+c)^2\right)^2:\,:\right).$$
With the condition $a^2+b^2=c^2$, this simplifies to
$$(b-c:a-c:c).$$
Note that
$$[BFC]=\frac12BC\times PF$$
so
$$p=\frac2a[BFC].$$
Using the Area Formula, we find that
$$[BFC]=\frac{(c-b)}{(a+b-c)}K$$
where $K$ is the area of $\triangle ABC$ (which, in this case, is $ab/2$).
Therefore,
$$p=\frac2a[BFC]=\frac{2(c-b)}{a(a+b-c)}K=\frac{2(c-b)}{a(a+b-c)}\left(\frac{ab}{2}\right)=
\frac{b(c-b)}{a+b-c}.$$
In the same way, we find
$$q=\frac{a(c-a)}{a+b-c}.$$
Thus,
$$p+q=\frac{ac+bc-(a^2+b^2)}{a+b-c}=\frac{ac+bc-c^2}{a+b-c}=c$$
and
$$pq=\frac{ab(c-b)(c-a)}{(a+b-c)^2}=\frac{a^2 b^2-a^2 b c-a b^2 c+a b c^2}{a^2+2 a b-2 a c+b^2-2 b c+c^2}$$
which simplifies to $ab/2$ when we substitute $a^2+b^2$ for $c^2$.
\end{proof}

\newpage

\begin{theorem}
\label{theorem:dpX84}
Let $E$ be the diagonal point of rhombus $ABCD$.
Let $F$, $G$, $H$, and $I$ be the $X_{40}$ points of $\triangle EAB$, $\triangle EBC$, 
$\triangle ECD$,  and $\triangle EDA$, respectively.
Let $F'$, $G'$, $H'$, and $I'$ be the $X_{84}$ points of $\triangle EAB$, $\triangle EBC$, 
$\triangle ECD$,  and $\triangle EDA$, respectively (Figure~\ref{fig:dpRhombusX84}).
Then $FGHI$ and $F'G'H'I'$ are congruent rectangles
and the three quadrilaterals have the same perimeter and the same area.
%$$\partial ABCD=\partial FGHI=\partial F'G'H'I'$$
\end{theorem}

\begin{figure}[h!t]
\centering
\includegraphics[width=0.4\linewidth]{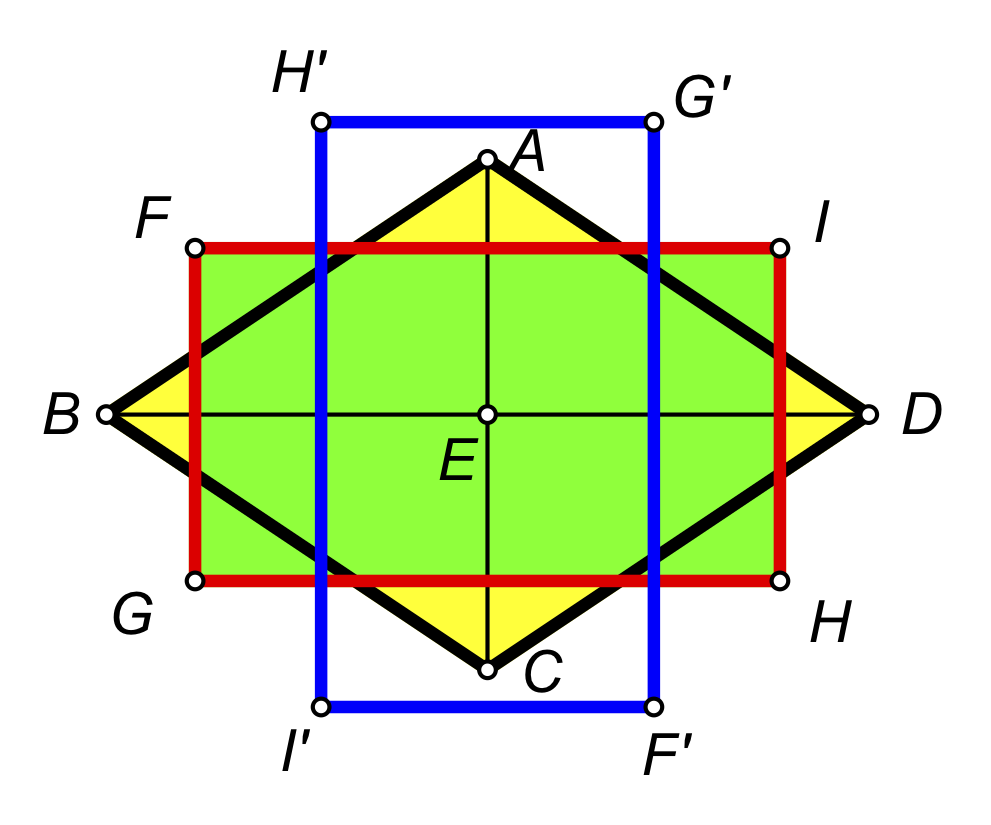}
\caption{$\displaystyle X_{40}\implies \partial ABCD=\partial FGHI$}
\label{fig:dpRhombusX84}
\end{figure}

\begin{proof}
By symmetry considerations (or using Theorem~5.50 from \cite{shapes}), we have
that the central quadrilaterals are rectangles with diagonal point $E$. The sides of
these rectangles are parallel to the diagonals of $ABCD$.
The diagonals of the rhombus divide each of the three figures into four congruent pieces.
We therefore only need to prove the appropriate result for one of these pieces.

\textbf{Part 1}: $X_{40}$

Let $P$ and $Q$ be the projections of $F$, the $X_{40}$ point of $\triangle ABE$, on
sides $BE$ and $CE$ of right triangle $ABE$, respectively, as shown
in Figure~\ref{fig:dpRhombusX40proof}.

\begin{figure}[h!t]
\centering
\includegraphics[width=0.4\linewidth]{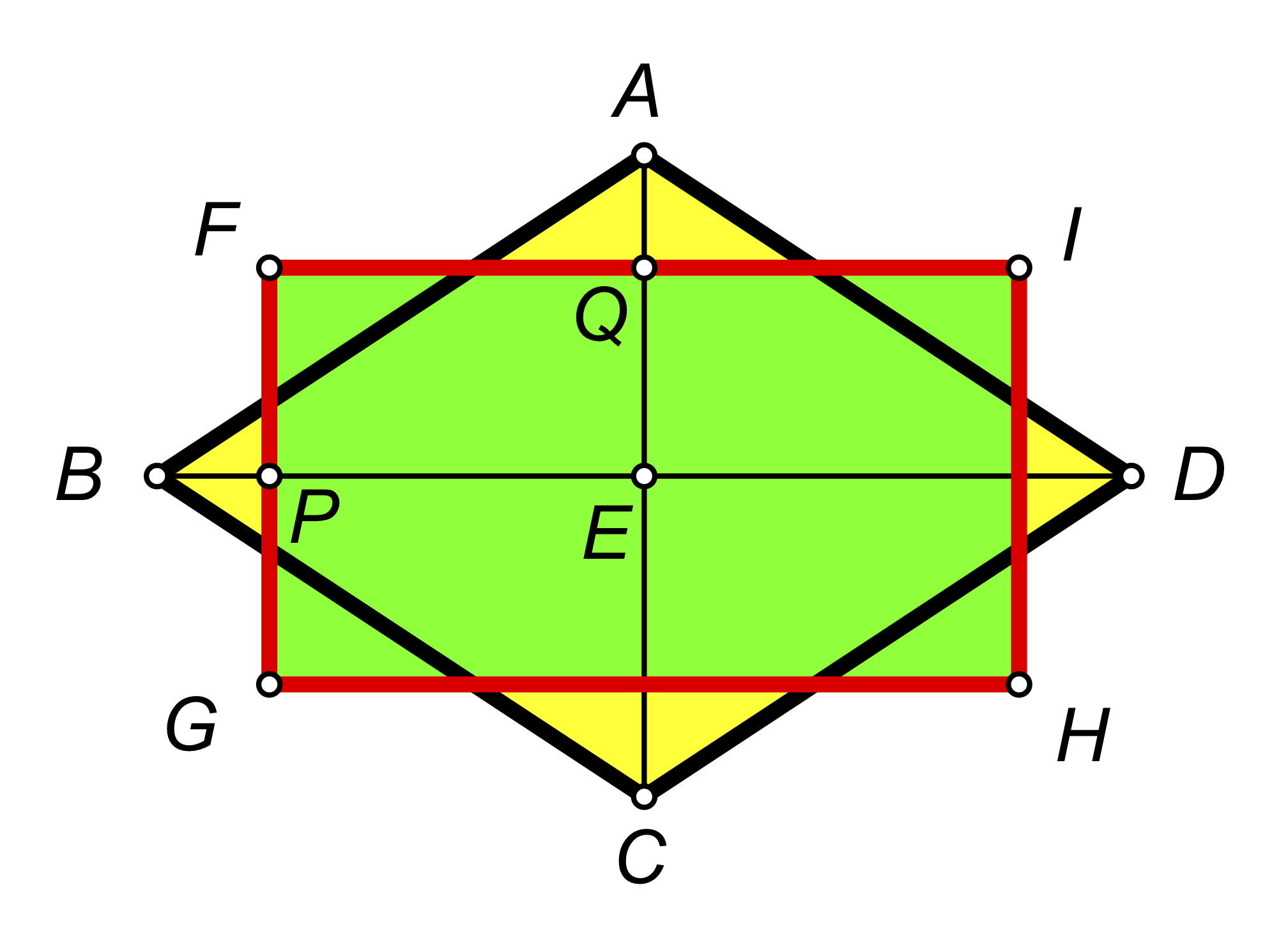}
\caption{}
\label{fig:dpRhombusX40proof}
\end{figure}

By Proposition~\ref{proposition:dpRightTriangleX40}, $FP+FQ=AB$.
Therefore, by symmetry, we have $$FG+GH+HI+IF=AB+BC+CD+DA.$$
Also, $FP\cdot FQ=(BE\cdot AE)/2$. So $[FPEQ]=[ABE]$.
Therefore, by symmetry, we have $[FGHI]=[ABCD]$.
The rectangle and the rhombus therefore have equal areas and equal perimeters.

\newpage
\textbf{Part 2}: $X_{84}$

Let $P$ and $Q$ be the projections of $F'$, the $X_{84}$ point of $\triangle ABE$, on
sides $BE$ and $CE$ of right triangle $ABE$, respectively,
as shown in Figure~\ref{fig:dpRhombusX84proof}.

\begin{figure}[h!t]
\centering
\includegraphics[width=0.4\linewidth]{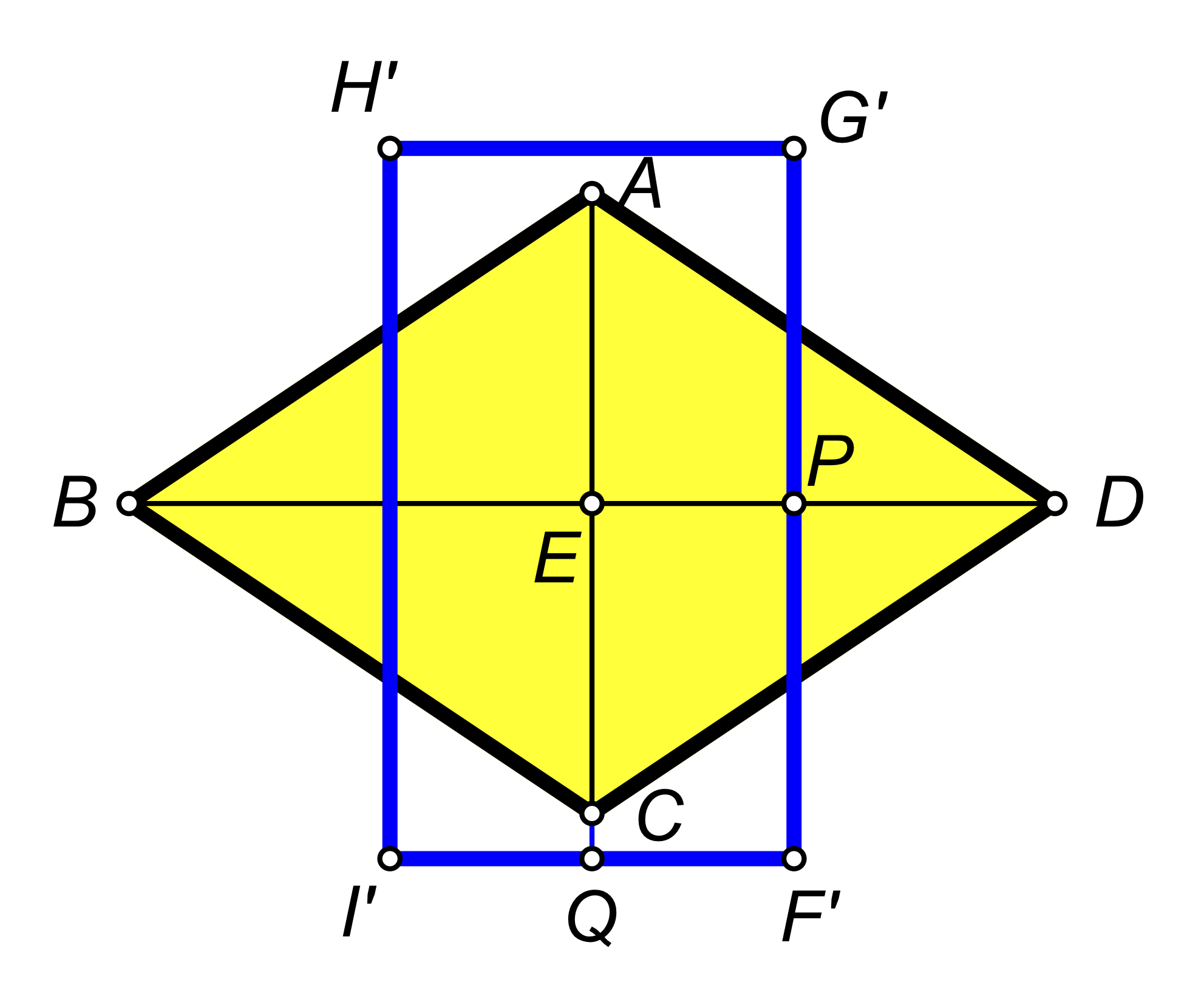}
\caption{}
\label{fig:dpRhombusX84proof}
\end{figure}

By Proposition~\ref{proposition:dpRightTriangleX84}, $F'P+F'Q=AB$.
Therefore, by symmetry, we have $$F'G'+G'H'+H'I'+I'F'=AB+BC+CD+DA.$$
Also, $F'P\cdot F'Q=(BE\cdot AE)/2$. So $[FPEQ]=[ABE]$.
Therefore, by symmetry, we have $[F'G'H'I']=[ABCD]$.
The rectangle and the rhombus therefore have equal areas and equal perimeters.

Finally, note that two rectangles that have equal areas and equal perimeters must
be congruent, so $ABCD\cong FGHI\cong F'G'H'I'$.
\end{proof}

\subsection{Proofs for Rectangles}\ 

We now give proofs for some of the results listed in Table~\ref{table:dp1} for rectangles.

The following result comes from \cite[Theorem~5.52]{shapes}.

\begin{lemma}
\label{thm:arbCenterRectangle}
For any triangle center,
if the reference quadrilateral is a rectangle,
then the central quadrilateral is a rhombus.
The two quadrilaterals have the same diagonal point.
The sides of the rhombus are parallel to diagonals of the rectangle
and are bisected by them
(Figure \ref{fig:dpRectangle}).
\end{lemma}

\begin{figure}[h!t]
\centering
\includegraphics[width=0.3\linewidth]{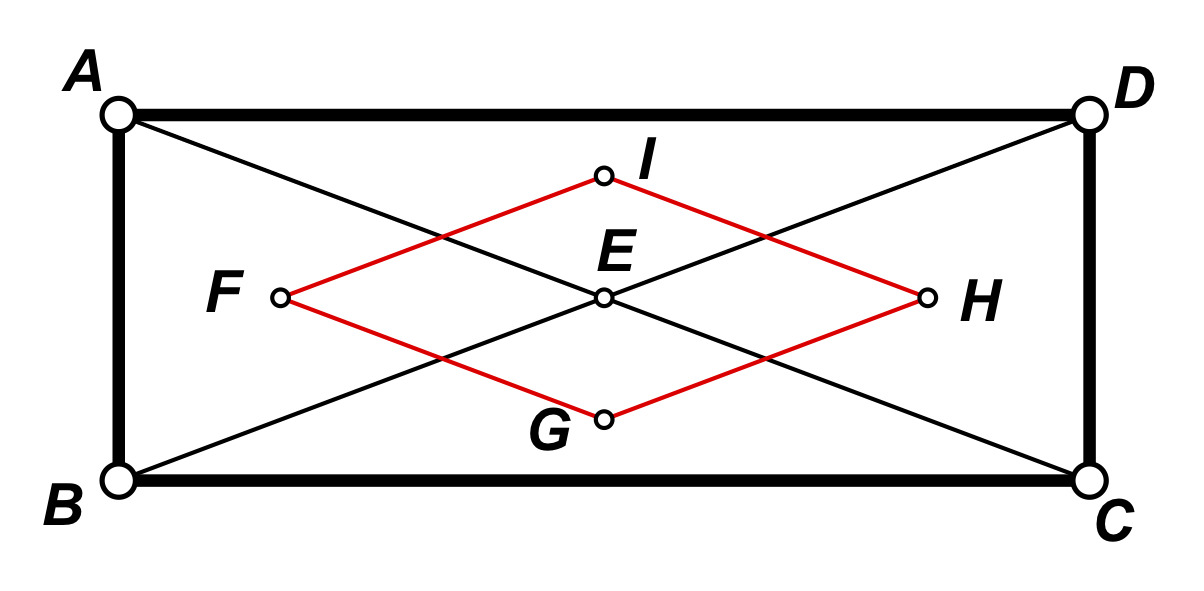}
\caption{rectangle $\implies$ rhombus}
\label{fig:dpRectangle}
\end{figure}

\begin{theorem}
\label{thm:isoscelesRatio}
Let $ABC$ be an isosceles triangle with $AB=AC$. Let $M$ be the midpoint of base $BC$.
Then the center $X_n$ lies on the line $AM$ and the ratio
of $X_nM$ to $AM$ is a constant for the values shown in
Table~\ref{table:diagonalPointIsoscelesRatios}:
\end{theorem}

\begin{table}[ht!]
\caption{Values of $n$ for which the ratio $X_nM:AM$ is a constant}
\label{table:diagonalPointIsoscelesRatios}
\begin{center}
\begin{tabular}{|l|c|}
\hline
$n$&\textbf{ratio}\\ \hline
\ru 2&$\frac13$\\ \hline
\ru 148&$-1$\\ \hline
\ru 149&$-1$\\ \hline
\ru 150&$-1$\\ \hline
\ru 290&$-\frac13$\\ \hline
\ru 402&$\frac12$\\ \hline
\ru 620&$\frac12$\\ \hline
\ru 671&$-\frac13$\\ \hline
\ru 903&$-\frac13$\\ \hline
\end{tabular}
\end{center}
\end{table}

\begin{proof}
The ratios $X_nM/AM$ are easily found (by computer) using Lemma~\ref{lemma:dpIsoscelesTriangle}.
The resulting ratio is simplified using the constraint that $b=c$.
Results where this ratio is not a constant are discarded.
\end{proof}

\begin{lemma}
\label{lemma:isoscelesRatioAll}
Let $ABC$ be an isosceles triangle with $AB=AC$. Let $M$ be the midpoint of base $BC$.
Let point $X$ lie on the line $AM$ such that (using signed distances) the ratio
$$\frac{XM}{AM}=k.$$
Then (still using signed distances),
$$\frac{AX}{AM}=1-k.$$
\end{lemma}

\begin{figure}[h!t]
\centering
\includegraphics[width=1\linewidth]{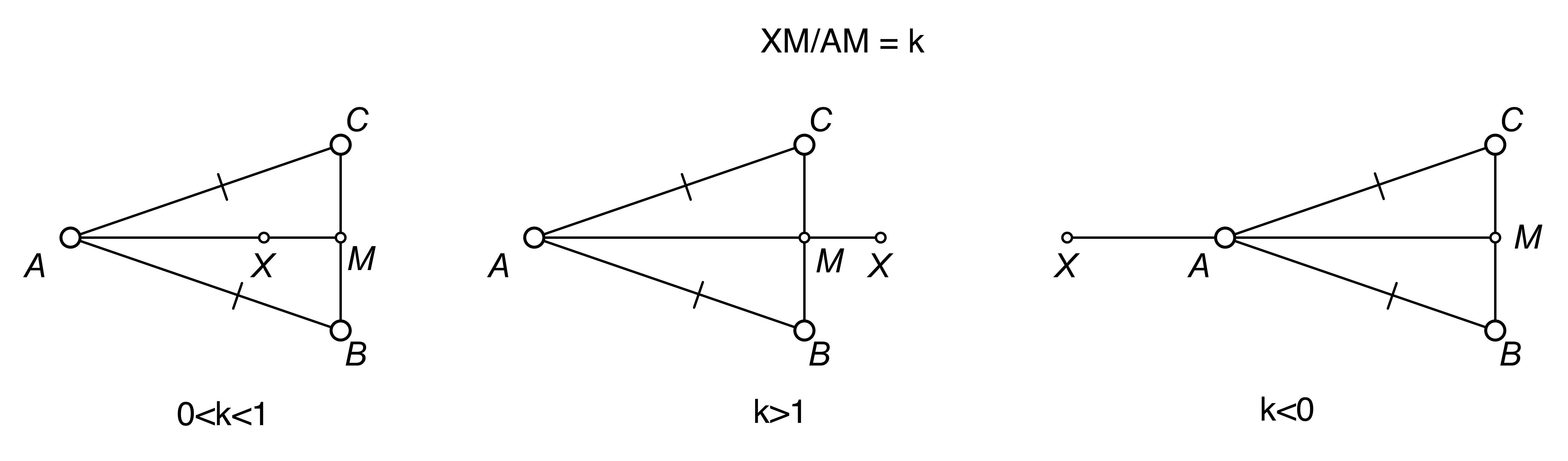}
\caption{Location of $X$ for various $k$}
\label{fig:dpIsoscelesCases}
\end{figure}

\begin{proof}
See Figure~\ref{fig:dpIsoscelesCases} for the three cases.
If $0<k<1$, then
$$\frac{AX}{AM}=\frac{AM-XM}{AM}=1-k.$$
If $k>1$, then
$$\frac{AX}{AM}=\frac{AM+MX}{AM}=\frac{AM-XM}{AM}=1-k.$$
If $k<0$, then
$$\frac{AX}{AM}=\frac{MX-MA}{AM}=-\left(\frac{XM-AM}{AM}\right)=-(k-1)=1-k.$$
\end{proof}

\relbox{Relationship $[ABCD]=8[FGHI]$}

\begin{theorem}
\label{thm:dpRectanglePos}
Let $E$ be the diagonal point of rectangle $ABCD$.
Let $X_n$ be a triangle center with the property that for all isosceles triangles with vertex $V$
and midpoint of base $M$, $X_nM/VM$ is a fixed positive constant $k$.
Let $F$, $G$, $H$, and $I$ be the $X_n$ points of $\triangle EAB$, $\triangle EBC$, 
$\triangle ECD$, and $\triangle EDA$, respectively.
Then
$$[ABCD]=\frac{2}{(1-k)^2}[FGHI].$$
\end{theorem}

\begin{figure}[h!t]
\centering
\includegraphics[width=0.4\linewidth]{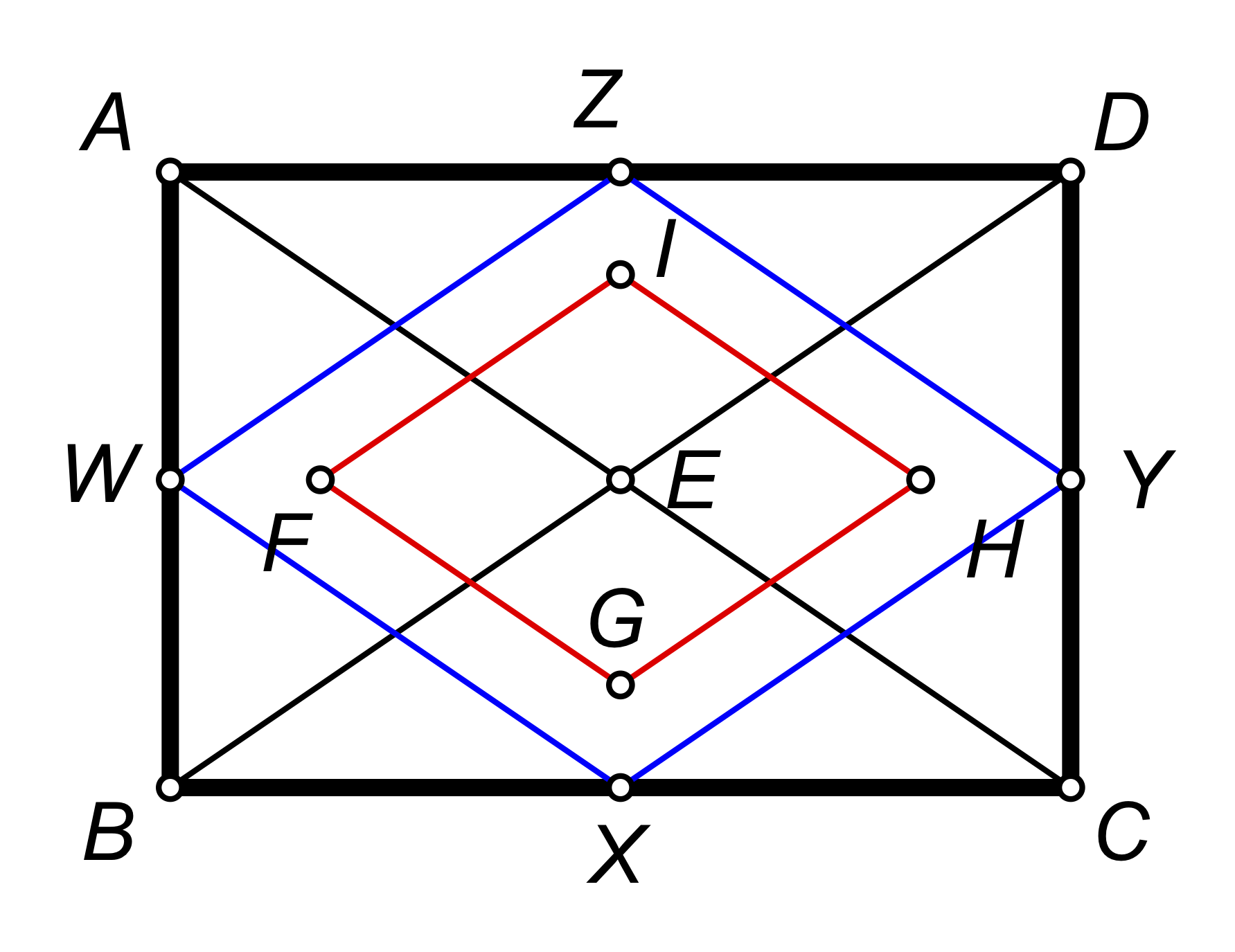}
\caption{}
\label{fig:dpRectanglePos}
\end{figure}

\begin{proof}
Since the diagonals of a rectangle are equal and bisect each other,
each of the radial triangles is isosceles with vertex $E$.
Let the midpoints of the sides of the rectangle be $W$, $X$, $Y$, and $Z$
as shown in Figure~\ref{fig:dpRectanglePos}.
Since $F$ is the $X_n$ point of $\triangle EAB$, by hypothesis,
$$\frac{FW}{EW}=k.$$
Since $k>0$
$$\frac{EF}{EW}=\frac{EW-FW}{EW}=1-\frac{FW}{EW}=1-k.$$
Similarly, $EG/EX=1-k$, $EH/EY=1-k$, and $EI/EZ=1-k$.
So quadrilaterals $FGHI$ and $WXYZ$ are homothetic, with $E$ the center of similitude
and ratio of similarity $1-k$.
Thus
$$[FGHI]=(1-k)^2[WXYZ].$$
But $[WXYZ]=\frac12[ABCD]$, so $[FGHI]=\frac{(1-k)^2}{2}[ABCD]$
or, equivalently, $[ABCD]=\frac{2}{(1-k)^2}[FGHI]$.
\end{proof}

When $n=402$ or $n=620$, $k=1/2$ and $[ABCD]=8[FGHI]$.

\newpage

\relbox{Relationship $[ABCD]=6[FGHI]$}

\begin{theorem}
\label{theorem:dpRectangleX395}
Let $E$ be the diagonal point of rectangle $ABCD$.
Let $F$, $G$, $H$, and $I$ be the $X_{395}$ points of $\triangle EAB$, $\triangle EBC$, 
$\triangle ECD$, and $\triangle EDA$, respectively (Figure~\ref{fig:dpRectangleX395}).
Then
$$[ABCD]=6[FGHI].$$
\end{theorem}

\begin{figure}[h!t]
\centering
\includegraphics[width=0.4\linewidth]{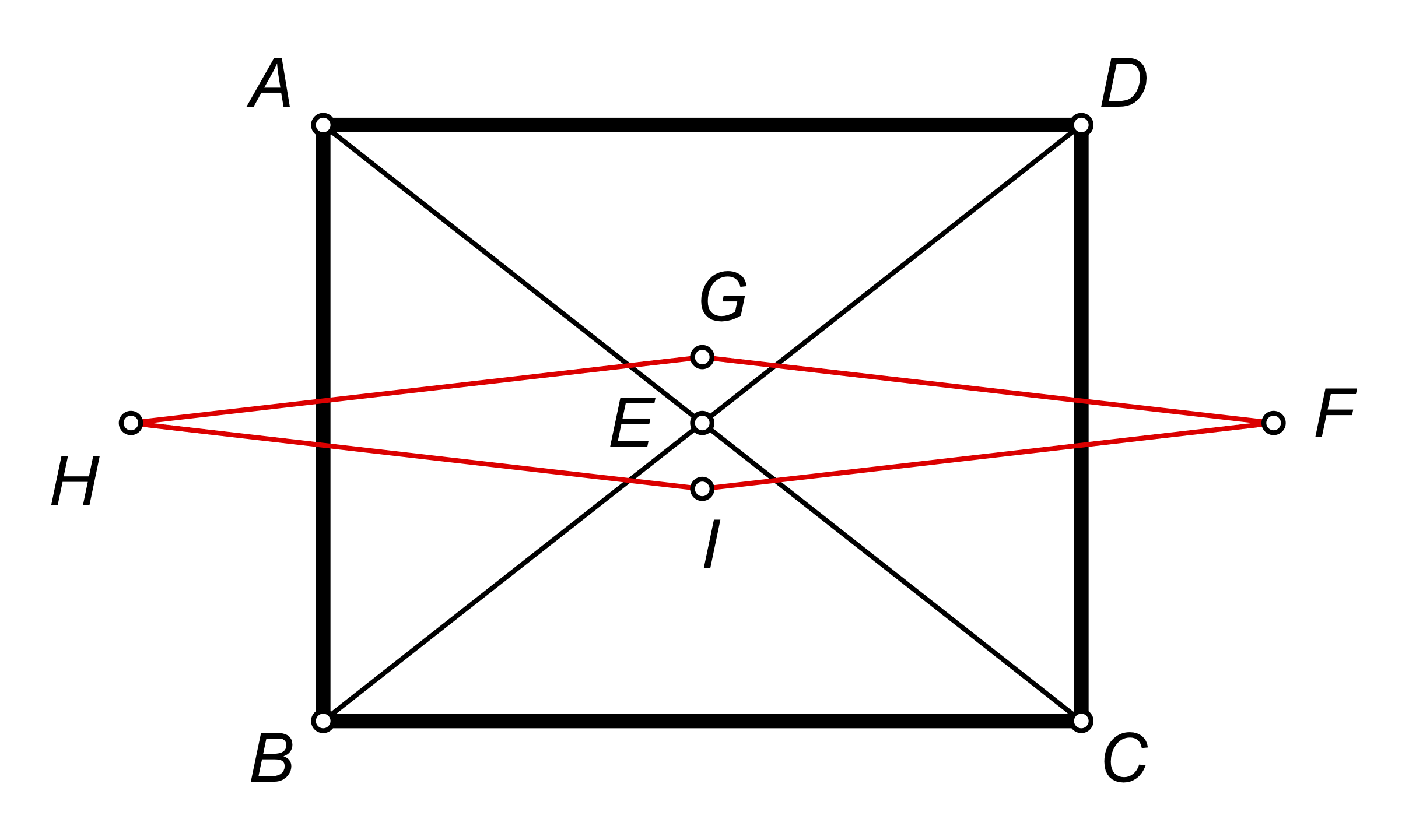}
\caption{rectangle, $X_{395}$ points $\implies$ $[ABCD]=6[FGHI]$}
\label{fig:dpRectangleX395}
\end{figure}

\begin{proof}
From \cite{ETC395}, we find that the barycentric coordinates for the $X_{395}$ point of a triangle are
\begin{equation}
\label{eq:uvw}
{(u:v:w)}=\left\{\sqrt{3} a^2-2 S:\sqrt{3} b^2-2 S:\sqrt{3} c^2-2 S\right\}
\end{equation}
where $a$, $b$, and $c$ are the lengths of the sides of that triangle and $S$ is twice the area of that triangle.

Set up a barycentric coordinate system as shown in Figure~\ref{fig:dpRectangleCoords}.
Let $BC=a$, $AB=c$, and $AC=b=\sqrt{a^2+c^2}$.
Since $ABC$ is a right triangle, $AE=BE=CE=b/2$.
For areas, we have $[ABC]=ac/2$, $[ABE]=ac/4$, and $[BCE]=ac/4$.

\begin{figure}[h!t]
\centering
\includegraphics[width=0.5\linewidth]{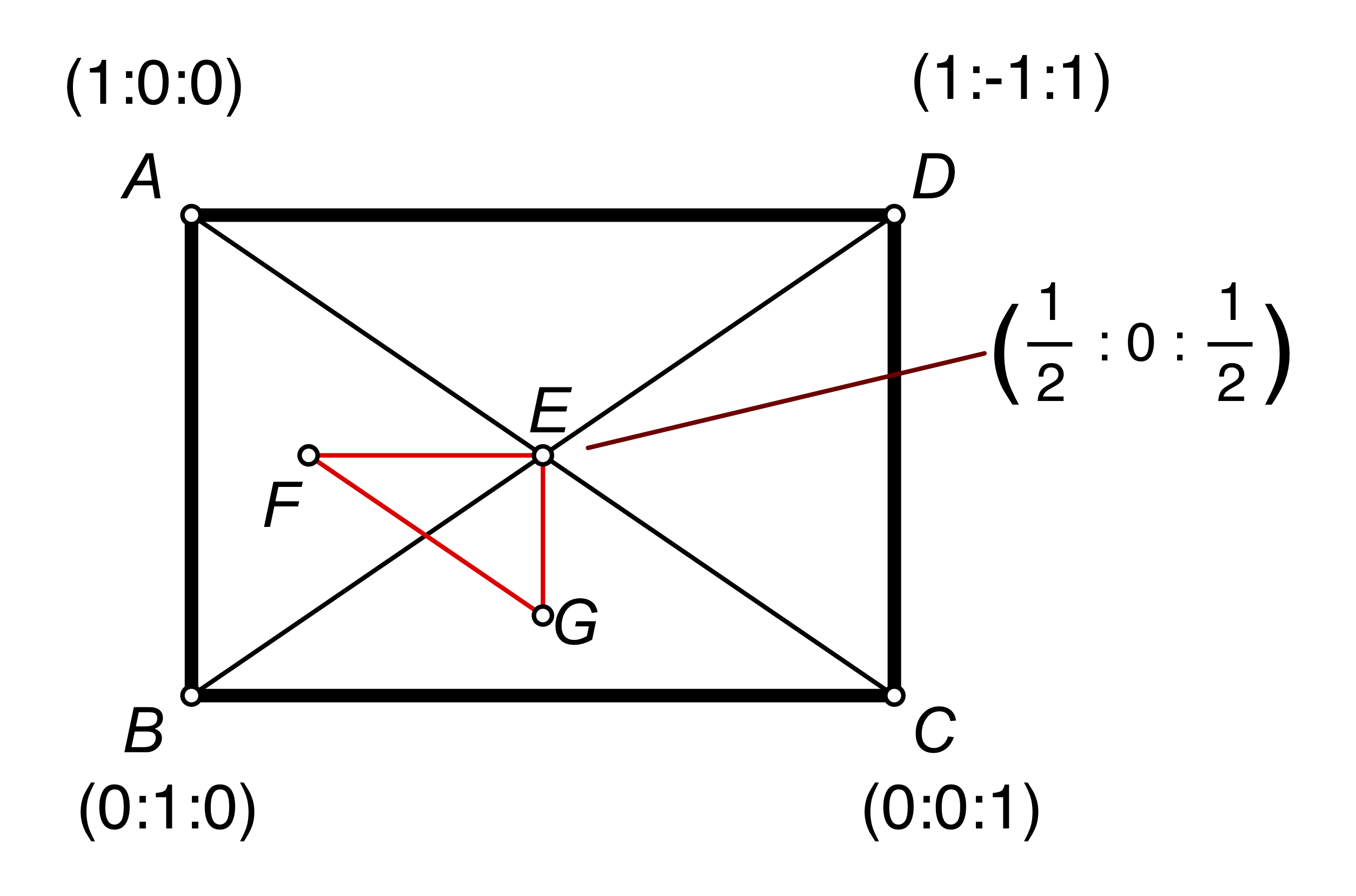}
\caption{barycentric coordinates and rectangle $ABCD$}
\label{fig:dpRectangleCoords}
\end{figure}

Since $F$ is the $X_{395}$ point of $\triangle ABE$,
$$F=uA+vB+wE$$
where $u$, $v$, $w$, and $S$ are given by Equation~(\ref{eq:uvw}),
except that the values of $a$, $b$, $c$, and $S$ are the sides and twice the area of $\triangle ABE$.
In other words, $a\to BE=b/2$, $b\to AE=b/2$, $c\to AB=c$, and $S\to 2[ABE]=ac/2$.
We get
$$F=\left(\frac{1}{4} \left(2 c \left(\sqrt{3} c-3 a\right)+\sqrt{3}
   b^2\right):\frac{\sqrt{3} b^2}{4}-a c:\frac{1}{2} c \left(\sqrt{3} c-a\right)\right).$$
Similarly, we find the coordinates for $G$ by using Equation~(\ref{eq:uvw})
applied to $\triangle BCE$, using the substitutions $a\to CE=b/2$, $b\to BE=b/2$, $c\to BC=a$,
and $S\to 2[BCE]=ac/2$. We get
$$G=\left(\frac{1}{2} a \left(\sqrt{3} a-c\right):\frac{\sqrt{3} b^2}{4}-a c:\frac{1}{4}
   \left(2 \sqrt{3} a^2-6 a c+\sqrt{3} b^2\right)\right).$$
Now we apply the Area Formula (Lemma~\ref{lemma:areaFormula}) to $\triangle EFG$, to get
$$[EFG]=\frac{K}{12}=\frac{[ABC]}{12}$$
after simplifying and using the fact that $b^2=a^2+c^2$ (because $\triangle ABC$ is a right triangle).

By Lemma~\ref{thm:arbCenterRectangle}, $FGHI$ is a rhombus, so $[EFG]=\frac14[FGHI]$.
Since $ABCD$ is a rectangle, $[ABCD]=2[ABC]$. Therefore
$$[ABCD]=2[ABC]=2(12[EFG])=2(12([FGHI]/4))=6[FGHI].$$
\end{proof}

\begin{theorem}
\label{theorem:dpRectangle8}
Let $E$ be the diagonal point of rectangle $ABCD$.
Let $F$, $G$, $H$, and $I$ be the $X_{396}$ points of $\triangle EAB$, $\triangle EBC$, 
$\triangle ECD$, and $\triangle EDA$, respectively.
Then
$$[ABCD]=6[FGHI].$$
\end{theorem}

The proof is the same as the proof for Theorem~\ref{theorem:dpRectangleX395} and is omitted.

\relbox{Relationship $[ABCD]=2[FGHI]$}

\begin{theorem}
\label{theorem:dpRectangleMidpoints}
Let $E$ be the diagonal point of rectangle $ABCD$.
Let $n$ be in the set
$$\{11, 115, 116, 122, 123, 124, 125, 127, 130, 134, 135, 136 $$
$$137, 139,
244, 245, 246,247, 338, 339, 865, 866, 867, 868\}.$$
Let $F$, $G$, $H$, and $I$ be the $X_n$ points of $\triangle EAB$, $\triangle EBC$, 
$\triangle ECD$,  and $\triangle EDA$, respectively (Figure~\ref{fig:dpRectangleMidpoints}).
Then $F$, $G$, $H$, and $I$ are the midpoints of the sides of the rectangle
and
$$[ABCD]=2[FGHI].$$
\end{theorem}

\begin{figure}[h!t]
\centering
\includegraphics[width=0.3\linewidth]{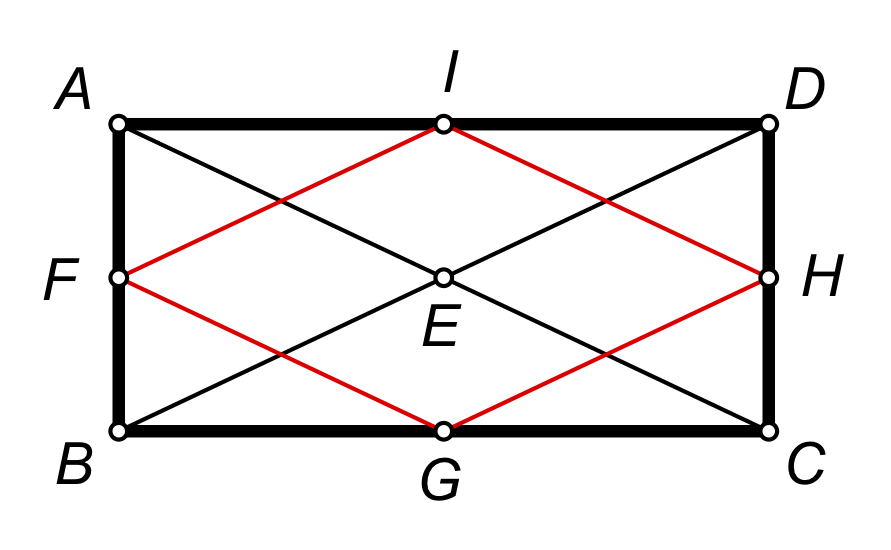}
\caption{rectangle $\implies$ $[ABCD]=2[FGHI]$}
\label{fig:dpRectangleMidpoints}
\end{figure}

\begin{proof}
Since the diagonals of a rectangle are equal and bisect each other,
each of the radial triangles is isosceles with vertex $E$.
Thus, by Lemma \ref{lemma:dpIsoscelesTriangleMidpoint}, 
$F$, $G$, $H$, and $I$ are the midpoints of the sides of the rectangle.
Then, by Lemma \ref{lemma:Varignon}, $[ABCD]=2[FGHI]$.
\end{proof}

This proves all the entries in Table~\ref{table:dp1} for rectangles with the relationship $[ABCD]=2[FGHI]$.

\relbox{Relationship $[ABCD]=\frac32[FGHI]$}

\begin{theorem}
\label{theorem:dpRectangleX616}
Let $E$ be the diagonal point of rectangle $ABCD$.
Let $F$, $G$, $H$, and $I$ be the $X_{616}$ points or the $X_{617}$ points of $\triangle EAB$, $\triangle EBC$, 
$\triangle ECD$, and $\triangle EDA$, respectively (Figure~\ref{fig:dpRectangleX616}).
Then
$$[ABCD]=\frac32[FGHI].$$
\end{theorem}

\begin{figure}[h!t]
\centering
\includegraphics[width=0.3\linewidth]{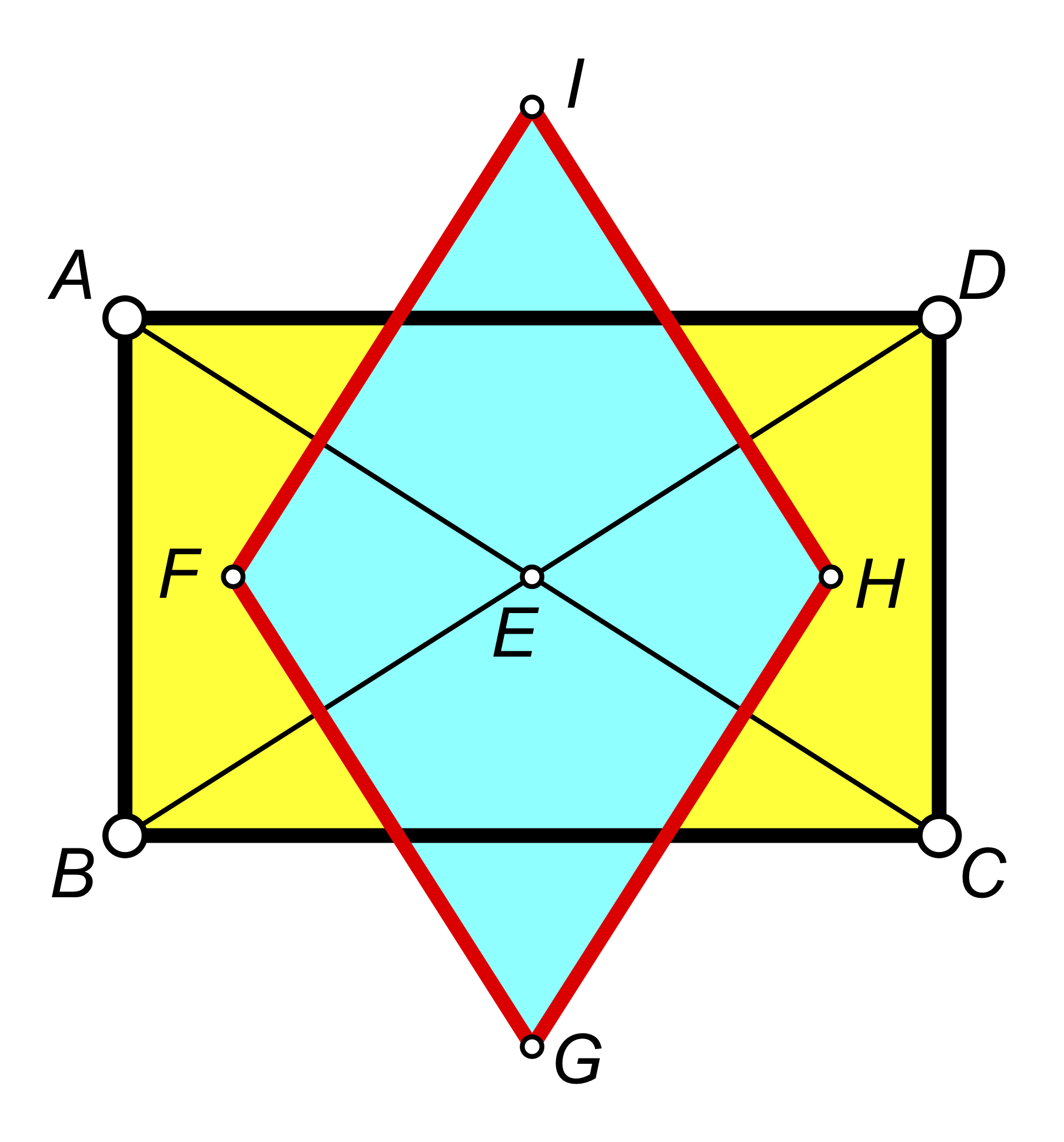}
\caption{rectangle, $X_{616}$ points $\implies$ $[ABCD]=\frac32[FGHI]$}
\label{fig:dpRectangleX616}
\end{figure}

\begin{proof}
The proof is the same as the proof for Theorem~\ref{theorem:dpRectangleX395} and the details are omitted.
\end{proof}

%\newpage

\relbox{Relationship $[ABCD]=\frac98[FGHI]$}

\begin{theorem}
\label{thm:dpRectangleNeg}
Let $E$ be the diagonal point of rectangle $ABCD$.
Let $X_n$ be a triangle center with the property that for all isosceles triangles with vertex $A$
and midpoint of base $M$, $X_nM/AM$ is a fixed negative constant $k$.
Let $F$, $G$, $H$, and $I$ be the $X_n$ points of $\triangle EAB$, $\triangle EBC$, 
$\triangle ECD$, and $\triangle EDA$, respectively.
Then
$$[ABCD]=\frac{2}{(1-k)^2}[FGHI].$$
\end{theorem}

\begin{figure}[h!t]
\centering
\includegraphics[width=0.5\linewidth]{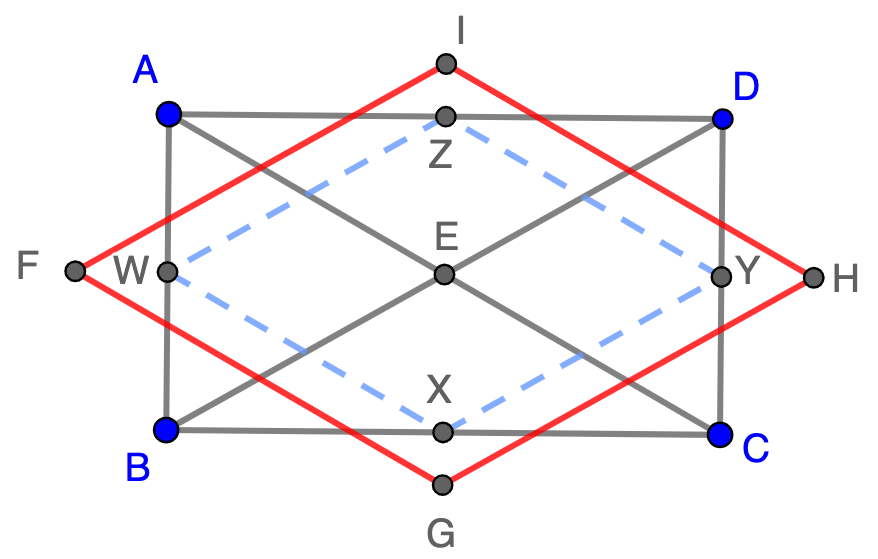}
\caption{}
\label{fig:dpRectangleNeg}
\end{figure}

\begin{proof}
Since the diagonals of a rectangle are equal and bisect each other,
each of the radial triangles is isosceles with vertex $E$.
Let the midpoints of the sides of the rectangle be $W$, $X$, $Y$, and $Z$
as shown in Figure~\ref{fig:dpRectangleNeg}.
Since $F$ is the $X_n$ point of $\triangle EAB$, by hypothesis,
$$\frac{FW}{EW}=k.$$
Since $k<0$
$$\frac{EF}{EW}=\frac{EW+WF}{EW}=\frac{EW-FW}{EW}=1-\frac{FW}{EW}=1-k.$$
Similarly, $EG/EX=1-k$, $EH/EY=1-k$, and $EI/EZ=1-k$.
So quadrilaterals $FGHI$ and $WXYZ$ are homothetic, with $E$ the center of similitude
and ratio of similarity $1-k$.
Thus
$$[FGHI]=(1-k)^2[WXYZ].$$
But $[WXYZ]=\frac12[ABCD]$, so $[FGHI]=\frac{(1-k)^2}{2}[ABCD]$
or, equivalently, $[ABCD]=\frac{2}{(1-k)^2}[FGHI]$.
\end{proof}

When $n=290$, $n=671$, or $n=903$, $k=-1/3$ and $[ABCD]=\frac98[FGHI]$.
These values do not appear in Table~\ref{table:dp1} because the ratios in
Table~\ref{table:dp1} are limited to those with denominators less than 6.

%\newpage

\relbox{Relationship $[ABCD]=\frac12[FGHI]$}

\begin{theorem}
\label{theorem:dpRectangleX148}
Let $E$ be the diagonal point of rectangle $ABCD$.
Let $n$ be 148, 149, or 150.
Let $F$, $G$, $H$, and $I$ be the $X_n$ points of $\triangle EAB$, $\triangle EBC$, 
$\triangle ECD$, and $\triangle EDA$, respectively.
Then
$$[ABCD]=\frac12[FGHI].$$
\end{theorem}

\begin{proof}
From Table~\ref{table:diagonalPointIsoscelesRatios} and Theorem~\ref{thm:dpRectangleNeg}
with $k=-1$, we have $\frac{2}{(1-k)^2}=\frac12$.
\end{proof}

\begin{theorem}
\label{theorem:dpRectangleX146}
Let $E$ be the diagonal point of rectangle $ABCD$.
Let $n$ be 146, 147, 150, 151, 152, or 153.
Let $F$, $G$, $H$, and $I$ be the $X_n$ points of $\triangle EAB$, $\triangle EBC$, 
$\triangle ECD$, and $\triangle EDA$, respectively.
Then
$$[ABCD]=\frac12[FGHI].$$
\end{theorem}

The proofs are the same as the proof for Theorem~\ref{theorem:dpRectangleX395} and are omitted.

\newpage

\subsection{Proofs for Squares}\

\relbox{Relationship $[ABCD]=6[FGHI]$}

\newcommand{\rv}{\rule[-10pt]{0pt}{28pt}}
\newcommand{\rw}{\rule[-14pt]{0pt}{34pt}}

By symmetry (or by Theorem~6.33 of \cite{shapes}), the central quadrilateral is a square when the reference quadrilateral is a square. It is straightforward to calculate $[ABCD]/[FGHI]$ for various centers.
The barycentric coordinates $(u:v:w)$ for the center can be found in \cite{ETC}.
Then the ratio $EF/EM$ can be found by Lemma~\ref{lemma:isoscelesRatioAll}
in terms of $u$, $v$, and $w$ (Figure~\ref{fig:dpSquare}). Since each radial triangle is an isosceles right
triangle, we can replace $a$, $b$, and $c$ by $\sqrt2$, 1, and 1, respectively, in this ratio.
Finally, the ratio of the areas of the squares is the square of the ratio of similitude, so
$$\frac{[ABCD]}{[FGHI]}=\left(\frac{EA}{EF}\right)^2=\left(\frac{EM\sqrt2}{EF}\right)^2=2\left(\frac{EM}{EF}\right)^2
=2\left(\frac{1}{\frac{EF}{EM}}\right)^2.$$
The results are tabulated in the tables on the following pages. We omit the cases where $EF/EM=0$, but otherwise
list all ratios, even when the result is true for quadrilaterals that are ancestors of a square.

\begin{figure}[h!t]
\centering
\includegraphics[width=0.3\linewidth]{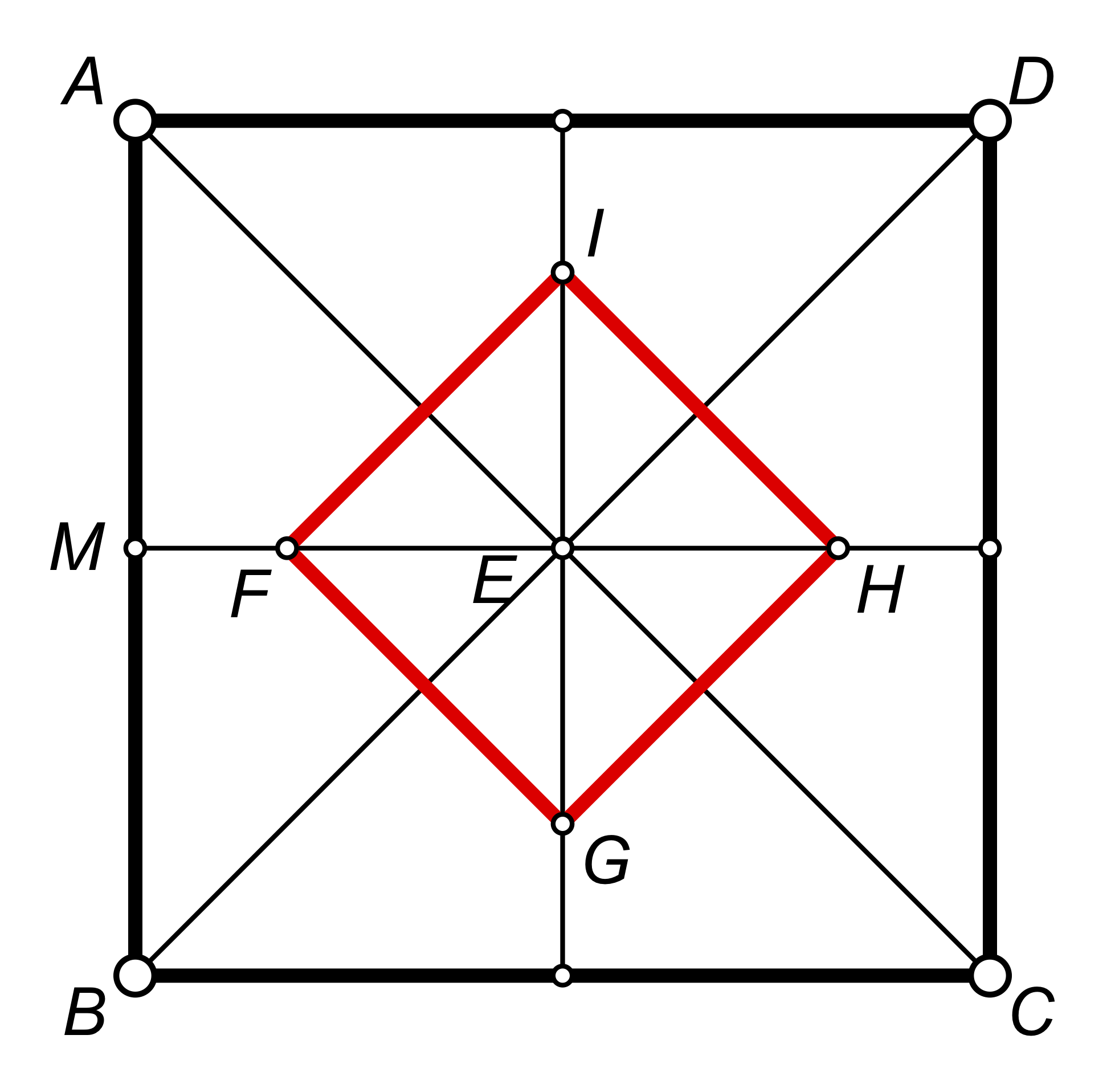}
\caption{A square and its central quadrilateral}
\label{fig:dpSquare}
\end{figure}

We can calculate $[ABCD]/[FGHI]$ even for centers that are not simple algebraic expressions in terms of $a$, $b$, and $c$.
For example, the Morley center of a triangle ($X_{356}$) has barycentric coordinates
$$\left(a\cos\frac{A}{3}+2a\cos\frac{B}{3}\cos\frac{C}{3}:b\cos\frac{B}{3}+2b\cos\frac{C}{3}\cos\frac{A}{3}:c\cos\frac{C}{3}+2c\cos\frac{A}{3}\cos\frac{B}{3}\right).$$
The quantities $A$, $B$, and $C$ normally are inverse trigonometric functions of $a$, $b$, and $c$.
However, when the triangle is an isosceles right triangle, with right angle at $A$, we have
$A=\pi/2$, $B=\pi/4$, and $C=\pi/4$.
Thus, the barycentric coordinates for the Morley center of an isosceles right triangle are
$$\left(\sqrt2\cos\frac{\pi}{6}+2\sqrt2\cos\frac{\pi}{12}\cos\frac{\pi}{12}:\cos\frac{\pi}{12}+2\cos\frac{\pi}{12}\cos\frac{\pi}{6}:\cos\frac{\pi}{12}+2\cos\frac{\pi}{6}\cos\frac{\pi}{12}\right)$$
$$=\left(\sqrt2+\sqrt6:\frac{\sqrt6+2\sqrt2}{2}:\frac{\sqrt6+2\sqrt2}{2}\right).$$
We therefore have the following theorem.

\begin{theorem}
\label{theorem:dpSquareX356}
Let $E$ be the diagonal point of square $ABCD$.
Let $F$, $G$, $H$, and $I$ be the $X_{356}$ points of $\triangle EAB$, $\triangle EBC$, 
$\triangle ECD$, and $\triangle EDA$, respectively.
Then
$$[ABCD]=6[FGHI].$$
\end{theorem}

\begin{proof}
By Lemma~\ref{lemma:dpIsoscelesTriangle},
$$\frac{FM}{EM}=\frac{u}{u+2v}$$
where $u=\sqrt2+\sqrt6$ and $v=\sqrt2+\sqrt6/2$.
Thus, $\frac{FM}{EM}=\frac{\sqrt2+\sqrt6}{3\sqrt2+2\sqrt6}=1-\frac{1}{\sqrt3}$.
From Lemma~\ref{lemma:isoscelesRatioAll}, we have
$$\frac{EF}{EM}=1-\left(1-\frac{1}{\sqrt3}\right)=\frac{1}{\sqrt3}.$$
But $EA=EM\sqrt2$, so
$$\frac{EA}{EF}=\frac{EM\sqrt2}{EF}=\frac{\sqrt2}{EF/EM}=\frac{\sqrt2}{1/\sqrt3}=\sqrt6.$$
Since $\frac{[ABCD]}{[FGHI]}=\left(\frac{EA}{EF}\right)^2$, we therefore have
$$\frac{[ABCD]}{[FGHI]}=\left(\sqrt6\right)^2=6.$$
\end{proof}

\relbox{Relationship $[ABCD]=k[FGHI]$ when $X_n$ is not algebraic}

By the same procedure, we get the following result for other values of $n$
for which the barycentric coordinates of $X_n$ are not algebraic.

\begin{theorem}
\label{theorem:dpSquareX357}
Let $E$ be the diagonal point of square $ABCD$.
Let $F$, $G$, $H$, and $I$ be the $X_{n}$ points of $\triangle EAB$, $\triangle EBC$, 
$\triangle ECD$, and $\triangle EDA$, respectively.\\[1pt]
If $n=357$, then $\displaystyle [ABCD]=\frac13\left(14+3\sqrt3\right)[FGHI]$.\\[3pt]
If $n=358$, then $\displaystyle [ABCD]=\left(14-5\sqrt3\right)[FGHI]$.\\[2pt]
If $n=359$, then $\displaystyle [ABCD]=\frac92[FGHI]$.\\[3pt]
If $n=360$, then $\displaystyle [ABCD]=8[FGHI]$.\\[2pt]
If $n=369$, then $\displaystyle [ABCD]=\frac14\left(27-10\sqrt2\right)[FGHI]$.
\end{theorem}

\newpage

\scalebox{0.45}
{
\begin{tabular}{|l|p{2.5in}|}
\hline
\multicolumn{2}{|c|}{\textbf{\color{blue}\large \strut Central Quadrilaterals formed by the Diagonal Point}}\\
\multicolumn{2}{|c|}{\textbf{\color{blue}\large \strut of a Square (part 1)}}\\ \hline
$[ABCD]/[FGHI]$&$n$\\ \hline
\ru $\frac{1}{2} \left(7-4 \sqrt{3}\right)$&622\\ \hline
\ru $\frac{1}{4} \left(3-2 \sqrt{2}\right)$&285,  309,  320, 350, 448\\ \hline
\ru $-2 \left(12 \sqrt{2}-17\right)$&205, 449, 853\\ \hline
\ru $\frac{1}{2} \left(3-2 \sqrt{2}\right)$&189, 239\\ \hline
\ru $-\frac{3}{2} \left(4 \sqrt{3}-7\right)$&383\\ \hline
\ru $\frac{1}{8}$&316, 352\\ \hline
\ru $3-2 \sqrt{2}$&223,  282,  484,  610,  857,  908,  909,  911,  923\\ \hline
\ru $-\frac{3}{4} \left(\sqrt{3}-2\right)$&299\\ \hline
\ru $\frac{2}{9}$&446\\ \hline
\ru $2-\sqrt{3}$&16\\ \hline
\ru $-2 \left(2 \sqrt{2}-3\right)$&44, 197, 227, 478, 672, 851, 861, 896, 899, 910\\ \hline
\ru $-\frac{4}{3} \left(\sqrt{3}-2\right)$&624\\ \hline
\ru $-\frac{8}{49} \left(6 \sqrt{2}-11\right)$&970\\ \hline
\ru $\frac{1}{2}$ &
   20, 22, 23, 70, 74, 94, 98, 102, 103, 104, 105, 106, 111, 146, 147, 148, 149, 150, 151, 152, 153, 160, 17
   5, 253, 280, 323, 325, 347, 385, 401, 477, 586, 638, 675, 697, 699, 701, 703, 705, 707, 709, 711, 713, 715
   , 717, 719, 721, 723, 725, 727, 729, 731, 733, 735, 737, 739, 741, 743, 745, 747, 753, 755, 759, 761, 767, 
   840, 841, 842, 843, 858, 953, 972\\ \hline
\ru $-4 \left(2 \sqrt{2}-3\right)$&198, 292\\ \hline
\rv $23-16 \sqrt{2}+2 \sqrt{2 \left(58-41 \sqrt{2}\right)}$&845\\ \hline
\ru $-3 \left(\sqrt{3}-2\right)$&14\\ \hline
\ru $179-126 \sqrt{2}$&170\\ \hline
\ru $\frac{1}{4} \left(9-4 \sqrt{2}\right)$&88, 411, 416, 673, 897\\ \hline
\rv $931-658 \sqrt{2}+4 \sqrt{2 \left(54146-38287 \sqrt{2}\right)}$&167\\ \hline
\ru $\frac{8}{9}$&67, 550, 625, 694\\ \hline
\ru $-\frac{18}{49} \left(6 \sqrt{2}-11\right)$&992\\ \hline
\ru $1$&36, 40, 80, 84, 238, 291, 859\\ \hline
\ru $\frac{1}{6} \left(13-4 \sqrt{3}\right)$&634\\ \hline
\rv $43-30 \sqrt{2}+4 \sqrt{58-41 \sqrt{2}}$&844\\ \hline
\ru $\frac{9}{8}$&290, 315, 376, 490, 671, 903\\ \hline
\rv $295-208 \sqrt{2}+2 \sqrt{2 \left(20714-14647 \sqrt{2}\right)}$&168\\ \hline
\ru $\frac{4}{49} \left(9+4 \sqrt{2}\right)$&963\\ \hline
\ru $\frac{1}{2} \left(11-6 \sqrt{2}\right)$&294, 329, 335, 573\\ \hline
\ru $\frac{32}{25}$&486, 548\\ \hline
\ru $\frac{1}{4} \left(14-5 \sqrt{3}\right)$&301\\ \hline
\ru $8 \left(3-2 \sqrt{2}\right)$&199\\ \hline
\rv $5-3 \sqrt{2}+2 \sqrt{10-7 \sqrt{2}}$&504\\ \hline
\ru $\frac{1}{4} \left(3+2 \sqrt{2}\right)$&322, 334, 944\\ \hline
\ru $\frac{3}{2}$&616, 617\\ \hline
\ru $9 \left(3-2 \sqrt{2}\right)$&165\\ \hline
\ru $\frac{1}{2} \left(9-4 \sqrt{2}\right)$&144\\ \hline
\ru $2$ &
   3, 11, 48, 49, 63, 69, 71, 72, 73, 77, 78, 97, 115, 116, 122, 123, 124, 125, 127, 130, 137, 184, 185, 187, 
   201, 212, 216, 217, 219, 222, 228, 237, 244, 245, 246, 248, 255, 265, 268, 271, 283, 287, 293, 295, 296, 3
   04, 305, 306, 307, 326, 328, 332, 336, 337, 338, 339, 343, 345, 348, 382, 394, 399, 408, 417, 418, 426, 44
   0, 441, 464, 465, 466, 488, 499, 577, 591, 603, 606, 615, 640, 682, 748, 820, 828, 836, 852, 856, 865, 866
   , 867, 868, 895, 902, 974\\ \hline
\rw $\frac{1}{2} \left(995+702 \sqrt{2}-4 \sqrt{2 \left(61466+43463 \sqrt{2}\right)}\right)$&400\\ \hline
\ru $121-84 \sqrt{2}$&728\\ \hline
\rv $7-4 \sqrt{2}+2 \sqrt{2 \left(10-7 \sqrt{2}\right)}$&164\\ \hline
\ru $\frac{9}{4}$&341\\ \hline
\ru $4 \left(43-30 \sqrt{2}\right)$&480\\ \hline
\ru $\frac{1}{2} \left(33-20 \sqrt{2}\right)$&346\\ \hline
\ru $\frac{1}{36} \left(73+12 \sqrt{2}\right)$&413\\ \hline
\ru $11-6 \sqrt{2}$&191, 200, 979\\ \hline
\rv $3-\sqrt{2}+2 \sqrt{3 \sqrt{2}-4}$&510\\ \hline
\ru $\frac{1}{2} \left(19-8 \sqrt{3}\right)$&627\\ \hline
\ru $\frac{4}{9} \left(3+2 \sqrt{2}\right)$&956\\ \hline
\ru $\frac{1}{4} \left(19-6 \sqrt{2}\right)$&312, 319, 572\\ \hline
\ru $\frac{49}{18}$&183, 252\\ \hline
\end{tabular}
\quad
\begin{tabular}{|l|p{2.5in}|}
\hline
\multicolumn{2}{|c|}{\textbf{\color{blue}\large \strut Central Quadrilaterals formed by the Diagonal Point}}\\
\multicolumn{2}{|c|}{\textbf{\color{blue}\large \strut of a Square (part 2)}}\\ \hline
$[ABCD]/[FGHI]$&$n$\\ \hline
\ru $16 \left(3-2 \sqrt{2}\right)$&220\\ \hline
\ru $\frac{4}{169} \left(146-17 \sqrt{3}\right)$&636\\ \hline
\ru $\frac{1}{16} \left(33+8 \sqrt{2}\right)$&314, 561, 679\\ \hline
\ru $\frac{3}{4} \left(2+\sqrt{3}\right)$&298\\ \hline
\ru $\frac{72}{25}$&549, 599, 626\\ \hline
\ru $\frac{1}{2} \left(3+2 \sqrt{2}\right)$&8, 313, 501, 947, 962\\ \hline
\ru $\frac{1}{3} \left(14-3 \sqrt{3}\right)$&18\\ \hline
\rv $19-10 \sqrt{2}-4 \sqrt{2 \left(10-7 \sqrt{2}\right)}$&166\\ \hline
\ru $\frac{1}{98} \left(561-184 \sqrt{2}\right)$&391\\ \hline
\ru $18 \left(3-2 \sqrt{2}\right)$&210\\ \hline
\ru $\frac{25}{8}$&76, 95, 277, 333, 353, 492, 631, 801\\ \hline
\ru $\frac{1}{144} \left(347+78 \sqrt{2}\right)$&871\\ \hline
\ru $\frac{8}{49} \left(11+6 \sqrt{2}\right)$&596, 960, 993\\ \hline
\ru $\frac{1}{4}\left(27-10\sqrt{2}\right)$&369\\ \hline
\ru $12 \left(2-\sqrt{3}\right)$&618\\ \hline
\rv $\frac{1}{4} \left(10-\sqrt{2}+4 \sqrt{2 \left(2-\sqrt{2}\right)}\right)$&556\\ \hline
\ru $\frac{1}{98} \left(233+60 \sqrt{2}\right)$&330\\ \hline
\ru $\frac{1}{6} \left(13+4 \sqrt{3}\right)$&633\\ \hline
\ru $9-4 \sqrt{2}$&9, 35, 321\\ \hline
\ru $\frac{1}{32} \left(83+18 \sqrt{2}\right)$&261, 310\\ \hline
\ru $\frac{1}{4} \left(26-7 \sqrt{3}\right)$&302\\ \hline
\ru $\frac{2}{49} \left(43+30 \sqrt{2}\right)$&958\\ \hline
\ru $\frac{1}{2} \left(41-24 \sqrt{2}\right)$&452\\ \hline
\ru $\frac{32}{9}$&140, 141, 182, 566, 641\\ \hline
\ru $2 \left(81-56 \sqrt{2}\right)$&594\\ \hline
\ru $\frac{1}{4} \left(9+4 \sqrt{2}\right)$&21, 75, 371, 497, 775, 991, 997\\ \hline
\ru $2+\sqrt{3}$&15\\ \hline
\ru $\frac{2}{49} \left(123-22 \sqrt{2}\right)$&949, 965\\ \hline
\ru $\frac{121}{32}$&308\\ \hline
\rv $11-6 \sqrt{2}+4 \sqrt{10-7 \sqrt{2}}$&503\\ \hline
\ru $\frac{18}{289} \left(33+20 \sqrt{2}\right)$&392\\ \hline
\rv $\frac{1}{2} \left(6-\sqrt{2}+4 \sqrt{2-\sqrt{2}}\right)$&188\\ \hline
\ru $\frac{1}{49} \left(163+18 \sqrt{2}\right)$&87, 936\\ \hline
\ru $\frac{4}{121} \left(146-17 \sqrt{3}\right)$&629\\ \hline
\ru $\frac{1}{98} \left(193+132 \sqrt{2}\right)$&390\\ \hline
\ru $\frac{1}{8} \left(51-14 \sqrt{2}\right)$&966\\ \hline
\ru $\frac{98}{25}$&574\\ \hline
\ru $2 \left(3249-2296 \sqrt{2}\right)$&762\\ \hline
\ru $\frac{1}{18} \left(51+14 \sqrt{2}\right)$&257\\ \hline
\ru $4$&10, 55, 241, 405, 500, 582, 950\\ \hline
\ru $\frac{2}{3} \left(13-4 \sqrt{3}\right)$&398\\ \hline
\ru $\frac{1}{16} \left(51+10 \sqrt{2}\right)$&274\\ \hline
\ru $\frac{200}{49}$&632\\ \hline
\rv $31-20 \sqrt{2}+2 \sqrt{2 \left(194-137 \sqrt{2}\right)}$&258\\ \hline
\ru $\frac{4}{169} \left(146+17 \sqrt{3}\right)$&635\\ \hline
\ru $\frac{1}{18} \left(41+24 \sqrt{2}\right)$&585\\ \hline
\rv $\frac{1}{2} \left(154+105 \sqrt{2}-8 \sqrt{2 \left(338+239 \sqrt{2}\right)}\right)$&557\\ \hline
\ru $\frac{1}{49} \left(137+48 \sqrt{2}\right)$&982\\ \hline
\ru $\frac{1}{8} \left(17+12 \sqrt{2}\right)$&1000\\ \hline
\ru $25 \left(3-2 \sqrt{2}\right)$&380, 846\\ \hline
\ru $\frac{2}{9} \left(11+6 \sqrt{2}\right)$&38\\ \hline
\ru $\frac{64}{49} \left(9-4 \sqrt{2}\right)$&45\\ \hline
\rv $\frac{1}{4} \left(8+3 \sqrt{2}+2 \sqrt{2 \left(2+\sqrt{2}\right)}\right)$&260\\ \hline
\ru $\frac{1}{49} \left(113+72 \sqrt{2}\right)$&984\\ \hline
\rv $3+2 \sqrt{2 \left(3 \sqrt{2}-4\right)}$&364\\ \hline
\ru $2 \left(121-84 \sqrt{2}\right)$&756\\ \hline
\ru $\frac{9}{2}$&2, 54, 284, 311, 349, 359, 569, 570, 581, 637, 639, 943\\ \hline
\ru $\frac{1}{64} \left(209+60 \sqrt{2}\right)$&873\\ \hline
\ru $\frac{1}{36} \left(107+42 \sqrt{2}\right)$&870\\ \hline
\ru $\frac{2}{49} \left(57+40 \sqrt{2}\right)$&996\\ \hline
\rv $23-8 \sqrt{2}-2 \sqrt{2 \left(50-31 \sqrt{2}\right)}$&363\\ \hline
\ru $33-20 \sqrt{2}$&964, 986\\ \hline
\rv $4-\sqrt{2}+2 \sqrt{2 \left(2-\sqrt{2}\right)}$&259\\ \hline
\ru $\frac{16}{49} \left(9+4 \sqrt{2}\right)$&496, 613, 988\\ \hline
\ru $\frac{1}{4} \left(11+6 \sqrt{2}\right)$&968\\ \hline
\end{tabular}
}

\newpage

\scalebox{0.45}
{
\begin{tabular}{|l|p{2.5in}|}
\hline
\multicolumn{2}{|c|}{\textbf{\color{blue}\large \strut Central Quadrilaterals formed by the Diagonal Point}}\\
\multicolumn{2}{|c|}{\textbf{\color{blue}\large \strut of a Square (part 3)}}\\ \hline
$[ABCD]/[FGHI]$&$n$\\ \hline
\rv $2 \left(8+5 \sqrt{2}-2 \sqrt{2 \left(10+7 \sqrt{2}\right)}\right)$&178\\ \hline
\ru $\frac{1}{9} \left(33+8 \sqrt{2}\right)$&256\\ \hline
\ru $\frac{4}{3} \left(2+\sqrt{3}\right)$&623\\ \hline
\ru $2 \left(11-6 \sqrt{2}\right)$&37, 41, 442, 498, 584, 601, 976\\ \hline
\ru $\frac{1}{36} \left(97+60 \sqrt{2}\right)$&409\\ \hline
\ru $\frac{1}{2} \left(4+4 \sqrt[4]{2}+\sqrt{2}\right)$&366\\ \hline
\ru $\frac{36}{529} \left(41+24 \sqrt{2}\right)$&551\\ \hline
\rv $5-\sqrt{2}+2 \sqrt{2-\sqrt{2}}$&362\\ \hline
\ru $\frac{128}{25}$&575, 695\\ \hline
\ru $\frac{1}{8} \left(27+10 \sqrt{2}\right)$&60, 86, 443\\ \hline
\ru $\frac{2}{961} \left(2217+188 \sqrt{2}\right)$&980\\ \hline
\ru $\frac{841}{162}$&592\\ \hline
\rv $\frac{1}{2} \left(18-\sqrt{2}-8 \sqrt{2-\sqrt{2}}\right)$&236\\ \hline
\ru $\frac{1}{2} \left(19-6 \sqrt{2}\right)$&612, 894\\ \hline
\ru $\frac{4}{49} \left(51+10 \sqrt{2}\right)$&142, 474, 939, 954, 975\\ \hline
\ru $14-5 \sqrt{3}$&61, 358\\ \hline
\rv $2 \left(11+8 \sqrt{2}-2 \sqrt{2 \left(24+17 \sqrt{2}\right)}\right)$&367\\ \hline
\ru $\frac{1}{4} \left(33-8 \sqrt{2}\right)$&941\\ \hline
\ru $\frac{50}{9}$&39, 233, 493, 567, 590\\ \hline
\ru $\frac{1}{16} \left(73+12 \sqrt{2}\right)$&751\\ \hline
\ru $\frac{1}{4} \left(14+5 \sqrt{3}\right)$&300\\ \hline
\ru $\frac{18}{289} \left(123-22 \sqrt{2}\right)$&374\\ \hline
\ru $\frac{4}{121} \left(146+17 \sqrt{3}\right)$&630\\ \hline
\ru $3+2 \sqrt{2}$&1, 377\\ \hline
\ru $\frac{288}{49}$&547, 597\\ \hline
\ru $6$&395, 396\\ \hline
\ru $\frac{2}{49} \left(233-60 \sqrt{2}\right)$&750\\ \hline
\ru $\frac{49}{8}$&83, 288, 327, 588\\ \hline
\ru $36 \left(3-2 \sqrt{2}\right)$&375\\ \hline
\ru $\frac{2}{49} \left(123+22 \sqrt{2}\right)$&940\\ \hline
\rv $\frac{1}{2} \left(5+\sqrt{2}+4 \sqrt{1+\sqrt{2}}\right)$&508\\ \hline
\ru $\frac{16}{49} \left(11+6 \sqrt{2}\right)$&495, 611\\ \hline
\ru $\frac{1}{3} \left(14+3 \sqrt{3}\right)$&17, 357\\ \hline
\ru $\frac{1}{2} \left(27-10 \sqrt{2}\right)$&938, 955\\ \hline
\rv $\frac{1}{2} \left(10-\sqrt{2}+4 \sqrt{2 \left(2-\sqrt{2}\right)}\right)$&483\\ \hline
\ru $\frac{162}{25}$&373\\ \hline
\ru $\frac{4}{9} \left(9+4 \sqrt{2}\right)$&893\\ \hline
\rv $\frac{\left(-2+2 \sqrt{2}+\left(2-\sqrt{2}\right)^{3/4}\right)^2}{2\left(\sqrt{2}-1\right)^2}$&507\\ \hline
\rv $32-19 \sqrt{2}+2 \sqrt{2 \left(194-137 \sqrt{2}\right)}$&289\\ \hline
\ru $\frac{1}{16} \left(129-16 \sqrt{2}\right)$&981\\ \hline
\ru $2 \left(9-4 \sqrt{2}\right)$&12, 42, 863, 922\\ \hline
\ru $\frac{121}{18}$&384\\ \hline
\ru $2+\sqrt{2}+2\ 2^{3/4}$&365\\ \hline
\rv $\frac{\left(1+2^{2/3} \sqrt[3]{\sqrt{2}-1}\right)^2}{\sqrt[3]{2}\left(\sqrt{2}-1\right)^{2/3}}$&506\\ \hline
\ru $\frac{1}{4} \left(19+6 \sqrt{2}\right)$&81, 85, 386, 404, 977\\ \hline
\ru $\frac{1}{2} \left(7+4 \sqrt{3}\right)$&621\\ \hline
\ru $41-24 \sqrt{2}$&171\\ \hline
\ru $\frac{18}{49} \left(11+6 \sqrt{2}\right)$&354, 969\\ \hline
\rv $\frac{1}{2} \left(227-154 \sqrt{2}+28 \sqrt{2 \left(58-41 \sqrt{2}\right)}\right)$&179\\ \hline
\rv $\frac{1}{2} \left(21+9 \sqrt{3}-2 \sqrt{62+35 \sqrt{3}}\right)$&559\\ \hline
\ru $\frac{1}{2} \left(9+4 \sqrt{2}\right)$&176, 192, 387, 989\\ \hline
\rv $\frac{1}{2} \left(6+\sqrt{2}+4 \sqrt{2+\sqrt{2}}\right)$&174\\ \hline
\ru $\frac{25}{49} \left(9+4 \sqrt{2}\right)$&987\\ \hline
\rv $\frac{1}{49} \left(1169-672 \sqrt{2}+8 \sqrt{33178-23201 \sqrt{2}}\right)$&180\\ \hline
\rv $\frac{3}{2} \left(5+3 \sqrt{2}-\sqrt{3 \left(3+2 \sqrt{2}\right)}\right)$&554\\ \hline
\ru $\frac{1}{8} \left(33+20 \sqrt{2}\right)$&995\\ \hline
\end{tabular}
\quad
\begin{tabular}{|l|p{2.5in}|}
\hline
\multicolumn{2}{|c|}{\textbf{\color{blue}\large \strut Central Quadrilaterals formed by the Diagonal Point}}\\
\multicolumn{2}{|c|}{\textbf{\color{blue}\large \strut of a Square (part 4)}}\\ \hline
$[ABCD]/[FGHI]$&$n$\\ \hline
\ru $\frac{2}{49} \left(163+18 \sqrt{2}\right)$&869\\ \hline
\ru $18 \left(57-40 \sqrt{2}\right)$&872\\ \hline
\ru $8$&5, 6, 169, 226, 344, 360, 389, 402, 482, 485, 494, 578, 620, 642, 942\\ \hline
\ru $\frac{9}{16} \left(9+4 \sqrt{2}\right)$&757\\ \hline
\ru $\frac{1}{9} \left(41+24 \sqrt{2}\right)$&985\\ \hline
\rv $2 \left(13+9 \sqrt{2}-2 \sqrt{58+41 \sqrt{2}}\right)$&177\\ \hline
\ru $\frac{1}{98} \left(561+184 \sqrt{2}\right)$&89\\ \hline
\rv $10+6 \sqrt{2}-\sqrt{3 \left(17+12 \sqrt{2}\right)}$&202\\ \hline
\ru $\frac{1}{4} \left(17+12 \sqrt{2}\right)$&82, 388\\ \hline
\ru $\frac{2}{49} \left(113+72 \sqrt{2}\right)$&999\\ \hline
\rv $3+\sqrt{2}+2 \sqrt{2 \left(1+\sqrt{2}\right)}$&509\\ \hline
\ru $9$&983\\ \hline
\ru $\frac{2}{9} \left(27+10 \sqrt{2}\right)$&904\\ \hline
\rv $7+2 \sqrt{2 \left(2-\sqrt{2}\right)}$&173\\ \hline
\ru $2 \left(33-20 \sqrt{2}\right)$&213, 605\\ \hline
\ru $\frac{1}{4} \left(26+7 \sqrt{3}\right)$&303\\ \hline
\ru $\frac{1}{2} \left(11+6 \sqrt{2}\right)$&7, 58, 614, 951\\ \hline
\ru $4 \left(11-6 \sqrt{2}\right)$&172\\ \hline
\ru $\frac{1}{64} \left(387+182 \sqrt{2}\right)$&763\\ \hline
\ru $\frac{81}{8}$&598\\ \hline
\ru $\frac{2}{49} \left(177+52 \sqrt{2}\right)$&967\\ \hline
\ru $19-6 \sqrt{2}$&43, 937\\ \hline
\rv $4+\sqrt{2}+2 \sqrt{2 \left(2+\sqrt{2}\right)}$&266\\ \hline
\ru $\frac{36}{49} \left(9+4 \sqrt{2}\right)$&553\\ \hline
\ru $\frac{1}{4} \left(33+8 \sqrt{2}\right)$&948\\ \hline
\ru $3 \left(2+\sqrt{3}\right)$&13\\ \hline
\rv $-2 \left(-34-23 \sqrt{2}+4 \sqrt{2 \left(58+41 \sqrt{2}\right)}\right)$&234\\ \hline
\ru $2 \left(3+2 \sqrt{2}\right)$&31, 65, 355, 602, 774, 855, 945, 998\\ \hline
\ru $\frac{1}{8} \left(57+28 \sqrt{2}\right)$&593\\ \hline
\ru $97-60 \sqrt{2}$&609\\ \hline
\ru $\frac{25}{2}$&194, 251, 262\\ \hline
\ru $\frac{9}{4} \left(3+2 \sqrt{2}\right)$&229, 959\\ \hline
\rv $\frac{1}{4} \left(18+7 \sqrt{2}+4 \sqrt{2 \left(10+7 \sqrt{2}\right)}\right)$&555\\ \hline
\ru $\frac{2}{3} \left(13+4 \sqrt{3}\right)$&397\\ \hline
\ru $2 \left(41-24 \sqrt{2}\right)$&181\\ \hline
\ru $9+4 \sqrt{2}$&57, 79, 379\\ \hline
\ru $\frac{25}{32} \left(11+6 \sqrt{2}\right)$&552\\ \hline
\ru $\frac{1}{2} \left(19+8 \sqrt{3}\right)$&628\\ \hline
\rv $11+2 \sqrt{2}+4 \sqrt{2-\sqrt{2}}$&361\\ \hline
\ru $\frac{1}{2} \left(17+12 \sqrt{2}\right)$&145, 595, 961, 994\\ \hline
\rv $\frac{1}{2} \left(18+\sqrt{2}+8 \sqrt{2+\sqrt{2}}\right)$&558\\ \hline
\ru $18$&32, 51, 195, 218, 381, 568, 800, 864, 973\\ \hline
\ru $\frac{1}{4} \left(41+24 \sqrt{2}\right)$&849\\ \hline
\ru $14+5 \sqrt{3}$&62\\ \hline
\ru $4 \left(3+2 \sqrt{2}\right)$&56, 214, 215, 481, 502, 946, 990\\ \hline
\rv $10+6 \sqrt{2}+\sqrt{3 \left(17+12 \sqrt{2}\right)}$&203\\ \hline
\ru $2 \left(9+4 \sqrt{2}\right)$&560, 678, 854\\ \hline
\ru $\frac{1}{2} \left(33+20 \sqrt{2}\right)$&279\\ \hline
\ru $32$&143, 263, 546, 576\\ \hline
\ru $2 \left(11+6 \sqrt{2}\right)$&583, 604\\ \hline
\ru $\frac{1}{2} \left(43+30 \sqrt{2}\right)$&957\\ \hline
\ru $12 \left(2+\sqrt{3}\right)$&619\\ \hline
\ru $9 \left(3+2 \sqrt{2}\right)$&267, 269, 978\\ \hline
\ru $2 \left(139-66 \sqrt{2}\right)$&600\\ \hline
\ru $\frac{1}{2} \left(131+90 \sqrt{2}\right)$&479\\ \hline
\rv $47+32 \sqrt{2}+6 \sqrt{2 \left(58+41 \sqrt{2}\right)}$&505\\ \hline
\ru $121+84 \sqrt{2}$&738\\ \hline
\end{tabular}
}

\newpage

%**************************************
%    Poncelet point
%**************************************

\section{Poncelet point}

In this section, we examine central quadrilaterals formed from the Poncelet
point of the reference quadrilateral.

The \emph{Poncelet point} (sometimes called the Euler-Poncelet point) of a quadrilateral
is the common point of the nine-point circles of the component triangles (half-triangles)
of the quadrilateral. A triangle formed from three vertices of a quadrilateral is called
a \emph{component triangle} of that quadrilateral.
The \emph{nine-point circle} of a triangle is the circle through the midpoints
of the sides of that triangle.

Figure \ref{fig:ppPonceletPoint} shows the Poncelet point of quadrilateral $ABCD$.
The yellow points represent the midpoints of the sides and diagonals of the quadrilateral.
The component triangles are $BCD$, $ACD$, $ABD$, and $ABC$.
The blue circles are the nine-point circles of these triangles.
The common point of the four circles is the Poncelet point (shown in green).

\begin{figure}[h!t]
\centering
\includegraphics[width=0.4\linewidth]{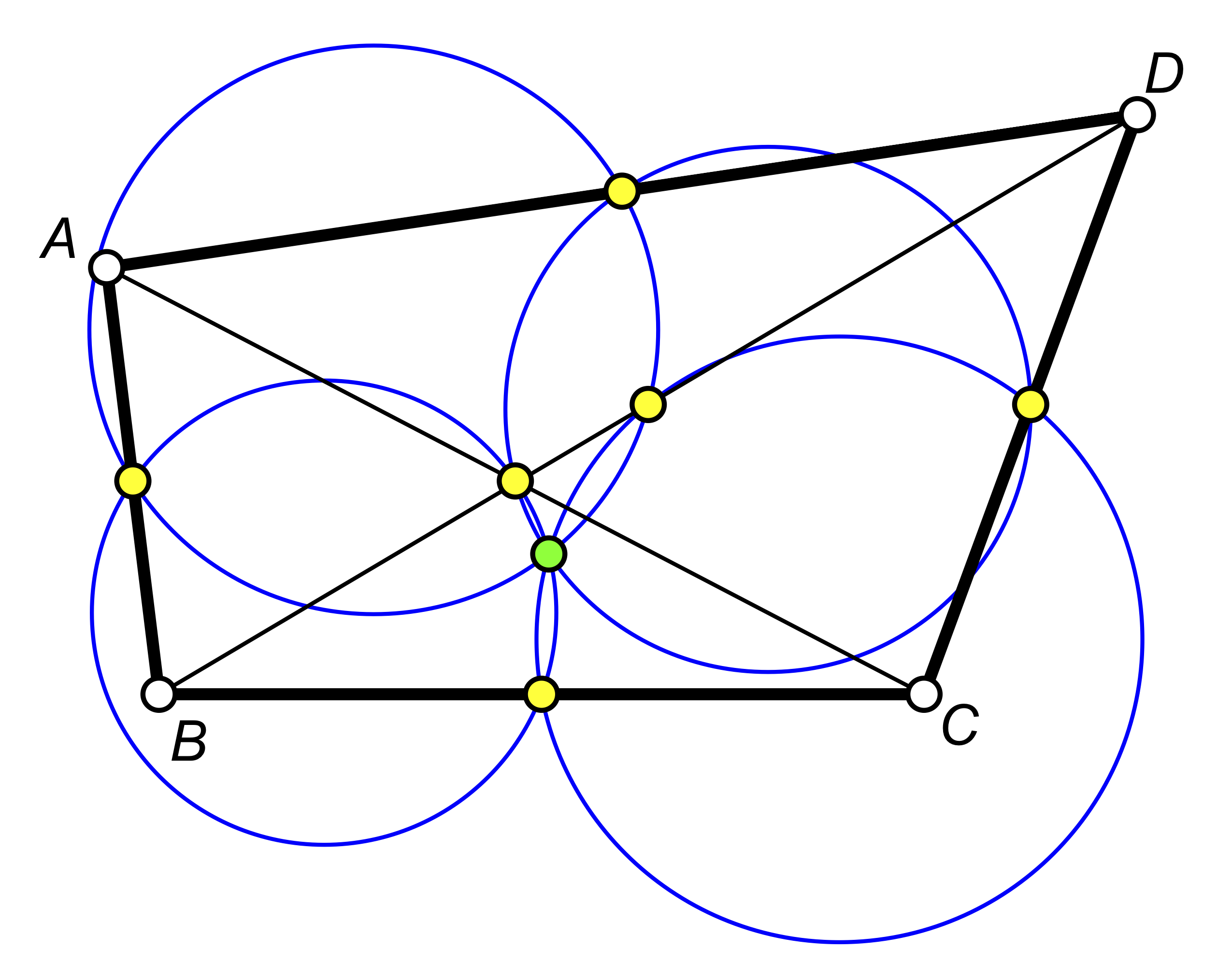}
\caption{The Poncelet point of quadrilateral $ABCD$}
\label{fig:ppPonceletPoint}
\end{figure}

\begin{proposition}
\label{prop:ppParallelogram}
The Poncelet point of a parallelogram coincides with the diagonal point.
\end{proposition}

\begin{figure}[h!t]
\centering
\includegraphics[width=0.5\linewidth]{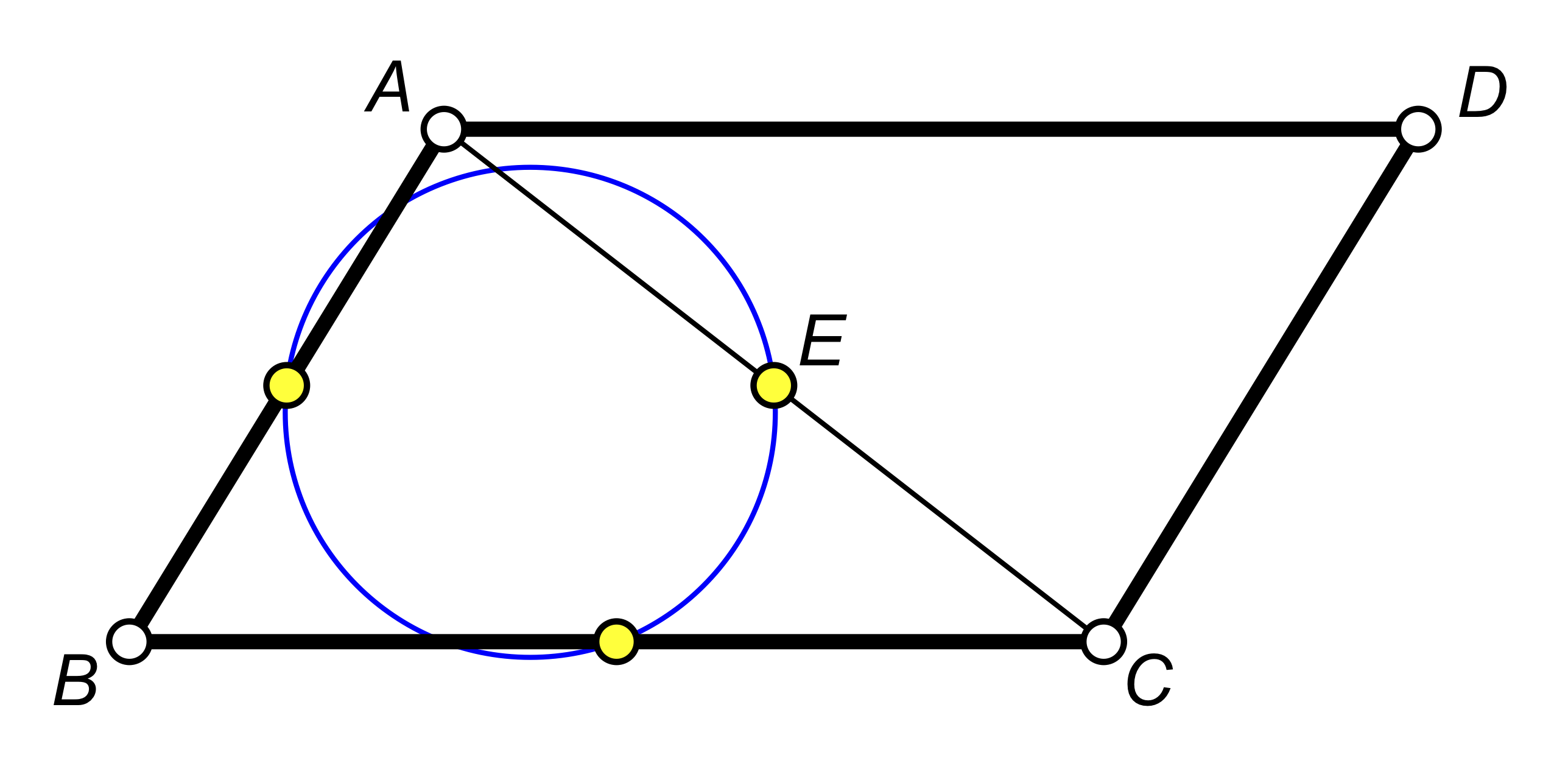}
\caption{Nine-point circle of component triangle $ABC$}
\label{fig:ppParallelogram}
\end{figure}

\begin{proof}
Since the diagonals of a parallelogram bisect each other, the diagonal point $E$, is the
midpoint of side $AC$ of component triangle $ABC$.
Thus, the nine-point circle of $\triangle ABC$ passes through $E$ (Figure~\ref{fig:ppParallelogram}).
Similarly, all the nine-point circles of the other component triangles pass through $E$. Hence the diagonal point is common to all four of these circles and
is therefore the Poncelet point of the quadrilateral.
\end{proof}

\newpage
The following result is well known.

\begin{lemma}
\label{lemma:ninePointCircle}
The nine-point circle of a triangle passes through the feet of the altitudes.
\end{lemma}

\begin{proposition}
\label{prop:ppOrtho}
The Poncelet point of an orthodiagonal quadrilateral coincides with the diagonal point.
\end{proposition}

\begin{figure}[h!t]
\centering
\includegraphics[width=0.4\linewidth]{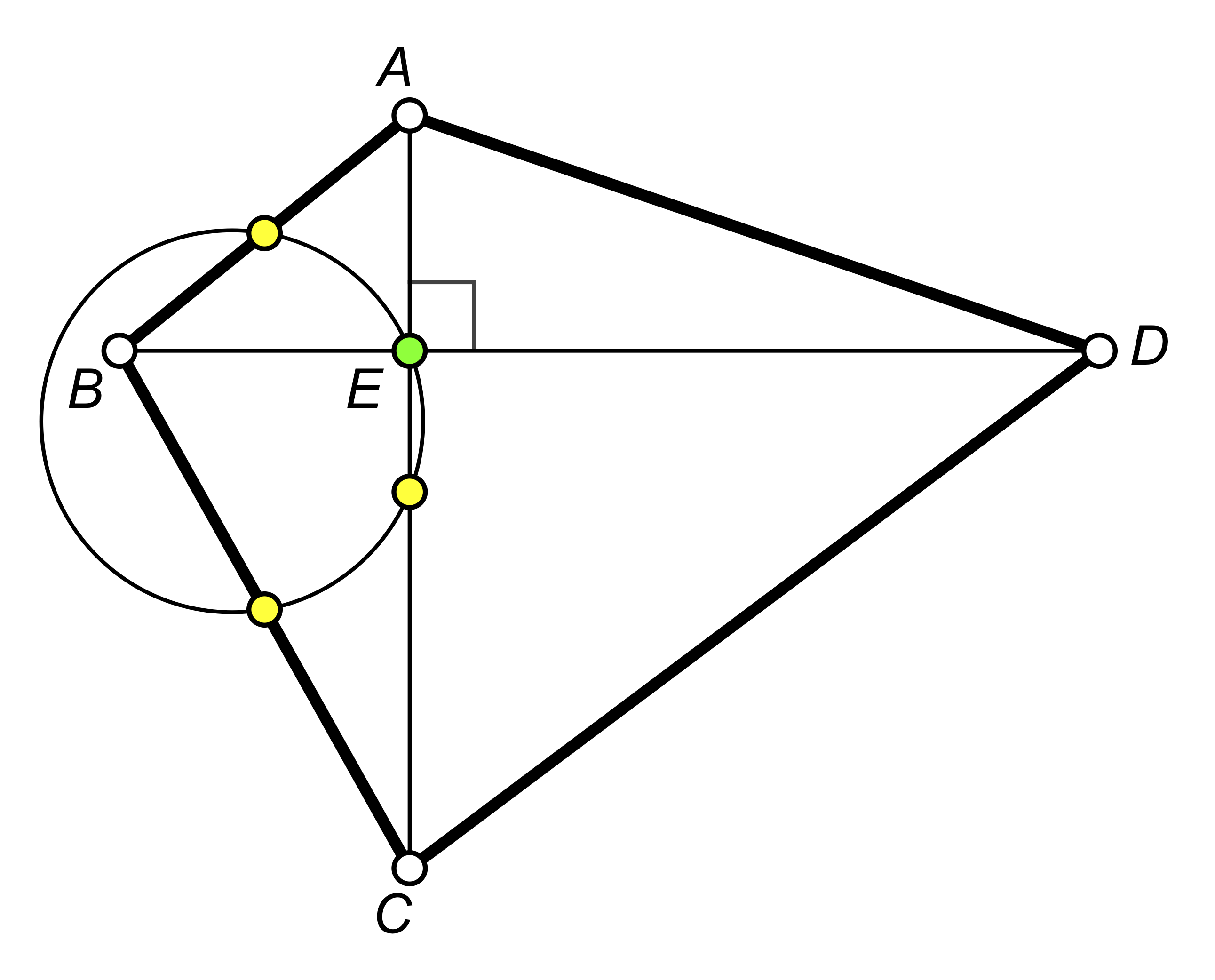}
\caption{Nine-point circle of component triangle $ABC$}
\label{fig:ppOrtho}
\end{figure}

\begin{proof}
Let $ABCD$ be an orthodiagonal quadrilateral with diagonal point $E$.
Then $BE$ is an altitude of component triangle $ABC$.
By Lemma~\ref{lemma:ninePointCircle}, the nine point circle of $\triangle ABC$
passes through $E$ (Figure~\ref{fig:ppOrtho}).
Similarly, the nine-point circles of the other component triangles also pass through $E$.
Thus, $E$ is the Poncelet point of quadrilateral $ABCD$.
\end{proof}

Our computer study examined the central quadrilaterals formed by the Poncelet point.
Since the Poncelet point coincides with the diagonal point of an orthodiagonal quadrilateral,
we omit results for orthodiagonal quadrilaterals.
Since the Poncelet point coincides with the diagonal point of a parallelogram,
we omit results for parallelograms.
We checked the central quadrilateral for all the first 1000 triangle centers (omitting points
at infinity) and all reference quadrilateral shapes listed in Table~\ref{table:quadrilaterals}.

The results found are listed in Table~\ref{table:Ponce}.
\medskip

\begin{table}[ht!]
\caption{}
\label{table:Ponce}
\begin{center}
\begin{tabular}{|l|l|p{2.2in}|}
\hline
\multicolumn{3}{|c|}{\color{blue}\textbf{\large \strut Central Quadrilaterals formed by the Poncelet Point}}\\ \hline
\textbf{Quadrilateral Type}&\textbf{Relationship}&\textbf{centers}\\ \hline
\ru Hjelmslev&$[ABCD]=\frac12[FGHI]$&20\\
\hline
\end{tabular}
\end{center}
\end{table}

The following result is well known, \cite{MathWorld-ninePointCircle}.

\begin{lemma}
\label{lemma:9pt}
The nine-point circle of a triangle bisects any line from the orthocenter to a point on the circumcircle.
\end{lemma}

%\newpage

\begin{lemma}
\label{lemma-ppHjelmslevMidpoint}
Let $E$ be the Poncelet point of quadrilateral $ABCD$ that has right angles at $B$ and $D$.
Then $E$ is the midpoint of $BD$.
\end{lemma}

\begin{figure}[h!t]
\centering
\includegraphics[width=0.4\linewidth]{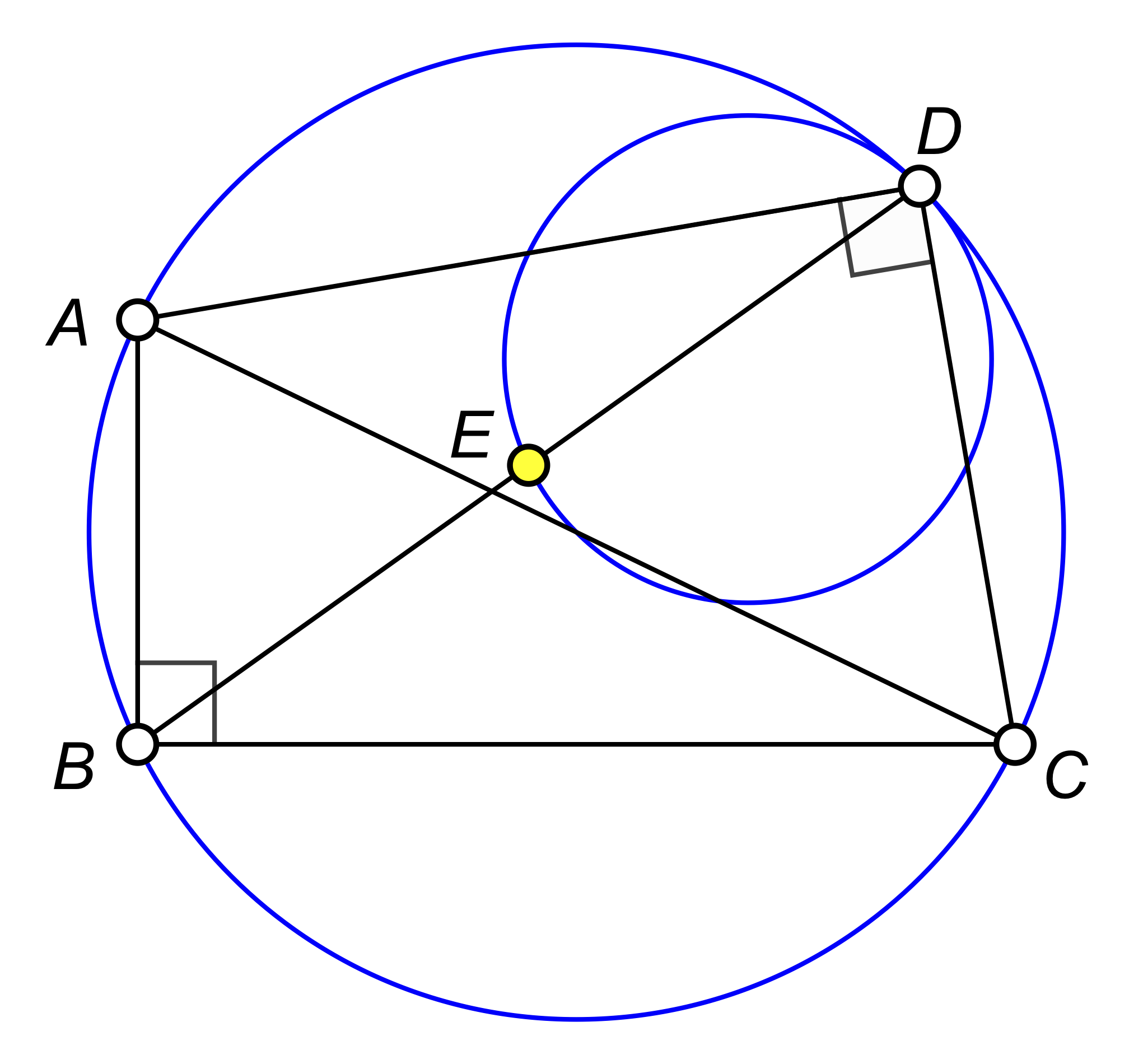}
\caption{Poncelet point of a Hjelmslev quadrilateral}
\label{fig:ppHjelmslevMidpoint}
\end{figure}

\begin{proof}
Since $\angle ABC$ is a right angle, $AC$ is a diameter of the circumcircle of $\triangle ABC$.
Since $\angle ADC$ is a right angle, $D$ lies on the circumcircle of $\triangle ABC$.
Since $\triangle ADC$ is a right triangle, its orthocenter is point $D$.
By Lemma~\ref{lemma:9pt}, the nine-point circle of $\triangle ADC$ bisects $BD$
(Figure~\ref{fig:ppHjelmslevMidpoint}). Let $E$ be the midpoint of $BD$.
Similarly, the nine-point circle of $\triangle ABC$ bisects $BD$.
So both nine-point circles pass through $E$.
By definition, the nine-point circles of triangles $ABD$ and $CBD$ pass through $E$.
Thus, $E$ is a common point of the nine-point circles of all four component triangles of
quadrilateral $ABCD$. Therefore, $E$ is the Poncelet point of $ABCD$.
\end{proof}

%\newpage

\relbox{Relationship $[ABCD]=\frac12[FGHI]$}

\begin{theorem}
\label{thm-ppHjelmslevX20}
Let $E$ be the Poncelet point of a Hjelmslev quadrilateral $ABCD$.
Let $F$, $G$, $H$, and $I$ be the de Longchamps points ($X_{20}$ points) of $\triangle EAB$, $\triangle EBC$, 
$\triangle ECD$,  and $\triangle EDA$, respectively (Figure~\ref{fig:ppHjelmslevX20}).
Then
$$[ABCD]=\frac12[FGHI].$$
\end{theorem}

\begin{figure}[h!t]
\centering
\includegraphics[width=0.4\linewidth]{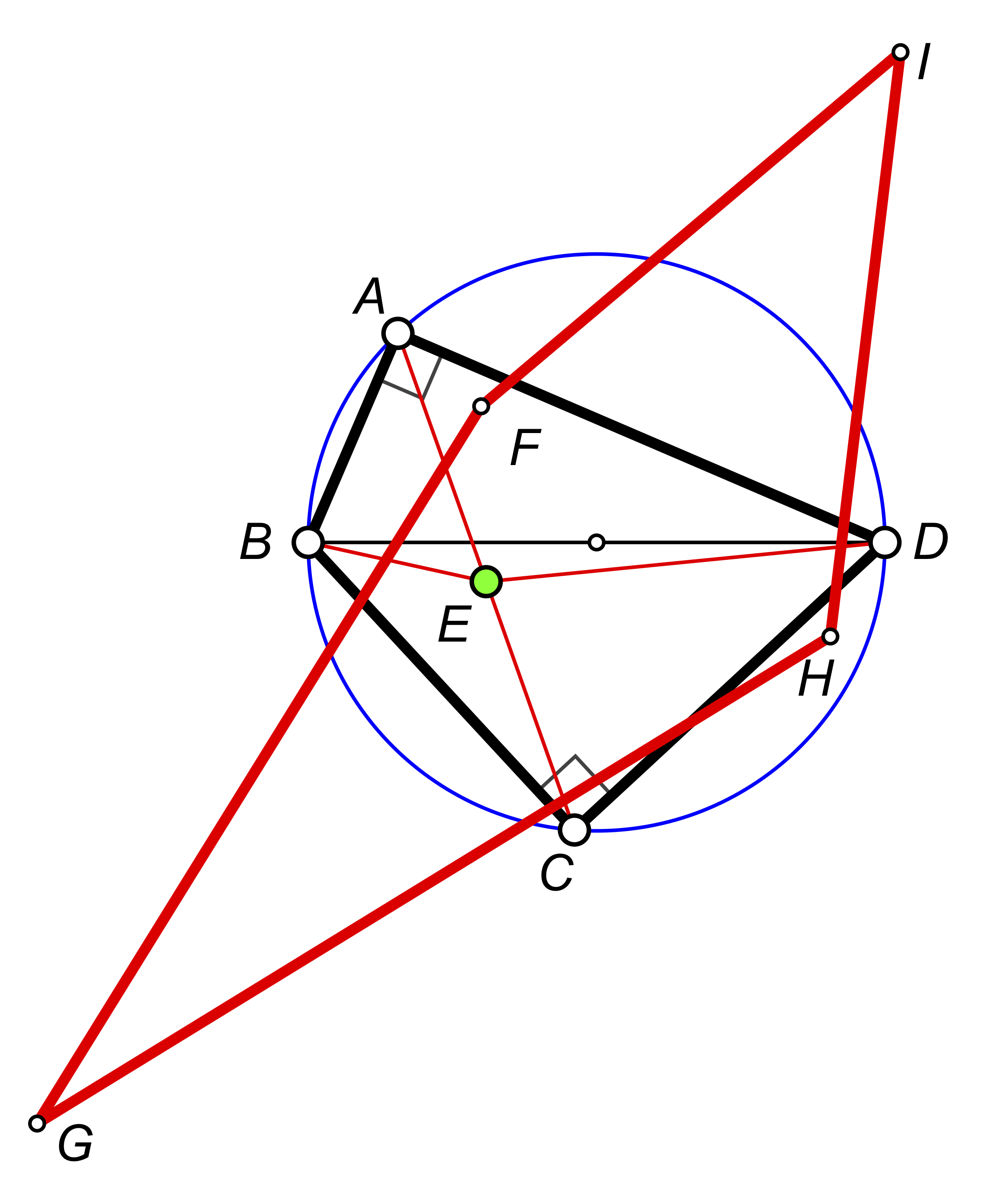}
\caption{Hjelmslev, $X_{20}$ points $\implies [ABCD]=\frac12[FGHI]$}
\label{fig:ppHjelmslevX20}
\end{figure}

\begin{proof}
Recall that a Hjelmslev quadrilateral is a quadrilateral with right angles at two opposite vertices.
Call the quadrilateral $ABCD$ with right angles at $B$ and $D$ (Figure~\ref{fig:ppCoordinatesMidpoint}).
We set up a barycentric coordinate system using $\triangle ABC$ as the reference triangle,
so that
$$
\begin{aligned}
A&=(1:0:0)\\
B&=(0:1:0)\\
C&=(0:0:1).
\end{aligned}
$$
We let the barycentric coordinates of $D$ be $(p:q:r)$ with $p+q+r=1$
and without loss of generality, assume $p>0$, $q<0$, and $r>0$.
Let $E$ be the Poncelet point of quadrilateral $ABCD$.
By Lemma~\ref{lemma-ppHjelmslevMidpoint}, $E$ is the midpoint of $BD$,
so has normalized barycentric coordinates $E=(\frac{p}{2}:\frac{q+1}{2}:\frac{r}{2})$ as shown in Figure~\ref{fig:ppCoordinatesMidpoint}.

\begin{figure}[h!t]
\centering
\includegraphics[width=0.45\linewidth]{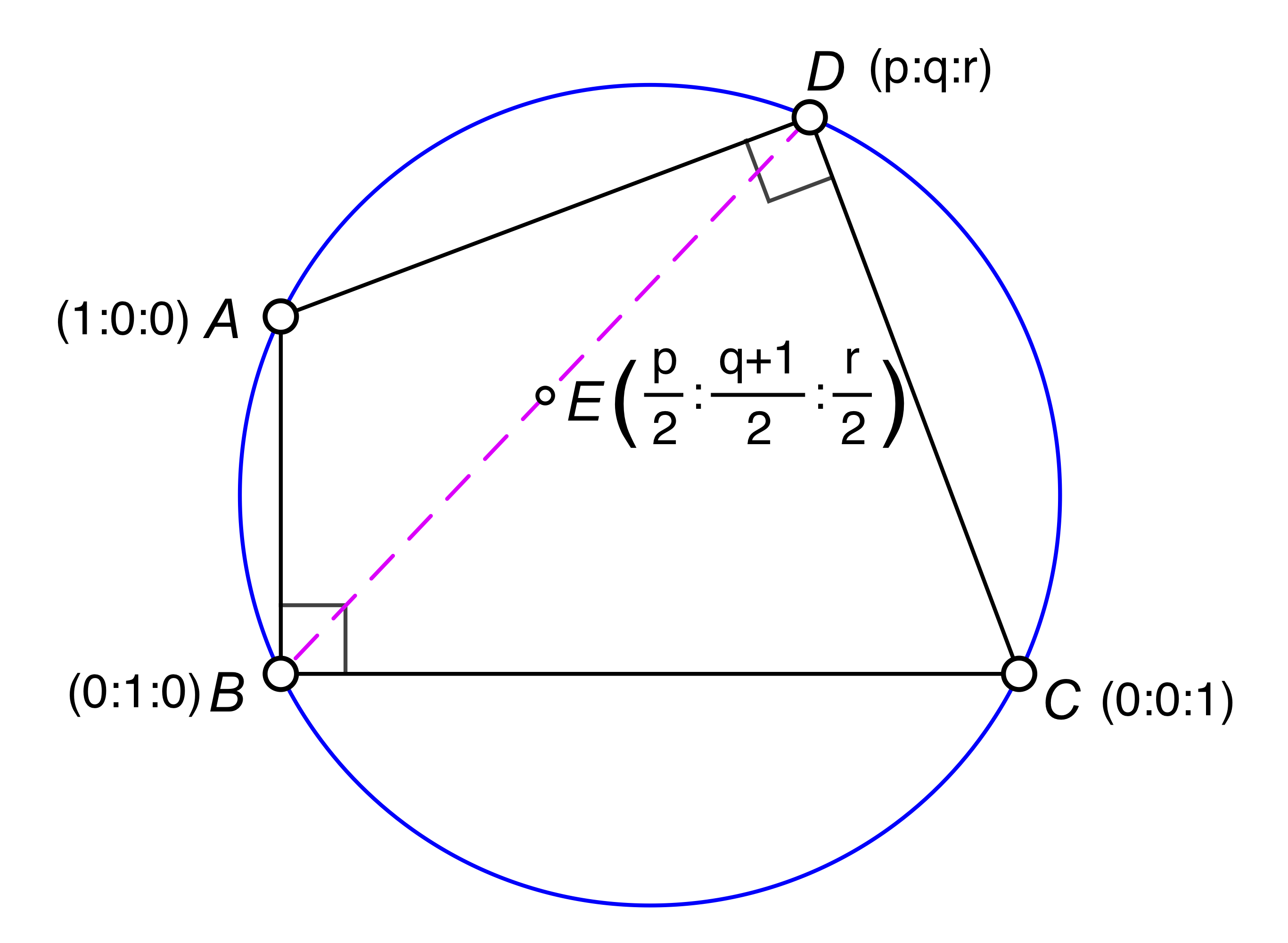}
\caption{Coordinate system for a Hjelmslev quadrilateral}
\label{fig:ppCoordinatesMidpoint}
\end{figure}

Using the Distance Formula (Lemma~\ref{lemma:distanceFormula}), we can compute the
distances between the various points. We get
$$
\begin{aligned}
AB&=c\\
BC&=a\\
AC&=b\\
AD&=\sqrt{-a^2 q r+b^2 r (q+r)+c^2 q (q+r)}\\
BD&=\sqrt{a^2 r (p+r)-b^2 p r+c^2 p (p+r)}\\
CD&=\sqrt{a^2 q (p+q)+b^2 p (p+q)-c^2 p q}\\
AE&=\frac{1}{2} \sqrt{-a^2 (q+1) r+b^2 r (q+r+1)+c^2 (q+1) (q+r+1)}\\
BE&=\frac{1}{2} \sqrt{a^2 r (p+r)-b^2 p r+c^2 p (p+r)}\\
CE&=\frac{1}{2} \sqrt{a^2 (q+1) (p+q+1)+b^2 p (p+q+1)-c^2 p (q+1)}\\
DE&=\frac{1}{2} \sqrt{a^2 r (p+r)-b^2 p r+c^2 p (p+r)}.
\end{aligned}
$$

Using the Area Formula (Lemma~\ref{lemma:areaFormula}), we can compute the areas of triangles $ABC$ and $ACD$.
We find
$$[ABC]=K\qquad\hbox{and}\qquad [ACD]=-qK$$
so that
$$[ABCD]=K(1-q).$$
Recall that $q$ is negative, so these areas are positive.

From \cite{ETC20}, we find that the barycentric coordinates for the $X_{20}$ point of a triangle with
sides of lengths $a$, $b$, and $c$ are $(x:y:z)$ where
$$
\begin{aligned}
x&=3 a^4-2 a^2 b^2-2 a^2 c^2-b^4+2 b^2 c^2-c^4\\
y&=-a^4-2 a^2 b^2+2 a^2 c^2+3 b^4-2 b^2 c^2-c^4\\
z&=-a^4+2 a^2 b^2-2 a^2c^2-b^4-2 b^2 c^2+3 c^4.
\end{aligned}
$$

We can use the Change of Coordinates
Formula (Lemma~\ref{lemma:changeOfCoordinates}) to find the coordinates for point $F$, the $X_{20}$ point of $\triangle ABE$, by substituting the lengths of $BE$, $AE$, and $AB$ for $a$, $b$, and $c$ in the expression for the
normalized barycentric coordinates for $X_{20}$.
In the same manner, we can find the barycentric coordinates for $G$, $H$, and $I$.

These barycentric coordinates are very complicated, but can be simplified using the fact that
quadrilateral $ABCD$ is a Hjelmslev quadrilateral.

Since $\triangle ABC$ is a right triangle, we have the relationship
\begin{equation}
\label{eq:a} a^2+c^2=b^2.
\end{equation}
Since $\triangle ADC$ is a right triangle, we have the relationship
$$AD^2+CD^2=AC^2.$$

In terms of $a$, $c$, $p$, and $q$, this is equivalent to
\begin{equation}
\label{eq:b}c^2=\frac{a^2 \left(p^2+2 p q-p+q^2-q\right)}{(1-p) p}
\end{equation}
where we have eliminated $b$ and $r$ since $b=\sqrt{a^2+c^2}$ and $r=1-p-q$.

Simplifying the formulas for the barycentric coordinates for $F$, $G$, $H$, and $I$
taking relationships (\ref{eq:a}) and (\ref{eq:b}) into account, we find that
$$
\begin{aligned}
F&=\left(1-\frac{p}{2}:\frac{p (-q)+p+3 q+1}{2 (p-1)}:\frac{p^2+p (q-2)-3 q-1}{2(p-1)}\right)\\
G&=\left(-\frac{p^2+p q-2 (q+1)}{2 (p+q)}:-\frac{p (q-1)+q^2+q+2}{2 (p+q)}:\frac{1}{2}(p+q+1)\right)\\
H&=\left(\frac{p}{2}+q+1,\frac{p (q-1)+2 q^2+q+1}{2 (p-1)},\frac{p^2+3 p q-2 p+2q^2-q+1}{2-2 p}\right)\\
I&=\left(\frac{p^2-p q+2 q}{2 (p+q)},\frac{p (q-1)-q (q+3)}{2 (p+q)},\frac{1}{2}(-p+q+3)\right).
\end{aligned}
$$

Using the Area Formula, we can compute the areas of triangles $FGH$ and $HIF$. We get
$$
\begin{aligned}\ 
[FGH]&=\frac{K \left(p (q+3)+q^2+q-2\right)}{p-1}\\
[HIF]&=-\frac{K (3 p q+p+(q-1) q)}{p-1}\\
\end{aligned}
$$
so that
$$[FGHI]=[FGH]+[HIF]=2K(1-q).$$

Thus,
$[FGHI]=2[ABCD]$.
\end{proof}

\newpage

%**************************************
%    Circumcenter
%**************************************

\section{Circumcenter}

In this section, we examine central quadrilaterals formed from the circumcenter
of the reference quadrilateral. Note that only cyclic quadrilaterals have circumcenters.
The \emph{circumcenter} of a cyclic quadrilateral is the center of the circle
through the vertices of the quadrilateral.

Our computer study examined the central quadrilaterals formed by the circumcenter.
Since the circumcenter of a rectangle coincides with the diagonal point of the rectangle,
we omit results for rectangles.
We checked the central quadrilateral for all the first 1000 triangle centers (omitting points
at infinity) and all reference quadrilateral shapes listed in Table~\ref{table:quadrilaterals}
that are cyclic.

The results found are listed in Table~\ref{table:circum}.
%\smallskip

\begin{table}[ht!]
\caption{}
\label{table:circum}
\begin{center}
\begin{tabular}{|l|l|p{2.2in}|}
\hline
\multicolumn{3}{|c|}{\textbf{\color{blue}\large \strut Central Quadrilaterals formed by the Circumcenter}}\\ \hline
\textbf{Quadrilateral Type}&\textbf{Relationship}&\textbf{centers}\\ \hline
\ru cyclic&$[ABCD]=8[FGHI]$&402, 620\\
\cline{2-3}
\ru &$[ABCD]=2[FGHI]$&11, 115, 116, 122--125, 127, 130, 134-137, 139, 244--247, 338, 339, 865--868\\
\cline{2-3}
\ru &$[ABCD]=\frac32[FGHI]$&616, 617\\
\cline{2-3}
\ru &$[ABCD]=\frac98[FGHI]$&290, 671, 903\\
\cline{2-3}
\ru &$[ABCD]=\frac12[FGHI]$&148--150\\
\hline
\end{tabular}
\end{center}
\end{table}

%\bigskip

\begin{theorem}
\label{thm:dpCyclic}
Let $E$ be the circumcenter of cyclic quadrilateral $ABCD$.
Let $X_n$ be a triangle center with the property that for all isosceles triangles with vertex $V$
and midpoint of base $M$, $X_nM/VM$ is a fixed positive constant $k$.
Let $F$, $G$, $H$, and $I$ be the $X_n$ points of $\triangle EAB$, $\triangle EBC$, 
$\triangle ECD$, and $\triangle EDA$, respectively.
Then
$$[ABCD]=\frac{2}{(1-k)^2}[FGHI].$$
\end{theorem}

\begin{proof}
The proof is the same as the proof of Theorem~\ref{thm:dpRectanglePos}.
\end{proof}

%When $n=402$ or $n=620$, $k=1/2$ and $[ABCD]=8[FGHI]$.

\relbox{Relationship $[ABCD]=8[FGHI]$}

\begin{theorem}
\label{thm:opCyclicX402}
Let $E$ be the circumcenter of cyclic quadrilateral $ABCD$.
Let $F$, $G$, $H$, and $I$ be the $X_{402}$ points or the $X_{620}$ points of $\triangle EAB$, $\triangle EBC$, 
$\triangle ECD$, and $\triangle EDA$, respectively (Figure~\ref{fig:opCyclicX402}).
Then
$$[ABCD]=8[FGHI].$$
\end{theorem}

\begin{figure}[h!t]
\centering
\includegraphics[width=0.35\linewidth]{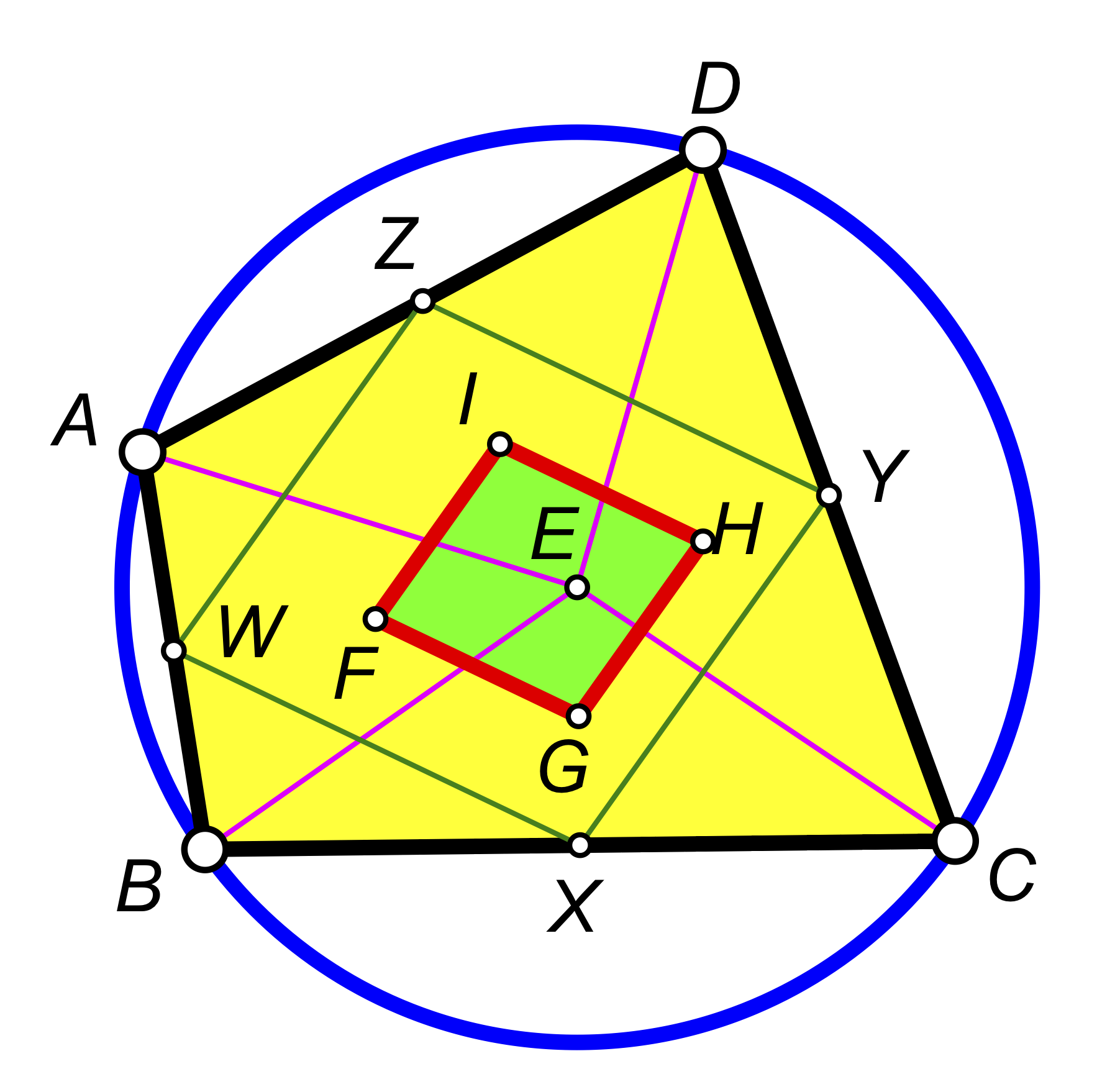}
\caption{$X_{402}$ points $\implies [ABCD]=8[FGHI]$}
\label{fig:opCyclicX402}
\end{figure}

\begin{proof}
Since $E$ is the center of the circle through points $A$, $B$, $C$, and $D$,
each of the radial triangles is isosceles with vertex $E$.
Let the midpoints of the sides of the quadrilateral be $W$, $X$, $Y$, and $Z$
as shown in Figure~\ref{fig:opCyclicX402}.
From Theorem~\ref{thm:ratio} and Table~\ref{table:diagonalPointRatios}, we find that
for $n=402$ and $n=620$, the ratio $X_nM/AM$ is a constant, $\frac12$, for all isosceles triangles
with vertex $A$ and midpoint of opposite side $M$.
Therefore, by Theorem~\ref{thm:dpCyclic}, with $k=\frac12$, we must have
$$[ABCD]=\frac{2}{(1-k)^2}[FGHI]=8[FGHI].$$
\end{proof}

\void{
\begin{proof}
Since $E$ is the center of the circle through points $A$, $B$, $C$, and $D$,
each of the radial triangles is isosceles with vertex $E$.
Let the midpoints of the sides of the quadrilateral be $W$, $X$, $Y$, and $Z$
as shown in Figure~\ref{fig:opCyclicX402}.

***NEED TO FIX ***

Since $F$ is the $X_n$ point of $\triangle EAB$, by hypothesis,
$$\frac{FW}{EW}=k.$$
Since $k>0$
$$\frac{EF}{EW}=\frac{EW-FW}{EW}=1-\frac{FW}{EW}=1-k.$$
Similarly, $EG/EX=1-k$, $EH/EY=1-k$, and $EI/EZ=1-k$.
So quadrilaterals $FGHI$ and $WXYZ$ are homothetic, with $E$ the center of similitude
and ratio of similarity $\frac12$.
Thus
$$[FGHI]=\frac14[WXYZ].$$
But $[WXYZ]=\frac12[ABCD]$, so $[FGHI]=\frac18[ABCD]$
\end{proof}

When $n=402$ or $n=620$, $k=1/2$ and $[ABCD]=8[FGHI]$.
}

\relbox{Relationship $[ABCD]=2[FGHI]$}

\begin{theorem}
\label{theorem:opCyclicMidpoints}
Let $E$ be the circumcenter of cyclic quadrilateral $ABCD$.
Let $n$ be in the set
$$\{11, 115, 116, 122, 123, 124, 125, 127, 130, 134, 135, 136 $$
$$137, 139,
244, 245, 246,247, 338, 339, 865, 866, 867, 868\}.$$
Let $F$, $G$, $H$, and $I$ be the $X_n$ points of $\triangle EAB$, $\triangle EBC$, 
$\triangle ECD$,  and $\triangle EDA$, respectively (Figure~\ref{fig:opCyclicMidpoints}).
Then $F$, $G$, $H$, and $I$ are the midpoints of the sides of the quadrilateral
and
$$[ABCD]=2[FGHI].$$
\end{theorem}

\begin{figure}[h!t]
\centering
\includegraphics[width=0.3\linewidth]{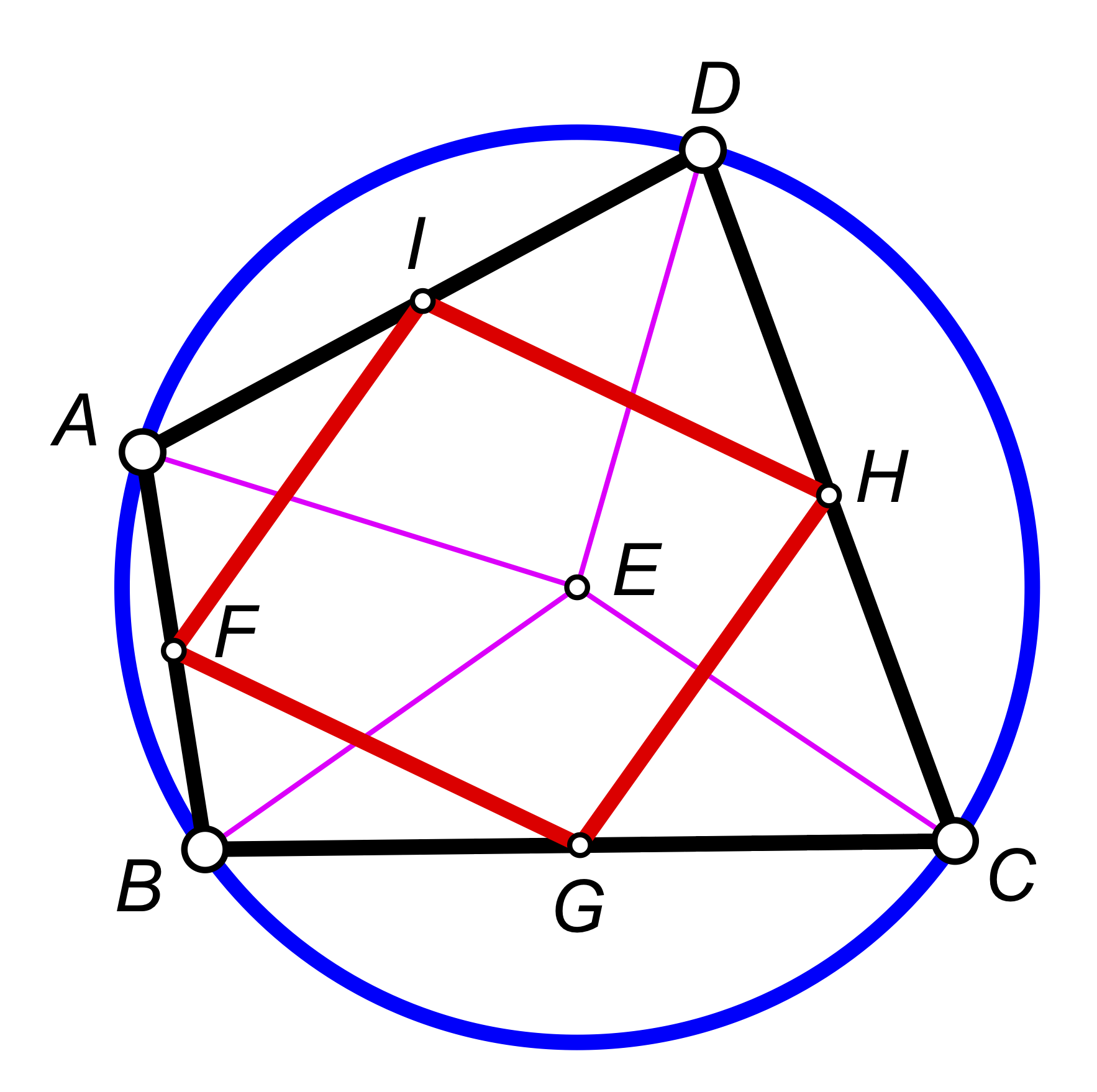}
\caption{$[ABCD]=2[FGHI]$}
\label{fig:opCyclicMidpoints}
\end{figure}

\begin{proof}
Since $E$ is the center of the circle through points $A$, $B$, $C$, and $D$,
each of the radial triangles is isosceles with vertex $E$.
Thus, by Lemma \ref{lemma:dpIsoscelesTriangleMidpoint}, 
$F$, $G$, $H$, and $I$ are the midpoints of the sides of the quadrilateral.
Then, by Lemma \ref{lemma:Varignon}, $[ABCD]=2[FGHI]$.
\end{proof}

\relbox{Relationship $[ABCD]=\frac32[FGHI]$}

\begin{proposition}[$X_{616}$ Property of an Isosceles Triangle]
\label{proposition:isoscelesX616}
Let $\triangle ABC$ be an isosceles triangle with $AB=AC$. If $F$ is the $X_{616}$ point of $\triangle ABC$, then $AF\perp BC$ and $AF=BC/\sqrt{3}$ (Figure~\ref{fig:opIsoscelesX616}).
\end{proposition}

\begin{figure}[h!t]
\centering
\includegraphics[width=0.3\linewidth]{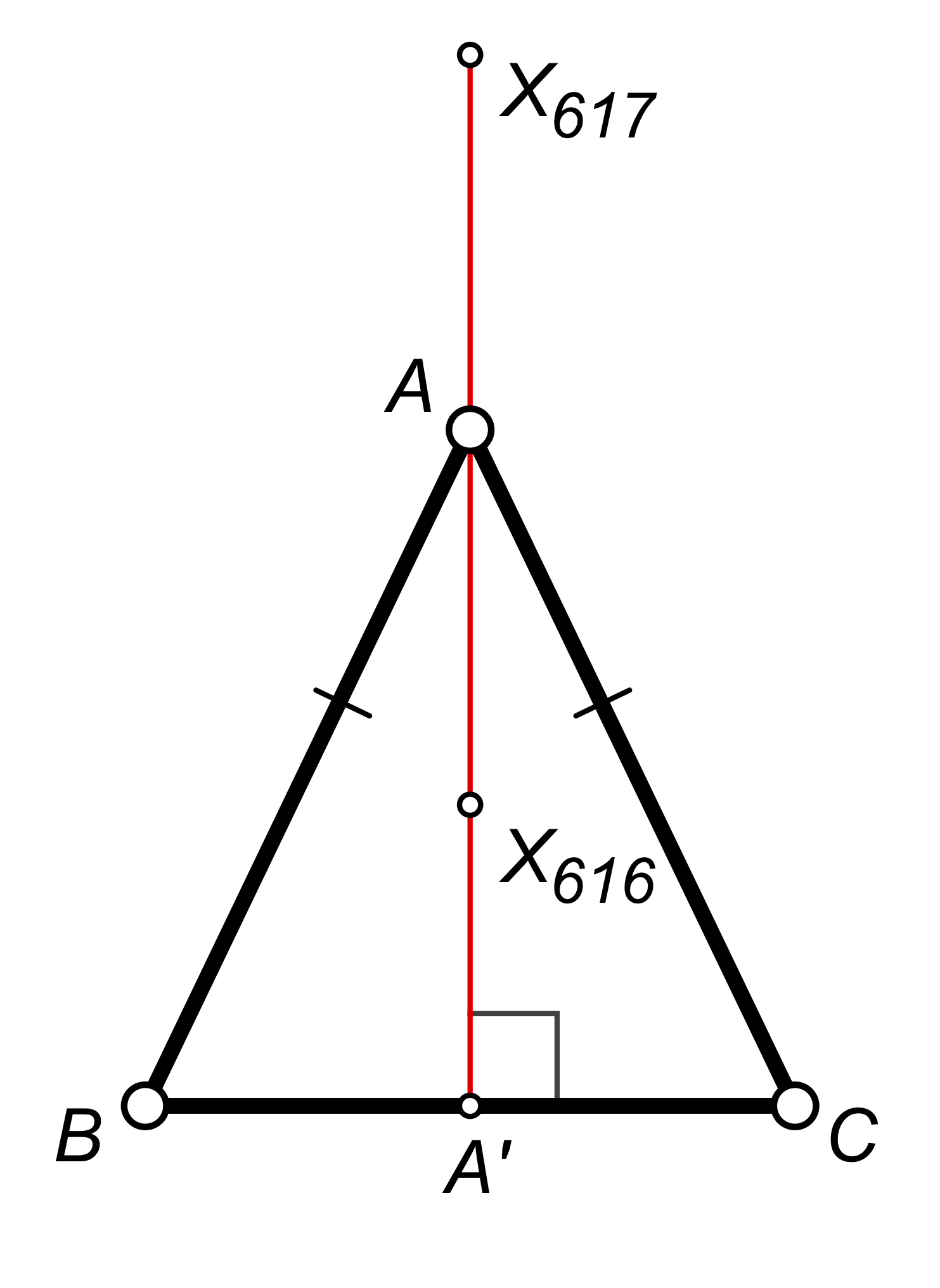}
\caption{$X_{616}$ and $X_{617}$ points of an isosceles triangle}
\label{fig:opIsoscelesX616}
\end{figure}

\begin{proof}
We use barycentric coordinates with respect to $\triangle ABC$. Let $A'$ be orthogonal projection of $A$ on the side $BC$.
Since $\triangle ABC$ is isosceles, $A'$ is the midpoint of $BC$, so $A'=(0:1:1)$. The barycentric coordinates of $F$ are
\begin{equation*}
   F=\left(5 a^4-a^2 \left(4 b^2+4 c^2-2 \sqrt{3} S\right)-b^4+2 b^2 \left(c^2-\sqrt{3} S\right)-c^4-2 \sqrt{3} c^2 S:\,:\right)
\end{equation*}
where $S$ denotes twice the area of $\triangle ABC$. Using the fact that $b=c$ we get
\begin{equation*}
\begin{split}
F= &\Biggl(5 a^4-a^2 \left(8 b^2-2 \sqrt{3} S\right)-2 b^4-2 \sqrt{3} b^2 S+2 b^2 \left(b^2-\sqrt{3} S\right): \\
   &-a^4-2 a^2 \left(b^2+\sqrt{3} S\right):-a^4-2 a^2 \left(b^2+\sqrt{3} S\right)\Biggr)
\end{split}
\end{equation*}

A simple calculation shows that the point $F$ lies on the line $AA'$ (which has equation $y=z$), so $AF\perp BC$.\\

Using the distance formula to get the length of $AF$ and substituting
$c=b$ and $S=\frac{1}{4} a \sqrt{4 b^2-a^2}$, we get

\void{
Using the distance formula we get
\begin{equation*}
   AF^2=\frac{a^6-a^4 b^2-a^4 c^2+2 a^2 b^4-3 a^2 b^2 c^2+2 a^2 c^4+b^6-b^4 c^2-2 \sqrt{3} b^4 S-b^2 c^4+4 \sqrt{3} b^2 c^2 S+c^6-2 \sqrt{3} c^4 S}{3 \left(a^4-a^2 b^2-a^2 c^2+b^4-b^2 c^2+c^4\right)}
\end{equation*} 
Putting $c=b$ and $S=\frac{1}{4} a \sqrt{4 b^2-a^2}$ in the above formula we obtain
}
$$AF^2=\frac{a^2}{3}=\frac{BC^2}{3}.$$
Thus, $AF=BC/\sqrt{3}$.
\end{proof}

\begin{proposition}[$X_{617}$ Property of an Isosceles Triangle]
\label{proposition:isoscelesX617}
Let $\triangle ABC$ be an isosceles triangle with $AB=AC$. If $F$ is the $X_{617}$ point of $\triangle ABC$, then $AF\perp BC$ and $AF=BC/\sqrt{3}$.
\end{proposition}

\begin{proof}[Proof] 
The proof is similar to the proof of Proposition~\ref{proposition:isoscelesX616}, so the details are omitted.
\end{proof}

\begin{theorem}
\label{thm:opCyclicX616}
Let $E$ be the circumcenter of cyclic quadrilateral $ABCD$.
Let $F$, $G$, $H$, and $I$ be the $X_{616}$ points of $\triangle EAB$, $\triangle EBC$, 
$\triangle ECD$, and $\triangle EDA$, respectively (Figure~\ref{fig:opCyclicX616}).
Then
$$[ABCD]=\frac32[FGHI].$$
\end{theorem}

\begin{figure}[h!t]
\centering
\includegraphics[width=0.4\linewidth]{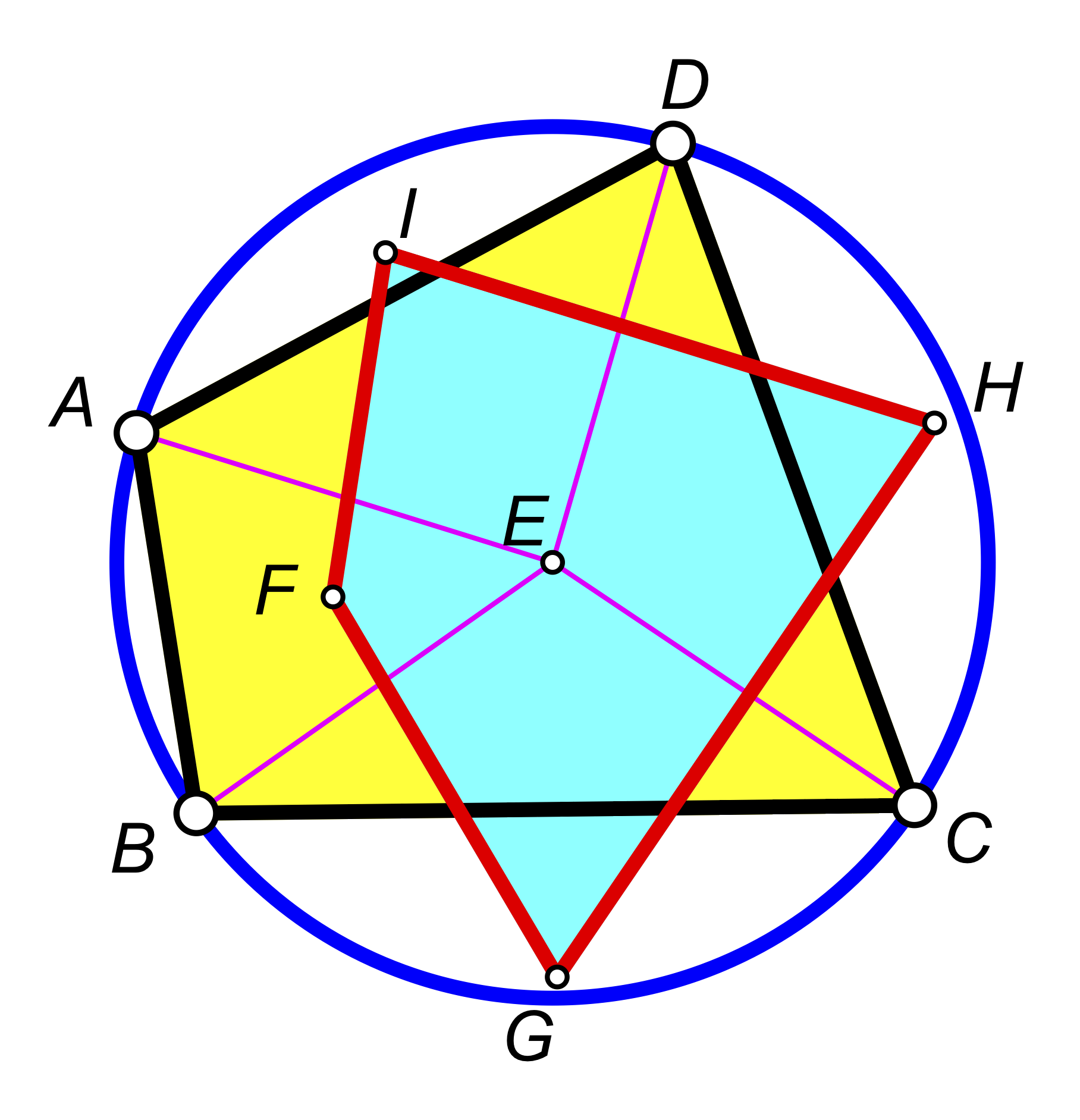}
\caption{$X_{616}$ points $\implies [ABCD]=\frac32[FGHI]$}
\label{fig:opCyclicX616}
\end{figure}

\begin{proof}
Let $AB=a$, $BC=b$, $CD=c$, and $DA=d$.
From Lemma~\ref{proposition:isoscelesX616}, we have $EF=a/\sqrt{3}$ and $EG=b/\sqrt{3}$
(Figure~\ref{fig:opCyclicX616proof}). 

\begin{figure}[h!t]
\centering
\includegraphics[scale=0.4]{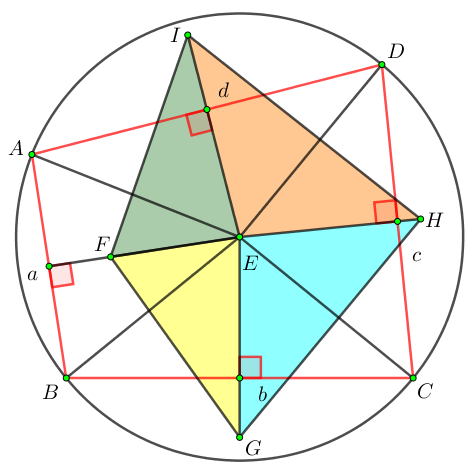}
\caption{}
\label{fig:opCyclicX616proof}
\end{figure}

Therefore,
$$
\begin{aligned}
{[EFG]}&=\frac{1}{2}\cdot EF\cdot EG\cdot \sin \left(\angle FEG\right)\\
&=\frac{1}{2}\cdot \frac{a}{\sqrt{3}}\cdot \frac{b}{\sqrt{3}}\cdot \sin \left(180^\circ-B\right)\\
&=\frac{1}{6}ab\sin B\\
&=\frac{1}{3}[ABC].
\end{aligned}
$$ 
Similarly, $[EGH]=\frac{1}{3}[BCD]$, $[EHI]=\frac{1}{3}[CDA]$, $[EIF]=\frac{1}{3}[DAB]$.
Therefore,
$$
\begin{aligned}
{[EFGH]}&=[EFG]+[EGH]+[EHI]+[EIF]\\
&=\frac{1}{3}\left([ABC]+[BCD]+[CDA]+[DAB]\right)
\end{aligned}
$$
from which, using the relations $[ABCD]=[ABC]+[CDA]=[BCD]+[DAB]$, we get the desired result.
\end{proof}

\textbf{Note.} Theorem~\ref{theorem:dpRectangleX616} is a special case of Theorem~\ref{thm:opCyclicX616}
since all rectangles are cyclic.

\begin{theorem}
\label{thm:opCyclicX617}
Let $E$ be the circumcenter of cyclic quadrilateral $ABCD$. Let $F$, $G$, $H$, and $I$ be the $X_{617}$ points of $\triangle EAB$, $\triangle EBC$, $\triangle ECD$,
and $\triangle EDA$, respectively. Then
\begin{equation*}
   [ABCD]=\frac{3}{2}[FGHI]
\end{equation*}
\end{theorem}

\begin{proof}
The proof is similar to the proof of Theorem~\ref{thm:opCyclicX616}, so the details are omitted.
\end{proof}

\relbox{Relationship $[ABCD]=\frac98[FGHI]$}

\begin{theorem}
\label{thm:opCyclicX290}
Let $E$ be the circumcenter of cyclic quadrilateral $ABCD$.
Let $n$ be 290, 671, or 903.
Let $F$, $G$, $H$, and $I$ be the $X_n$ points of $\triangle EAB$, $\triangle EBC$, 
$\triangle ECD$, and $\triangle EDA$, respectively (Figure~\ref{fig:opCyclicX290}).
Then
$$[ABCD]=\frac98[FGHI].$$
\end{theorem}

\begin{figure}[h!t]
\centering
\includegraphics[width=0.4\linewidth]{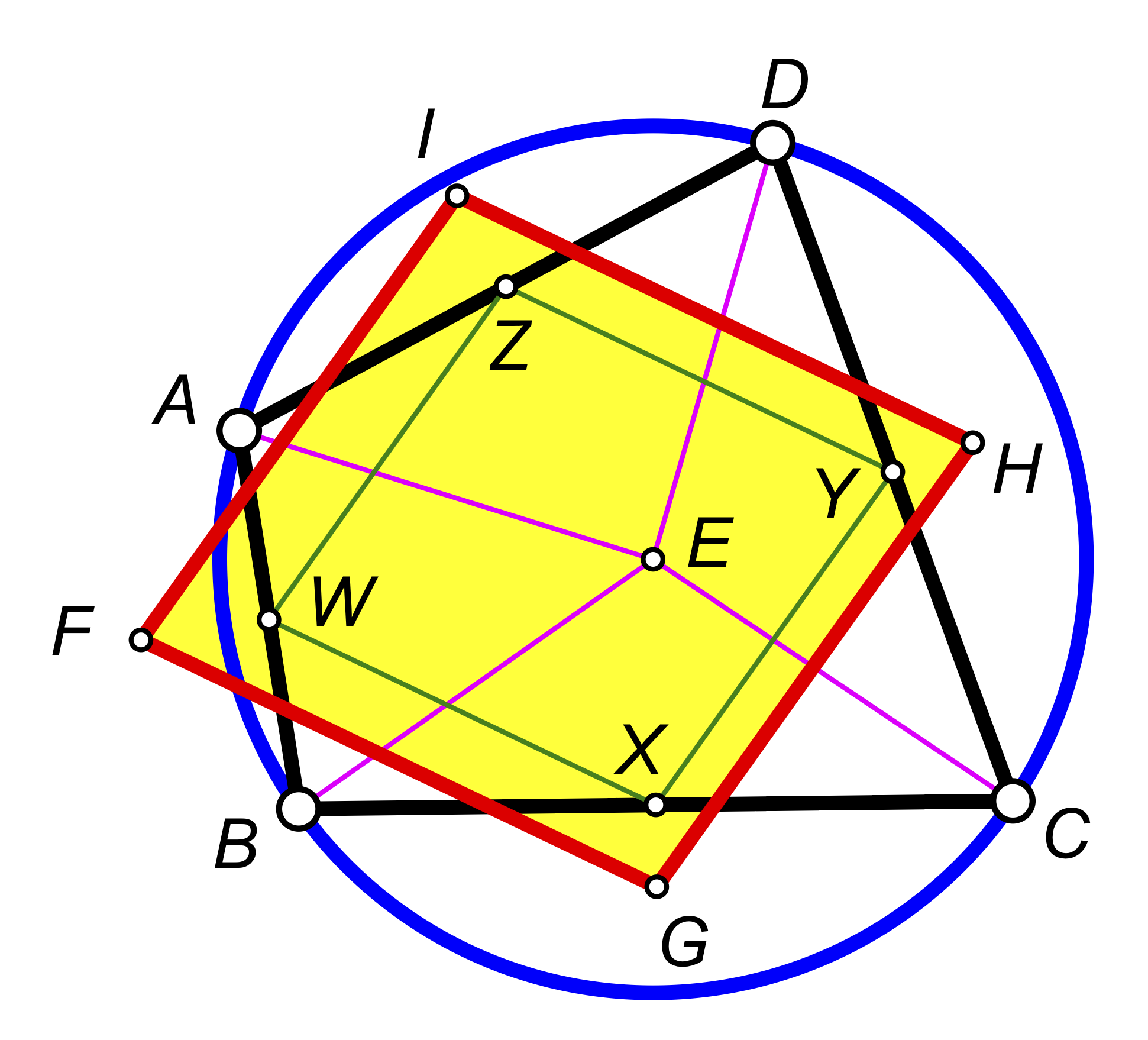}
\caption{$X_{290}$ points $\implies [ABCD]=\frac98[FGHI]$}
\label{fig:opCyclicX290}
\end{figure}

\begin{proof}
The proof is the same as the proof of Theorem~\ref{thm:opCyclicX402}, except $k=-\frac13$
and $\frac{2}{(1-k)^2}=\frac98$.
\end{proof}

\newpage

\relbox{Relationship $[ABCD]=\frac12[FGHI]$}

\begin{theorem}
\label{thm:opCyclicX149}
Let $E$ be the circumcenter of cyclic quadrilateral $ABCD$.
Let $n$ be 148, 149, or 150.
Let $F$, $G$, $H$, and $I$ be the $X_n$ points of $\triangle EAB$, $\triangle EBC$, 
$\triangle ECD$, and $\triangle EDA$, respectively (Figure~\ref{fig:opCyclicX149}).
Then
$$[ABCD]=\frac12[FGHI].$$
\end{theorem}

\begin{figure}[h!t]
\centering
\includegraphics[width=0.4\linewidth]{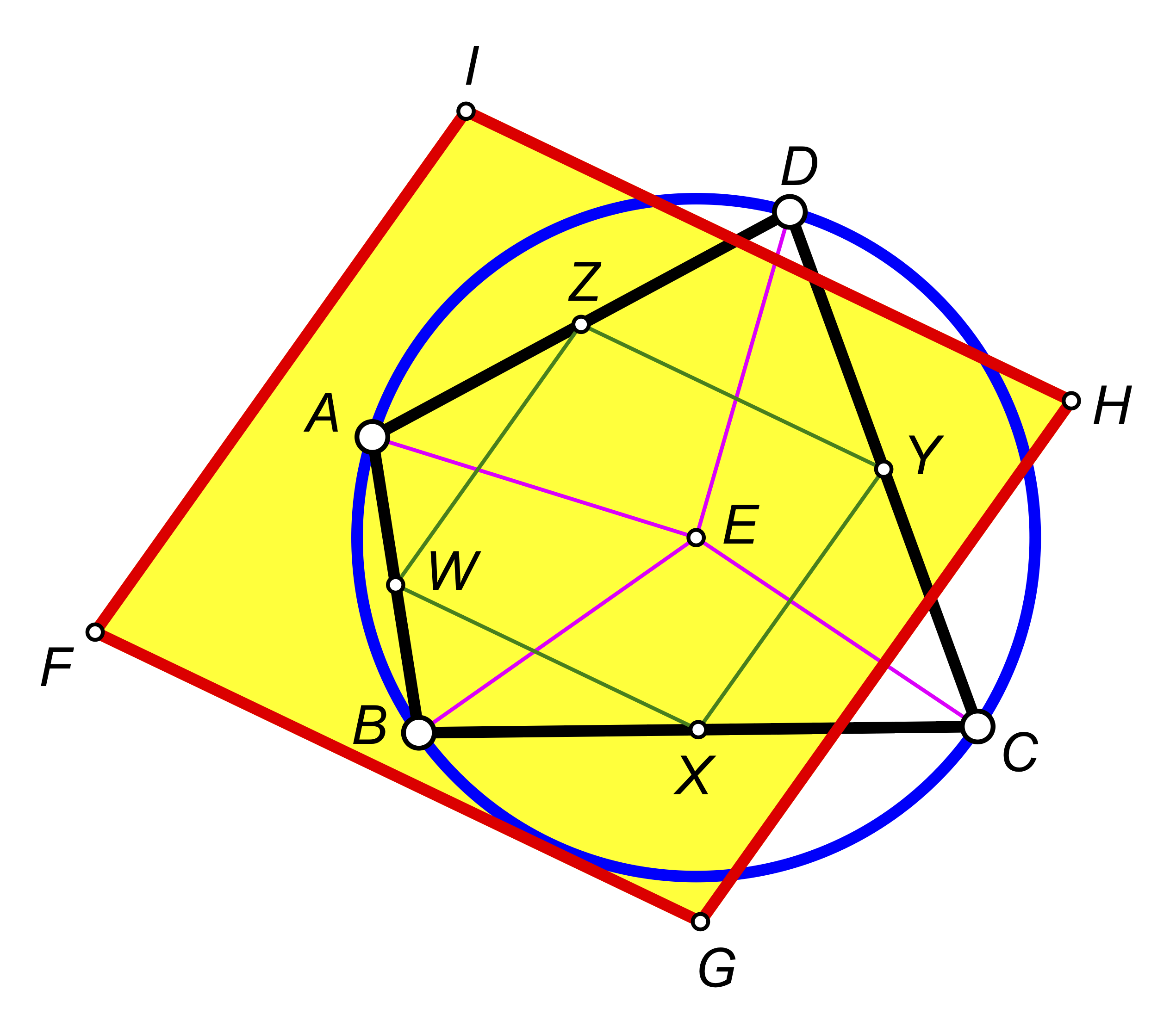}
\caption{$X_{149}$ points $\implies [ABCD]=\frac12[FGHI]$}
\label{fig:opCyclicX149}
\end{figure}

\begin{proof}
The proof is the same as the proof of Theorem~\ref{thm:opCyclicX402}, except $k=-1$
and $\frac{2}{(1-k)^2}=\frac12$.
\end{proof}

\newpage

%**************************************
%    Steiner point
%**************************************

\section{Steiner point}

In this section, we examine central quadrilaterals formed from the Steiner
point of the reference quadrilateral.

A \emph{midray circle} of a quadrilateral
is the circle through the midpoints of the line segments joining one vertex
of the quadrilateral to the other vertices (Figure~\ref{fig:spMidrayCircle}).

\begin{figure}[h!t]
\centering
\includegraphics[width=0.35\linewidth]{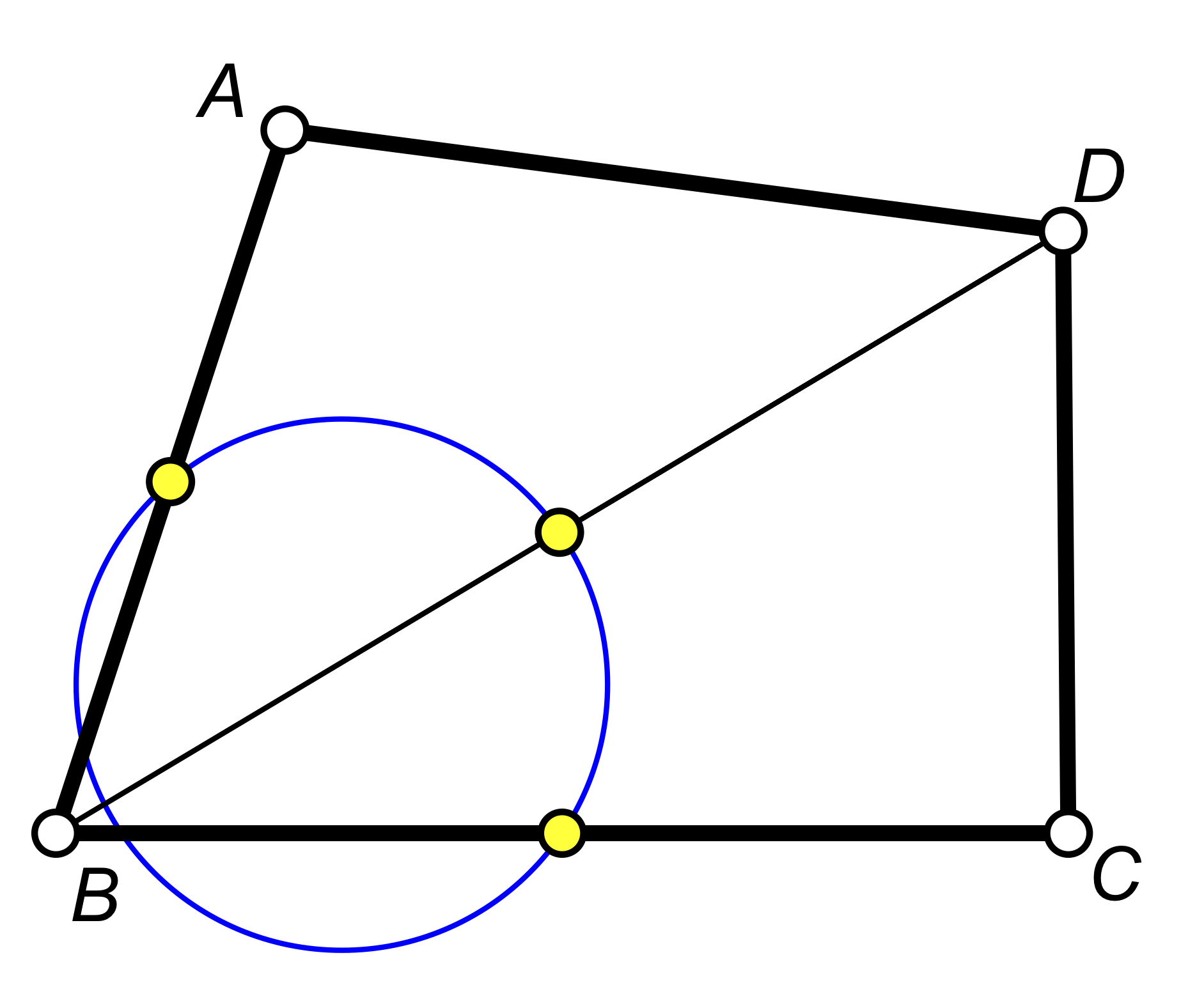}
\caption{Midray circle of quadrilateral $ABCD$ relative to vertex $B$}
\label{fig:spMidrayCircle}
\end{figure}

The \emph{Steiner point} (sometimes called the Gergonne-Steiner point) of a quadrilateral
is the common point of the midray circles of the quadrilateral.

\begin{figure}[h!t]
\centering
\includegraphics[width=0.4\linewidth]{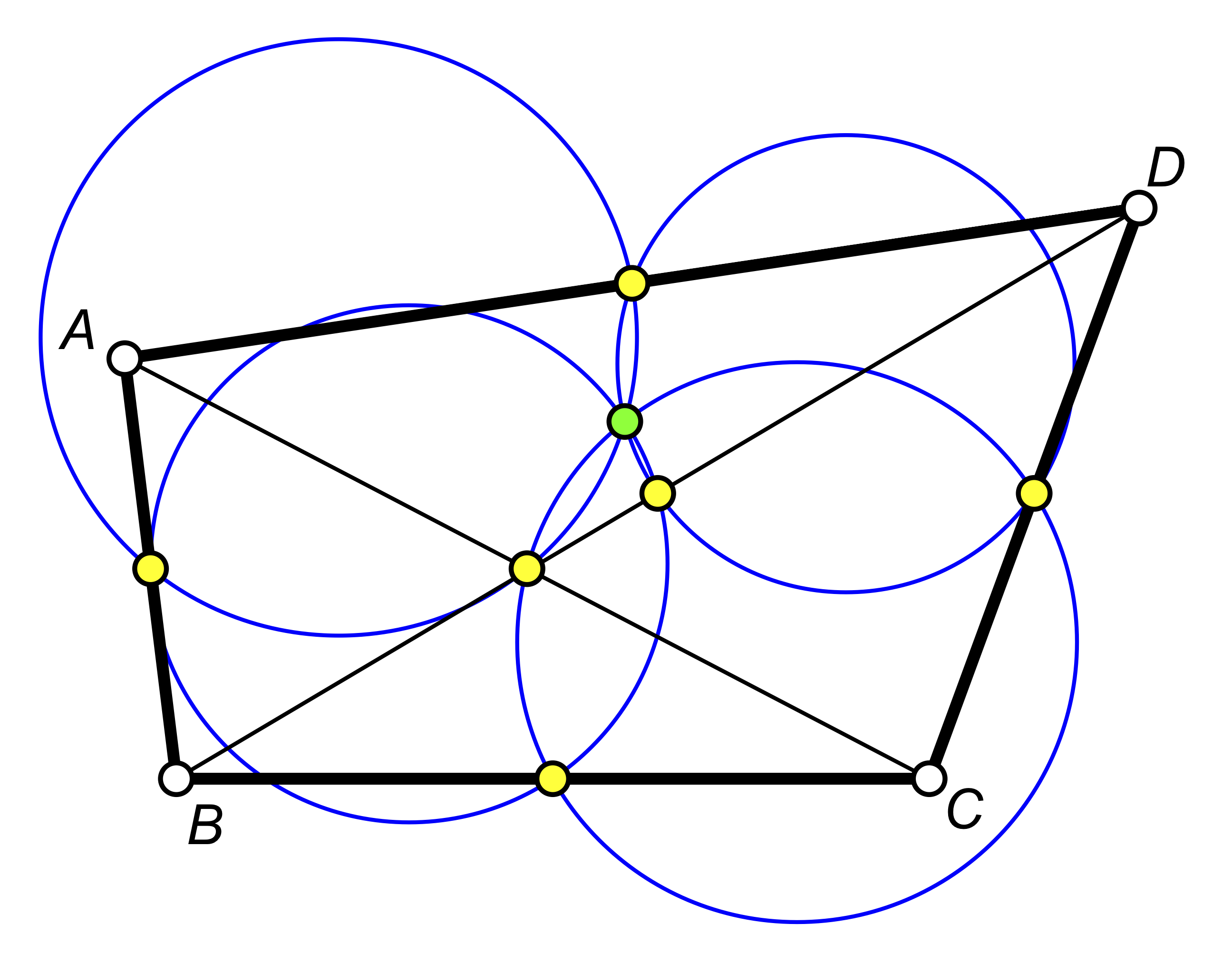}
\caption{Steiner point of quadrilateral $ABCD$}
\label{fig:spSteinerPoint}
\end{figure}

Figure \ref{fig:spSteinerPoint} shows the Steiner point of quadrilateral $ABCD$.
The yellow points represent the midpoints of the sides and diagonals of the quadrilateral.
The blue circles are the midray circles.
The common point of the four circles is the Steiner point (shown in green).

\void{
\begin{lemma}
\label{lemma:rightTriangleNinePoint}
The nine-point circle of a right triangle passes through the vertex of the right angle.
\end{lemma}

\begin{proof}
by Lemma~\ref{lemma:ninePointCircle}, the nine-point circle passes through
the feet of the altitudes. The result then follows since the vertex of a right triangle
is the foot of two altitudes.
\end{proof}
}

\begin{proposition}
\label{prop:spParallelogram}
The Steiner point of a parallelogram coincides with the diagonal point.
\end{proposition}

\begin{proof}
The diagonals of a parallelogram bisect each other.
Every midray circle passes through the midpoint of a diagonal.
Therefore all midray circles pass through the diagonal point of the quadrilateral.
\end{proof}

%\newpage

\begin{proposition}
\label{prop:spCyclic}
The Steiner point of a cyclic quadrilateral coincides with the circumcenter of that quadrilateral.
\end{proposition}

\begin{figure}[h!t]
\centering
\includegraphics[width=0.4\linewidth]{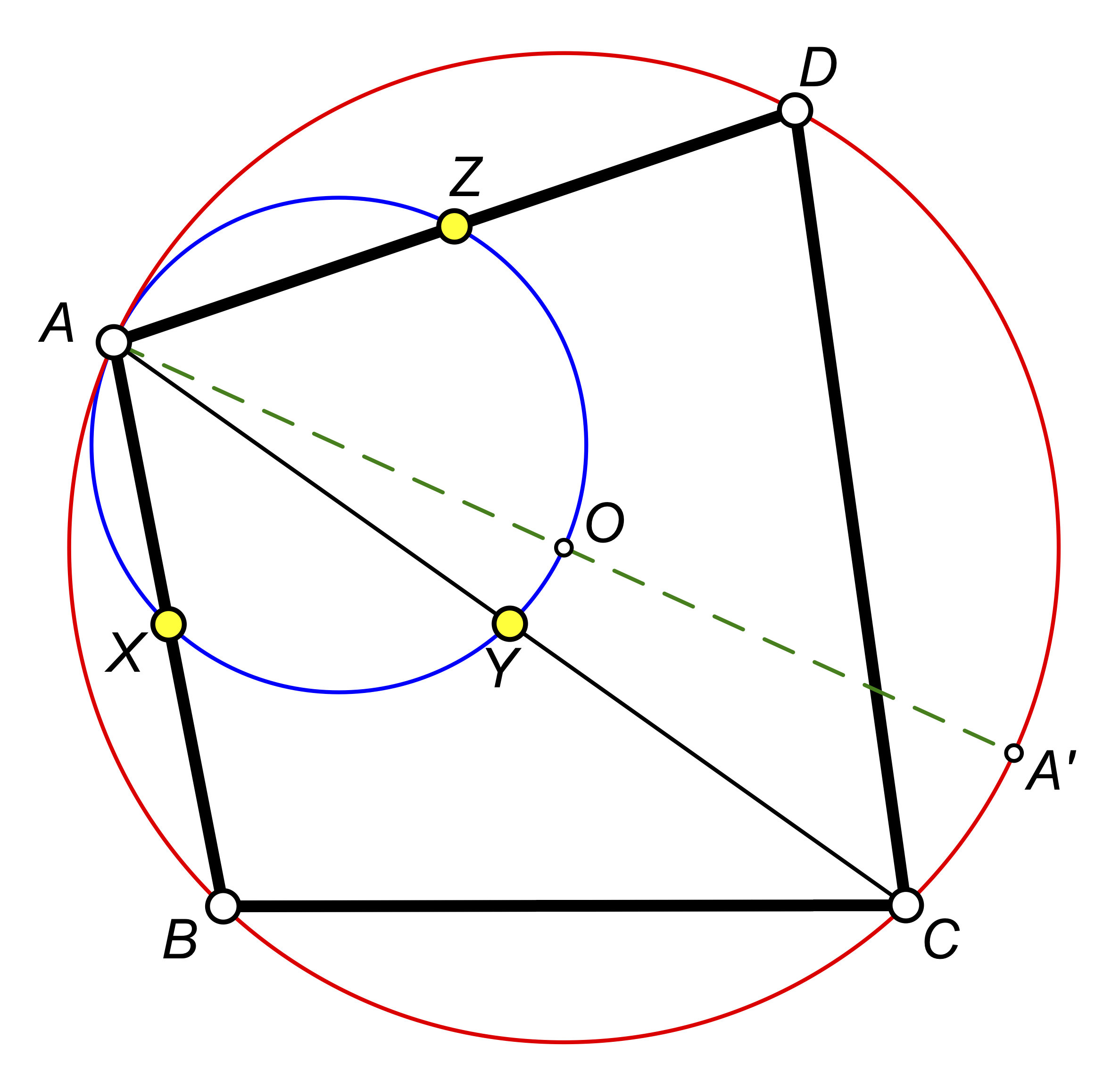}
\caption{Relationship between midray circle and circumcircle}
\label{fig:spCyclic}
\end{figure}

\begin{proof}
Let $ABCD$ be a cyclic quadrilateral with circumcenter $O$.
Let $X$, $Y$, and $Z$ be the midpoints of $AB$, $AC$, and $AD$, respectively.
Then the circle through $X$, $Y$, and $Z$ is a midray circle of quadrilateral $ABCD$.
Let $A'$ be the point on the circumcircle diametrically opposite $A$ (Figure~\ref{fig:spCyclic}).
The homothety with center $A$ and ratio of similarity $\frac12$ maps $B$ into $X$, $C$ into $Y$,
$D$ into $Z$, and $A'$ into $O$.
This homothety therefore maps the circumcircle into the midray circle.
Thus, the circumcenter of the quadrilateral, $O$, lies on the midray circle.
Since this is true for all the midray circles, the point of intersection of the midray circles
(the Steiner point) must be $O$.
\end{proof}

%\newpage

The following result comes from \cite{Suppa9953}.

\begin{lemma}
\label{lemma:spOrtho}
Let $P$ be any point on altitude $AH$ of $\triangle ABC$.
Let $X$ and $X'$ be the midpoints of $AB$ and $AC$, respectively.
Let $\omega$ be the circumcircle of $\triangle XPX'$.
Let $PP'$ be the chord of $\omega$ through $P$ that is parallel to $BC$.
Let $H'$ be the orthogonal projection of $P'$ on $BC$.
Then $H'$ is the midpoint of $BC$.
\end{lemma}

\begin{figure}[h!t]
\centering
\includegraphics[width=0.45\linewidth]{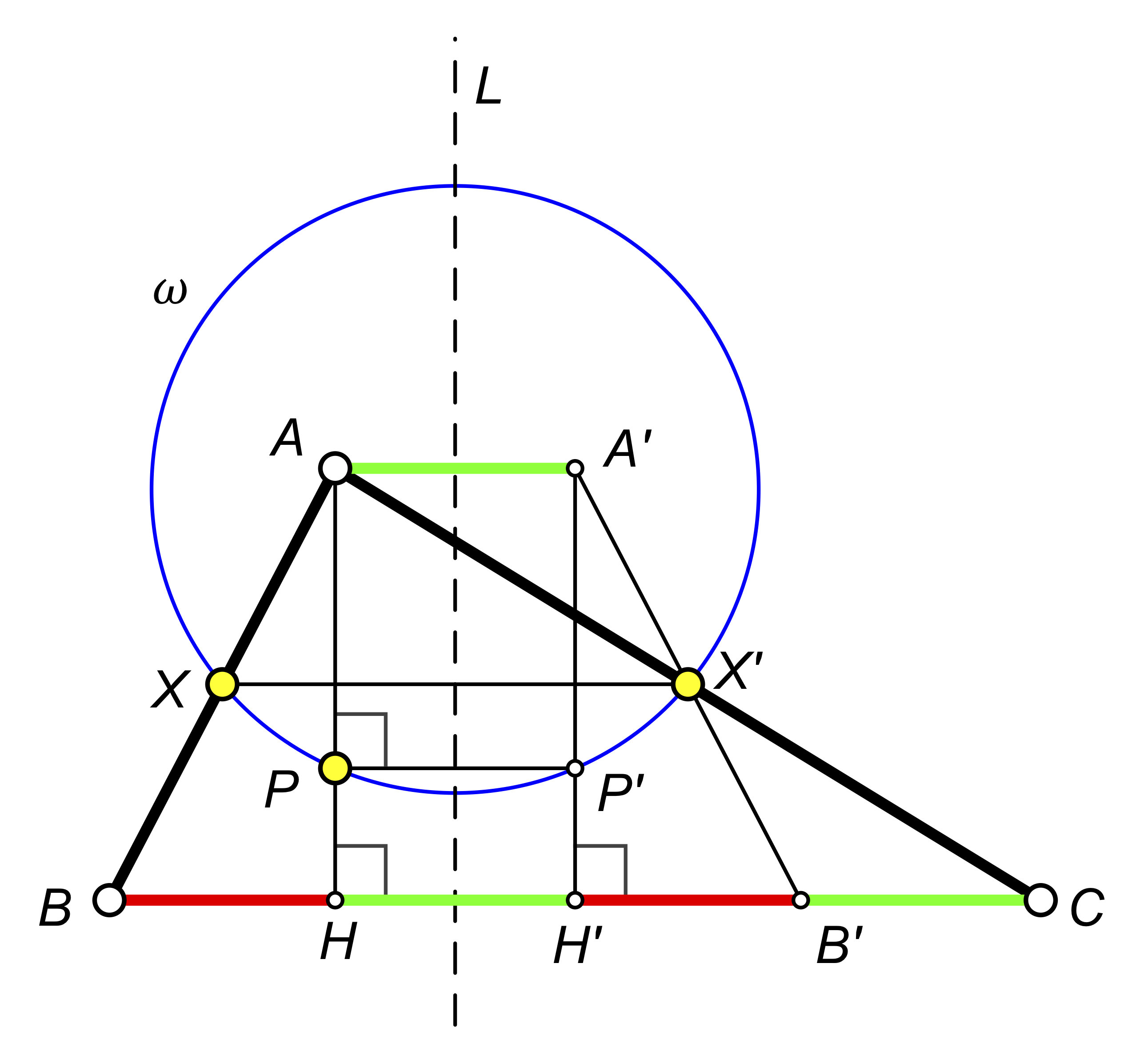}
\caption{}
\label{fig:spOrthoProof}
\end{figure}

\begin{proof}
Let $L$ be the perpendicular bisector of $XX'$.
Let $A'$ be the reflection of $A$ about $L$.
Let $B'$ be the reflection of $B$ about $L$.
Since $X'$ is the reflection of $X$ about $L$, $A'B'$ passes through $X'$ (Figure~\ref{fig:spOrthoProof}).
Then $\triangle AX'A'\cong\triangle CX'B'$. Thus, $AA'=B'C$.
Also, $HH'=AA'$, so $HH'=B'C'$.
By symmetry, $BH=H'B'$.
Hence, $BH'=BH+HH'=HB'+B'C=H'C$ and $H'$ is the midpoint of $BC$.
\end{proof}

\begin{proposition}
\label{prop:spOrtho}
The Steiner point of an orthodiagonal quadrilateral coincides with the point
of intersection of the perpendicular bisectors of the diagonals.
\end{proposition}

\begin{figure}[h!t]
\centering
\includegraphics[width=0.45\linewidth]{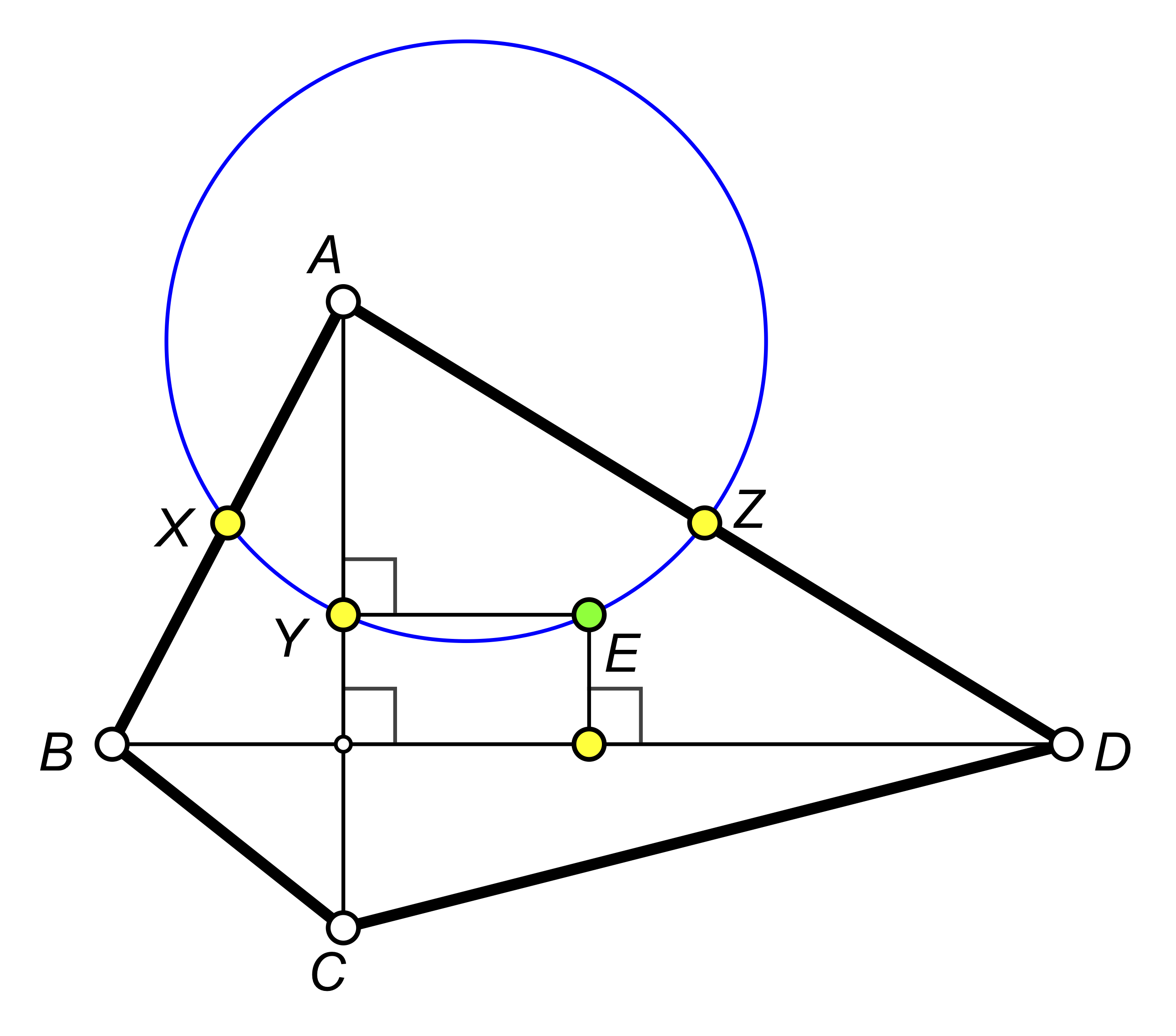}
\caption{Steiner point of an orthodiagonal quadrilateral}
\label{fig:spOrtho}
\end{figure}

\begin{proof}
Let the orthodiagonal quadrilateral be $ABCD$ and let $X$, $Y$, and $Z$
be the midpoints of $AB$, $AC$, and $AD$, respectively.
Let $E$ be the intersection point of the perpendicular bisectors of diagonals $AC$ and $BD$
(Figure~\ref{fig:spOrtho}).
From Lemma~\ref{lemma:spOrtho}, we  can conclude that the midray circle through $X$, $Y$, and $Z$
passes through $E$. Similarly, the other midray circles pass through $E$.
Thus, $E$ is the Steiner point of quadrilateral $ABCD$.
\end{proof}

\textit{Note}. The point
of intersection of the perpendicular bisectors of the diagonals of a quadrilateral
is known as the \textit{quasi circumcenter} (QG-P5 in \cite{EQF}) of the quadrilateral.

When the orthodiagonal quadrilateral is a kite, we get the following result.

\begin{corollary}
\label{prop:spKite}
Let $ABCD$ be a kite in which $BD$ bisects $AC$.
Then the Steiner point of $ABCD$ is the midpoint of $BD$ (Figure~\ref{fig:spKite}).
\end{corollary}

\begin{figure}[h!t]
\centering
\includegraphics[width=0.4\linewidth]{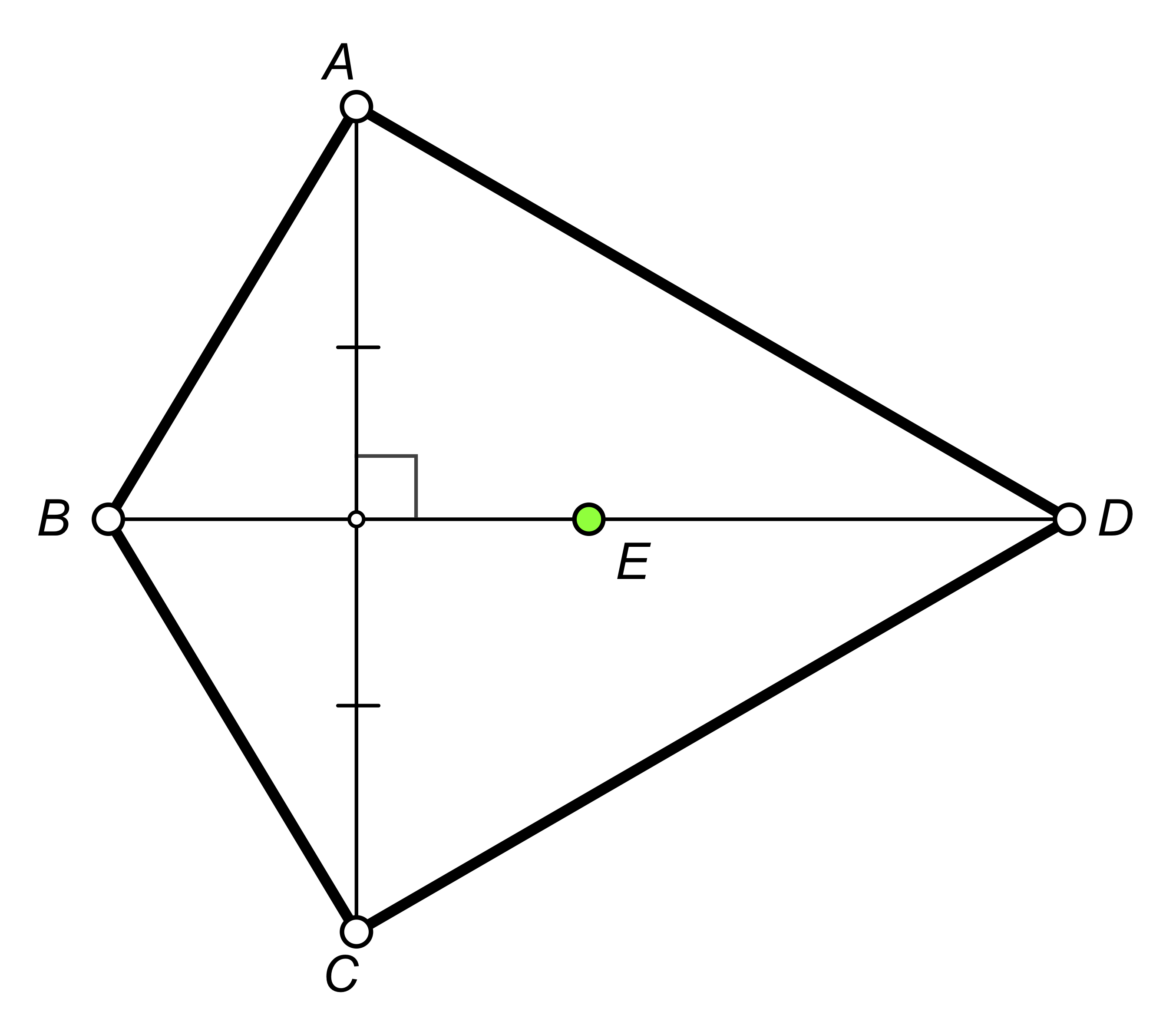}
\caption{Steiner point $E$ of kite is midpoint of $BD$}
\label{fig:spKite}
\end{figure}

Our computer study examined the central quadrilaterals formed by the Steiner point.
Since the Steiner point coincides with the diagonal point of a parallelogram,
we omit results for parallelograms.
Since the Steiner point of a cyclic quadrilateral coincides with the circumcenter of
that quadrilateral, we omit results for cyclic quadrilaterals.
We checked the central quadrilateral for all the first 1000 triangle centers (omitting points
at infinity) and all reference quadrilateral shapes listed in Table~\ref{table:quadrilaterals}.

The results found are listed in Table~\ref{table:Steiner}.

\begin{table}[ht!]
\caption{}
\label{table:Steiner}
\begin{center}
\begin{tabular}{|l|l|p{2.2in}|}
\hline
\multicolumn{3}{|c|}{\textbf{\color{blue}\large \strut Central Quadrilaterals formed by the Steiner Point}}\\ \hline
\textbf{Quadrilateral Type}&\textbf{Relationship}&\textbf{centers}\\ \hline
\void{
\ru cyclic&$[ABCD]=8[FGHI]$&402, 620\\
\cline{2-3}
\ru &$[ABCD]=2[FGHI]$&11, 115, 116, 122--125, 127, 136, 137, 244--247, 338, 339, 865--868\\
\cline{2-3}
\ru &$[ABCD]=\frac32[FGHI]$&616, 617\\
\cline{2-3}
\ru &$[ABCD]=\frac12[FGHI]$&148--150\\
\hline
}
\ru equidiagonal kite&$[ABCD]=8[FGHI]$&642\\
\cline{2-3}
\ru &$[ABCD]=2[FGHI]$&486\\
\hline
\end{tabular}
\end{center}
\end{table}

%\newpage

\relbox{Relationship $[ABCD]=8[FGHI]$}

The following result by Peter Moses comes from \cite{ETC642}.

\begin{lemma}
\label{lemma:x642}
Erect squares inwards on the sides of triangle $\triangle ABC$. The circumcenter of the centers of the squares is the center $X_{642}$ of $\triangle ABC$. 
\end{lemma}

\begin{theorem}
\label{thm-equiKiteX642}
Let $E$ be the Steiner point of equidiagonal kite $ABCD$ with $AB=BC$ and $AD=CD$.
Let $F$, $G$, $H$, and $I$ be the $X_{642}$ points of $\triangle EAB$, $\triangle EBC$, 
$\triangle ECD$, and $\triangle EDA$, respectively (Figure~\ref{fig:spKiteX642}).
Then $FGHI$ is a square homothetic to the Varignon parallelogram of $ABCD$ and
$$[ABCD]=8[FGHI].$$
The two squares have the same diagonal point.
\end{theorem}

\begin{figure}[h!t]
\centering
\includegraphics[width=0.5\linewidth]{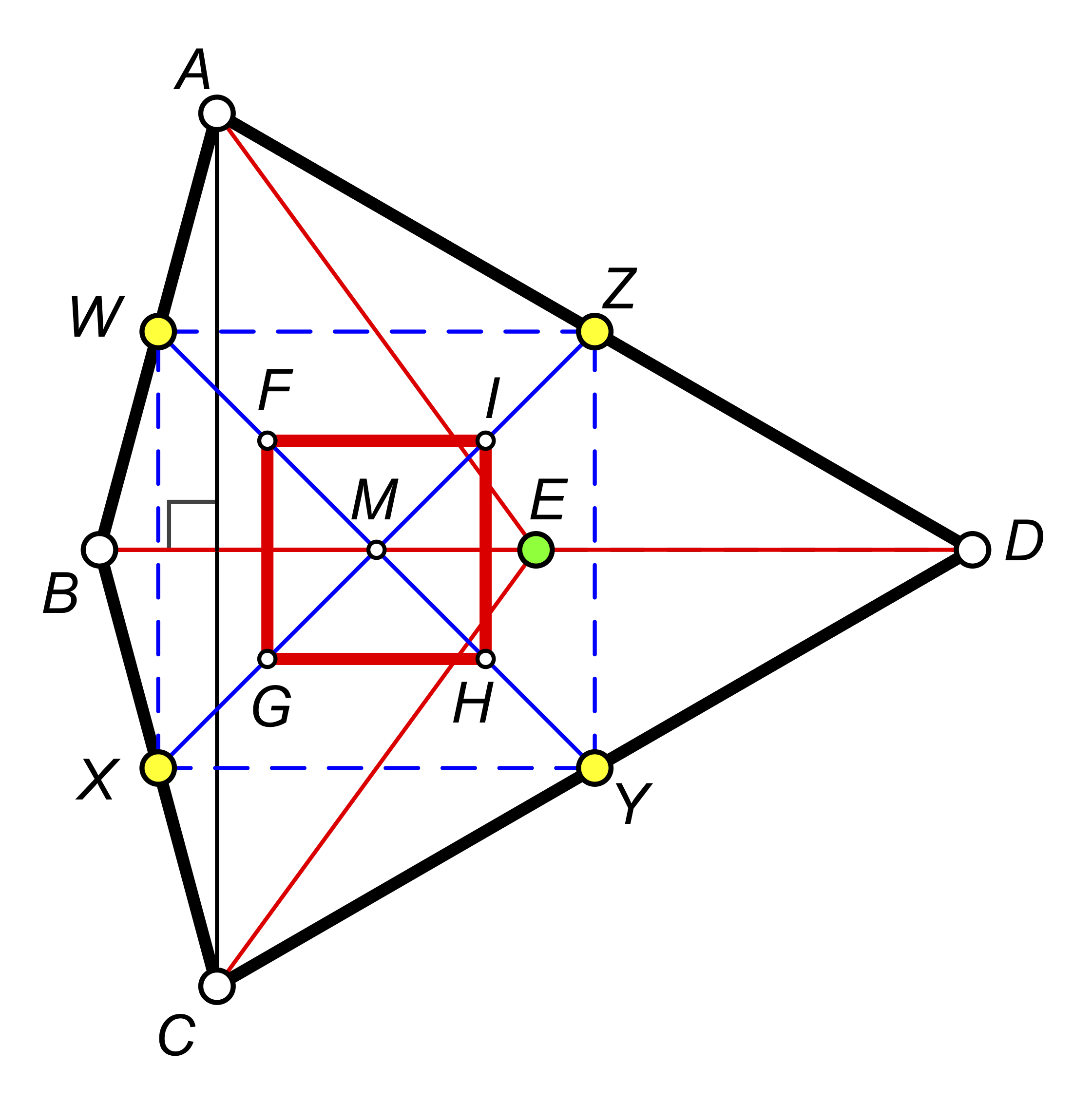}
\caption{Equidiagonal kite, $X_{642}$ points $\implies [ABCD]=8[FGHI]$}
\label{fig:spKiteX642}
\end{figure}

%\newpage

\begin{proof}
We use Cartesian coordinates with the origin at point $E$, with $x$-axis the line $BD$.
Without loss of generality, assume that $AC=BD=2$ and $AB<AD$.
Since $ABCD$ is a kite, $BD\perp AC$ and $BD$ bisects $AC$.
By Corollary~\ref{prop:spKite}, we have $B=(-1,0)$ and $D=(1,0)$.
Let $K$ be the point of intersection of $AC$ and $BD$.
Since $ABCD$ is equidiagonal, $AK=CK=1$ and we can let
$K=(-u,0)$, $A=(-u,1)$, and $C=(-u,-1)$, with $u>0$
as shown in Figure~{\ref{fig:spKiteCoords}}.

\begin{figure}[h!t]
\centering
\includegraphics[width=0.65\linewidth]{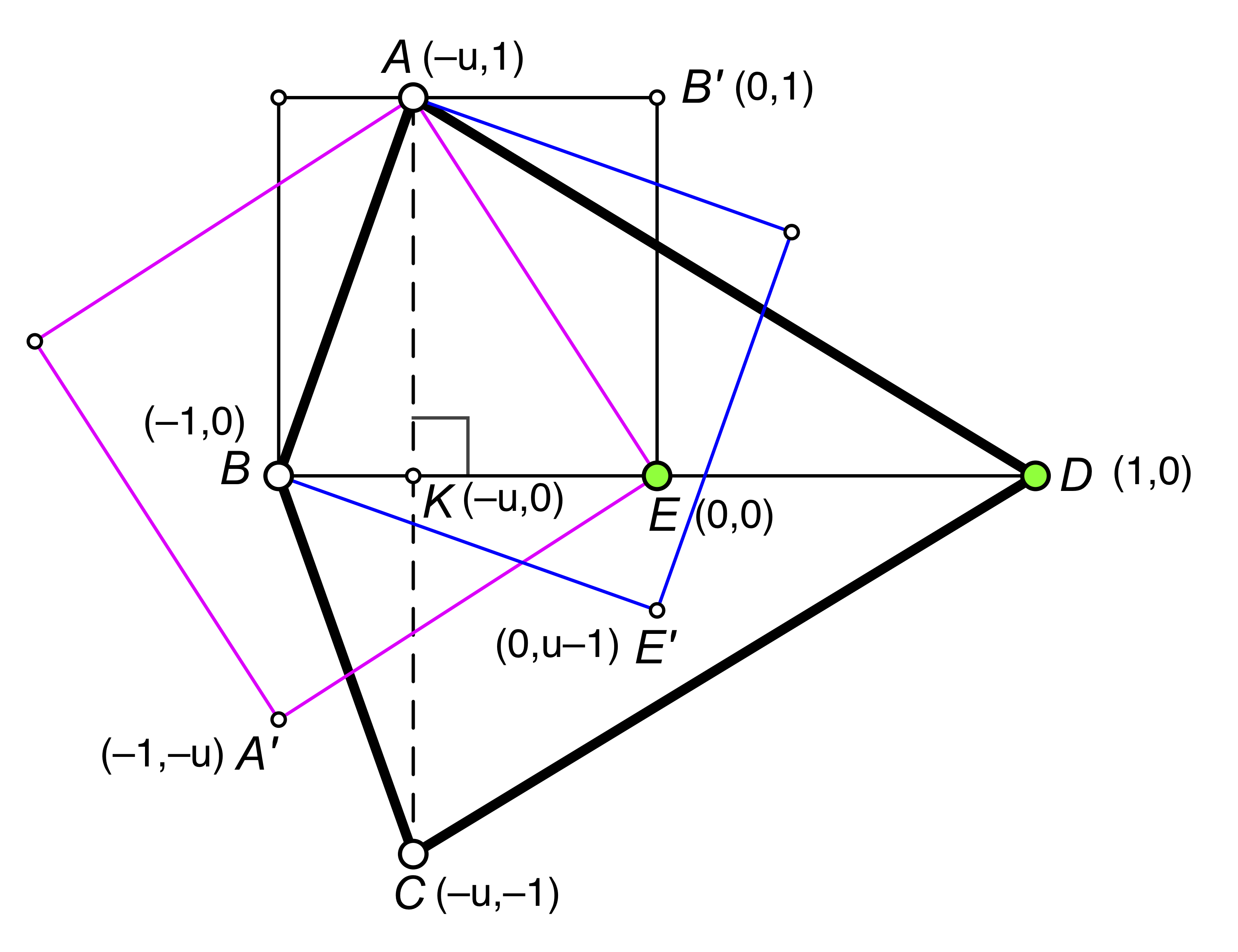}
\caption{Coordinate set-up for an equidiagonal kite}
\label{fig:spKiteCoords}
\end{figure}

%\newpage

In order to get the coordinates of point $F$, we use Lemma~\ref{lemma:x642}.
See Figure~\ref{fig:spKiteCenters}.

\begin{figure}[h!t]
\centering
\includegraphics[width=0.5\linewidth]{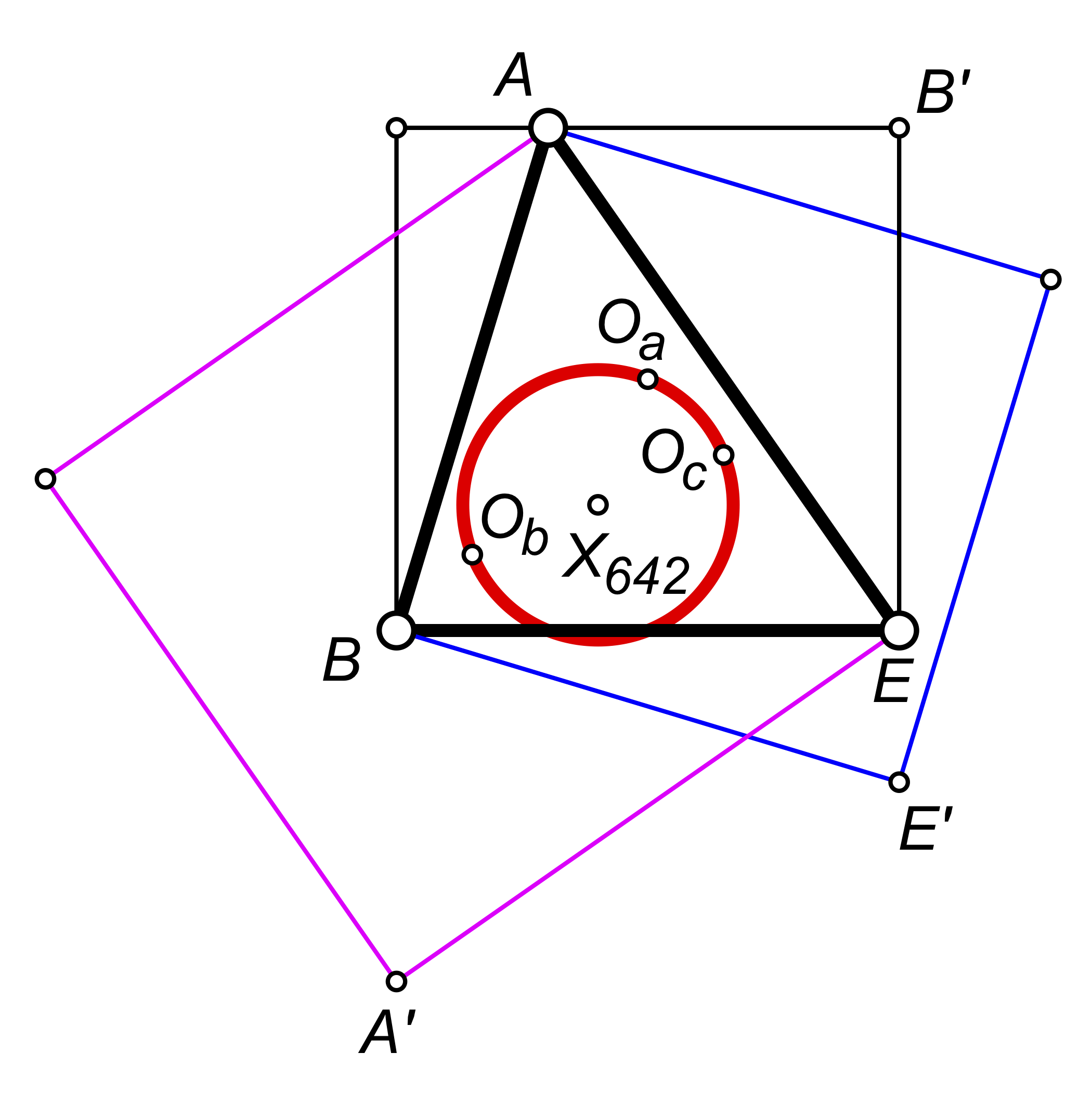}
\caption{}
\label{fig:spKiteCenters}
\end{figure}

The center of the square constructed inwards on the side $BE$ is the midpoint of the segment $BB'$ where $B'=(0,1)$.
Therefore, $O_a=\left(-\frac{1}{2},\frac{1}{2}\right)$.

The center of the square constructed inwards on the side $AE$ is the midpoint of the segment $AA'$ where $A'=(-1,-u)$ since $\triangle EBA'\cong\triangle EB'A$.
Therefore, $O_b=\left(\frac{-u-1}{2},\frac{-u+1}{2}\right)$.

The center of the square constructed inwards on the side $AB$ is the midpoint of the segment $AE'$ where $E'=(0,u-1)$ since $\triangle BEE'\cong\triangle ABK$, so $EE'=BK$.
Therefore, $O_c=\left(-\frac{u}{2},\frac{u}{2}\right)$.

The point $F$ is the circumcenter of $\triangle O_aO_bO_e$.
It coincides with the point of intersection of the perpendicular bisectors of $O_aO_b$ and $O_aO_c$.
The midpoint of $O_aO_b$ is easily found, as well as the slope of $O_aO_b$.
The slope of the perpendicular bisector of $O_aO_b$ is the negative reciprocal of the slope of $O_aO_b$.
Using the Point Slope Formula, we find that the equation of the perpendicular
bisector of $O_aO_b$ is $2x+2y+u=0$. Similarly, the equation of the perpendicular
bisector of $O_aO_c$ is $2x-2y+u=0$.

Solving these two equations gives us the coordinates for point $F$.
The point $G$ is the reflection of $F$ with respect to $BD$.
The coordinates of $F$ and $G$ are therefore
\begin{equation*}
  F=\left(\frac{-2u-1}{4},\frac{1}{4}\right),\qquad  G=\left(\frac{-2u-1}{4},-\frac{1}{4}\right).
\end{equation*}

In the same manner, we find that the coordinates of $H$ and $I$ are
\begin{equation*}
  H=\left(\frac{-2u+1}{4},-\frac{1}{4}\right),\qquad  I=\left(\frac{-2u+1}{4},\frac{1}{4}\right).
\end{equation*}

Therefore, $FG=FI=\frac{1}{2}$, so $FGHI$ is a square (notice that $FGHI$ is a parallelogram by construction).
The center of $FGHI$ is the midpoint $M=\left(-\frac{u}{2},0\right)$ of $FH$.
The coordinates for the midpoints of the sides $AB$, $BC$, $CD$, and $DA$ are
$$W=\left(\frac{-u-1}{2},\frac{1}{2}\right),\qquad  X=\left(\frac{-u-1}{2},-\frac{1}{2}\right),$$
$$Y=\left(\frac{-u+1}{2},-\frac{1}{2}\right),\qquad  Z=\left(\frac{-u+1}{2},\frac{1}{2}\right).$$
The center of $WXYZ$ is $\left(-\frac{u}{2},0\right)$.
The square $FGHI$ and the Varignon parallelogram $WXYZ$ of $ABCD$ have the same center and parallel sides, so they are homothetic. The ratio of similarity is $k=\frac{1}{2}$.

Finally, since $[ABCD]=2$ and $[FGHI]=\frac{1}{4}$, we have $[ABCD]=8[FGHI]$.
\end{proof}

\relbox{Relationship $[ABCD]=2[FGHI]$}

\begin{lemma}
\label{lemma:spSquareVecten}
Let $P$ be any point on side $AD$ of square $ABCD$.
Then the inner Vecten point ($X_{486}$ point) of $\triangle PBC$
coincides with the diagonal point of the square (Figure~\ref{fig:spSquareVecten}).
\end{lemma}

\begin{figure}[h!t]
\centering
\includegraphics[width=0.3\linewidth]{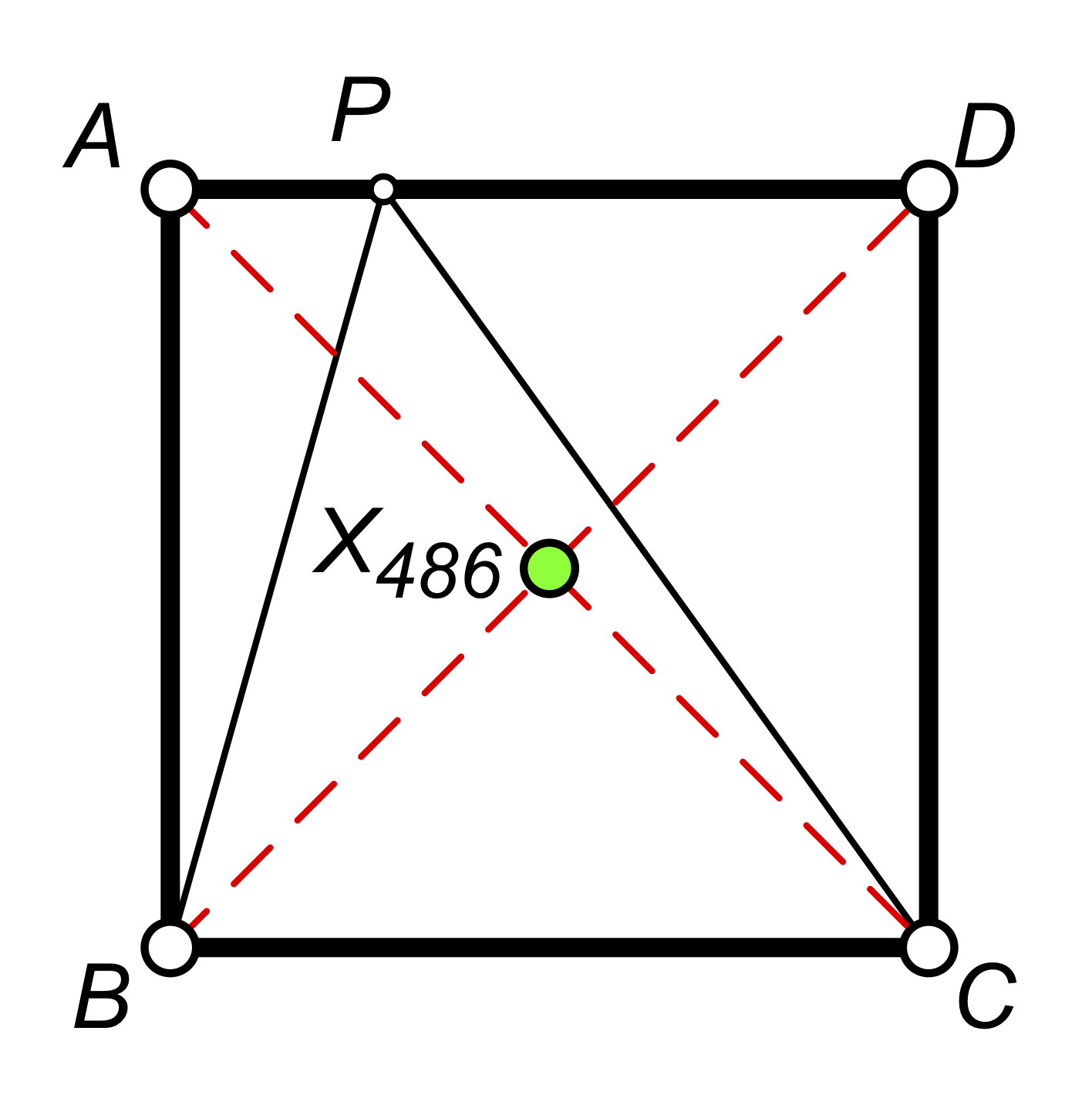}
\caption{Vecten point of $\triangle PBC$ in square $ABCD$}
\label{fig:spSquareVecten}
\end{figure}

\begin{proof}
From \cite{ETC486} we know that the inner Vecten point is the intersection of $PO_a$ and $CO_c$
where $O_a$ is the center of the square erected internally on side $BC$ of $\triangle PBC$ and
$O_c$ is the center of the square erected internally on side $BP$ of $\triangle PBC$
(Figure~\ref{fig:spSquareVectenProof}).

\begin{figure}[h!t]
\centering
\includegraphics[width=0.3\linewidth]{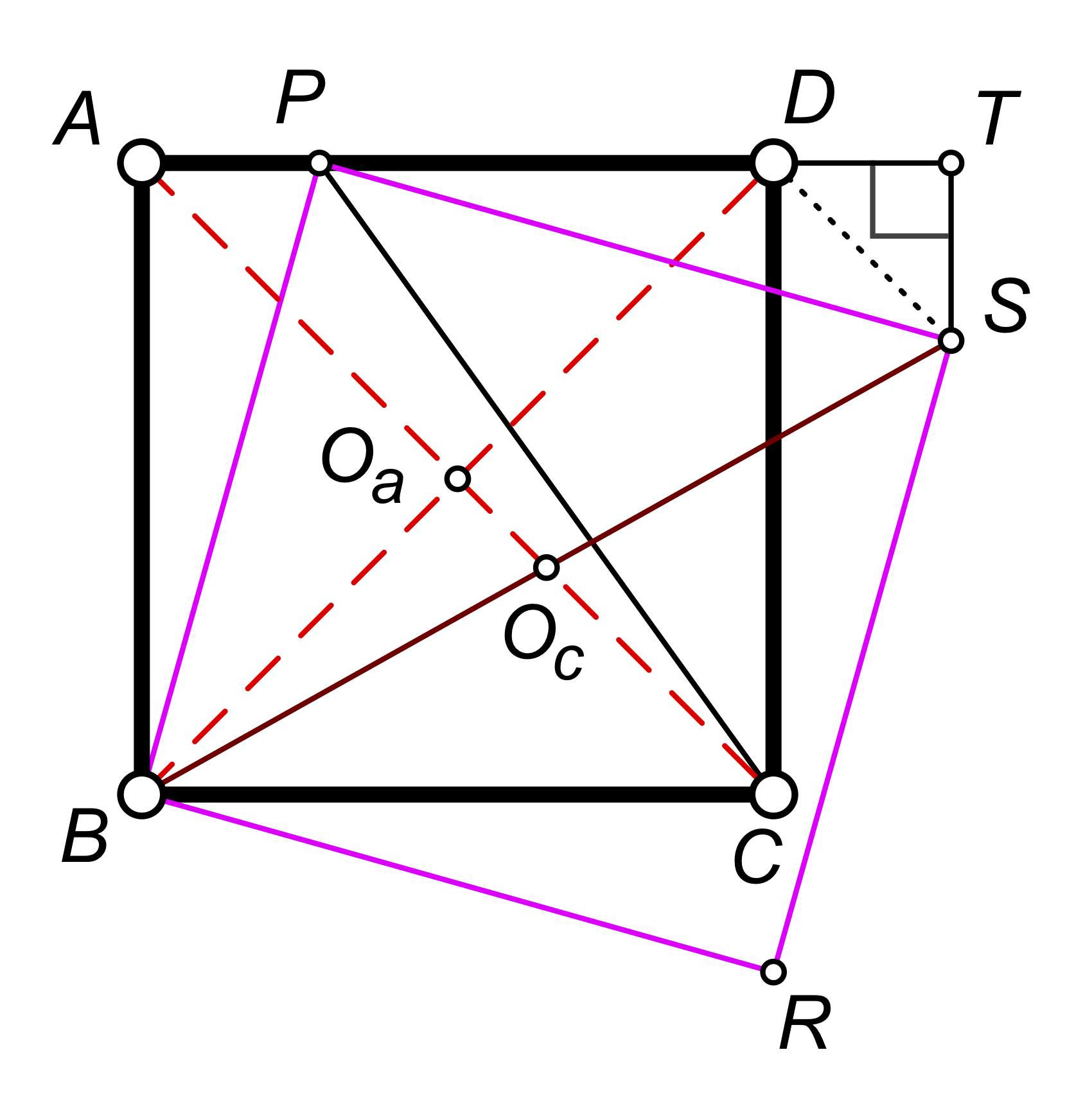}
\caption{Squares erected on sides $PB$ and $BC$}
\label{fig:spSquareVectenProof}
\end{figure}

The line $PO_a$ clearly passes through $O_a$, so we need only show that $CO_c$ also passes through $O_a$.
Let $PBRS$ be the square erected internally on side $BP$ of $\triangle PBC$.
Drop a perpendicular from $S$ to $AD$ meeting it at $T$.
Right triangles $BAP$ and $PTS$ have equal hypotenuses.
Angles $\angle PBA$ and $\angle SPT$ are equal because they are both complementary to $\angle APB$.
Thus, $\triangle BAP\cong\triangle PTS$. Hence $PT=AB$ and $TS=AP$.
Since $AB=AD$, this implies that $DT=PT-PD=AB-PD=AD-PD=AP$.
Because $DT=TS$, $\angle SDT=45\degrees$.
But $\angle CAT=45\degrees$, so $DS\parallel AC$.
In $\triangle BDS$, $O_a$ is the midpoint of $BD$ and $O_c$ is the midpoint of $BS$,
so $O_aO_c\parallel DS$.
Thus, $O_aO_c$ coincides with $AC$ and hence $CO_c$ passes through $O_a$.
\end{proof}

%\newpage

\begin{theorem}
\label{thm-equiKiteX486}
Let $E$ be the Steiner point of equidiagonal kite $ABCD$ with $AB=BC$ and $AD=CD$.
Let $F$, $G$, $H$, and $I$ be the inner Vecten points ($X_{486}$ points) of $\triangle EAB$, $\triangle EBC$, 
$\triangle ECD$,  and $\triangle EDA$, respectively (Figure~\ref{fig:spKiteX486}).
Then $FGHI$ is a square congruent to the Varignon parallelogram of $ABCD$ and
$$[ABCD]=2[FGHI].$$
\end{theorem}

\begin{figure}[h!t]
\centering
\includegraphics[width=0.3\linewidth]{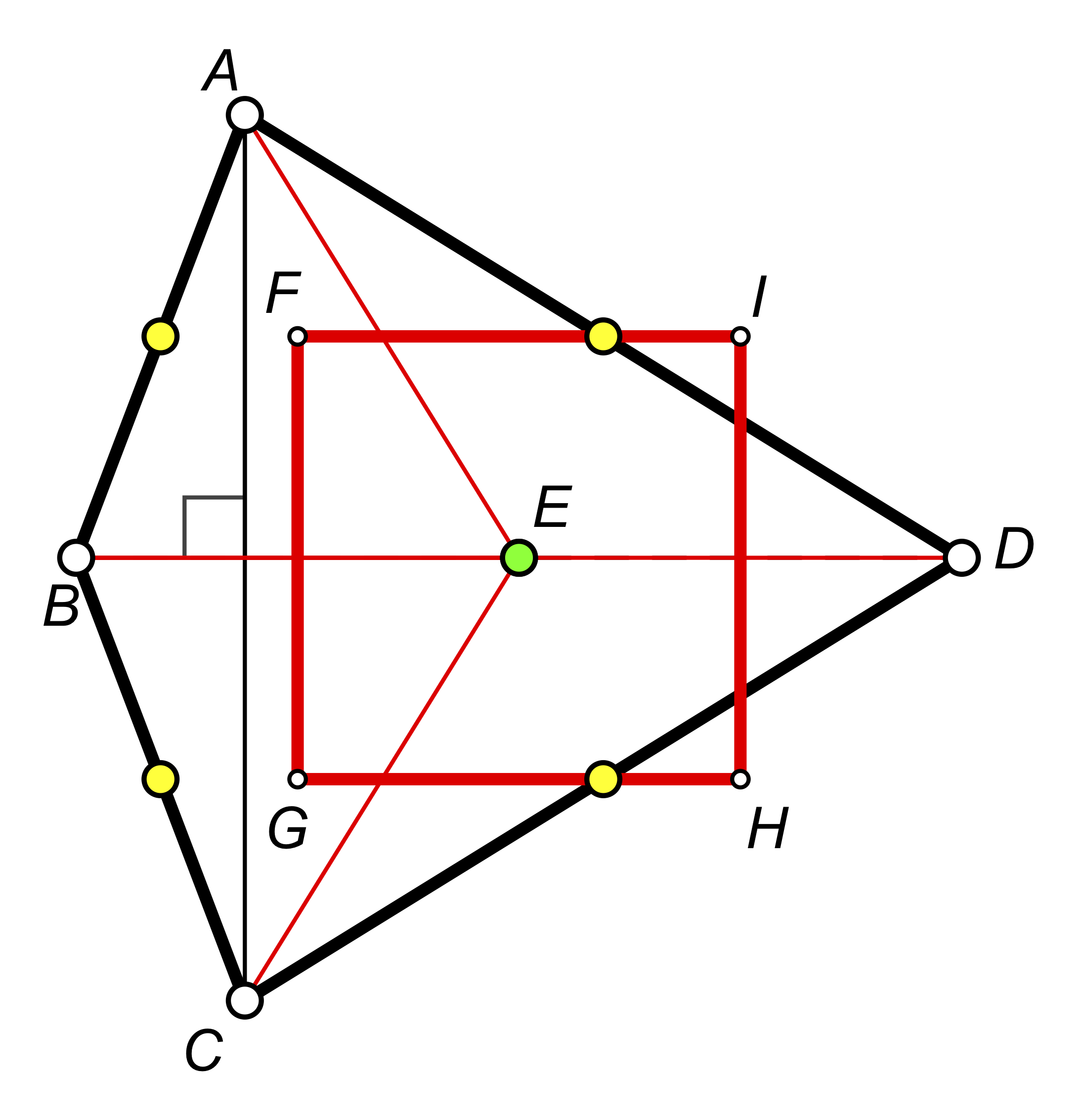}
\caption{Equidiagonal kite, X486 points $\implies [ABCD]=2[FGHI]$}
\label{fig:spKiteX486}
\end{figure}

\begin{proof}
Since $ABCD$ is a kite, $BD\perp AC$ and $BD$ and $BD$ bisects $AD$.
By Corollary~\ref{prop:spKite}, $E$ is the midpoint of $BD$.
Let $K$ be the point of intersection of $AC$ and $BD$.
Erect perpendiculars to $BD$ at $B$, $E$ and $D$.
Erect perpendiculars to $AC$ at $A$, and $C$.
The points of intersection of these perpendiculars are $P$, $Q$, $R$, $S$, $T$, and $U$
as shown in Figure~\ref{fig:spKiteX486Proof}.

\begin{figure}[h!t]
\centering
\includegraphics[width=0.3\linewidth]{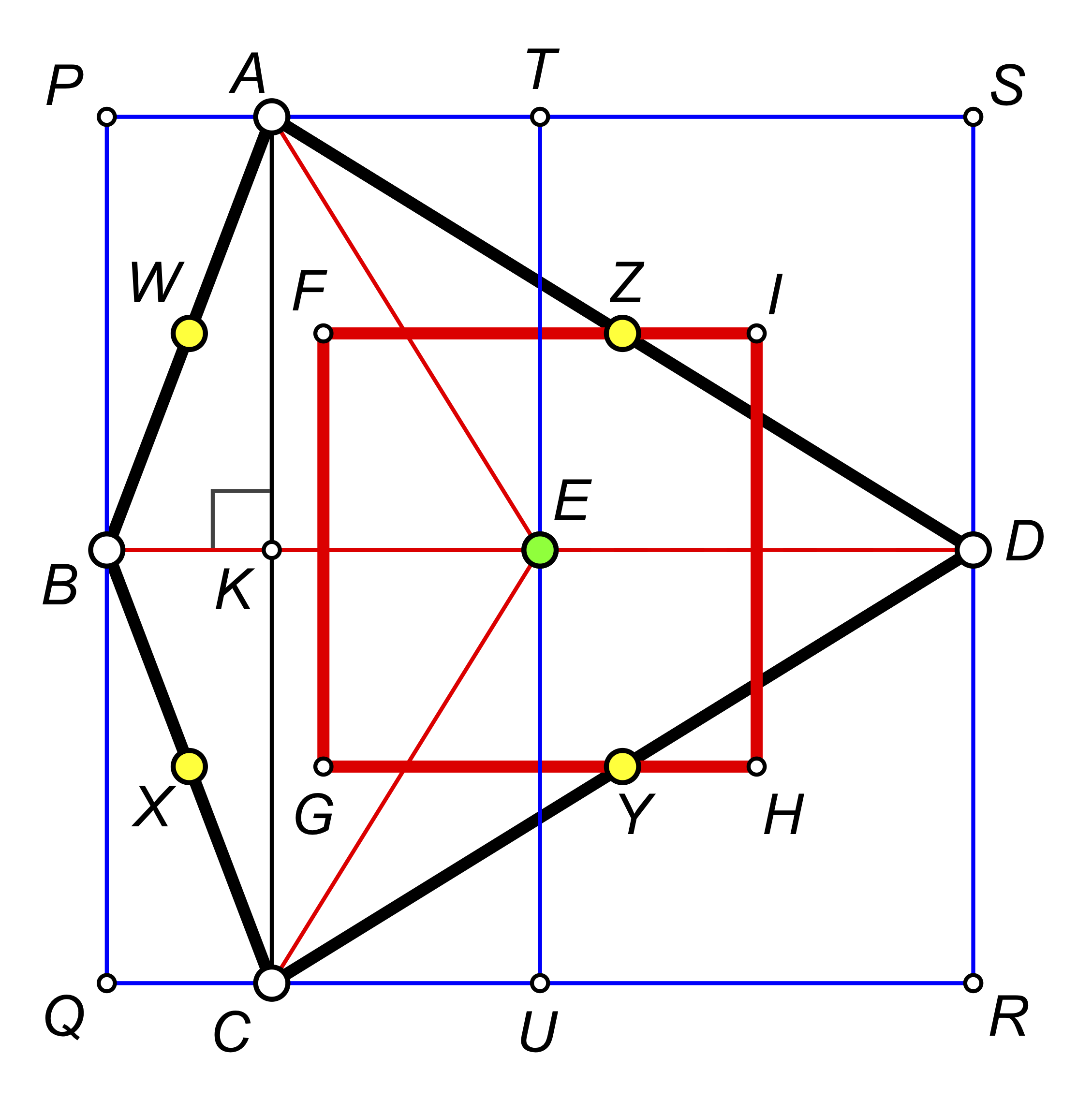}
\caption{}
\label{fig:spKiteX486Proof}
\end{figure}

Since $ABCD$ is equidiagonal, $AC=BD$ or $AK=KC$.
This implies that quadrilaterals $PBET$, $TEDS$, $BQUE$, and $EURD$ are all squares.

By Lemma~\ref{lemma:spSquareVecten}, $F$, $G$, $H$, and $I$ are the centers of these squares,
from which it follows that $FGHI$ is a square congruent to each of these squares.
Square $FGHI$ has center at $E$ and side of length equal to $BE$.
If $W$, $X$, $Y$, and $Z$ are the midpoints of the sides of $ABCD$, then $WXYZ$ is the Varignon parallelogram
of $ABCD$. Since $WZ=\frac12BD=BE$, square $FGHI\cong WXYZ$.

Finally, since $[ABCD]=2[WXYZ]$ and $[FGHI]=[WXYZ]$, we have $[ABCD]=2[FGHI]$.
\end{proof}

\void{
\textbf{Note.} The lengths of the segments in the equidiagonal kite $ABCD$ are related as follows.
If we let $BC=a$, $AC=b$, and $AB=c$, then $c=a$. If we let $BP=x$, $AE=y$, and $AD=z$, then
$$
\begin{aligned}
x&=\sqrt{a^2-b^2/4}\\
AP=CP=BE=DE&=b/2\\
y&=\sqrt{b^2/4+(b/2-x)^2}\\
z&=\sqrt{b^2/4+(b-x)^2}\\
PD&=b-x\\
[AEB]=[BEC]=[CED]=[DEA]&=\frac12(AP\times BE)=\frac{b^2}{8}\\
[ABCD]&=b^2/2
\end{aligned}
$$
as shown in Figure~\ref{fig:spKiteLengths}.

\begin{figure}[h!t]
\centering
\includegraphics[width=0.4\linewidth]{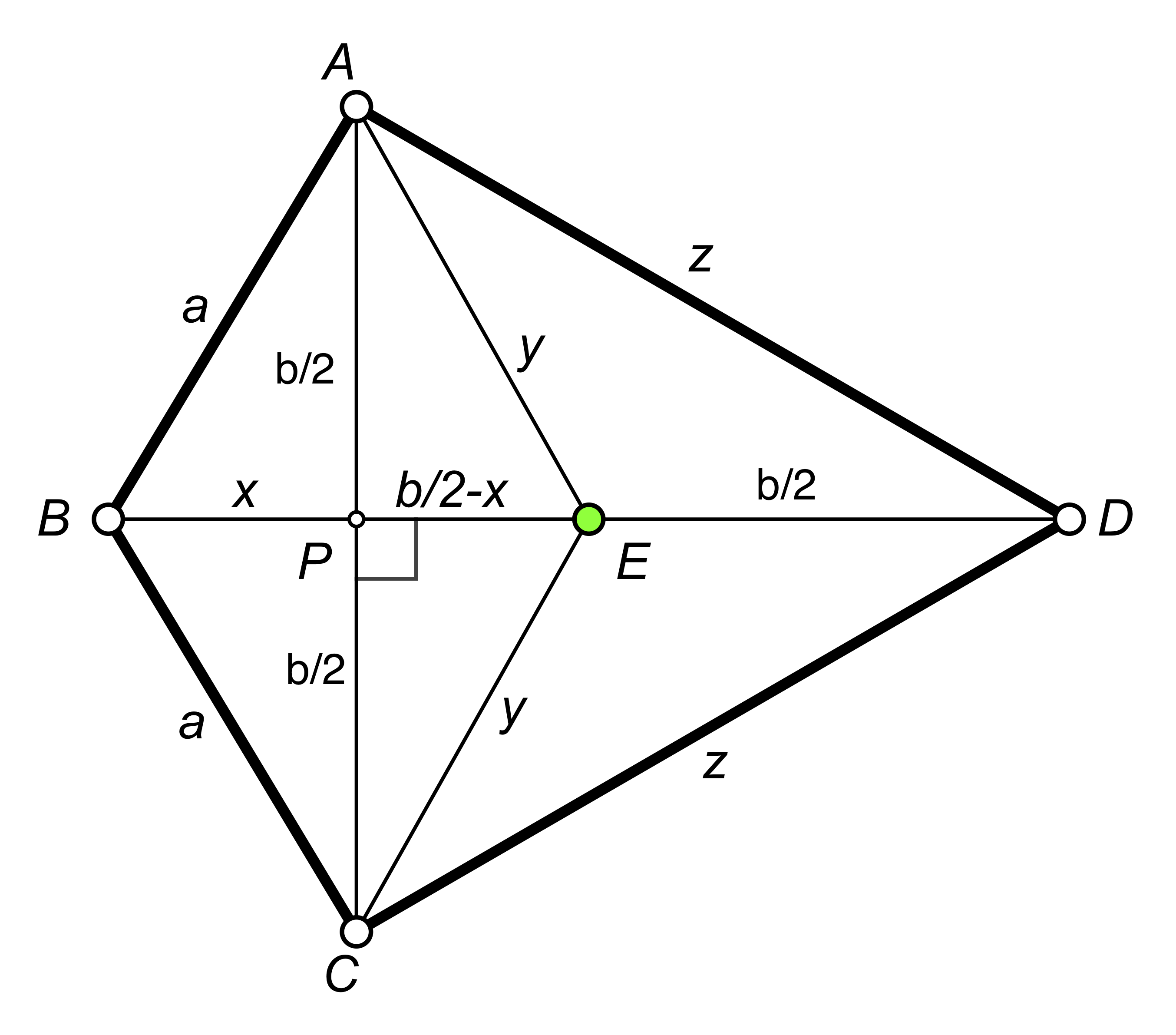}
\caption{Lengths in Equidiagonal kite with Steiner point $E$}
\label{fig:spKiteLengths}
\end{figure}
}

\newpage

%**************************************
%    Centroid
%**************************************

\section{Centroid}

In this section, we examine central quadrilaterals formed from the centroid
of the reference quadrilateral.

A \textit{bimedian} of a quadrilateral is the line segment
joining the midpoints of two opposite sides (Figure~\ref{fig:gpBimedian}).

\begin{figure}[h!t]
\centering
\includegraphics[width=0.4\linewidth]{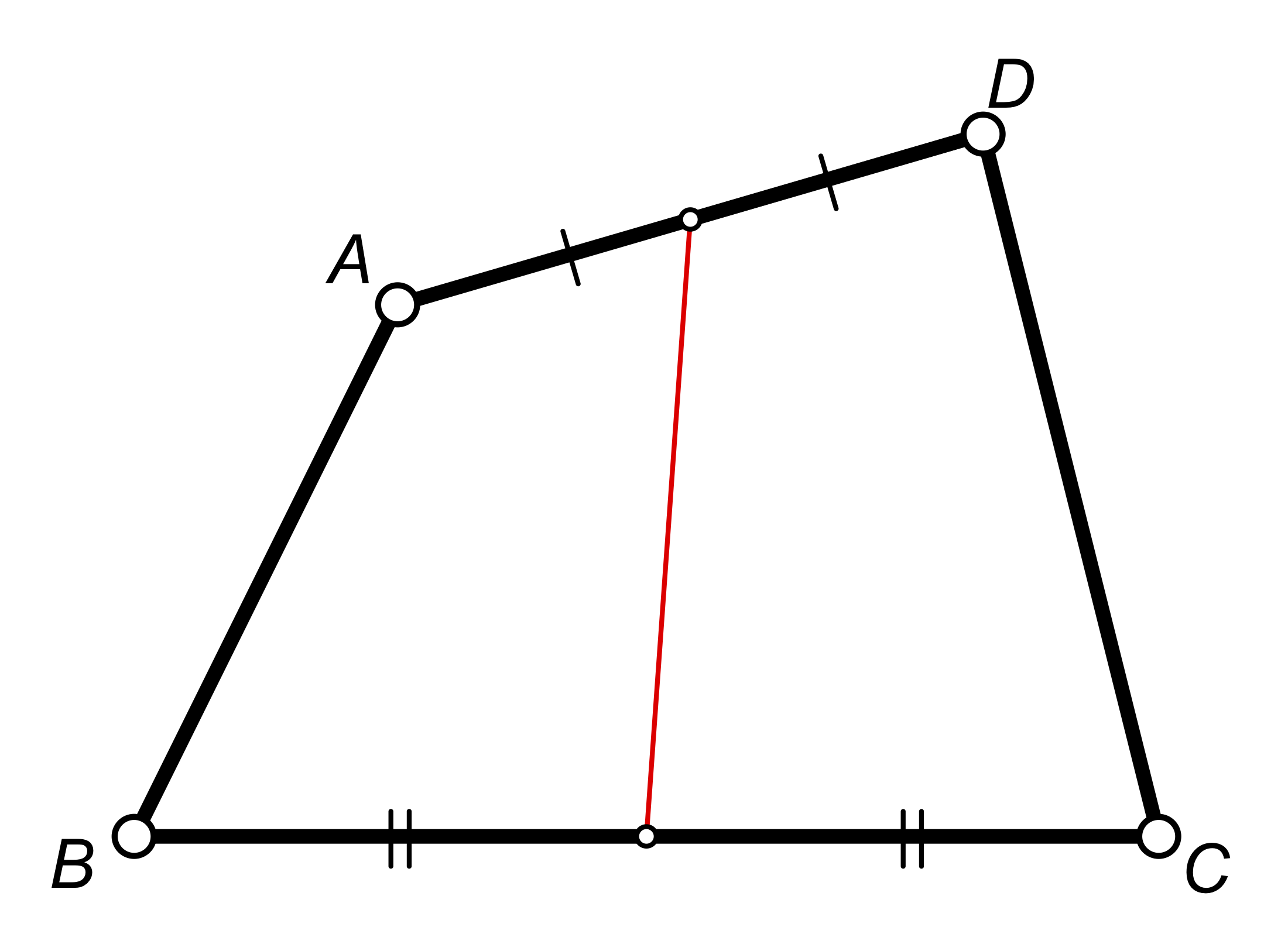}
\caption{Bimedian of a quadrilateral}
\label{fig:gpBimedian}
\end{figure}

The \textit{centroid} (or vertex centroid) of a quadrilateral is the point of intersection
of the bimedians (Figure~\ref{fig:gpCentroid}). The centroid bisects each bimedian.

\begin{figure}[h!t]
\centering
\includegraphics[width=0.4\linewidth]{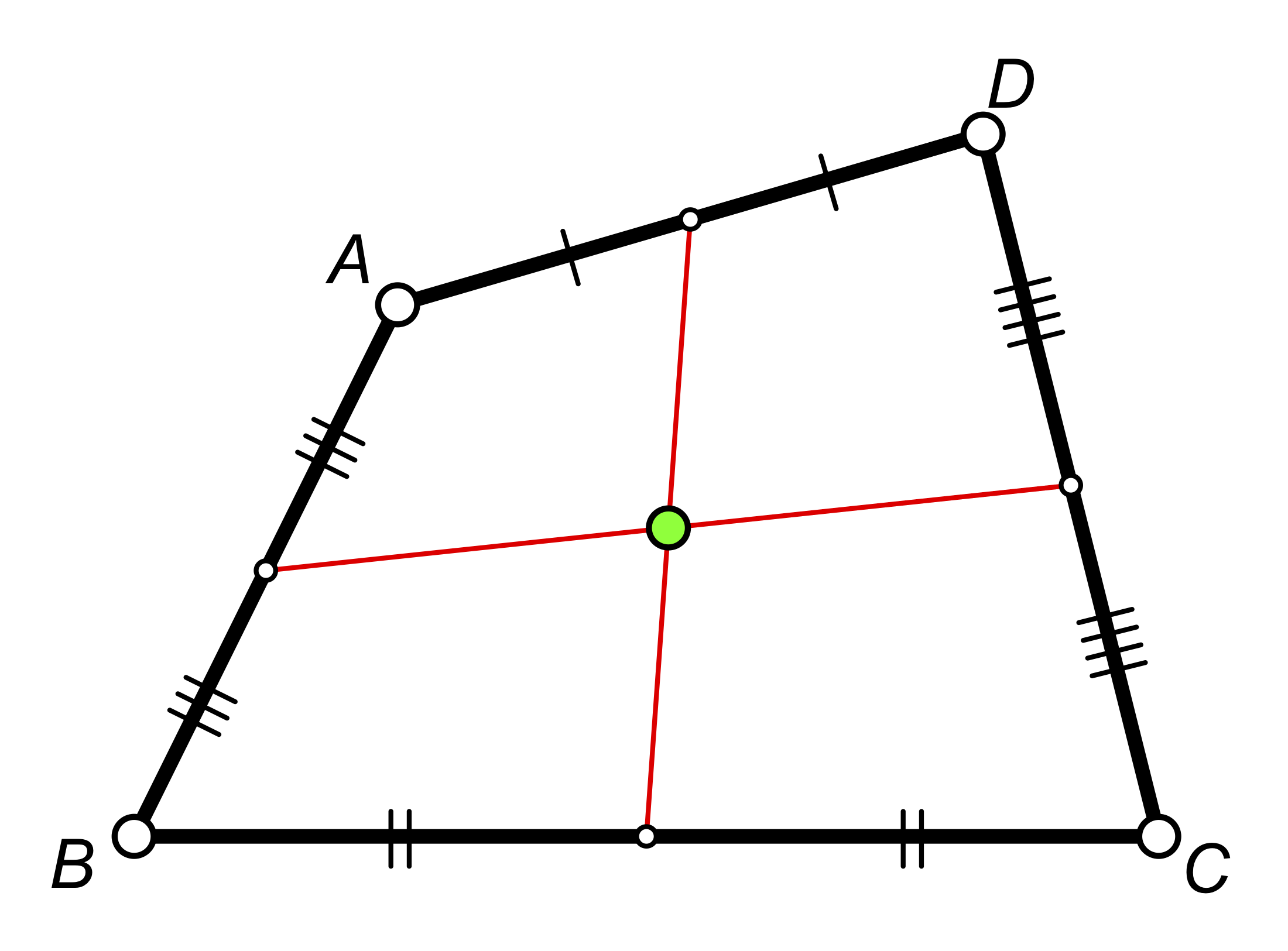}
\caption{Centroid of a quadrilateral}
\label{fig:gpCentroid}
\end{figure}
 
Our computer study examined the central quadrilaterals formed by the centroid.
Since it is easy to prove that the centroid coincides with the diagonal point of a parallelogram,
we omit results for parallelograms.
We checked the central quadrilateral for all the first 1000 triangle centers (omitting points
at infinity) and all reference quadrilateral shapes listed in Table~\ref{table:quadrilaterals}.

The results found are listed in the following table.
\medskip

\begin{table}[ht!]
\caption{}
\begin{center}
\begin{tabular}{|l|l|p{2.2in}|}
\hline
\multicolumn{3}{|c|}{\textbf{\color{blue}\large \strut Central Quadrilaterals formed by the Centroid}}\\ \hline
\textbf{Quadrilateral Type}&\textbf{Relationship}&\textbf{centers}\\ \hline
\ru bicentric trapezoid&$[ABCD]=8[FGHI]$&402\\
\cline{2-3}
\ru &$[ABCD]=2[FGHI]$&122, 123, 127, 339\\
\cline{2-3}
\ru &$[ABCD]=\frac12[FGHI]$&74, 477\\
\hline
\end{tabular}
\end{center}
\end{table}

\begin{lemma}
\label{lemma:cyclicTrapIsosceles}
A cyclic trapezoid is isosceles.
\end{lemma}

\begin{figure}[h!t]
\centering
\includegraphics[width=0.4\linewidth]{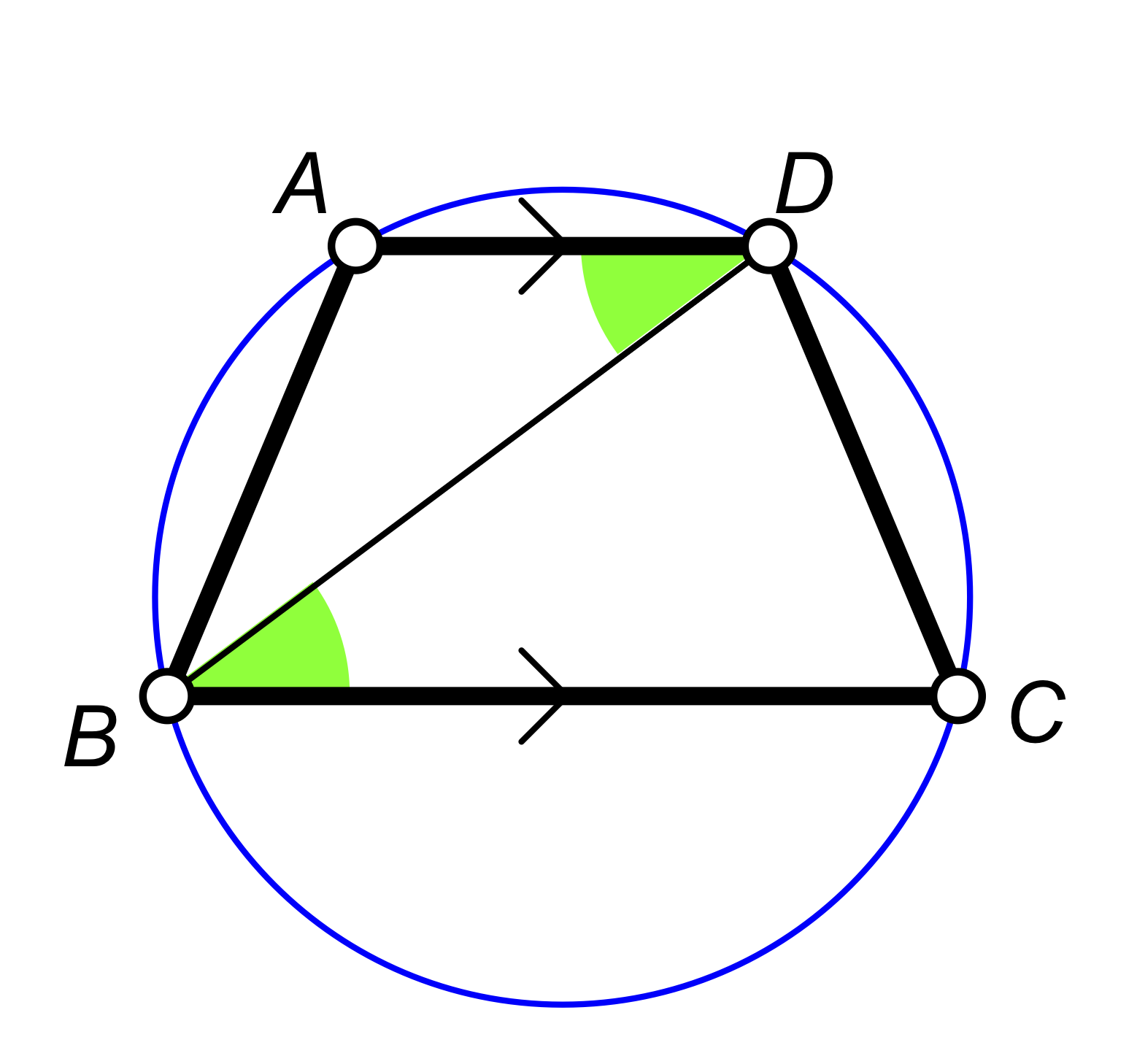}
\caption{$AB=CD$}
\label{fig:gpCyclicTrap}
\end{figure}

\begin{proof}
Let $ABCD$ be a cyclic trapezoid with $AD\parallel BC$ (Figure~\ref{fig:gpCyclicTrap}).
Since $AD\parallel BC$, we must have $\angle ADB=\angle DBC$.
Arcs intercepted by equal inscribed angles have the same measure, so minor arcs $AB$ and $CD$ are congruent.
Equal arcs have equal chords, so $AB=CD$.
\end{proof}

\begin{lemma}
\label{lemma:biTrapPerp}
Let $ABCD$ be a tangential trapezoid with $AD\parallel BC$.
Let $E$ be the incenter of $ABCD$.
Then $AE\perp BE$ (Figure~\ref{fig:gpTangentialTrap}).
\end{lemma}

\begin{figure}[h!t]
\centering
\includegraphics[width=0.35\linewidth]{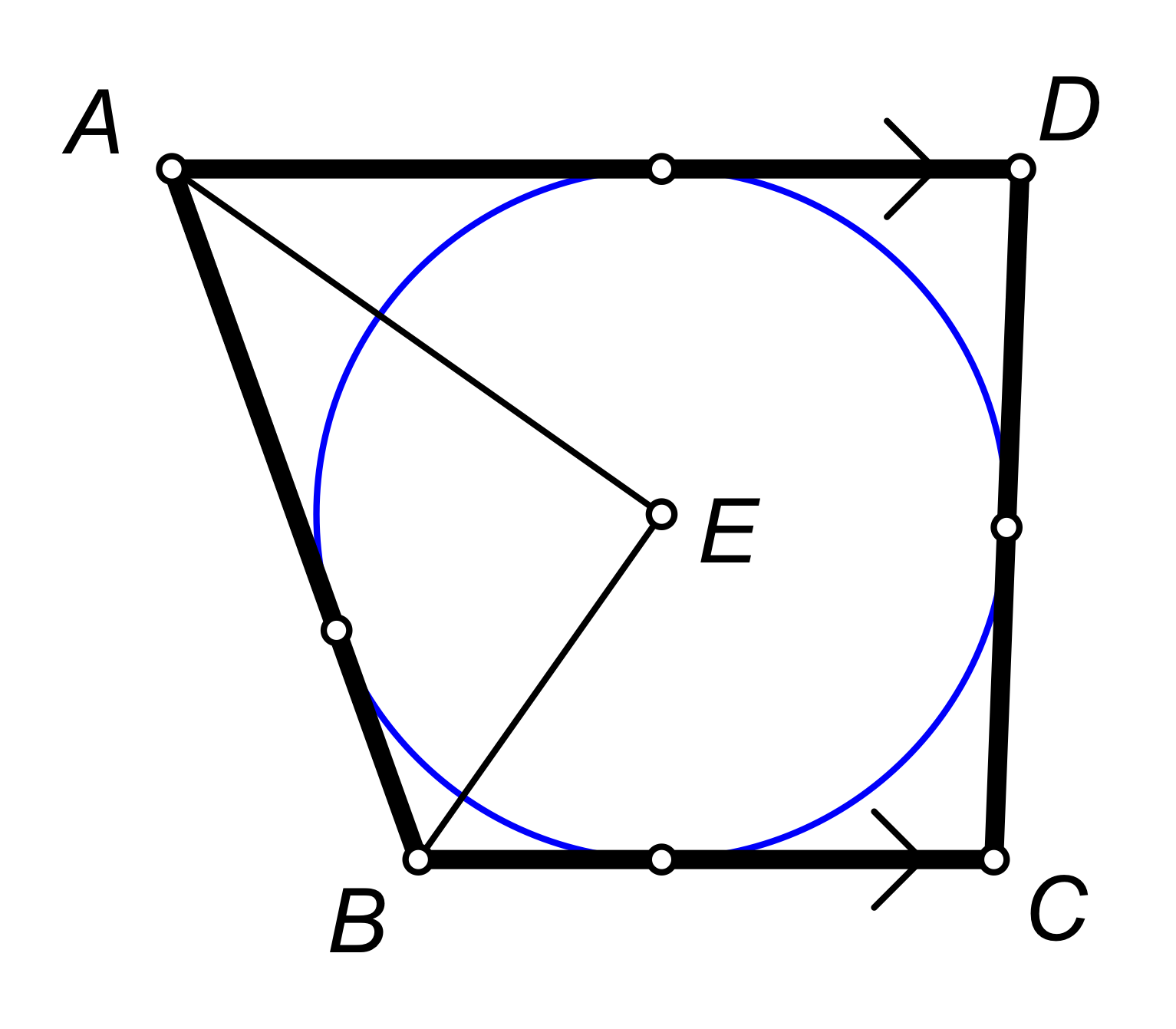}
\caption{$AE\perp BE$}
\label{fig:gpTangentialTrap}
\end{figure}

\begin{proof}
Since $AD\parallel BC$, $\angle BAD+\angle CBA=180\degrees$.
Since $E$ is the incenter, $AE$ bisects $\angle BAD$ and $BE$ bisects $\angle CBA$.
Therefore,
$$\angle EBA+\angle BAE=\frac12\left(\angle CBA+\angle BAD\right)=\frac12(180\degrees)=90\degrees.$$
The sum of the angles of $\triangle AEB$ is $180\degrees$.
Hence $\angle AEB=90\degrees$, so $AE\perp BE$.
\end{proof}

\newpage

\begin{lemma}
\label{lemma:biTrapIsos}
Let $ABCD$ be a bicentric trapezoid with $AD\parallel BC$.
Let $E$ be the incenter of $ABCD$.
Then $AE=DE$ and $BE=CD$ (Figure~\ref{fig:gpBiTrapIsos}).
\end{lemma}

\begin{figure}[h!t]
\centering
\includegraphics[width=0.35\linewidth]{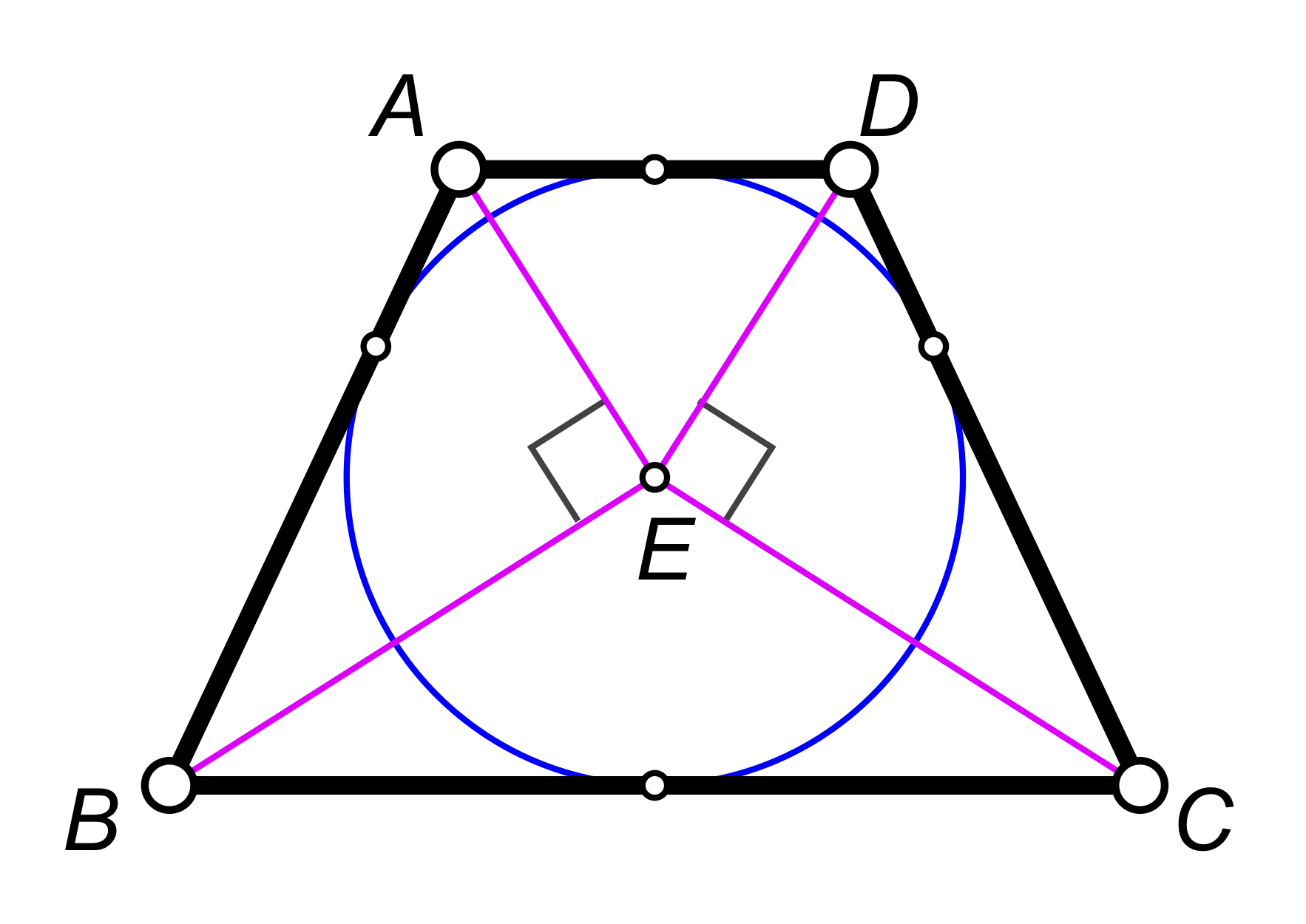}
\caption{$BE=CE$}
\label{fig:gpBiTrapIsos}
\end{figure}

\begin{proof}
By Lemma~\ref{lemma:biTrapPerp}, $AE\perp BE$ and $DE\perp CE$.
By Lemma~\ref{lemma:cyclicTrapIsosceles}, $AB=DC$.
Thus $\triangle AEB\cong \triangle DEC$ and so $BE=CE$.
Similarly, $AE=DE$.
\end{proof}

\begin{lemma}
\label{lemma:bicentricTrapCentroid}
The centroid of a bicentric trapezoid coincides with its incenter.
\end{lemma}

\begin{figure}[h!t]
\centering
\includegraphics[width=0.4\linewidth]{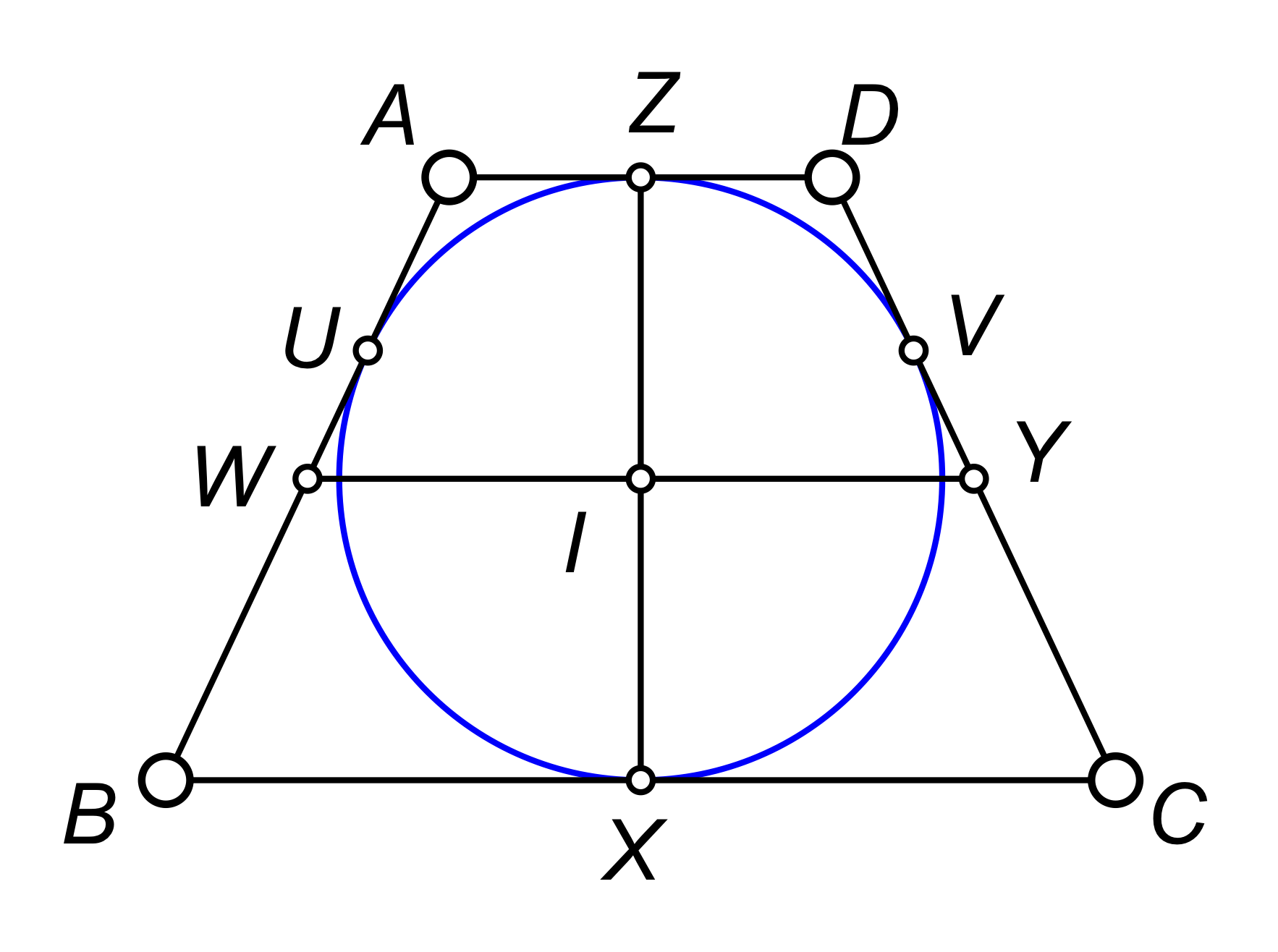}
\caption{$I$ is the centroid of bicentric trapezoid $ABCD$}
\label{fig:gpBicentricTrapCentroid}
\end{figure}

\begin{proof}
Let the bicentric trapezoid be $ABCD$ with $AD\parallel BC$.
Let the midpoints of the sides be $W$, $X$, $Y$, and $Z$ as shown in Figure~\ref{fig:gpBicentricTrapCentroid}.
Since a bicentric quadrilateral is cyclic, by Lemma~\ref{lemma:cyclicTrapIsosceles}, $AB=DC$.
By definition, the centroid of quadrilateral $ABCD$ is point $I$, the intersection of $WY$ and $XZ$.
Since $WY\parallel AD\parallel BC$, and $WY$ bisects both $AB$ and $CD$, it must also bisect $XZ$.
Therefore, $IX=IZ$.
Since a bicentric quadrilateral is tangential, by Lemma~\ref{lemma:biTrapPerp}, $AI\perp BI$.
Hence $\triangle AIB$ is a right triangle and $IW$ is the median to the hypotenuse. Thus, $IW=\frac12AB$.
Similarly, $IY=\frac12CD$. Consequently, $IW=IY$.

Let $IU$ be the altitude to the hypotenuse of right triangle $AIB$.
Similarly, $\triangle CID$ is a right triangle and let $IV$ be the altitude to its hypotenuse.
Since $XZ$ is the perpendicular bisector of $AD$ and $BC$, we can conclude that $IA=ID$ and $IB=IC$.
Thus, $\triangle AIB\cong\triangle DIC$. Corresponding parts of congruent figures are congruent,
so $IU=IV$.
The two tangents to a circle from an external point are equal, so $AU=AZ$.
Triangles $AIU$ and $AIZ$ are congruent since $\angle IUA=\angle AZI=90\degrees$ and $AU=AZ$.
Hence, $IU=IZ$.

We have now shown that $IU=IZ=IV=IX$, so $I$ is the incenter of  $ABCD$.
\end{proof}

\newpage

\relbox{Relationship $[ABCD]=8[FGHI]$}

\begin{theorem}
\label{thm-bicentricTrapX402}
Let $E$ be the centroid of a bicentric trapezoid $ABCD$.
Let $F$, $G$, $H$, and $I$ be the $X_{402}$ points of $\triangle EAB$, $\triangle EBC$, 
$\triangle ECD$,  and $\triangle EDA$, respectively (Figure~\ref{fig:gpBicentricTrapX402}).
Then
$$[ABCD]=8[FGHI].$$
\end{theorem}

\begin{figure}[h!t]
\centering
\includegraphics[width=0.4\linewidth]{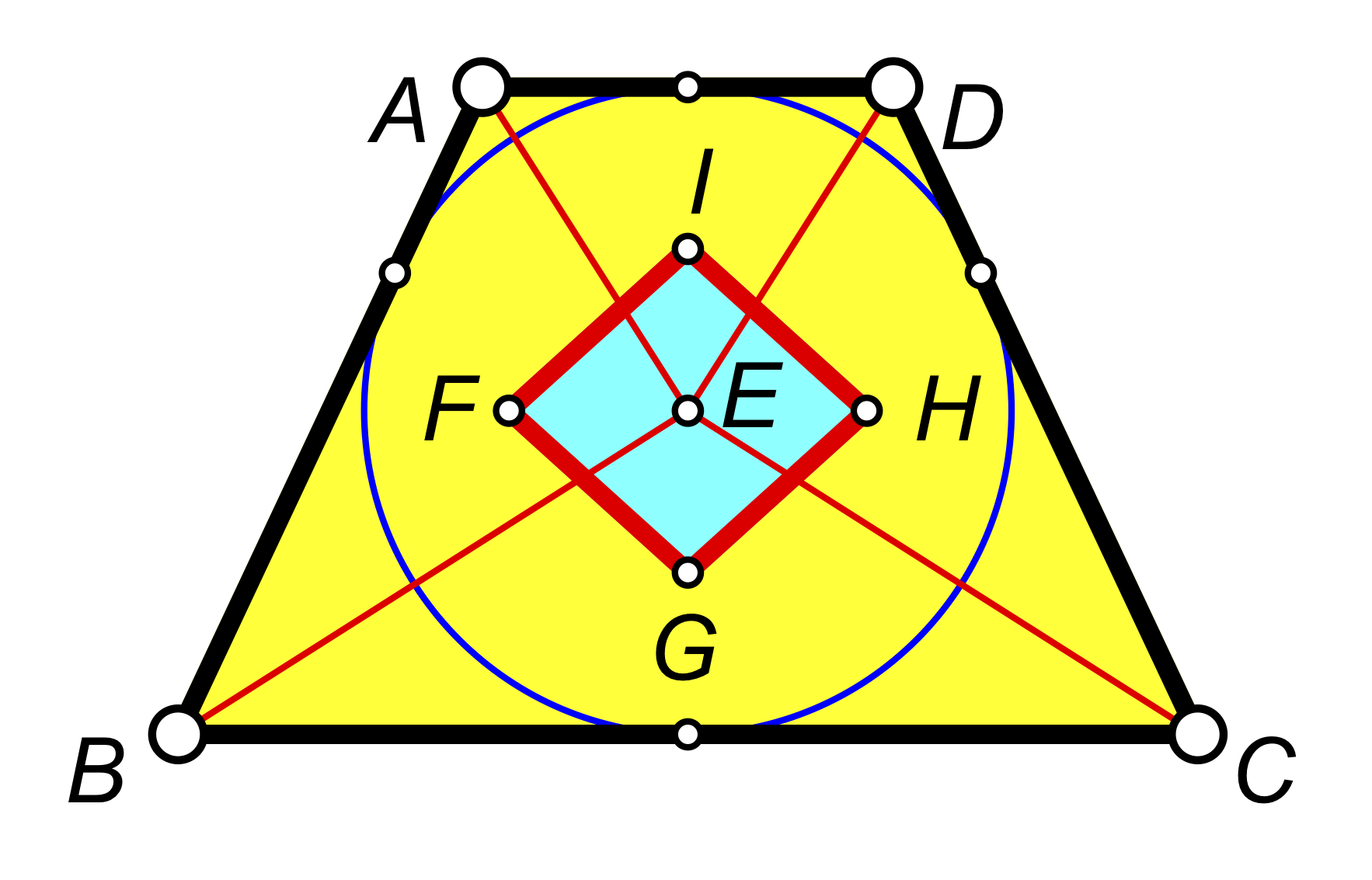}
\caption{bicentric trapezoid, $X_{402}\implies [ABCD]=8[FGHI]$}
\label{fig:gpBicentricTrapX402}
\end{figure}

\begin{proof}

Let $W$, $X$, $Y$, and $Z$ be the midpoints of the sides of the bicentric trapezoid
as shown in Figure~\ref{fig:gpBicentricTrapX402proof}.

\begin{figure}[h!t]
\centering
\includegraphics[width=0.4\linewidth]{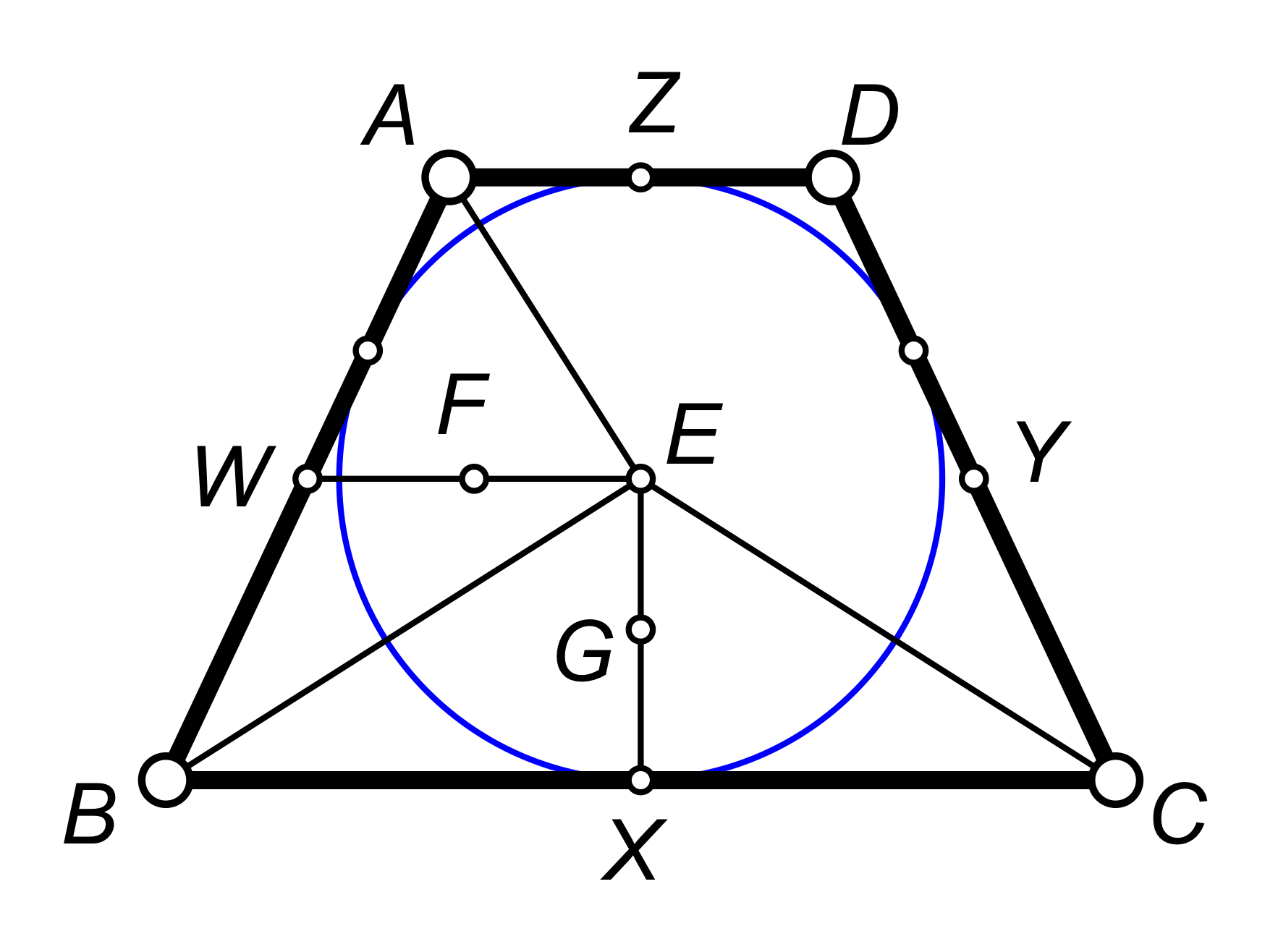}
\caption{}
\label{fig:gpBicentricTrapX402proof}
\end{figure}

Since $E$ is the centroid of $ABCD$, it lies on the bimedian $XZ$ which is the perpendicular bisector
of $BC$. Therefore $EB=EC$ and $\triangle EBC$ is isosceles.
By Theorem~\ref{thm:isoscelesRatio} and Table~\ref{table:diagonalPointIsoscelesRatios},
$G$ is the midpoint of $EX$.
By Lemma~\ref{lemma:biTrapPerp}, $EA\perp EB$, so $\triangle AEB$ is a right triangle.
Note that $W$ is the midpoint of the hypotenuse.
By Lemma~\ref{lemma:dpX402}, the $X_{402}$ point, $F$, coincides with the $X_5$ point.
By Lemma~\ref{lemma:midpointMedian}, the $X_5$ point coincides with the midpoint of
the median to the hypotenuse. Therefore, $F$ is the midpoint of $EW$.

Similarly, $H$ is the midpoint of $EY$ and $I$ is the midpoint of $EZ$.
Thus, quadrilateral $FGHI$ is homothetic to quadrilateral $WXYZ$ with ratio of similarity $\frac12$.
Thus $[WXYZ]=4[FGHI]$. But $WXYZ$ is the Varignon parallelogram of $ABCD$, so $[ABCD]=2[WXYZ]$.
Consequently, $[ABCD]=2[WXYZ]=8[FGHI]$.
\end{proof}

\newpage

\relbox{Relationship $[ABCD]=\frac12[FGHI]$}

\begin{proposition}[$X_{74}$ Property of a Right Triangle]
\label{proposition:gpX74rightTriangle}
Let $ABC$ be a right triangle with right angle at C.
Let $P$ be the $X_{74}$ point of $\triangle ABC$.
Then $BCAP$ is an orthogonal kite and $CP=2ab/c$ (Figure~\ref{fig:gpX74rightTriangle}).
\end{proposition}

\begin{figure}[h!t]
\centering
\includegraphics[width=0.3\linewidth]{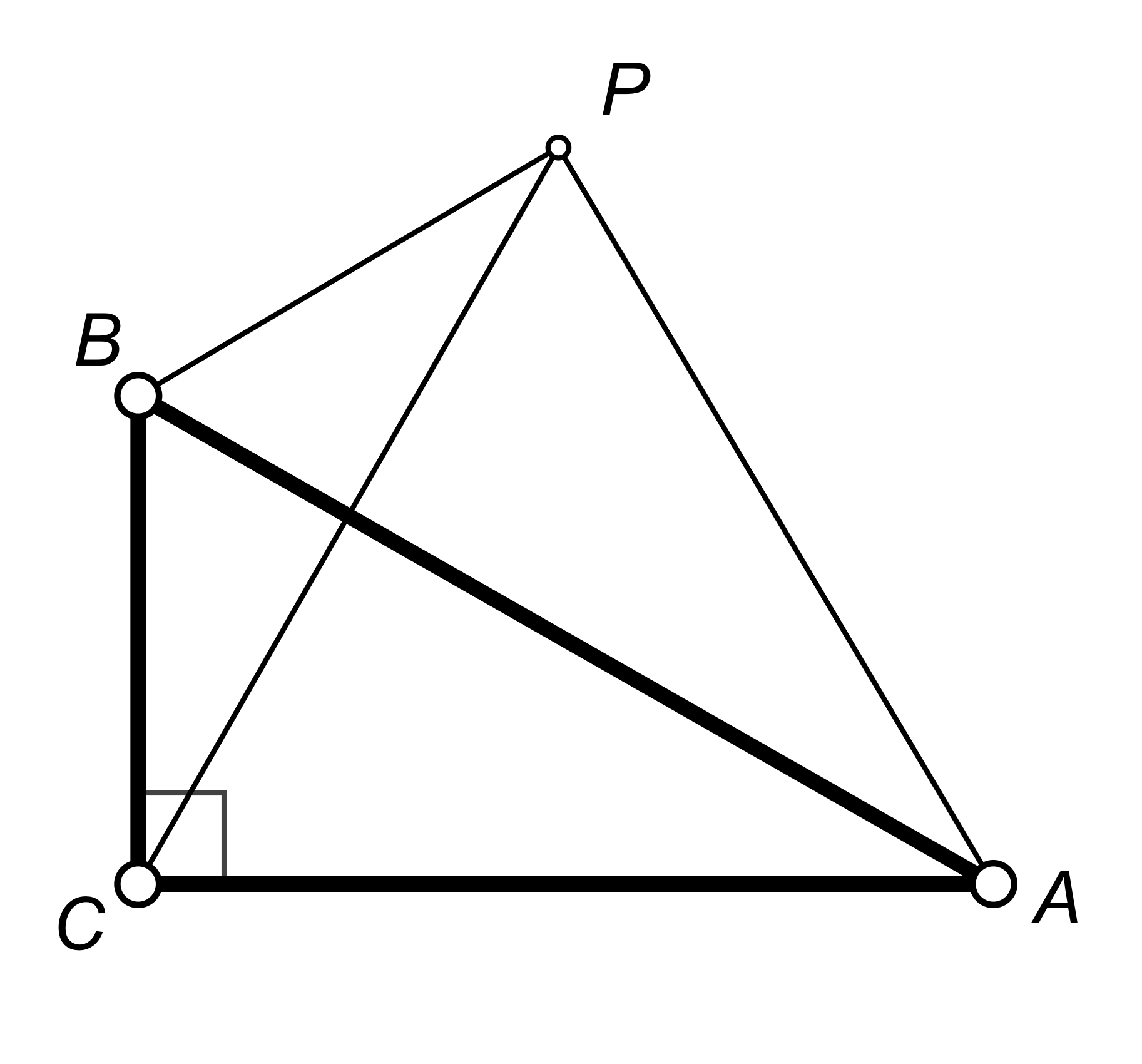}
\caption{$X_{74}$ point of a right triangle}
\label{fig:gpX74rightTriangle}
\end{figure}

\begin{proof}
According to \cite{ETC74}, the barycentric coordinates for the $X_{74}$ point of a triangle are
$$\Bigl(a^2 \left(a^4-2 a^2 b^2+a^2 c^2+b^4+b^2 c^2-2 c^4\right) \left(a^4+a^2 b^2-2 a^2 c^2-2 b^4+b^2 c^2+c^4\right):\,:\Bigr).$$
With the condition that $a^2+b^2=c^2$, this simplifies to
$$P=\Bigl(2a^2:2b^2:c^2\Bigr).$$
Using the Distance Formula and the condition $a^2+b^2=c^2$, we find that $BP=a$. Similarly, $AP=b$.
Thus, $BCAP$ is an orthogonal kite.
The length $CP$ is therefore twice the length of the altitude from $A$, which has length $ab/c$.
\end{proof}

\begin{proposition}[$X_{74}$ Property of an Isosceles Triangle]
\label{proposition:gpX74isoscelesTriangle}
Let $ABC$ be an isosceles triangle with $AC=BC$.
Let $P$ be the $X_{74}$ point of $\triangle ABC$.
Then $BCAP$ is a cyclic kite and $CP=b^2/CM$ where $M$ is the midpoint of $AB$(Figure~\ref{fig:gpX74isoscelesTriangle}).
\end{proposition}

\begin{figure}[h!t]
\centering
\includegraphics[width=0.25\linewidth]{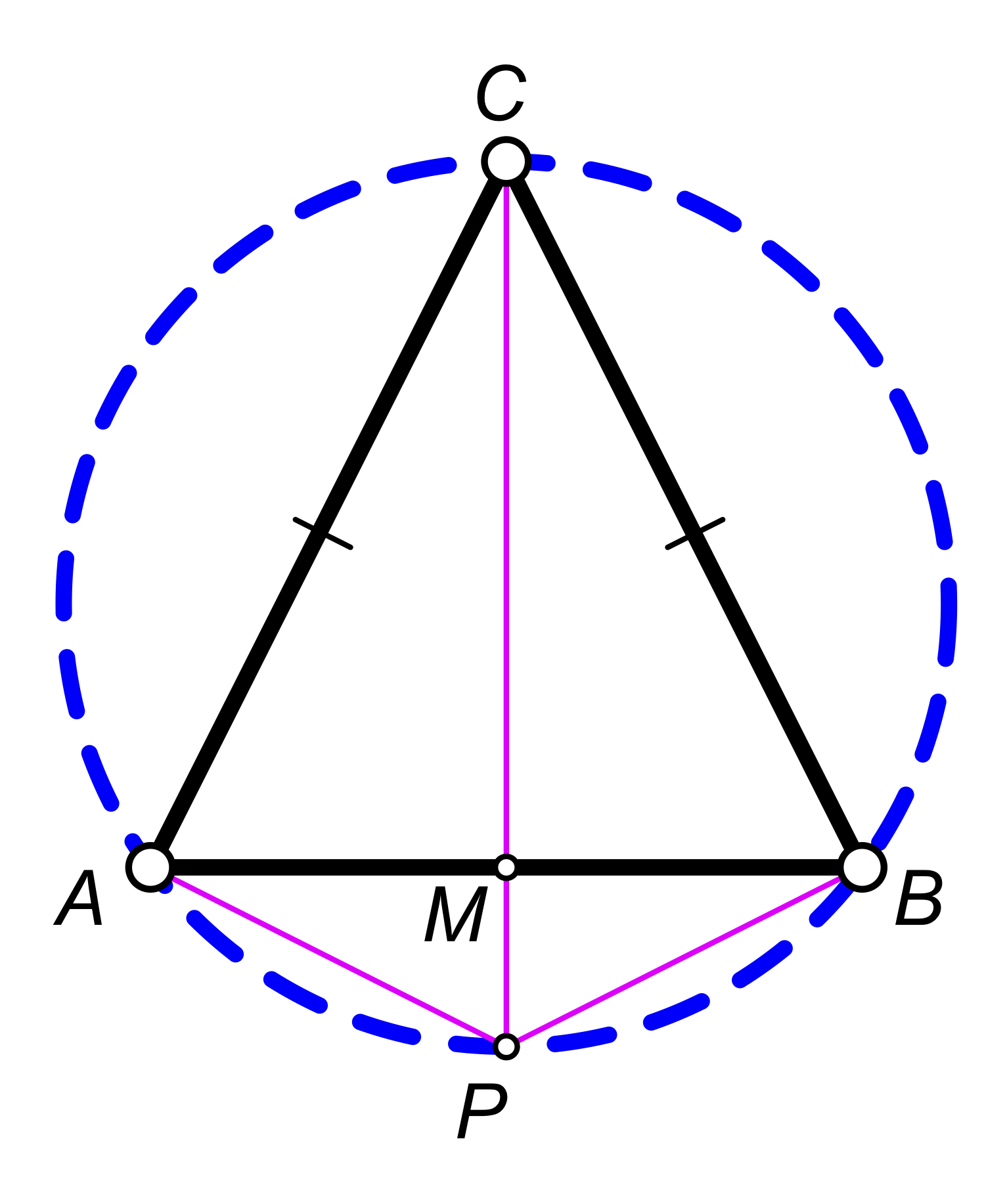}
\caption{$X_{74}$ point of a right triangle}
\label{fig:gpX74isoscelesTriangle}
\end{figure}

\begin{proof}
The barycentric coordinates for the $X_{74}$ point of a triangle were given in the proof of
Proposition~\ref{proposition:gpX74rightTriangle}.
With the condition that $a=b$, this simplifies to
$$P=\Bigl(2b^2,2b^2,-c^2\Bigr).$$
Let $M$ be the midpoint of $AB$, so that $M=(1:1:0)$ and $AM=c/2$.
Using the Distance Formula and the condition $a^2+b^2=c^2$, we find that $CM\times MP$
simplifies to $c^2/4$. But this is equal to $AM\times BM$.
Thus, $BCAP$ is cyclic.
\end{proof}

Note that all kites are tangential, so $BCAP$ is actually a bicentric kite.

\begin{theorem}
\label{thm-bicentricTrapX74}
Let $E$ be the centroid of a bicentric trapezoid $ABCD$.
Let $F$, $G$, $H$, and $I$ be the $X_{74}$ points of $\triangle EAB$, $\triangle EBC$, 
$\triangle ECD$,  and $\triangle EDA$, respectively (Figure~\ref{fig:gpBicentricTrapX74}).
Then $FGHI$ is a kite and
$$[ABCD]=\frac12[FGHI].$$
\end{theorem}

\begin{figure}[h!t]
\centering
\includegraphics[width=0.4\linewidth]{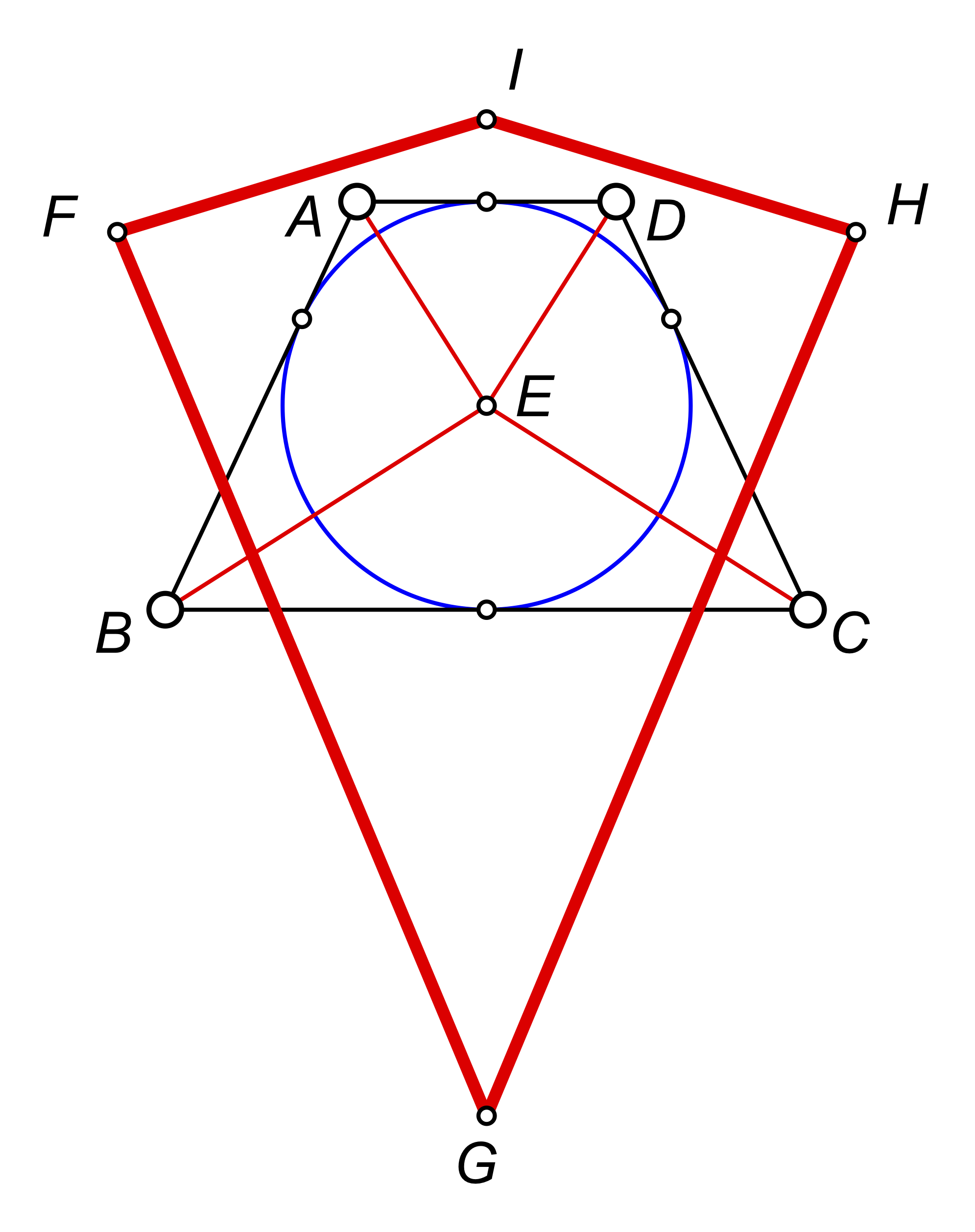}
\caption{Bicentric trapezoid, $X_{74}$ points $\implies [ABCD]=\frac12[FGHI]$}
\label{fig:gpBicentricTrapX74}
\end{figure}

\begin{proof}
Let $BC=2a$, $AD=2b$, and let $r$ and $p$ be the inradius and the semiperimeter of $ABCD$.
Let $X$, $Y$, $Z$, and $T$ be the points where the incircle touches the sides of $ABCD$
as shown in Figure~\ref{fig:gpBiTrap1}.

\begin{figure}[h!t]
\centering
\includegraphics[scale=0.6]{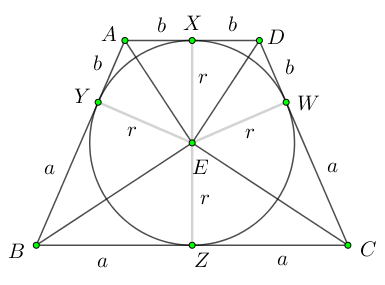}
\caption{bicentric trapezoid and incircle}
\label{fig:gpBiTrap1}
\end{figure}

We have $[ABCD]=rp=2r(a+b)$ and $r^2=ab$. We will now calculate the area of $FGHI$.
See Figure~\ref{fig:gpBiTrap2}.

\begin{figure}[h!t]
\centering
\includegraphics[scale=0.5]{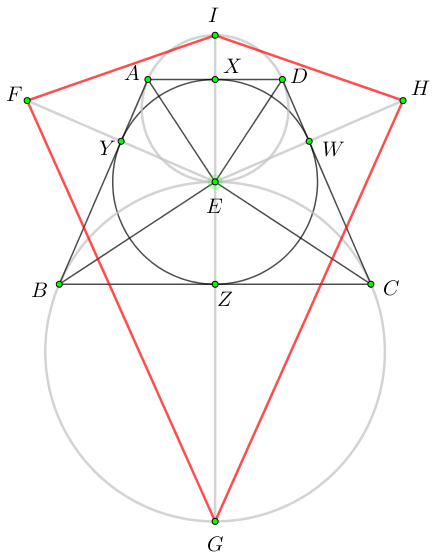}
\caption{Bicentric trapezoid and incircle}
\label{fig:gpBiTrap2}
\end{figure}

By Proposition~\ref{proposition:gpX74isoscelesTriangle}, we have $I\in \odot(AED)$ and $G\in\odot(BCE)$.
Therefore, by the intersecting chords theorem, we get
$$
   IX\cdot XE=AX\cdot XD \quad \Rightarrow \quad IX=\frac{AX\cdot XD}{EX}=\frac{b^2}{r}
$$
and
$$
   GZ\cdot ZE=BZ\cdot ZC \quad \Rightarrow \quad GZ=\frac{BZ\cdot ZC}{ZE}=\frac{a^2}{r}.
$$
Hence,
$$
 EI=EX+XI=r+\frac{b^2}{r}=\frac{r^2+b^2}{r}=\frac{ab+b^2}{r}  
$$
and
$$
   EG=EZ+ZG=r+\frac{a^2}{r}=\frac{r^2+a^2}{r}=\frac{ab+a^2}{r}.
$$
By Proposition~\ref{proposition:gpX74rightTriangle}, we have $EF=2r$, so
$$
   [IEF]=\frac{1}{2}EI\cdot EF\cdot \sin\left(\angle IEF\right)=\frac{1}{2}
\left(\frac{ab+b^2}{r}\right)\cdot 2r \sin B=b(a+b)\sin\left(\angle IEF\right).
$$
Since $\angle BYE=\angle EZB=90\degrees$, quadrilateral $BYEZ$ is cyclic and so
$\angle IEF=\angle B$.
The value $\sin B$ is the height of $A$ above $BC$ divided by $AB$, so $\sin B=2r/(a+b)$.
Using this relation, we obtain
$$
   [IEF]=b(a+b)\cdot \frac{2r}{a+b}=2rb.
$$
Similarly, we have $[FEG]=2ra$. Finally,
$$
   [FGHI]=2\left([IEF]+[FEG]\right)=2\left(2rb+2ra\right)=4r(a+b)=2[ABCD],
$$
and we are done.
\end{proof}

\begin{lemma}
\label{lemma:isoscelesX477}
The $X_{477}$ point of an isosceles triangle coincides with its $X_{74}$ point.
\end{lemma}

\begin{proof}
With the condition that $a=b$, the barycentric coordinates for the $X_{477}$ point simplifies to
$$P=\Bigl(2b^2,2b^2,-c^2\Bigr).$$
These are the same coordinates as the $X_{74}$ point in an isosceles triangle.
\end{proof}

\begin{lemma}
\label{lemma:rightTriangleX477}
The $X_{477}$ point of right triangle is the reflection of its $X_{74}$ point
about the median to the hypotenuse (Figure~\ref{fig:gpRightTriangleX477}).
\end{lemma}

\begin{figure}[h!t]
\centering
\includegraphics[width=0.35\linewidth]{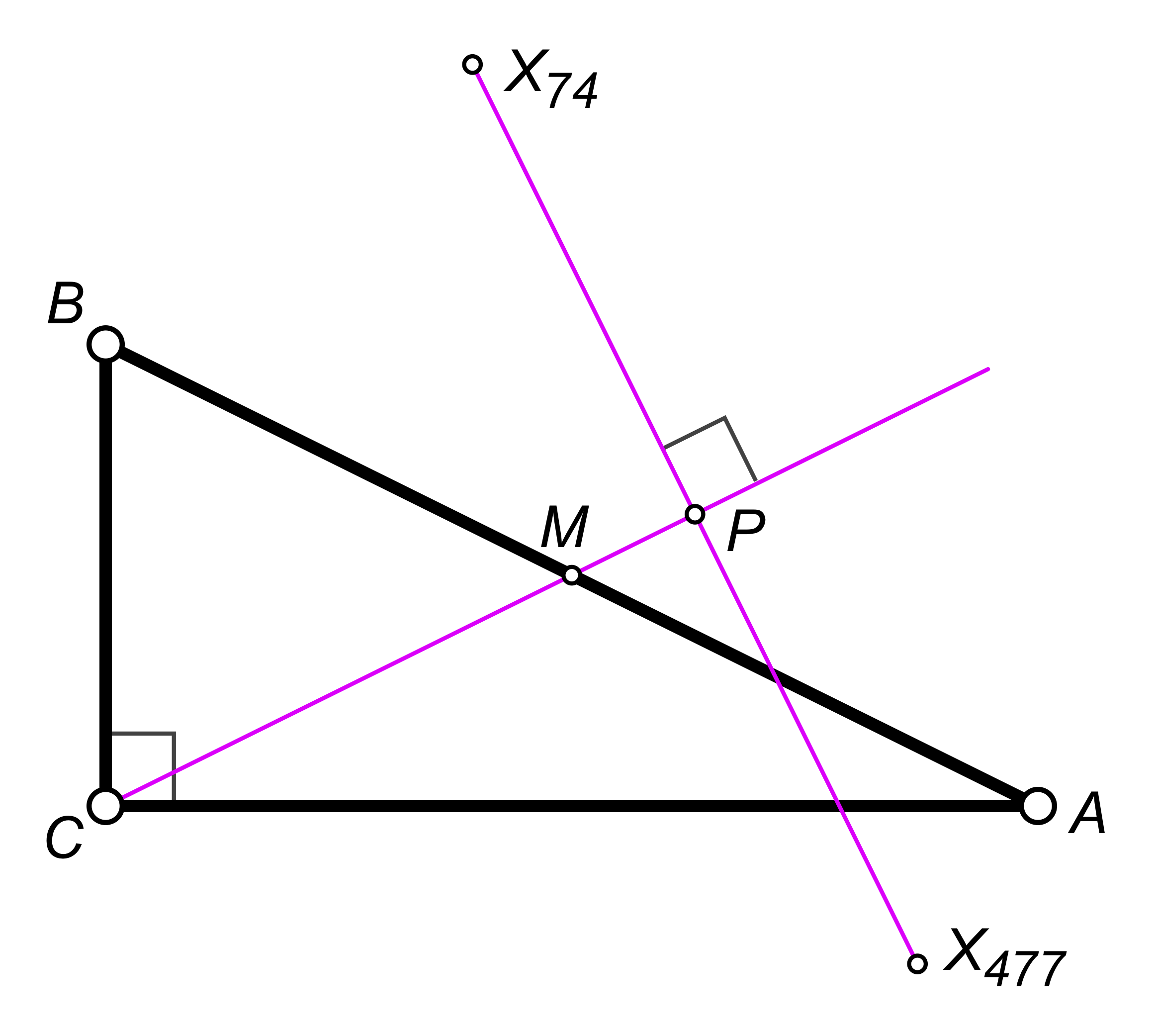}
\caption{$X_{74}$ and $X_{477}$ points in a right triangle}
\label{fig:gpRightTriangleX477}
\end{figure}

\begin{proof}
Let the right triangle be $\triangle ABC$ with right angle at $C$.
Let $M=(1:1:0)$ be the midpoint of the hypotenuse.
The equation of line $CM$ is $x=y$.
The distance formula shows that $CX_{74}=2ab/c$ and $CX_{477}=2ab/c$, so $CX_{74}=CX_{477}$.
We need only show that the midpoint, $P$, of $X_{74}X_{477}$ lies on $CM$.
Calculating the midpoint of $X_{74}$ and $X_{477}$ subject to the condition
$a^2=b^2+c^2$, we find that
$$P=\left(4a^2b^2:4a^2b^2:a^4-6a^2b^2+b^4\right).$$
Since the $x$ and $y$ components are equal, this proves that $P$ lies on $CM$.
\end{proof}

%\newpage

\begin{theorem}
\label{thm-bicentricTrapX477}
Let $E$ be the centroid of a bicentric trapezoid $ABCD$.
Let $F$, $G$, $H$, and $I$ be the $X_{477}$ points of $\triangle EAB$, $\triangle EBC$, 
$\triangle ECD$,  and $\triangle EDA$, respectively (Figure~\ref{fig:gpBicentricTrapX477}).
Then $FGHI$ is a kite with $FI=HI$, $FG=HG$, and
$$[ABCD]=\frac12[FGHI].$$
\end{theorem}

\begin{figure}[h!t]
\centering
\includegraphics[width=0.4\linewidth]{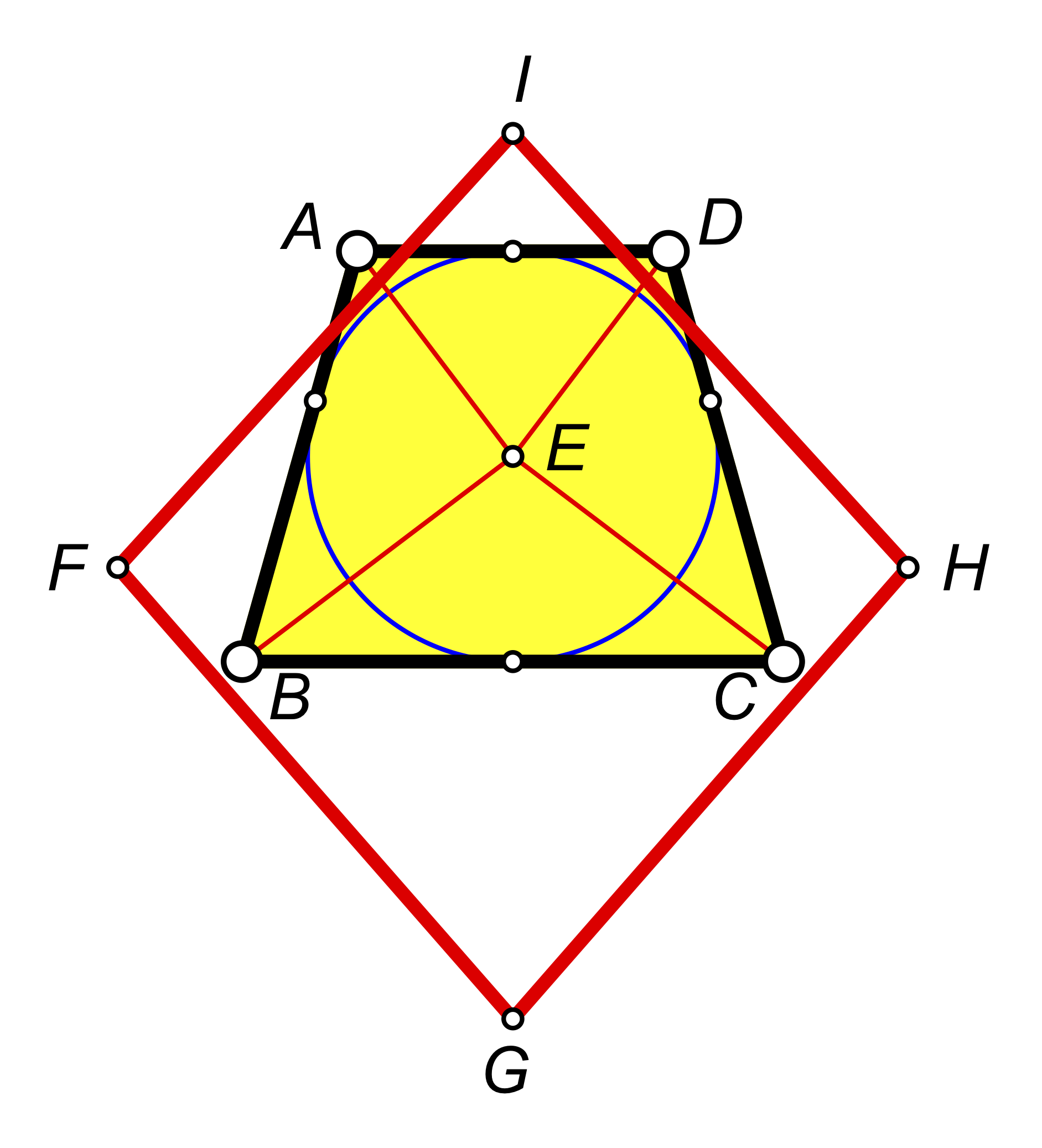}
\caption{Bicentric trapezoid, $X_{477}$ points $\implies [ABCD]=\frac12[FGHI]$}
\label{fig:gpBicentricTrapX477}
\end{figure}

\begin{proof}
Let $L$ be the line through $E$ parallel to $BC$.
Note that $\triangle AEB$ is a right triangle and the median to the hypotenuse is parallel to $BC$.
Let $F'$, $G'$, $H'$, and $I'$ be the $X_{74}$ points of
the radial triangles (Figure~\ref{fig:gpBicentricTrapX477X74}).

\begin{figure}[h!t]
\centering
\includegraphics[width=0.35\linewidth]{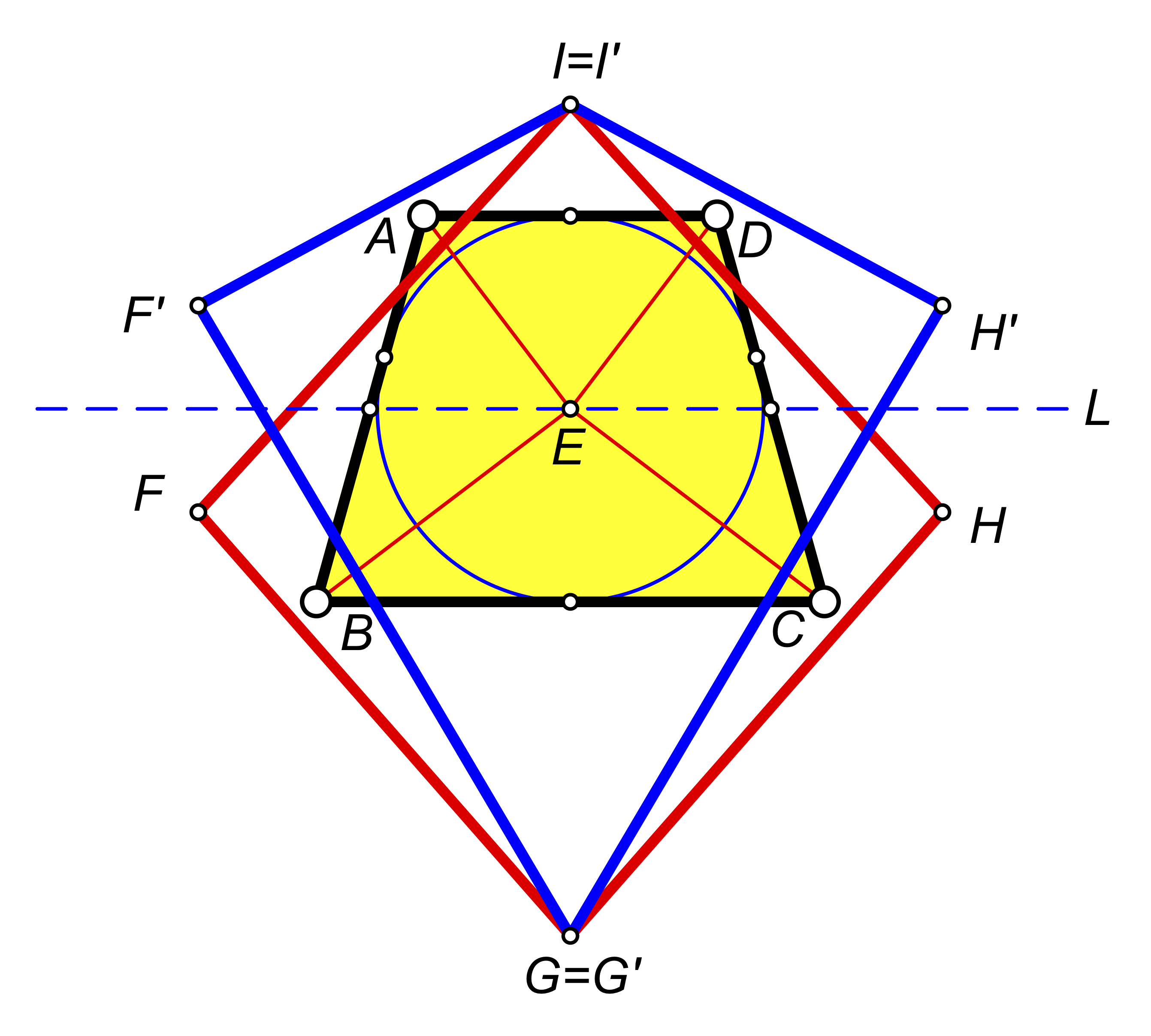}
\caption{Bicentric trapezoid with $X_{477}$ points and $X_{74}$ points}
\label{fig:gpBicentricTrapX477X74}
\end{figure}

By Lemma~\ref{lemma:isoscelesX477}, $G'$ coincides with $G$ and $I'$ coincides with $I$.
By Lemma~\ref{lemma:rightTriangleX477}, $F'$ is the reflection of $F$ about $L$
and $H'$ is the reflection of $H$ about $L$.
Thus, $FH=F'H'$.
By Lemma~\ref{lemma:orthodiag},
$$[FGHI]=\frac12FH\cdot GI=\frac12F'H'\cdot G'I'=[F'G'H'I'].$$
Thus,
$$[ABCD]=\frac12[F'G'H'I']=\frac12[FGHI]$$
by Theorem~\ref{thm-bicentricTrapX74}.
\end{proof}

\relbox{Relationship $[ABCD]=2[FGHI]$}

\begin{theorem}
\label{thm-bicentricTrapX122}
Let $E$ be the centroid of a bicentric trapezoid $ABCD$.
Let $n$ be 122, 123, 127, or 339.
Let $F$, $G$, $H$, and $I$ be the $X_{n}$ points of $\triangle EAB$, $\triangle EBC$, 
$\triangle ECD$,  and $\triangle EDA$, respectively.
Then $FGHI$ is the Varignon parallelogram of $ABCD$ (Figure~\ref{fig:gpBicentricTrapX127}) and
$$[ABCD]=2[FGHI].$$
\end{theorem}

\begin{figure}[h!t]
\centering
\includegraphics[width=0.4\linewidth]{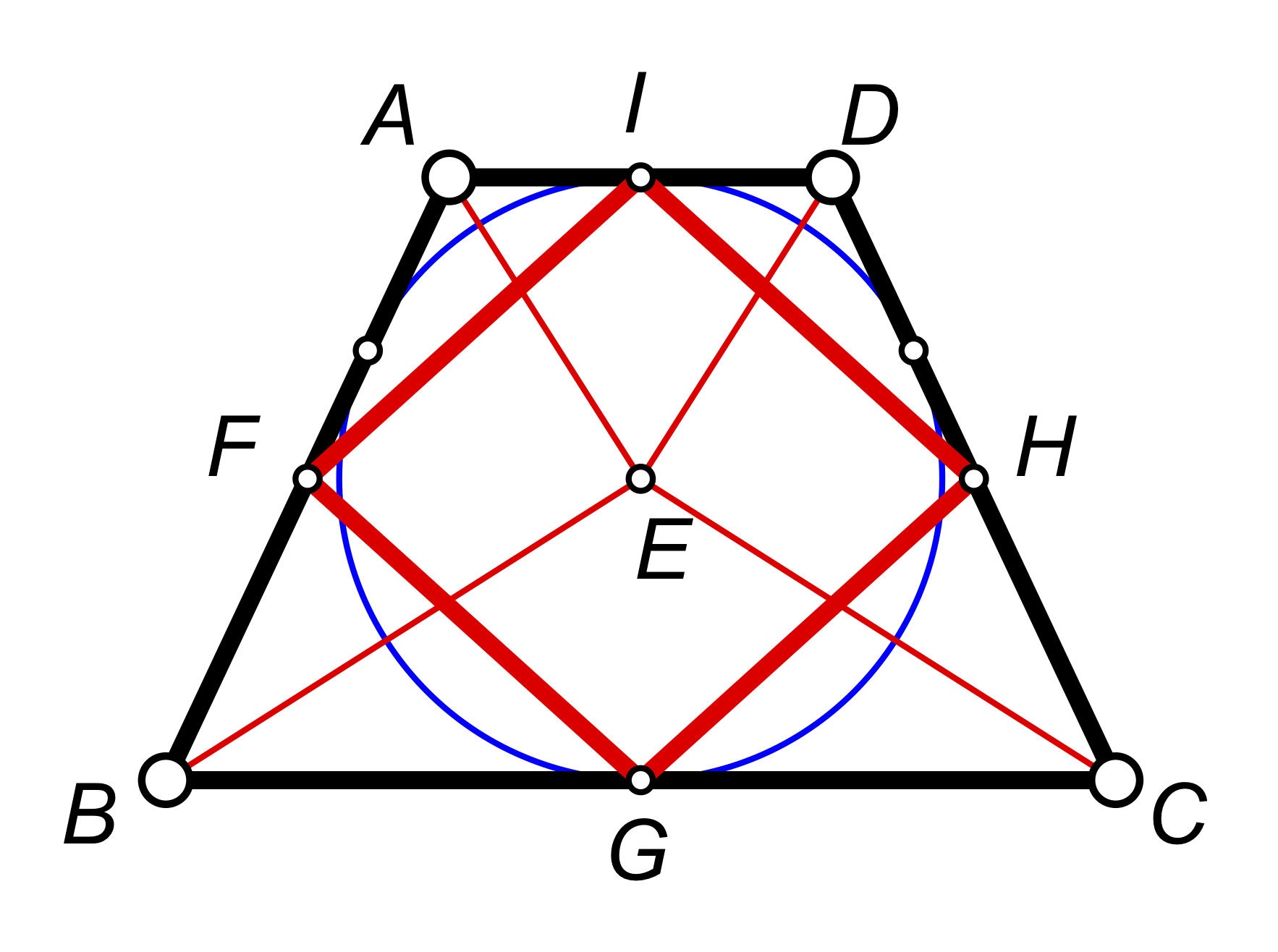}
\caption{Bicentric trapezoid, $X_{127}$ points $\implies [ABCD]=\frac12[FGHI]$}
\label{fig:gpBicentricTrapX127}
\end{figure}

\begin{proof}
By Lemma~\ref{lemma:biTrapPerp}, $\triangle AEB$ is a right triangle.
By Lemma~\ref{thm:atE}, $F$ is the midpoint of $AB$.
Similarly, $H$ is the midpoint of $CD$.
By Lemma~\ref{lemma:dpIsoscelesTriangleMidpoint}, $\triangle BEC$ is isosceles.
By Lemma~\ref{lemma:dpIsoscelesTriangleMidpoint}, $G$ is the midpoint of $BC$.
Similarly, $I$ is the midpoint of $AD$.
Hence $FGHI$ is the Varignon parallelogram of $ABCD$ and
$$[ABCD]=2[FGHI].$$
\end{proof}

\newpage

%**************************************
%   Anticenter
%**************************************

\section{Anticenter}

In this section, we examine central quadrilaterals formed from the anticenter
of the reference quadrilateral. Note that only cyclic quadrilaterals have anticenters.
A \emph{maltitude} of a quadrilateral is a line from the midpoint of one side
perpendicular to the opposite side (Figure~\ref{fig:apMaltitude}).

\begin{figure}[h!t]
\centering
\includegraphics[width=0.4\linewidth]{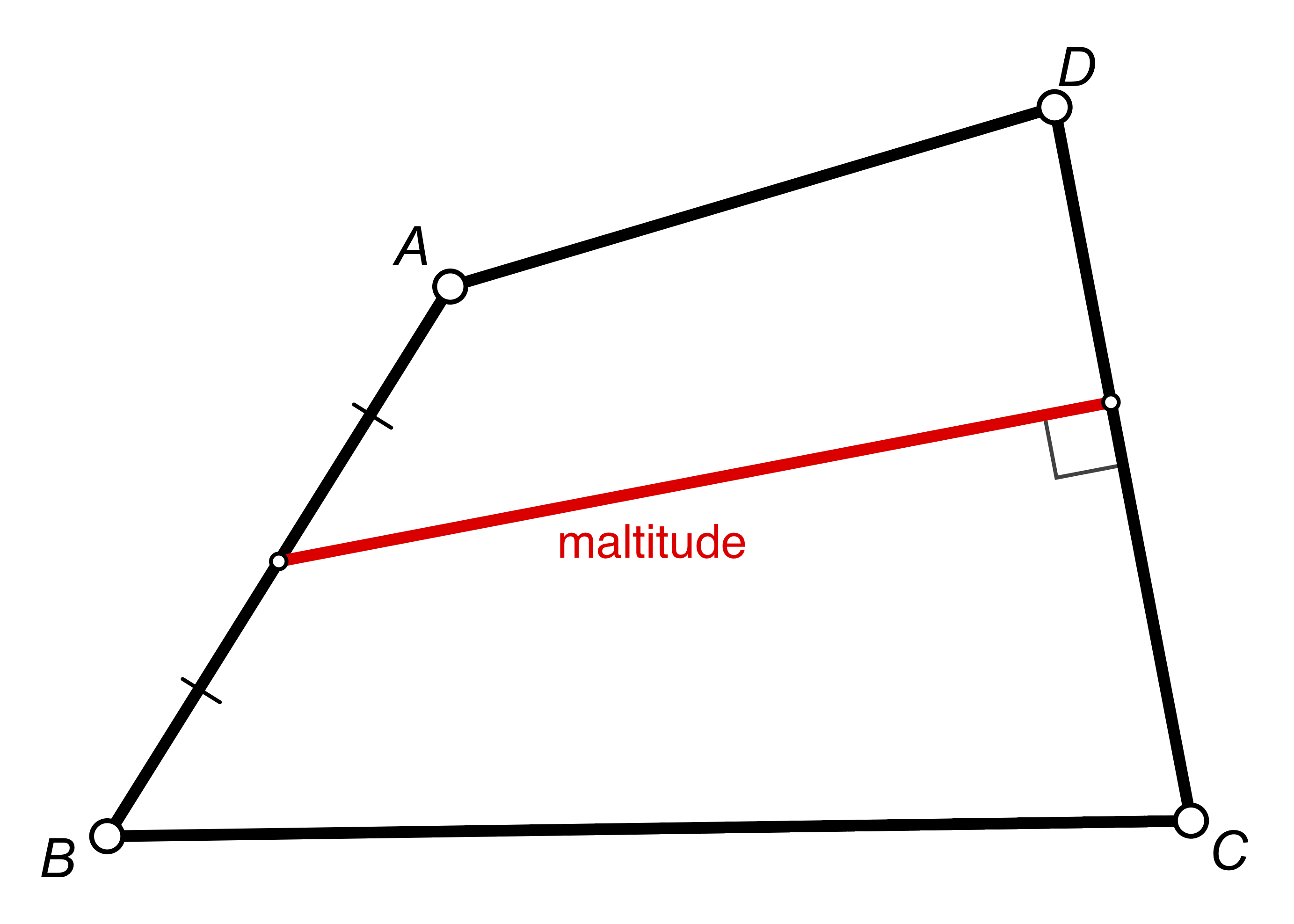}
\caption{Maltitude of a quadrilateral}
\label{fig:apMaltitude}
\end{figure}

The four maltitudes of a cyclic quadrilateral
concur at a point called the \emph{anticenter} of the quadrilateral (Figure~\ref{fig:apAnticenter}).

\begin{figure}[h!t]
\centering
\includegraphics[width=0.4\linewidth]{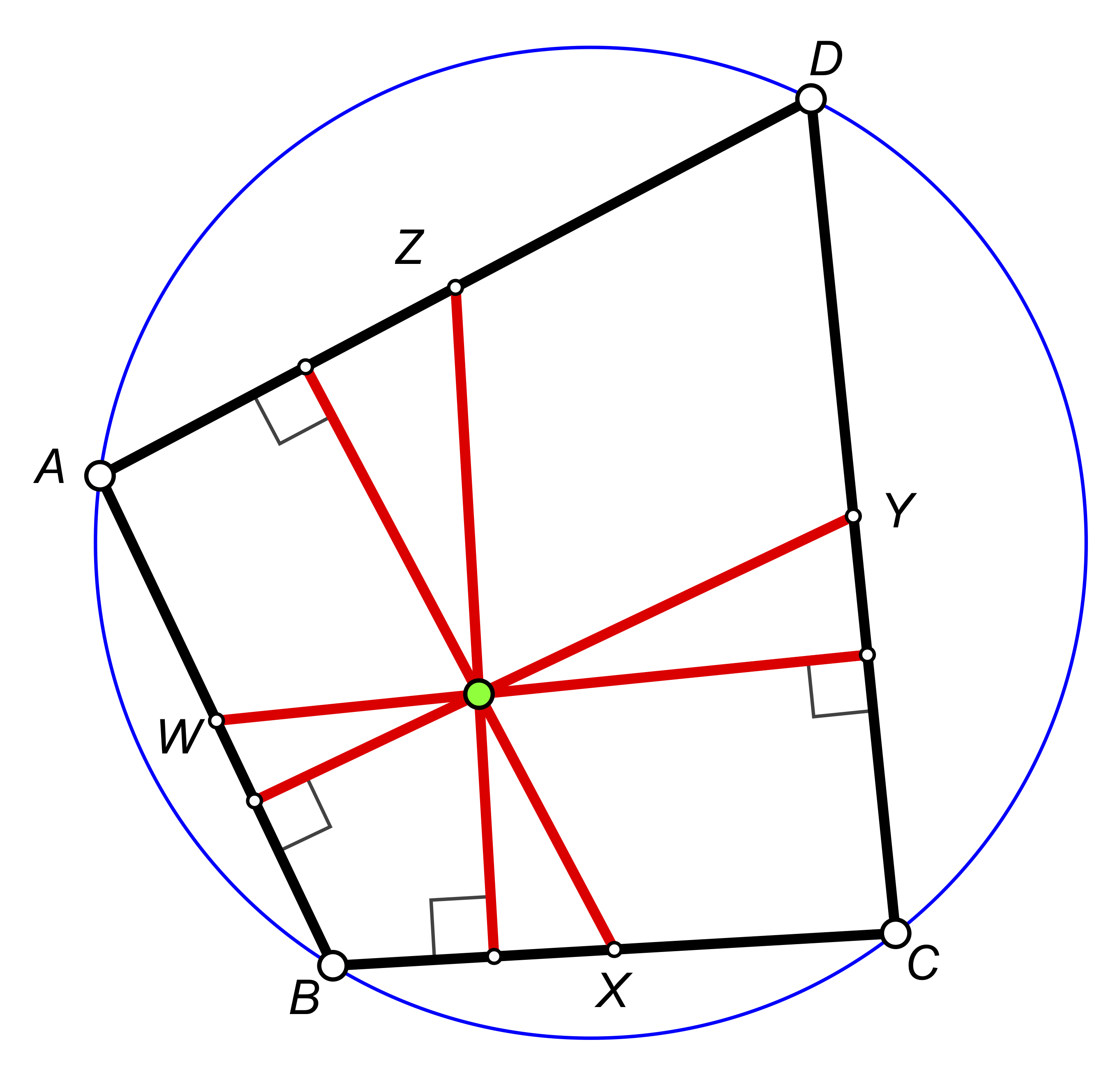}
\caption{Anticenter of a cyclic quadrilateral}
\label{fig:apAnticenter}
\end{figure}

The following result is well known, \cite{MathWorld-Brahmagupta}.

\begin{lemma}[Brahmagupta's Theorem]
The anticenter of a cyclic orthodiagonal quadrilateral coincides with the diagonal point
(Figure~\ref{fig:apBrahmagupta}).
\end{lemma}

\begin{figure}[h!t]
\centering
\includegraphics[width=0.4\linewidth]{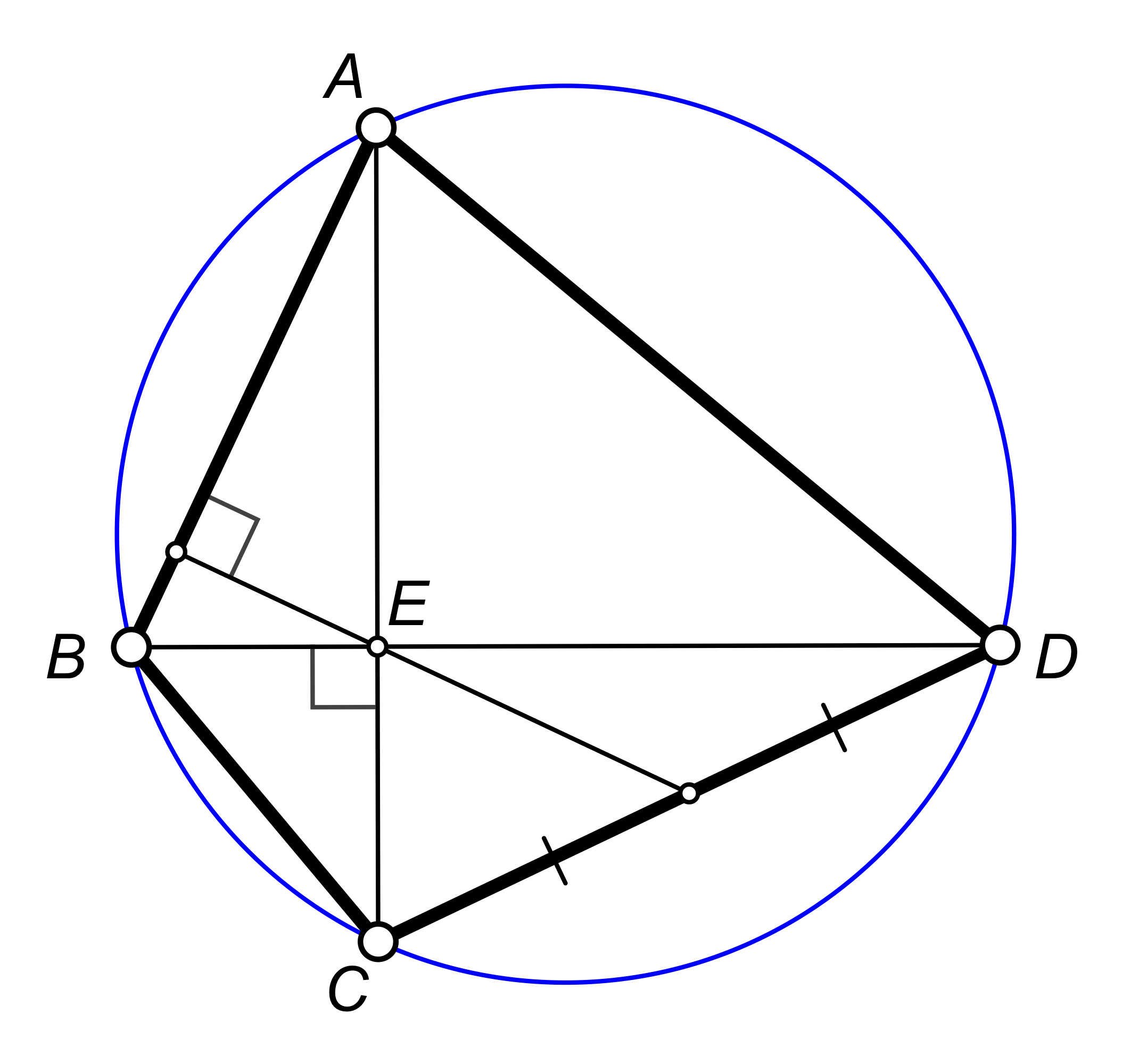}
\caption{Brahmagupta's Theorem}
\label{fig:apBrahmagupta}
\end{figure}

The following result is well known, \cite{QA-P2}.

\begin{lemma}
The anticenter of a cyclic quadrilateral coincides with the Poncelet point.
\end{lemma}

Our computer study examined the central quadrilaterals formed by the anticenter.
Since the anticenter of a cyclic quadrilateral coincides with the Poncelet point,
we omit results for cyclic quadrilaterals.
\void{
Since the anticenter of a cyclic orthodiagonal quadrilateral coincides with the diagonal point,
we omit results for cyclic orthodiagonal quadrilaterals.
}
We checked the central quadrilateral for all the first 1000 triangle centers (omitting points
at infinity) and all reference quadrilateral shapes listed in Table~\ref{table:quadrilaterals}
that are cyclic.

No other results were found.
\medskip
	
\begin{table}[ht!]
\caption{}
\begin{center}
\begin{tabular}{|l|l|p{2.2in}|}
\hline
\multicolumn{3}{|c|}{\color{blue}\textbf{\large \strut Central Quadrilaterals formed by the Anticenter}}\\ \hline
\textbf{Quadrilateral Type}&\textbf{Relationship}&\textbf{centers}\\ \hline
\multicolumn{3}{|c|}{No relationships were found}\\
\hline
\end{tabular}
\end{center}
\end{table}

%**************************************
%   Orthocenter
%**************************************

\section{Orthocenter}

In this section, we examine central quadrilaterals formed from the orthocenter
of the reference quadrilateral. Note that only cyclic quadrilaterals have orthocenters.
We define an \emph{altitude} of a quadrilateral to be a line from a vertex to the orthocenter
of the triangle formed by the other three vertices. The four altitudes of a cyclic quadrilateral
concur at a point that we will call the \emph{orthocenter} of the quadrilateral.
(These are nonstandard definitions.)

Our computer study examined the central quadrilaterals formed by the orthocenter.
Since the orthocenter of a rectangle coincides with the diagonal point,
we omit results for rectangles.
We checked the central quadrilateral for all the first 1000 triangle centers (omitting points
at infinity) and all reference quadrilateral shapes listed in Table~\ref{table:quadrilaterals}
that are cyclic.

No other results were found
\medskip

\begin{table}[ht!]
\caption{}
\begin{center}
\begin{tabular}{|l|l|p{2.2in}|}
\hline
\multicolumn{3}{|c|}{\textbf{\color{blue}\large \strut Central Quadrilaterals formed by the Orthocenter}}\\ \hline
\textbf{Quadrilateral Type}&\textbf{Relationship}&\textbf{centers}\\ \hline
\multicolumn{3}{|c|}{No relationships were found.}\\
\hline
\end{tabular}
\end{center}
\end{table}

\newpage

%**************************************
%   Incenter
%**************************************

\section{Incenter}

In this section, we examine central quadrilaterals formed from the incenter
of the reference quadrilateral. Note that only tangential quadrilaterals have incenters.
A \emph{tangential quadrilateral} in one in which a circle can be inscribed, touching all
four sides. The center of this circle is called the \emph{incenter} of the quadrilateral.
The circle is called the \emph{incircle}.

Our computer study examined the central quadrilaterals formed by the incenter.
Since the incenter of a rhombus coincides with the diagonal point,
we omit results for rhombi.
In a bicentric trapezoid, the incenter coincides with the centroid (Lemma~\ref{lemma:bicentricTrapCentroid}),
so we have excluded results for bicentric trapezoids that are true
when the radiator is the centroid.
We checked the central quadrilateral for all the first 1000 triangle centers (omitting points
at infinity) and all reference quadrilateral shapes listed in Table~\ref{table:quadrilaterals}
that are tangential.

No other results were found.
\medskip

\begin{table}[ht!]
\caption{}
\begin{center}
\begin{tabular}{|l|l|p{2.2in}|}
\hline
\multicolumn{3}{|c|}{\textbf{\color{blue}\large \strut Central Quadrilaterals formed by the Incenter}}\\ \hline
\textbf{Quadrilateral Type}&\textbf{Relationship}&\textbf{centers}\\ \hline
\multicolumn{3}{|c|}{No relationships were found.}\\
\hline
\end{tabular}
\end{center}
\end{table}

%**************************************
%   3rd Diagonal Midpoint
%**************************************

\section{Midpoint of the 3rd Diagonal}

In this section, we examine central quadrilaterals formed from the midpoint of the 3rd diagonal
of the reference quadrilateral. If $ABCD$ is a convex quadrilateral with no two sides parallel,
let $AB$ meet $CD$ at $P$ and let $BC$ meet $DA$ at $Q$. Then line segment $PQ$ is called
the \emph{3rd diagonal} of $ABCD$. Note that this line segment only exists when the reference quadrilateral
is not a trapezoid.

Our computer study examined the central quadrilaterals formed by the midpoint of the 3rd diagonal.
We checked the central quadrilateral for all the first 1000 triangle centers (omitting points
at infinity) and all reference quadrilateral shapes listed in Table~\ref{table:quadrilaterals}
that are not trapezoids.

No other results were found.
\medskip

\begin{table}[ht!]
\caption{}
\begin{center}
\begin{tabular}{|l|l|p{2.2in}|}
\hline
\multicolumn{3}{|c|}{\textbf{\color{blue}\large \strut Central Quadrilaterals formed by}}\\
\multicolumn{3}{|c|}{\textbf{\color{blue}\large \strut the midpoint of the 3rd diagonal}}\\ \hline
\textbf{Quadrilateral Type}&\textbf{Relationship}&\textbf{centers}\\ \hline
\multicolumn{3}{|c|}{No relationships were found.}\\
\hline
\end{tabular}
\end{center}
\end{table}

%**************************************
%   Miquel point
%**************************************

%\section{Miquel point}

\newpage
\goodbreak

%\printindex
\end{document}